\documentclass[12pt]{puthesis}

\pdfoutput=1

\usepackage{graphicx}
\usepackage{epsfig}
\usepackage{amsfonts}
\usepackage{amssymb}
\usepackage{amsmath}   
\usepackage{amsthm}
\usepackage{latexsym}
\usepackage{subfigure}
\usepackage{psfrag} 
\usepackage[vcentering,dvips]{geometry}

\newtheorem{theorem}{Theorem}[section]
\newtheorem{lemma}[theorem]{Lemma}
\newtheorem{corollary}[theorem]{Corollary}
\newtheorem{proposition}[theorem]{Proposition}

\newtheorem{observation}[theorem]{Observation}

\newtheorem{remark}[theorem]{Remark}
\newtheorem{definition}[theorem]{Definition}

\newcommand{\sgn}{\operatorname{sgn}}

\title{Freedom through Imperfection: \\
exploiting the flexibility offered by redundancy in signal processing}
\submitted{September 2009}  
\author{Rachel Ward}
\adviser{Ingrid Daubechies}

\abstract{
\noindent This thesis consists of four chapters.  The first two chapters pertain to the design of stable quantization methods for analog to digital conversion, while the third and fourth chapters concern problems related to compressive sensing.  
\\
\\
\noindent In the first chapter, we study the $\beta$-encoder and golden ratio encoder, and show that these quantization schemes are robust with respect to uncertainty in the multiplication element of their implementation, whereby $x \rightarrow \beta x$.  In particular, we use a result from analytic number theory to show that the possibly unknown value of $\beta$ can be reconstructed as the unique root of a power series with coefficients in $\{-1,0,1\}$ obtained from the quantization output of test input $x$ and its negative, $-x$.    
\\
\\
\noindent The focus of our attention in Chapter 2 is Sigma Delta ($\Sigma \Delta$) modulation, a course quantization method in analog to digital conversion that is widely used for its simplicity of implementation and robustness to component imperfections.  A persistent problem in $\Sigma \Delta$ quantization of audio signals is the occurrence of undesirable \emph{tones} arising from periodicities in the bit output; these tones are particularly intolerable in the space between successive audio tracks.  As one of the contributions of this thesis, we show that the standard second order 1-bit $\Sigma \Delta$ scheme can be modified so as to eliminate such periodicities, and this modification can be achieved without sacrificing accuracy or introducing significant complexity.  
\\
\\
\noindent The emerging area of compressed sensing guarantees that sufficiently sparse signals can be recovered from significantly fewer measurements than their underlying dimension suggests, based on the observation that an $N$ dimensional real-valued vector can be reconstructed efficiently to within a factor of its best $k$-term approximation by taking $m = k\log{N}$ measurements, or inner products.  However, the quality of such a reconstruction is not assured if the underlying sparsity level of the signal is unknown. In Chapter 3, we show that sharp bounds on the errors between a signal and $p$ estimates (corresponding to a different sparsity hypotheses, say) can be achieved by reserving only $O(\log{p})$ measurements of the total $m$ measurements for this purpose.  The proposed technique is reminiscent of cross validation in statistics and learning theory, and theoretical justification in the context of compressed sensing is provided by the Johnson Lindenstrauss lemma.  
\\
\\
\noindent In Chapter 4, we study the relation between nonconvex minimization problems arising in compressed sensing and those known as \emph{free-discontinuity} problems, which are often encountered in the areas of image segmentation and edge detection.  We adapt recent thresholding techniques in the area of sparse recovery to a class of nonconvex functionals that contains both compressed sensing and free-discontinuity-type problems.   Along the way, we establish a connection between these two types of problems, using the notion of gamma convergence, and gain new insights into the minimizing solutions of such functionals.
}

\acknowledgements{First and foremost, I would like to thank my graduate advisor, Ingrid Daubechies, for directing me towards interesting and challenging problems, and for giving me the confidence to attack them.  You are truly an inspiration, and I feel so proud and fortunate to have had the opportunity to work with you. \\

\noindent  I would also like to thank the signal processing community as a whole for being friendly and warm to young people in the field.  I have had many wonderful interactions and established nice collaborations as a result.  In particular, I would like to thank Massimo Fornasier for being a great colleague and friend.  Working with you is so much fun! Your attitude is infectious, and I can only hope to pass it along to others.  I am also particularly thankful to Robert Calderbank, Sinan Gunturk, Albert Cohen, Thao Nguyen, and Ozgur Yilmaz for enlightening discussions that strengthened this thesis immensely.  \\

\noindent I am thankful to the graduate students in PACM for making me feel part of a family.  I have yet to encounter another department that gets along as well as we do.  I am thankful to Arie for helping me simplify several of the arguments in this thesis, but more importantly for being by my side during the dark times, and for yelling at me when I take work too seriously.  And finally, to my amazing family: your love and support inspires and strengthens me every day. }

\dedication{To my father, for introducing me to the strange and wonderful world of mathematics.}

\begin{document}


\begin{center}.

\vspace{60mm}

\emph{``But what did the bird see in the clear stream below? His own image; no longer a dark, gray bird, ugly and disagreeable to look at, but a graceful and beautiful swan."}
\end{center}
\indent \indent \indent \indent \indent \indent \indent \indent \indent \indent \indent \indent \indent \indent \indent \indent \indent  - The Ugly Duckling \\
\vspace{40mm}

\pagebreak

\section{Overview}
The first part of the thesis is concerned with the acquisition of real-valued, continuous-time signals, the likes of which we unknowingly process every day:  speech, music, video, and wireless communications are all produced in this form, for example.  At the same time, it is becoming increasingly efficient to store and manipulate information in \emph{digital} format.  This is partly due to the greater robustness of digital signals versus analog signals with respect to noise and fluctuation: any slight variation in an analog signal can cause a considerable amount of distortion, but for digital sequences can be safely rounded off to the nearest element in the discrete alphabet.  The design and implementation of fast, stable, and accurate algorithms for conversion between the continuous and discrete domain is therefore crucial.   
\\
\\
\noindent The process of \emph{analog-to-digital conversion} usually consists of two parts: \emph{sampling} and \emph{quantization}.  Sampling consists of converting the continuous-time signal of interest $f(t)$ into a discrete-time signal $f (t_n)$.  According to the Shannon - Nyquist interpolation formula, this operation is invertible if the analog signal of interest is uniformly bounded and absolutely integrable, $f \in L^{\infty}(\mathbb{R}) \cap L^1(\mathbb{R})$, and \emph{bandlimited}, meaning that its Fourier transform $\hat{f}(\omega) = \int_{-\infty}^{\infty} f(t) \exp{(-i \omega t)} dt$ vanishes outside a bounded interval $[-\Omega/2, \Omega/2]$.  The importance of this result lies in the fact that many signals of practical engineering interest are well-approximated by bandlimited functions.  For example, speech signals can be modelled as bandlimited functions whose bandwidth $\Omega$ is $4$ KHz, while audio signals are well modelled as bandlimited functions whose bandwidth is $20$ KHz \cite{BandLimit}.  
\\
\\
Specifically, $\Omega$-bandlimited functions $f \in L^{\infty}(\mathbb{R}) \cap L^1(\mathbb{R})$ can be recovered from sufficiently dense evenly-spaced samples according to the following reconstruction formula, which holds for $\lambda \geq 1$:
\begin{equation}
f(t) = \frac{1}{\lambda} \sum_n f \Big( \frac{\pi n}{\lambda \Omega} \Big) g \Big( t - \frac{\pi n}{\lambda \Omega} \Big); 
\label{reconstruct}
\end{equation}  
here, the reconstruction filter $g$ can be any function satisfying $\| \hat{g} \|_{\infty} \leq 1$, $\hat{g}(\omega) = 1$ for $|\omega| \leq \Omega/2$, and $\hat{g}(\omega) = 0$ for $|\omega| \geq \lambda \Omega/2$.  At  critical sampling rate $\lambda = 1$, or \emph{Nyquist rate}, these restrictions determine the Fourier transform $\hat{g}$ as the indicator function for the interval $[-\Omega/2,\Omega/2]$, yielding the \emph{sinc} filter, $g(t) = \Omega$sinc$(\Omega t) = \sin{(\Omega \pi t)} / \pi t$.  As sinc$(t)$ is not absolutely summable, sampling at the critical rate $\lambda = 1$ renders the reconstruction \eqref{reconstruct} unstable in the presence of inevitable additive error on the samples; if noisy input $f_n = f ( \frac{\pi n}{\Omega}) + \epsilon_n$ is instead observed, and $| \epsilon_n | = \epsilon$, then the sign pattern of the noise sequence $(\epsilon_n)$ can be chosen adversarily so that the resulting series $\sum_n f_n g( t - \frac{\pi n}{\Omega})$ does not even converge \cite{DD03}.  
\\
\\
\noindent In practice, reconstruction is stabilized by oversampling at a fixed rate corresponding to $\lambda > 1$.  In this setting, one has the freedom to design $g$ so that its Fourier transform $\hat{g}$ is smooth, in which case $g$ will have fast decay, and only those samples $f(\frac{\pi n}{\Omega})$ for which $| \frac{\pi n}{\Omega} - t|$ is sufficiently small will contribute significantly towards the reconstruction of $f(t)$. 

\vspace{5mm}
\noindent {\bf\normalsize Quantization in analog to digital conversion}
\vspace{2mm}

\noindent The second step in analog to digital conversion is quantization, whereby the real-valued sequence $\big( f^{\lambda}_n \big) = \big( f( \frac{\pi n}{\lambda \Omega}) \big)$ is encoded in a bitstream $(q_{(n,m)})$, with $K$ bits $(q_{(n,m)})_{m=1}^K$ allocated to each sample $f^{\lambda}_n$.  From these bits, the original function $f(t)$ can be reconstructed in the analog domain, albeit imperfectly, by replacing each sample $f_n^{\lambda}$ in \eqref{reconstruct} by a  function of the $K$-bit sequence $(q_{(n,m)})_{m=1}^K$.  For example, this sequence could be taken to be the first $K$ bits in the binary expansion of $f_n^{\lambda}$ - such \emph{pulse code modulation} quantization schemes are often employed in practice.  In Chapter 1, we will study a related quantization scheme, the $\beta$-encoder, whereby expansions of $f_n^{\lambda}$ in base $1 < \beta < 2$ are considered in place of binary expansions (where $\beta = 2$).  Note that quantization, unlike sampling, is no longer an invertible operation.  However, in any reasonable quantization scheme, the accuracy of the approximation to $f$ from the quantization sequence $(q_{(n,m)})$ will increase as as function of $B = K\lambda$, the number of bits spent per Nyquist interval.  That is, oversampling allows for more accurate as well as more stable reconstruction, but at the expense of additional sampling resources.   The balance between accuracy, stability, and sampling efficiency in signal acquisition will be a recurring theme throughout this thesis.
\\
\\
In Chapters 1 and 2, we focus on the design of stable quantization schemes, assuming that we may spend many bits per Nyquist interval in acquiring samples $(f_n^{\lambda})$.  This is the case for audio signal processing, where signals of interest have reasonable bandwidth $\Omega \leq 40$ KHz, and sampling at a rate of e.g. $32$ bits per Nyquist interval corresponds to using fewer than $1$ million bits per second.  In determining the stability of a particular quantization scheme to nonidealities in its implementation, of particular importance is its sensitivity to imperfections in the \emph{quantizer} element $Q:\mathbb{R} \rightarrow \{-1,1\}$ that is necessary for conversion from the analog to digital representation.  $Q$  is usually taken to be the signum function $Q(u) = \sgn(u)$, the behavior of which an actual quantizer, having finite precision and subject to thermal fluctuations, can only approximate.  Actual quantizer elements usually come with a pre-assigned tolerance $\epsilon$: for input with magnitude below this tolerance, $| u | \leq \epsilon$, the output $Q_{\epsilon}(u) \in \{-1,1\}$ is not reliable. We will say that a quantization scheme is {\it robust} with respect to quantization error if, for some $\epsilon > 0$, the worst approximation error produced by the quantization scheme can be made arbitrarily small by allowing a sufficiently large bit budget, even if the quantizer used in its implementation is imprecise to within a tolerance $\epsilon$.   Such robustness can be achieved if the quantization sequence $(q_{(n,m)})$ constitutes a \emph{redundant} representation of the original $f(t)$, so that an incorrect bit $q'_{(m,n)}$ caused by quantization error can be redeemed by an appropriate choice of subsequent bits.  Of course, any quantization scheme of interest will require additional components in its implementation, such as adders, multipliers, integrators, and so forth; robustness with respect to imperfections in all these elements must be taken into account in evaluating the reconstruction accuracy of a particular scheme.  
\\
\\
\noindent In Chapter 1, we investigate robustness properties of the $\beta$-encoder, a quantization scheme that was recently designed to be robust with respect to quantization errors, while also having asymptotically optimal reconstruction accuracy.  We show that  the $\beta$-encoder is also robust with respect to the other components in its implementation, securing the validity of its reconstruction accuracy.  Interestingly, such robustness is achieved by enforcing a slight shift in the quantizer element, $Q(u) = \sgn{u + \tau}$; that is, we take advantage of the \emph{freedom to choose} an imperfect quantizer, offered by redundancy, in order to ensure additional robustness guarantees.  
\\
\\
\noindent In Chapter 2, we focus on a different notion of stability in the design of quantization schemes.  Sigma Delta ($\Sigma \Delta$) quantization, and in particular one-bit $\Sigma \Delta$ modulation, is a \emph{coarse} quantization method, as each sample value $f_n^{\lambda}$ in \eqref{reconstruct} is replaced by a single bit $q_n \in \{-1,1\}$ that depends on all previous $f_k^{\lambda}$.  Because of their robustness properties and ease of implementation, $\Sigma \Delta$ schemes are widely popular in practice, despite suffering from an entirely different kind of problem: in audio signal processing, periodic oscillatory patterns in the bit output  $q_n$ often occur, producing idle tone components in response to stretches of zero input $f_n^{\lambda} = 0$, or even small amplitude sinusoidal inputs. Such tones are audible to the listener, and are so pronounced in low-order $\Sigma \Delta$ quantizers that such schemes are often avoided in audio applications for more complex higher order $\Sigma \Delta$ schemes \cite{SDChaos}.  As one of the contributions of this thesis, we introduce a second-order $\Sigma \Delta$ scheme that eliminates such periodicities by damping the internal variables of the system once the input $f_n^{\lambda}$ is identically zero.  This approach is in line with the philosophy of Chapter 1: to eliminate the idle tones, we exploit the freedom to induce a small amount of  leakage to the system without sacrificing accuracy, as offered by redundancy of the $\Sigma \Delta$ quantization output.

\vspace{5mm}
\noindent {\bf\normalsize Beyond Nyquist sampling}
\vspace{2mm}

\noindent The second part of the thesis concerns \emph{compressed sensing},  a fast emerging area of applied mathematics that represents a new approach to traditional sampling paradigms.  At the core of this field is the remarkable fact that a large class of underdetermined linear systems $y = \Phi x$ are invertible, and can be inverted efficiently, subject to the constraint that $x$ is sufficiently \emph{sparse}, or has sufficiently small nonzero support $\| x \|_0 = | \{ i : x_i \neq 0 | \}$.  In particular, a matrix $\Phi \in \mathbb{R}^{m \times N}$ with unit normed columns is said to have the ($k,\delta$)-\emph{restricted isometry property}, or satisfy $k$-RIP for short, if all singular values of any $k$ column submatrix of $\Phi$ lie in the interval $[1 - \delta, 1 + \delta]$ for a given constant $\delta < 1$.  For $2k$-RIP matrices, a $k$-sparse vector $x$ can be recovered from its lower-dimensional image $y = \Phi x$ as the minimal $\ell_1$ norm solution,
\begin{equation}
x = \arg \min_{\Phi u = y} \| u \|_{\ell^N_1}.
\label{convex1}
\end{equation}
That is, $2k$-RIP for the matrix $\Phi$ furnishes an equivalence between the minimal $\ell_0$ norm and minimal $\ell_1$ norm solution from the affine space $\{u \in \mathbb{R}^N: \Phi u = y \}$ if a $k$-sparse solution exists; while finding the former solution represents an NP-hard problem in general, the latter solution can be recovered efficiently using linear programming methods.  Even more can be said: for vectors $x$ satisfying $y = \Phi x$ that are approximately but not exactly $k$-sparse, the error $\| x - \hat{x} \|_{\ell_2^N}$ incurred by the approximation $\hat{x} = \arg \min_{\Phi u = y} \| u \|_1$ using $2k$-RIP matrix $\Phi$ is still small, thanks to the following stability result: 
\begin{equation}
\| x - \hat{x} \|_{\ell^N_2} \leq \frac{C_{\delta}}{\sqrt{k}}\inf_{\|z \|_0 \leq k} \|x - z\|_{\ell_1^N}.
\label{approxerr0}
\end{equation}
In words, the error $\| x - \hat{x} \|_{\ell^N_2}$ is bounded by a fraction of the best possible approximation error between $x$ and the set of $k$-sparse vectors in the metric of $\ell_1^N$  \cite{Candes08}. 
\\
\\
{\bf Constructing $k$-RIP matrices} With high probability, an $m \times N$ random matrix whose coefficients are drawn as independent and identical realizations of a Gaussian or Bernoulli random variable will satisfy $k$-RIP of optimal order,
 \begin{equation}
 k  = O \big( m/\log{(N/m)}) \big).
 \label{order}
 \end{equation}
  While no deterministic constructions of RIP matrices have been found of this order, many matrices with a smaller degree of randomness satisfy the restricted isometry property to almost optimal order.   Most notably, with high probability an $m \times N$ matrix obtained by selecting $m$ rows at random from the $N \times N$ discrete Fourier matrix satisfies $k$-RIP of order $k =  O \big(m / (\log{N}\log{(\log{N})}) \big) $ \cite{RV06}.   This particular result has an interesting interpretation that connects back to the first two chapters: signals that are well-approximated as elements from the class of \emph{sparse trigonometric polynomials},
\begin{eqnarray}
 f(t) = \sum_{\omega=-N}^{\omega = N} a_{\omega} \exp{(2\pi i \omega t)} \textrm{ , } t \in [0,1) \textrm{ , } a_{\omega} \in \mathbb{R}, \| a \|_0 \leq d,
 \end{eqnarray}
 can be efficiently reconstructed from only $m = O(d\log^3{N}))$ samples; when $d$ is small, this represents an \emph{exponentially smaller} number than the $O(N)$ samples required for the Shannon-Whittaker reconstruction formula \eqref{reconstruct} to hold.  This is significant in applications such as radar, navigation, and satellite communications, where signals of interest often have a sparse frequency representation, while measurements are at the same time  expensive to implement.

\vspace{5mm}
\noindent {\bf\normalsize Cross validation in compressed sensing}
\vspace{2mm}

\noindent From the reconstruction formula \eqref{approxerr0} and order relation \eqref{order}, the quality of the approximation $\hat{x} = \arg \min_{\Phi u = y} \| u \|_1$ to an unknown signal $x$ in the affine space $\Phi x =  y$ depends on the approximability of $x$ is by its $k = \frac{m}{\log{N/m}}$ largest coordinates.  In the literature, a known and fixed value for $k$ is assumed to well-approximate all signals in the class of interest, while more realistically, $k$ may be a parameter of the unknown input.   Natural images, for instance, are generally well-approximated by only a few discrete Wavelet basis elements $\cite{DeVore92}$; however, certain heavily textured images, such as those containing hair or sand, are not particularly compressible in such bases.  In statistics, parameter selection and noise estimation are commonly achieved through \emph{cross validation}, whereby the available data  is separated into independent \emph{testing} and \emph{training} sets.  In Chapter 3, we show that cross validation incorporates naturally into the compressed sensing paradigm, as the random measurements often used for compressed sensing are the same measurements that provide almost-isometric lower dimensional embeddings of generic point sets, as guaranteed by the Johnson-Lindenstrauss lemma.   More precisely, we take a set of $m$ measurements of the unknown $x$, and use $m - r$ of these measurements, $\Phi x$, in a compressed sensing decoding algorithm to return a sequence $(\hat{x}_1, \hat{x}_2, ... )$ of candidate approximations to $x$.  The remaining $r$ measurements, $\Psi x$, are then used to identify from among this sequence a `best' approximation $\hat{x} = \hat{x}_j$, along with an estimate of the sparsity level of $x$.   The proposed method for error estimation in compressed sensing is \emph{extremely cheap}:  approximation errors of up to $p$ distinct approximations $(\hat{x}_j)_{j=1}^p$ can be estimated with high accuracy at the expense of only $10\log{p}$ samples.  Whereas in Chapters 1 and 2 we exploited redundancy of representation afforded by oversampling to provide stability results for quantization schemes, we now exploit the stability of random measurements to validate the assumption that the representation at hand is redundant. 

\vspace{5mm}
\noindent {\bf\normalsize Compressed sensing and free-discontinuity problems}
\vspace{2mm}

\noindent Chapter 4 explores the connection between minimization problems arising in compressed sensing and those corresponding to \emph{free-discontinuity problems}, which describe situations where the solution of interest is defined by a function and a lower dimensional set consisting of the discontinuities of the function.  Such situation arise, for instance, in crack detection from fracture mechanics \cite{ro08-1} or in certain digital image segmentation problems \cite{essh02}.  The best-known example of a free-discontinuity problem is that of minimizing the so-called Mumford-Shah functional \cite{chdm99}, which is defined by
$$
J(u,K):= \int_{\Omega \setminus K} \left [ | \nabla u |^2 + \alpha ( u - g )^2 \right ] dx +\beta \mathcal{H}^{d-1}(K \cap \Omega);
$$
here, the set $\Omega$ is a bounded open subset of $\mathbb{R}^d$, the constants $\alpha, \beta >0$ are fixed, $g \in L^\infty(\Omega)$, and $\mathcal H^N$ denotes the $N$-dimensional Hausdorff measure.  In the context of visual analysis, the function $g$ is a given noisy image that is to be approximated by the minimizing function $u \in W^{1,2}(\Omega \setminus K)$; the set $K$ is simultaneously is used in order to {\it segment} the image into connected components. 
\\
\\
The Mumford-Shah functional $J(u,K)$ is not easily handled, theoretically or numerically speaking. This difficulty has given rise to the development of approximation methods for the Mumford-Shah functional and its minimizers where \emph{sets} are no longer involved \cite{chdm99}.  In the one-dimensional setting, minimizers of $J(u,K)$ can be approximated in a strong sense by minimizers of a discrete functional which, reformulated in terms of the discrete derivative, and generalized to incorporate inverse problems, takes the form of a  \emph{selective least squares} problem,
\begin{equation}
\textrm{Minimize }\mathcal J_r(u):= \| T u - f \|_{\ell_2^m}^2 + \sum^N_{i=1} \min \left \{|u_i|^2, r \right \}.
\label{regMSintro}
\end{equation} 
Note that the unknown discrete derivative vector, assumed to be `small' and smoothly varying except on the  discontinuity set of the solution, should be well-approximated by the minimizer of $\mathcal J_r$.  However, despite successful numerical results \cite{chve01} observed for such selective least squares problems, no rigorous results on the existence of minimizers, let alone on the convergence of several proposed algorithms to such minimizers, are currently available in the literature.  
\\
\\
\noindent In Chapter 4, we establish a connection between the selective least squares problem \eqref{regMSintro} and the $\ell_0$ minimization problem from compressed sensing,
\begin{equation}
\textrm{Minimize }\mathcal J_r(u):= \| T u - f \|_{\ell_2^m}^2 + \| u \|_0 
\label{intro0}
\end{equation} 
using the notion of gamma convergence.  Moreover, we derive an iterative thresholding algorithm for finding solutions to the selective least squares functional ${\cal{J}}_r$, motivated by the recent application of such algorithms for minimizing the $\ell_0$ regularized functional \eqref{intro0}.  Our algorithm is shown to converge to local minimizers $\bar{u}$ of the functional $\mathcal J_r$ that are characterized by certain fixed point conditions, including the following \emph{gap} property:  for each $j$, either $|\bar{u}_j| \geq  \sqrt{2}r$ or $|\bar{u}_j| \leq \frac{1}{\sqrt{2}}r$.   We show in addition that any global minimizer of $\mathcal J_r$ must satisfy such fixed point conditions, giving mathematical justification to the observation that minimizers of $\mathcal J_r$ tend to be ``cartoon"-like, or segmented into regions of small gradient, separated by edges corresponding to large gradient.    Moreover, we show that minimizers are restricted to regions where the non-convex functional $\mathcal J_r$ is locally strictly convex.  Thus, global minimizers are always isolated, although not necessarily unique, whereas local minimizers may constitute a continuum of unstable equilibria.  These observations sheds light on fundamental properties, virtues, and limitations, of regularization by means of the Mumford-Shah functional, and provide a rigorous justification of the numerical results appearing in the literature.


\pagebreak
\noindent {\bf \large Contributions of this thesis}
\vspace{4mm}

\noindent {\bf \normalsize Chapter 1.} The $\beta$-encoder and golden ratio encoder were recently introduced as quantization schemes in analog to digital conversion that are simultaneously robust with respect to quantizer imperfections in their analog implementation and efficient in their bit-budget use.  However, it was not clear at first \cite{DY06} that these schemes were also robust with respect to imprecisions in the multiplier element.  Using a result in analytic number theory concerning the roots of power series with coefficients in $\{-1,0,1\}$, we show that these schemes are indeed robust as such, and the underlying value of $\beta$ can be reconstructed as the unique root on $[0,1]$ of a power series obtained from the quantization output of test input $(x,-x)$.

\vspace{5mm}
\noindent {\bf \normalsize Chapter 2.} Sigma Delta ($\Sigma \Delta$) modulation is a course quantization method in analog to digital conversion that is widely used for its simplicity of implementation and robustness properties; however, a persistent problem in $\Sigma \Delta$ modulation is the occurrence of undesirable idle tones arising from oscillatory patterns in the bit output.  Such oscillations are omnipresent, and automatically arise along stretches of zero input, the likes of which are unavoidable in audio signal processing.  Responding to a question left open in \cite{OYthesis}, we introduce a family of \emph{quiet} second-order 2-bit asymmetrically damped $\Sigma \Delta$ schemes whose quantization output becomes constant at the onset of zero input, effectively eliminating the idle tones that arise in this context. 

\vspace{5mm}
\noindent {\bf \normalsize Chapter 3.} The emerging area of compressed sensing boasts efficient sensing techniques based on the phenomenon that an $N$ dimensional real-valued vector can be reconstructed efficiently to within a factor of its best $k$-term approximation error by taking only $m = k\log{N}$ measurements. However, the quality of such a reconstruction is not assured if the underlying sparsity level of the signal is unknown. We show that sharp bounds on the errors between a signal and $p$ approximations to the signal (corresponding to a different sparsity hypotheses, say) can be achieved by reserving only $O(\log{p})$ measurements of the total $m$ measurements for this purpose.  This error estimation technique is reminiscent of cross validation in statistics and learning theory, and theoretical justification in the context of compressed sensing is provided by the Johnson Lindenstrauss lemma.

\vspace{5mm}
\noindent {\bf \normalsize Chapter 4.}  Inspired by recent thresholding techniques in compressed sensing, we develop an algorithm for minimizing a discrete version of the Mumford Shah (MS) functional arising in image segmentation \cite{ca95}.  Although regularization methods involving similar nonconvex functionals work well in practice \cite{ro07}, ours is the first proven to converge to local minimizers of the discrete MS functional.  The proposed algorithm can be adapted to a more general class of nonconvex functionals that includes the $\ell_0$-regularized functional from compressed sensing.  Finally, we show that the $\ell_0$ regularized functional can be viewed as the limit of a sequence of free-discontinuity type problems, illuminating an intimate connection between the two problems.

\chapter{Robust quantization in analog to digital conversion}  
\section{Analog to digital conversion: an introduction}

Digital signals are omnipresent; one important reason is that transmission and storage for digital signals  is much more robust with respect to noise and fluctuation than for analog signals.  Because many signals of interest to us are produced in analog form, a transition from analog to digital is therefore necessary.
\\
\\
Analog-to-digital (A/D) conversion consists of two parts: {\it sampling} and {\it quantization}.  The process of sampling, followed by quantization, can be schematically represented by 
\begin{center}
$f(t) \Rightarrow \big( f(t_n) \big)_{n \in Z} \Rightarrow  \big( (q_{(n,m)})_{m=1}^K \big)_{n \in Z}$
\end{center}
where $q_{(n,m)} = \big( q_{(n,1)}, q_{(n,2)},..., q_{(n,K)} \big)$ is a finite sequence of digits from a finite (and usually binary) alphabet.   It is reasonable to assume that the real-valued analog signal of interest, $f$, is uniformly bounded and absolutely integrable, $f \in L^{\infty}(\mathbb{R}) \cap L^1(\mathbb{R})$, and is {\it bandlimited}; i.e., its Fourier transform $\hat{f}(\omega) = \int_{-\infty}^{\infty} f(t) \exp{(-i \omega t)} dt$ vanishes outside some bounded interval $[-\Omega/2, \Omega/2]$.  Under this assumption, the sampling step in A/D conversion is an invertible operation, as $f(t)$ can be reconstructed from the samples $\big( f(\frac{n \pi}{\Omega \lambda})\big)_{n \in Z}$ as long as $\lambda \geq 1$, via the reconstruction formula
\begin{equation}
f(t) = \frac{1}{\lambda} \sum_n f \Big( \frac{n \pi}{\lambda \Omega} \Big) g \Big( t - \frac{n \pi}{\lambda \Omega} \Big).
\label{reconstruct}
\end{equation}  
Above, $g$ can be any function such that $\| g \|_{\infty} \leq 1, \hat{g} = 1$ for $|\xi| \leq \Omega$ and $\hat{g} = 0$ for $|\xi| \geq \lambda \Omega$.  If $\hat{g}$ is smooth, then $g$ will have fast decay and the reconstruction $\eqref{reconstruct}$ will be almost entirely local $\cite{DD03}$.   By scaling and normalizing appropriately, we can reduce any bandlimited function in $L^{\infty}(\mathbb{R}) \cap L^1(\mathbb{R})$ to an element of the class $S = \{f : \|f\|_{\infty} \leq 1, \|f\|_1 \leq 1, \textrm{supp }{\hat{f}} \in [-\pi,\pi] \}$, in which case the reconstruction formula $\eqref{reconstruct}$ reduces to
\begin{equation}
f(t) = \frac{1}{\lambda} \sum_n f \Big( \frac{n}{\lambda} \Big) g \Big( t - \frac{n}{\lambda} \Big).
\label{reconstruct+}
\end{equation}  
The second step in A/D conversion is {\it quantization}, in which sample values $f(\frac{n}{\lambda})$ are replaced by finite bitstreams $(q_{(n,m)})^{K}_{m=1}$ and associated function $\Delta: \{0,1\}^K \rightarrow \mathbb{R}$, so that the original function can be reconstructed at a later time by the approximation
\begin{equation}
\tilde{f}(t) = \frac{1}{\lambda} \sum_n \Delta(q_{(n,.))}) g \Big( t - \frac{n \pi}{\lambda \Omega} \Big).
\label{ftilde}
\end{equation}
 Quantization, unlike sampling, is not in general an invertible operation.   That is, if $E$ is the operator which maps functions in $S$ to bitstreams (the {\it encoding} operator) associated with a particular quantization scheme, and $D$ is the {\it decoding} operator which maps bitstreams back to functions in $S$, then typically $f \neq D(Ef)$.   In order to measure the performance of a particular quantization scheme, one considers the {\it distortion} of the scheme, given by $d(S;E,D) := \sup_{f \in S}\| f - D(Ef) \|$, where $\| . \|$ is the norm of interest; e.g.,the $L^{\infty}$ norm on  $\mathbb{R}$.  In any reasonable quantization scheme, the distortion will decrease as the number of bits per unit interval  (the so-called {\it bit budget}) increases; schemes whose distortion decreases faster as a function of the bit budget are generally considered superior quantization schemes.
\\
\\
If the relation between the distortion $d$ and the bit budget $B$ were the only important measure associated with a quantization scheme, then the widely-used {\it pulse code modulation} (PCM) quantization algorithm would always be preferred in practice.  Given a function $f$ in $S$, a $N$-bit PCM algorithm simply replaces each sample value $f(\frac{n}{\lambda})$ with $N$ bits: one bit for its sign, followed by the first $N-1$ bits of the binary expansion of $|f(\frac{n}{\lambda})|$.  One can show that for signals $f$ in $S$, and for fixed $\lambda > 1$, the distortion $d(f,E,D) \leq C_{\lambda}2^{-N} = C_{\lambda}2^{-B/\lambda}$ for an $N$-bit PCM, when, for instance, $\| . \|$ is the $L^{\infty}$ norm on $\mathbb{R}$.  
\begin{flushleft}
{\normalsize\bf Sigma Delta ($\Sigma \Delta$) quantization}
\end{flushleft} 
\noindent
Another popular quantization algorithm, {\it $\Sigma \Delta $ modulation}, has distortion that decays only like an inverse polynomial in the bit budget $B$.  One-bit $\Sigma \Delta$ modulation replaces each sample $f_n^{\lambda}$ in $\eqref{reconstruct}$ by a single bit, $b_n \in \{-1,1\}$ (In multibit $\Sigma \Delta$ schemes, the coefficients $b(n)$ replacing the $f_n^{\lambda}$ can assume a larger range of discrete values, and thus require several bits).  In contrast to PCM, the $b_n$ in $\Sigma \Delta$ depend not only on $f_n^{\lambda}$, but on all previous $f_k^{\lambda}$, $k \leq n$.   Consider first-order, one-bit $\Sigma \Delta$, where the $b_n$ are generated recursively by the scheme,
\begin{eqnarray}
b_n &=& Q(u_{n-1} + f^{\lambda}_n) \nonumber \\
u_n &=& u_{n-1} + f^{\lambda}_n - b_n.
\label{1storder}
\end{eqnarray} 
The \emph{quantizer} $ Q(x) = \sgn(x) $ just returns the sign of its argument, which can be taken to be either $1$ or $-1$ when $x = 0$.  Assuming $|f^{\lambda}_n| \leq 1$ (which is true of functions $f \in S$),  and initializing $|u_0| < 1$, a simple inductive argument guarantees that $|u_n| < 1$ for all $n$ \cite{DD03}.  Using this, along with the observation that $B = \lambda$ for one-bit quantization schemes, we arrive at an error estimate $\| f - \tilde{f}^B \|_{\infty} \leq 2\|\frac{dg}{dt}\|_1B^{-1}$  for the first-order $\Sigma \Delta$ scheme $\eqref{1storder}$.  More generally, the $K^{th}$ order analog of the recursion $\eqref{1storder}$ equates $f_n^{\lambda} - b_n$ to the $K^{th}$ order (as opposed to first order) difference of a bounded sequence $(u_n)_{n \in \mathbb{Z}}$, and has corresponding decay $\| f - \tilde{f}^B \|_{\infty} \leq C_{K,g}B^{-K}$;  in particular, the second-order one-bit $\Sigma \Delta$ quantization scheme can be recast in the form
 \begin{eqnarray}
 b_n &=& Q(F(u_{n-1},v_{n-1}, f^{\lambda}_n)) \nonumber \\
u_n &=& u_{n-1} + f^{\lambda}_n - b_n \nonumber \\
v_n &=& v_{n-1} + u_n.
\label{2ndorder}
 \end{eqnarray}
The function $F$ can be any map that guarantees the boundedness of the state variables $(u_n, v_n)$; for example, the linear function $F(u_n,v_n) = u_n + \gamma v_n$ suffices for an appropriate range of $\gamma$ \cite{OYthesis}.
\begin{flushleft}
{\normalsize\bf Robustness for $\Sigma \Delta$ quantization}
\end{flushleft} 
\noindent
Clearly, the $O(2^{-B})$ decay of PCM will outperform the $O(B^{-K})$ decay of a $K$-th order one-bit $\Sigma \Delta$ scheme, for any order $K$.  Yet, $\Sigma \Delta$ schemes of as low as second order are preferred in practice over high order PCM quantization schemes.  This can be partly understood by the fact that $\Sigma \Delta$ schemes are less sensitive than PCM schemes to inevitable errors in the analog components of their implementation.   The notion of $\emph{robustness}$ of $\Sigma \Delta$ schemes to quantization error, which was first studied in \cite{DD03}, is connected to how well $\Sigma \Delta$ schemes exploit the redundancy of the representation $f(t) = \frac{1}{\lambda} \sum_n c_n  g(t - t_n^{\lambda} )$ afforded by oversampling $\lambda > 1$.  By redundancy, we refer to the fact that the functions $\big( g_n(t) \big)_{n \in \mathbb{Z}}  = \big( g(t - t_n^{\lambda})\big)_{n \in \mathbb{Z}}$ form a frame \footnote{Technically, the shifts of $g$ are not a frame for the space of $\Omega$-bandlimited functions because these functions don't live in this space. But this is only a technicality due to the definition of a frame; the function $g$ lives in the larger Hilbert space of 
$\Omega \lambda$-bandlimited functions, or simply $L^2$ for that matter, and its shifts satisfy the frame property for the smaller space of $\Omega$-bandlimited functions.} for the Hilbert space of $\Omega$-bandlimited functions, and their frame redundancy increases with $\lambda$.    One source of quantization error incurred in both $\Sigma \Delta$ and PCM quantization comes from the quantizer function itself:  the function $Q(u) = \sgn{u}$ that is used in both the $\Sigma \Delta$ schemes and in the recursive algorithms used to generate binary expansions in PCM cannot be built to have infinite precision.  In fact, quantizer elements for A/D circuits generally   come with a prescribed tolerance $\nu$ for which the output $Q(u)$ of such quantizers should not be trusted once  $|u| < \nu$ (and of course, quantizers with lower tolerance are more expensive).  As the binary expansion of almost every real number is unique, an incorrect bit assignment, and especially an incorrect initial bit, in the truncated binary expansion of a sample $f_n^{\lambda}$ will cause an error in the resulting approximation that obviates the possibility of $O(2^{-B})$ decay.  $\Sigma \Delta$ schemes, on the other hand, keep track of prior samples $f_k^{\lambda}$, $k \leq n$, in such a way that errors caused by imprecise quantizers can be corrected later, and the $O(B^{-K})$ maintained, by exploiting the redundant information in the samples $f_n^{\lambda}$ $\cite{DD03}$.   
\\
\\
This discussion brings us to the question: \emph{Is it possible to have the best of both worlds}?  That is, can one design a quantization scheme that has exponential reconstruction guarantees of PCM while also being robust with respect to quantization imperfections like $\Sigma \Delta$?  This question was answered affirmatively in \cite{DDGV06}, with the introduction of the $\beta$-encoder.
 
 \section{On the robustness of beta-encoders and golden ratio encoders}

Beta-encoders are similar to PCM in that they replace each sampled function value $x = f_n^{\lambda}$ with a truncated series expansion in a base $\beta$, where $1 < \beta < 2$, and with binary coefficients.  Clearly, if $\beta = 2$, then this algorithm coincides with PCM.  However, whereas the binary expansion of almost every real number is unique, for every $\beta \in (1,2)$, there exist a continuum of different $\beta$ expansions of almost every $x$ in $(0,1]$ (see $\cite{Sid03}$).  It is precisely this redundancy that gives beta-encoders the freedom to correct errors caused by imprecise quantizers shared by $\Sigma \Delta$ schemes.  Whereas in $\Sigma \Delta$, a higher degree of robustness is achieved via finer sampling, beta-encoders are made more robust by choosing a smaller value of $\beta$ as the base for expansion.  
\\
\\
Although beta-encoders as discussed in $\cite{DDGV06}$ are robust with respect to quantizer imperfections, these encoders are not as robust with respect to imprecisions in other components of their circuit implementation.  Typically, beta-encoders require a multiplier in which real numbers are multiplied by $\beta$.  Like all analog circuit components, this multiplier will be imprecise; that is, although a known value $\beta_0$ may be set in the circuit implementation of the encoder, thermal fluctuations and other physical limitations will have the effect of changing the true multiplier to an unknown value $\beta \in [\beta_{low}, \beta_{high}]$ within an interval of the pre-set value $\beta_0$.   The true value $\beta$ will vary from device to device, and will also change slowly in time within a single device.  This variability, left unaccounted for, disqualifies the beta-encoder as a viable quantization method since the value of $\beta$ must be known with exponential precision in order to reconstruct a good approximation to the original signal from the recovered bit streams.
\\
\\
We overcome this potential limitation of the beta-encoder by introducing a method for recovering $\beta$ from the encoded bitstreams of a real number $x \in [-1,1]$ and its negative, $-x$.  Our method incorporates the techniques used in $\cite{DY06}$, but our analysis is simplified using a transversality condition, as defined in $\cite{solom95}$, for power series with coefficients in $\{-1,0,1\}$.  As the value of $\beta$ can fluctuate within an interval $[\beta_{low}, \beta_{high}]$ over time, our recovery technique can be repeated at regular intervals during quantization (e.g., after the quantization of every 10 samples).  
\\
\\
The golden ratio encoder (GRE) was proposed in $\cite{GREsubmit}$ as a quantizer that shares the same robustness and exponential rate-distortion properties as beta-encoders, but that does not require an explicit multiplier in its circuit implementation.  GRE functions like a beta-encoder in that it produces beta-expansions of real numbers; however, in GRE, $\beta$ is fixed at the value of the golden ratio, $\beta = \phi = \frac{1+\sqrt{5}}{2}$.  The relation $\phi^2 = \phi + 1$ characterizing the golden ratio permits elimination of the multiplier from the encoding algorithm.  Even though GRE does not require a precise multiplier, component imperfections such as integrator leakage in the implementation of GRE may still cause the true value of $\beta$ to be slightly larger than $\phi$; in practice it is reasonable to assume $\beta \in [\phi, 1.1\phi]$.   Our method for recovering $\beta$ in general beta-encoders can be easily extended to recovering $\beta$ in the golden ratio encoder.
\\
\\
Our work in this section will be organized as follows:
\begin{enumerate}
\item In sections $1.2.1$ and $1.2.2$, we review relevant background on beta-encoders and golden ratio encoders, respectively.
\item In section $1.2.3$, we introduce a more realistic model of the golden ratio encoder that takes into account the effects of integrator leak in the delay elements of the circuit.  We show that the output of this revised model still correspond to truncated beta-expansions of the input, but in an unknown base $\beta$ that differs from the pre-set value.
\item Section $1.2.4$ describes a way to recover the unknown value of $\beta$ up to arbitrary precision using the bit streams of a `test' number $x \in [-1,1]$, and $-x$.  The recovery scheme reduces to finding the root of a polynomial with coefficients in $\{-1,0,1 \}$.
\item Section $1.2.5$ extends the recovery procedure of the previous section to the setting of beta-encoders having leakage in the (single) delay element of their implementation.  We show that the analysis in this case is completely analogous to that of Section $1.2.4$.
\end{enumerate}

\subsection{The beta-encoder}
In this section, we summarize certain properties of beta-encoders (or $\beta$-encoders) with error correction, from the perspective of encoders which produce beta expansions with coefficients in $\{-1,1\}$.  For more details on beta-encoders, we refer the reader to $\cite{DDGV06}$.
\\
\\
We start from the observation that given $\beta \in (1,2]$, every real number $x \in [-\frac{1}{\beta - 1},\frac{1}{\beta - 1}]$ admits a sequence $(b_j)_{j \in N}$, with $b_j \in \{-1,1\}$, such that 
\begin{equation}
x = \sum_{j=1}^{\infty} b_j \beta^{-j}.
\label{ex}
\end{equation}
Under the transformation $\tilde{b}_j = \frac{b_j + 1}{2}$, $\eqref{ex}$ is equivalent to the observation that every real number $y \in [0,\frac{1}{\beta - 1}]$ admits a beta-expansion in base $\beta \in (1,2]$, with coefficients $\tilde{b}_j \in \{0,1\}$.  Accordingly, all of the results that follow in this section have straightforward analogs in terms of $\{0,1\}$-beta-expansions; see $\cite{DDGV06}$ for more details.
\\
\\
One way to compute a sequence $(b_j)_{j \in N}$ that satisfies $\eqref{ex}$ is to run the iteration
\begin{eqnarray}
u_1 &=& \beta x \nonumber \\
b_1 &=& Q(u_1) \nonumber \\
\textrm{for $j \geq 1:$ } u_{j+1} &=& \beta(u_j - b_j) \nonumber \\
b_{j+1} &=& Q(u_{j+1}) 
\label{iter}
\end{eqnarray}
where the quantizer $Q$ is simply the sign-function,
\begin{equation}
	Q(u)
           = \left\{\begin{array}{cl}
	-1,  & u \leq 0 \\
	1,   &u > 0. \\
	   \end{array}
	   \right.
           \label{(Q)}
\end{equation}
For $\beta = 2$, the expansion $\eqref{ex}$ is unique for almost every $x \in [-1,1]$; however, for $ \beta \in (1,2)$, there exist uncountably many expansions of the type $\eqref{ex}$ for any $x \in [-1,1]$  (see $\cite{Sid03}$).  Because of this redundancy in representation, beta encoders are robust with respect to quantization error, while PCM schemes are not.  We now explain in more detail what we mean by quantization error.  The quantizer $Q$ in $\eqref{(Q)}$ is an idealized quantizer; in practice, one has to deal with quantizers that only approximate this ideal behavior.  A more realistic model is obtained by replacing $Q$ in $\eqref{(Q)}$ with a `flaky' version $Q^{\nu}$, for which we know only that
\begin{equation}
Q^{\nu}(u)  = \left\{\begin{array}{cl}
	-1,  & u  <  -\nu \\
	1,   &u  \geq \nu \\
	-1 \textrm{ or } 1, &-\nu \leq u  \leq \nu.
	   \end{array}\right.
         \label{(fqs)}
\end{equation}
In practice, $\nu$ is a quantity that is not known exactly, but over the magnitude of which we have some control, e.g. $|\nu| \leq \epsilon$ for a known $\epsilon$.  This value $\epsilon$ is called the {\it tolerance} of the quantizer.   We shall call a quantization scheme {\it robust} with respect to quantization error if, for some $\epsilon > 0$, the worst approximation error produced by the quantization scheme can be made arbitrarily small by allowing a sufficiently large bit budget, even if the quantizer used in its implementation is imprecise to within a tolerance $\epsilon$.  According to this definition, the naive $\{-1,1\}$-binary expansion is not robust.  More specifically, suppose that a flaky quantizer $Q^\nu$ is used in $\eqref{iter}$ to compute the base-2 expansion of a number $x \in [-1,1]$ which is sufficiently small that $|2x| \leq \nu$.  Since $2x$ is within the flaky zone for $Q^{\nu}$, if $b_1 = Q^{\nu}(2x)$ is assigned incorrectly; i.e., if $b_1$ differs from the sign of $x$, then no matter how the remaining bits are assigned, the difference between $x$ and the number represented by the computed bits will be at least $|x|$.   This is not the case if $1 < \beta < 2$, as shown by the following whose proof can be found in $\cite{DDGV06}$:

\begin{theorem}
Let $\epsilon> 0$ and $x \in [-1,1]$.  Suppose that in the beta-encoding of $x$, the procedure $\eqref{iter}$ is followed, but the quantizer $Q^{\nu}$ is used instead of the ideal $Q$ at each occurence, with $\nu$ satisfying $\nu \leq  \epsilon$.  Denote by $(b_j)_{j\in N}$ the bit sequence produced by applying this encoding to the number $x$.  If $\beta$ satisfies
\begin{center}
$1 < \beta < \frac{2+\epsilon}{\epsilon+1},$
\end{center}
then
\begin{equation}
|x - \sum_{j=1}^N b_j \beta^{-j} | \leq C\beta^{-N}
\end{equation}
with $C = \epsilon + 1$.
\label{b-encod+} 
\end{theorem}
For a given tolerance $\epsilon > 0$, running the recursion $\eqref{iter}$ with a quantizer $Q^{\nu}$ of tolerance $\epsilon$ and a value of $\beta$ in $(1,\frac{2+\epsilon}{\epsilon+1})$ produces bitstreams $(b_j)$ corresponding to a beta-expansion of the input $x$ in base $\beta$; however, the precise value of $\beta$ must be known in order recover $x$ from such a beta-expansion.  As mentioned in the introduction and detailed in the following section, component imperfections may cause the circuit to behave as if a different value of $\beta$ is used, and this value will possibly be changing slowly over time within a known range, $[\beta_{low}, \beta_{high}]$.   In $\cite{DY06}$, a method is proposed whereby an exponentially precise approximation $\tilde{\gamma}$ to the value of $\gamma = \beta^{-1}$ at any given time can be encoded and transmitted to the decoder without actually physically measuring its value, via the encoded bitstreams of a real number $x \in [0,1)$ and $1- x$.  This decoding method can be repeated at regular time intervals during the quantization procedure, to account for the possible time-variance of $\gamma$.  That an exponentially precise approximation $\tilde{\gamma}$ to $\gamma$ is sufficient to reconstruct subsequent samples $f(t_n)$ with exponential precision is the content of the following theorem, which is essentially a restatement of Theorem 5 in $\cite{DY06}$. 

\begin{theorem}[Daubechies, Yilmaz]
Consider $x \in [0,1)$ and $(b_j)_{j \in N} \in \{0,1\}$, or $x \in [-1,1]$ and $(b_j)_{j \in N} \in \{-1,1\}$.  Suppose $\gamma \in (1/2,1)$ is such that $x = \sum_{j=1}^{\infty} b_j \gamma^j$.  Suppose $\tilde{\gamma}$ is such that $|\gamma - \tilde{\gamma}| \leq C_1 \gamma^N$ for some fixed $C_1 > 0$.    Then $\tilde{x}_N := \sum_{j=1}^N b_j \tilde{\gamma}^j$ satisfies
  \begin{equation}
  |x - \tilde{x}_N| \leq C_2 \gamma^N \end{equation} 
\label{daub}
where $C_2$ is a constant which depends only on $\gamma$ and $C_1$.
\end{theorem}

Although the approach proposed in $\cite{DY06}$ for estimating $\beta$ from bitstreams overcomes potential approximation error caused by imprecise multipliers in the circuit implementation of the beta-encoder, new robustness issues are nevertheless introduced. Typically, one cannot ensure that the reference level 1 in $1-x$ is known with high precision.  To circumvent this problem, the authors consider other heuristics whereby $\beta$ is recovered from clever averaging of multiple pairs $x_j$ and $1-x_j$.  These heuristics do not require that the reference level 1 in $1-x$ be precise; however, these approaches become quite complicated in and of themselves, and any sort of analytical analysis of their performance becomes quite difficult.  As one of the contributions of the present work, we present an approach for recovering $\beta$ that is inspired by the approach in \cite{DY06} but does not require a precise reference level, yet still allows for exponentially precise approximations to $\beta$.

\subsection{The golden ratio encoder (GRE)}
 As shown in the previous section, beta-encoders are robust with respect to imperfect quantizer elements, and their approximation error decays exponentially with the number of bits $N$.   To attain this exponential precision, $\beta$ must be measured with high precision, which is quite complicated in practice.   These complications motivated the invention of the golden ratio encoder (GRE) of $\cite{GREsubmit}$, which has the same robustness and rate-distortion properties as beta-encoders, but uses an additional delay element in place of precise multiplier in its implementation.  That is, if one implements the recursion $\eqref{iter}$ with $\beta = \phi = \frac{1 + \sqrt{5}}{2}$, then using the relation $\phi^2 = \phi + 1$, one obtains the recursion formula $u_{n+1} = u_{n-1} + u_n - (b_{n-1} + \phi b_n)$.  If the term $b_{n-1} + \phi b_n$ in this formula is removed, then the resulting recursion $v_{n+1} = v_n + v_{n-1}$ should look familiar; indeed, with initial conditions $(v_0,v_1) = (0,1)$, this recursion generates the Fibonacci numbers $v_n$, and it is well-known that the sequence $\frac{v_{n+1}}{v_n} \rightarrow \phi$ as $n \rightarrow \infty$.   If $b_{n-1} + \phi b_n$ is instead replaced by a single bit taking values in $\{-1,1\}$, then we are led to the following scheme:
\begin{eqnarray}
            u_{0} &=& x    \nonumber \\
           u_{1} &=& 0   \nonumber \\
           \textrm{for $n \geq 0:$ }    	
      b_{n} &=& Q(u_{n}, u_{n+1})      \nonumber \\	
      u_{n+2} &=& u_{n+1} + u_{n} - b_{n}
	  \label{(simple recurs)} 
        \end{eqnarray} 
In this paper, we will consider quantizers $Q$ in $\eqref{(simple recurs)}$ of the form $Q_{\alpha}$, where
\begin{equation}
Q_{\alpha}(u,v)  = \left\{\begin{array}{cl}
	-1,  & u + \alpha v  <  0\\
	1,   &u + \alpha v \geq 0
	   \end{array}\right.
           \label{(f2)}
\end{equation}
along with their flaky analogs,
\begin{equation}
Q_{\alpha}^{\nu}(u,v)  = \left\{\begin{array}{cl}
	-1,  & u + \alpha v  <  -\nu \\
	1,   &u + \alpha v \geq \nu \\
	-1 \textrm{ or } 1, &-\nu \leq u + \alpha v \leq \nu
	   \end{array}\right.
           \label{(fq2)}
\end{equation}
In $\cite{GREsubmit}$, the authors consider the recursion formula $\eqref{(simple recurs)}$ implemented with flaky $\{0,1\}$-quantizers  of the form $\bar{Q}_{\alpha}^{\nu,\iota}(u,v) = \Big[ Q^{\nu}(u + \alpha v - \iota) + 1 \Big]/2$.  For the simplicity of presentation, we will consider only the $\{-1,1\}$-quantizers $\eqref{(fq2)}$, but many of our results extend straightforwardly to quantizers of the type $\bar{Q}_{\alpha}^{\nu,\iota}$. 
\\
\\
The following theorem was proved in $\cite{GREsubmit}$; it shows that as long as $x$ and $Q$ are such that the state sequence $u = \{u_n\}_{n=0}^{\infty}$ remains bounded, a golden ratio encoder (corresponding to the recursion $\eqref{(simple recurs)}$) will produce a bitstream $(b_j)$ corresponding to a beta-expansion of $x$ in base $\beta = \phi$, just as does the beta-encoder from which the GRE was derived.  

\begin{theorem}[Daubechies, G{\"u}nt{\"u}rk, Wang, Yilmaz]
Consider the recursion $\eqref{(simple recurs)}$.  Suppose the 1-bit quantizer $Q$ which outputs bits $(b_j)$ in $\{-1,1\}$ is of the type $Q_{\alpha}^{\nu}$ such that the state sequence $u = \{u_n\}_{n=0}^{\infty}$ with $u_0 = x$ and $u_1 = 0$ is bounded.   Then
\begin{equation}
 |x - \sum_{n=0}^{N} b_n {\phi}^{-n} | \leq \phi^{-N+1}
\label{expand}
\end{equation}
Here $\phi$ is the golden ratio.
\end{theorem}

\begin{proof} 
Note that
\begin{eqnarray}
\sum_{n=0}^{N} b_{n} \phi^{-n} &=& \sum_{n=0}^{N} (u_{n} + u_{n+1} - u_{n+2}) \phi^{-n} \nonumber \\
&=& \sum_{n=0}^{N} u_n \phi ^{-n} +  \sum_{n=1}^{N+1} u_n \phi^{-(n-1)}  - \sum_{n=2}^{N+2} u_n \phi ^{-(n-2)} \nonumber \\
\nonumber \\
&=& u_0 +  (1 + \phi - \phi^2) \sum_{n=2}^{N} \phi^{-n} u_n  \nonumber \\  
&&+ u_1(\phi^{-1} + 1) + \phi^{-N} \Big(u_{N+1} (1 - \phi) - u_{N+2} \Big) \nonumber \\
&=& u_0 + \phi^{-N} \Big(u_{N+1} (1 - \phi) - u_{N+2} \Big) \nonumber \\
&=& x  +\phi^{-N}\Big( u_{N+1} (1 - \phi) - u_{N+2} \Big).  \nonumber
\end{eqnarray}
The second to last equality uses the relation $1+ \phi - \phi^2 = 0$, and the last equality is obtained by setting $u_0 = x$ and $u_1 = 0$.  Since the state sequence $u = \{u_n\}_{n=0}^{\infty}$ is bounded, it follows that $x =  \sum_{n=0}^{\infty} b_n \phi^{-n}$ .   Thus,
\begin{eqnarray}
 |x - \sum_{n=0}^{N} b_n {\phi}^{-n} | &=&  |\sum_{n=N+1}^{\infty} b_n \phi^{-n} | \nonumber \\
 &\leq& \frac{\phi^{-(N+1)}}{1-\phi^{-1}} \nonumber \\
 &=& \phi^{-N+1}. \nonumber
 \end{eqnarray}
 \end{proof} 
Although the implementation of GRE requires 2 more adders and one more delay element than the implementation of the beta-encoder, the multiplier element $\alpha$ in GRE does not have to be precise (see section 6), whereas imprecisions in the multiplier element $\beta$ of the beta-encoder result in beta-expansions of the input $x$ in a different base $\beta'$.  

  
\subsection{GRE: A revised scheme incorporating integrator leak}
In modeling the golden ratio encoder by the system $\eqref{(simple recurs)}$, we are assuming that the delay elements used in its circuit implementation are ideal.  A more realistic model would take into account the effect of integrator leak, which is inevitable in any practical circuit implementation (see $\cite{GregTem}$ for more details).  After one clock time, the stored input in the first delay is reduced to $\lambda_1$ times its original value, while the stored input in the second delay is replaced by $\lambda_2$ times its original value.  In virtually all circuits of interest, no more than $10$ percent of the stored input is leaked at each timestep; that is,  we can safely assume that $\lambda_1$ and $\lambda_2$ are parameters in the interval [.9,1].  The precise values of these parameters may change in time; however, as virtually all practical A/D converters produce over 1000 bits per second (and some can produce over 1 billion bits per second),  we may safely assume that $\lambda_1$ and $\lambda_2$ are constant throughout the quantization of at least every 10 samples.  
\\
\\
Fixing an input value $x \in [-1,1]$, we arrive at the following revised description of the golden ratio encoder (revised GRE): 
\begin{eqnarray}
           u_{0} &=& 0    \nonumber \\
           u_{1} &=& x   \nonumber \\
           \textrm{for $n \geq 0:$ }    
         b_{n} &=& Q(\lambda_1\lambda_2 u_{n}, \lambda_1 u_{n+1})      \nonumber \\		
      u_{n+2} &=& \lambda_1 \lambda_2 u_{n} + \lambda_1 u_{n+1} - b_{n} 
	  \label{(lambda recurs)}
\end{eqnarray}	
Obviously, $\lambda_1 = \lambda_2 = 1$ corresponds to the original model $\eqref{(simple recurs)}$.  It is reasonable to assume in practice that $(\lambda_1,\lambda_2) \in M := [.95,1]^2$, and in virtually all cases $(\lambda_1,\lambda_2) \in V := [.9,1]^2$.
\\
\\
We will show that the revised scheme $\eqref{(lambda recurs)}$ still produces beta-expansions of the input $x$, but in a slightly different base $\gamma = \beta^{-1} = \frac{-\lambda_1 + \sqrt{\lambda_1^2 + 4 \lambda_1 \lambda_2}}{2 \lambda_1 \lambda_2}$, which increases away from $\phi^{-1}$ as the parameters $\lambda_1$ and $\lambda_2$ decrease.   Key in the proof of Theorem $\ref{expand}$ was the use of the relation $\phi^2 - \phi - 1 = 0$ to reduce $\sum_{n=0}^N (u_n + u_{n+1} - u_{n+2}) \phi^{-n}$ to the sum of the input $x$, and a remainder term that becomes arbitrarily small with increasing $N$.   Accordingly, the relation $1  - \lambda_1 \gamma - \lambda_1 \lambda_2 \gamma^2 = 0$ gives $\sum_{n=0}^N (\lambda_1 \lambda_2 u_{n} + \lambda_1 u_{n+1} - u_{n+2}) \gamma^{n+1} = x + R(N)$, where $R(N)$ goes to $0$ as $N$ goes to infinity.   

\begin{theorem}    Suppose the 1-bit quantizer $Q$ in $\eqref{(lambda recurs)}$ of type $\eqref{(fq2)}$ is such that the state sequence $u = \{u_n\}_{n=0}^{\infty}$ with $u_0 = 0$ and $u_1 = x$ is bounded.  Consider $\gamma = \frac{-\lambda_1 + \sqrt{\lambda_1^2 + 4 \lambda_1 \lambda_2}}{2 \lambda_1 \lambda_2}$. Then
\begin{center}
$|x - \sum_{n=0}^{N} b_n \gamma^{n+1} | \leq C_{\gamma} \gamma^{N} $
 \end{center}
 where $C_{\gamma} = \frac{\gamma}{1 - \gamma}$.
\label{exp}
\end{theorem}	

\begin{proof}
\begin{eqnarray}
\sum_{n=0}^{N} b_{n} \gamma^{n+1} &=& \sum_{n=0}^{N} ( \lambda_1 \lambda_2 u_{n} + \lambda_1 u_{n+1} - u_{n+2}) \gamma ^{n+1} \nonumber \\
&=&  (\lambda_1 \lambda_2 \gamma^2 + \lambda_1 \gamma  - 1) \sum_{n=2}^{N-1} \gamma^{n} u_n \nonumber \\
&  & + \lambda_1 \lambda_2 u_0 \gamma + u_1(\lambda_1 \gamma  + \lambda_1 \lambda_2 \gamma^2) \nonumber \\
&  & + \gamma^{N} \Big(u_{N} (\lambda_1 - \gamma^{-1}) - u_{N+1} \Big) \nonumber \\
&=&  x  +\gamma^{N}\Big( u_{N} (\lambda_1 - \gamma^{-1}) - u_{N+1} \Big). \nonumber
\end{eqnarray}
The last equality is obtained by setting $u_0 = 0$ and $u_1 = x$. 
Since the $u_n$ are bounded, it follows as in the proof of $\eqref{(simple recurs)}$ that 
\begin{eqnarray}
 |x - \gamma \sum_{n=0}^{N} b_n \gamma^n | &\leq& \frac{\gamma^{N+1}}{1 - \gamma}. 
 \nonumber
\end{eqnarray}
 \end{proof}

Theorem $\ref{exp}$ implies that if $\gamma$ is known, then the revised GRE scheme $\eqref{(lambda recurs)}$ still gives exponential approximations to the input signal $x$, provided that the $u_n$ are indeed bounded.   The following theorem gives an explicit range for the parameters $(\nu, \alpha)$ which results in bounded sequences $u_n$ when the input $x \in [-1,1]$, independent of the values of the leakage parameters $(\lambda_1,\lambda_2)$ in the set $V = [.9,1]^2$.  This parameter range is only slightly more restrictive than that derived in $\cite{GREsubmit}$ for the ideal GRE scheme $\eqref{(simple recurs)}$; that is, the admissable parameter range for $\alpha$ and $\nu$ is essentially robust with respect to leakage in the delay elements of the GRE circuit implementation.  
\begin{theorem}
Let $x \in [-1,1]$, and $(\lambda_1, \lambda_2) \in [.9, 1]^2$.  Suppose that the GRE scheme $\eqref{(lambda recurs)}$ is followed, and the quantizer $Q_{\alpha}^{\nu}(u,v)$ is used, with $\nu$ possibly varying at each occurrence, but always satisfying $\nu \leq \epsilon$ for some fixed tolerance $\epsilon \leq .337$.    If the parameter $\alpha$ takes values in the interval $[1.198 + 1.479 \epsilon, 2.053 - 1.058 \epsilon]$, then the resulting state sequence $(u_j)_{j \in N}$ is bounded.
\label{range}
\end{theorem}
We leave the proof of Theorem $\ref{range}$ to the appendix.  This result in some sense parallels  Theorem $\ref{b-encod+}$ in that an admissable range for the "multiplication" parameter  ($\alpha$ in GRE, $\beta$ in beta-encoders) is specified for a given quantizer tolerance $\epsilon$; however, we stress that the specific value of $\beta$ is needed in order to recover the input from the bitstream $(b_j)$ in beta-encoders.  In contrast, a GRE encoder can be built with a multiplier $\alpha$ set at any value within the range $[1.198 + 1.479 \epsilon, 2.053 - 1.058 \epsilon]$, and as long as this multiplier element  has enough precision that the true value of $\alpha$ will not stray from this interval, then the resulting bitstreams $(b_j)$ will always represent a series expansion of the input $x$ in base $\gamma$ of Theorem $\ref{exp}$, which does not depend on $\alpha$.  The base $\gamma$ does however depend on the leakage parameters $\lambda_1$ and $\lambda_2$, which are not known a priori to the decoder and can also vary in time from input to input: as discussed earlier, the only information available to the decoder a priori is that $(\lambda_1,\lambda_2) \in [.9,1]^2$ in virtually all cases of interest, and $(\lambda_1,\lambda_2) \in [.95,1]^2$, or $\gamma \in [\phi^{-1}, \frac{20}{19} \phi^{-1}] \approx [.618, .6505]$ in most cases of interest.  In the next section, we show that the upper bound of $.6505$ is sufficiently small that the value of $\gamma$ can be recovered with exponential precision from the encoded bitstreams of a pair of real numbers $(x,-x)$.  In this sense, GRE encoders are robust with respect to leakage in the delay elements, imprecisions in the multiplier $\alpha$, { \it and } quantization error.  

\subsection{Determining $\gamma = 1/\beta$ up to exponential precision }
\subsubsection{Approximating $\gamma$ using an encoded bitstream for x = 0}
Recall that by Theorem $\ref{daub}$, exponentially precise approximations $\tilde{\gamma}$ to the root $\gamma$ in Theorem $\ref{exp}$ are sufficient in order to reconstruct subsequent input $x_n = f(t_n) \in [-1,1]$ whose bit streams are expansions in root $\gamma$ with exponential precision.  In this section, we present a method to approximate $\gamma$ with such precision using the only information at our disposal at the decoding end of the quantization scheme: the encoded bitstreams of real numbers $x \in [-1,1]$.  More precisely, we will be able to recover the value $\gamma = \beta^{-1} $ using only a single bitstream corresponding to a beta-expansion of the number 0.  It is easy to adapt this method to slow variations of $\gamma$ in time, as one can repeat the following procedure at regular time intervals during quantization, and update the value of $\gamma$ accordingly.  
\\
\\
 The analysis that follows will rely on the following theorem by Peres and Solomyak $\cite{solom96}$:
\begin{theorem}[Peres-Solomyak] [$\delta$-transversality]
Consider the intervals $I_{\rho} = [0,\rho]$.  If $\rho \leq .6491...$, then for any $g$ of the form
\begin{equation}
g(x) = 1 + \sum_{j=1}^{\infty} b_j x^j \textrm{,                    }  b_j \in \{-1,0,1\}
\label{g}
\end{equation}
and any $x \in I_{\rho}$, there exists a $\delta_{\rho} > 0$ such that if $g(x) < \delta_{\rho}$ then $g'(x) < -\delta_{\rho}$.  Furthermore, as $\rho$ increases to $.6491...$, $\delta_{\rho}$ decreases to 0.
\label{trans}
\end{theorem}

Theorem $\ref{trans}$ has the following straightforward corollary:
\begin{corollary}
If $g(x)$ is a polynomial or power series belonging to the class $B$ given by
\begin{equation}
B = \{ \pm 1 + \sum_{j=1}^{\infty} b_j x^j  \textrm{,     }: b_j \in \{-1,0,1\} \}
\label{g+}
\end{equation}
then $g$ can have no more than one root on the interval $(0, .6491]$.  Furthermore, if such a root exists, then this root must be simple.
\label{root}
\end{corollary}
In $\cite{solom96}$, Peres and Solomyak used Theorem $\ref{trans}$ to show that the distribution $\nu_\lambda$ of the random series  $\sum \pm \lambda^n$ is absolutely continuous for a.e. $\lambda \in (1/2, 1)$.   The estimates in Theorem $\ref{trans}$ are obtained by computing the smallest double zero of a larger class of power series $\bar{B} := \{1 + \sum_{n=1}^{\infty} b_nx^n : b_n \in [-1,1] \}$.   The specific upper bound $\rho = .6491...$ for which $\delta$-transversality holds on the interval $I_{\rho}$ is the tightest bound that can be reached using their method of proof, but the true upper bound cannot be much larger; in $\cite{zeros06}$, a power series $g(t)$ belonging to the class $B$ in $\eqref{g+}$ is constructed which has a double zero at $t_0 \approx .68$ (i.e., $g(t_0) = 0$ and $g'(t_0) = 0$), and having a double root obviously contradicts $\delta$-transversality.
\\
\\
We now show how to use Theorem $\ref{trans}$ to recover $\gamma$ from a bitstream $(b_j)_{j=1}^{\infty}$ produced from the GRE scheme $\eqref{(lambda recurs)}$ corresponding to input $x=0$.  We assume that the parameters $(\alpha,\nu)$ of the quantizer $Q_{\alpha}^{\nu}$ used in this implementation are within the range provided by Theorem $\ref{range}$.  Note that such a bitstream $(b_j)_{j=1}^{\infty}$ is not unique if $\nu > 0$, or if more than one pair $(\lambda_1,\lambda_2)$ correspond to the same value of $\gamma$.   Nevertheless, Theorem $\ref{exp}$ implies that $ \sum_{j=1}^{\infty} b_j \gamma^{j} = 0$, so that $\gamma$ is a root of the power series $F(t) = b_1 + \sum_{j=1}^{\infty} b_{j+1} t^j$.  Suppose we know that $\beta \geq 1.54056$, or that $\gamma \leq .6491$.  Since $F(t)$ belongs to the class $B$ of Corollary $\eqref{root}$, $\gamma$ must necessarily be the { \it smallest} positive root of $F(t)$.  In reality one does not have access to the entire bitsream $(b_j)_{j=1}^{\infty}$, but only a finite sequence $(b_j)_{j=1}^{N}$.  It is natural to ask whether we can approximate the first root of $F(t)$ by the first root of the polynomials $P_n(t) = b_1 + \sum_{j=1}^{n} b_{j+1} t^j$.   Since the $P_n$ still belong to the class $B$ of Corollary $\eqref{root}$, $|P_n(t)|$ has {\it at most} one zero on the interval $[\phi^{-1},.6491] \approx [.618, .6491]$.  The following theorem shows that, if it is known a priori that $\gamma \leq .6491 - \epsilon$ for some $\epsilon > 0$, then for $n$ sufficiently large, $|P_n(t)|$ is guaranteed to have a root $\gamma_n$ in $[0, .6491]$, and  $| \gamma - \gamma_n |$ decreases exponentially as $n$ increases.  

\begin{theorem}
Suppose that for some $\epsilon > 0$,  it is known that $\gamma \leq \gamma_{high} = .6491 - \epsilon$.   Let $\delta > 0$ be such that $\delta$-transversality holds on the interval $[0, \gamma_{high}]$.  Let $N$ be the smallest integer such that $\gamma^{N+1} \leq (1 - \gamma) \epsilon \delta$.  Then for $n \geq N$, \begin{enumerate}
\renewcommand{\labelenumi}{(\alph{enumi})}
\item $P_n$ has a unique root $\gamma_n$ in $[0,.6491]$
\item $|\gamma - \gamma_n| \leq C_1 \gamma^n$, where $C_1 = \frac{\gamma}{\delta(1 - \gamma)}$.
\end{enumerate}  
\label{mytheorem}
\end{theorem}

\begin{proof}
Without loss of generality, we assume $b_1 = 1$, and divide the proof into 2 cases: (1) $P_n(\gamma) \leq 0$, and (2) $P_n(\gamma)  > 0$.   The proof is the same for $b_1 = -1$, except that the cases are reversed.
\\
{\bf Case (1)}: Suppose $P_n(\gamma) \leq 0$.  In this case, many of the restrictions in the theorem are not necessary; in fact, the theorem holds here for all $n > 0$ and $\epsilon \geq 0$.    $P_n(t)$ has opposite signs at $t=0$ and $t=\gamma$, so $P_n$ must have at least one root $\gamma_n$ in between.  Moreover, this root is unique by Theorem $\ref{trans}$.   To prove part (b) of the theorem, observe that  Theorem $\ref{trans}$ implies that if $P_n(t_0) \leq \delta$ at some $t_0 \in (0,.6491)$, then $P_n'(t) \leq - \delta$ in the interval $[t_0,.6491)$.  In particular, $P_n(t) \leq \delta$ and $P_n'(t) \leq - \delta$ for $t \in [\gamma_n, \gamma]$.  By the Mean Value Theorem, $|P_n(\gamma)| = |P_n(\gamma) - P_n(\gamma_n)| = |P_n'(\xi)| |\gamma - \gamma_n|$ for some $\xi \in [\gamma_n, \gamma]$, so that
\begin{eqnarray}
    |\gamma - \gamma_n| &\leq&  \frac{|P_n(\gamma)|}{\inf_{\xi \in [\gamma_n,\gamma]} |P_n'(\xi)|} \nonumber \\
    &\leq&  \frac{|P_n(\gamma)|}{\delta} \nonumber \\ 
    &\leq&  \frac{\gamma^{n+1}}{(1 - \gamma)\delta} .
    \nonumber
  \end{eqnarray}  
 The inequality $|P_n(\gamma)| \leq \frac{\gamma^{n+1}}{1-\gamma}$ follows from Theorem $\ref{exp}$.
\\
{\bf Case (2)}:  If $P_n(\gamma) > 0$, then Theorem $\ref{trans}$ implies that the first positive root of $P_n$, if it exists, must be greater than $\gamma$.   Whereas in Case (1), $P_n$ was guaranteed to have a root $\gamma_n \leq \gamma$ for all $n \geq 0$, in this case $P_n$ might not even have a root in $(0,1)$; for example, the first positive root of the polynomial $P_4(t) = 1 - t - t^2 + t^3$ occurs at $t = 1$.   However, if $n$ is sufficiently large, $P_n$ will have a root in $[\gamma, .6491]$.  Precisely, let $n \geq N$, where $N$ is the smallest integer such that $\gamma^{N+1} \leq (1 - \gamma) \epsilon \delta$.  Then $P_n(\gamma) \leq \frac{\gamma^{n+1}}{1-\gamma} \leq \frac{\gamma^{N+1}}{1-\gamma}  \leq \epsilon \delta \leq \delta$.  Since $\gamma + \epsilon \leq .6491$, Theorem $\ref{trans}$ gives that $P_n'(t) \leq - \delta$ for $t \in [\gamma, \gamma + \epsilon]$, and    
\begin{eqnarray}
P_n(\gamma + \epsilon) &\leq& P_n(\gamma) + \epsilon \sup_{t \in [\gamma, \gamma+\epsilon] } P_n'(t) \nonumber \\ &\leq& \epsilon \delta - \epsilon \delta \nonumber \\  &=& 0 \nonumber.
\end{eqnarray}
We have shown that $P_n(\gamma) > 0$ and $P_n(\gamma + \epsilon) \leq 0$, so $P_n$ is guaranteed to have a root $\gamma_n$ in the interval $[\gamma, \gamma + \epsilon]$, which is in fact the unique root of $P_n$ in $[0,.6491]$.  Furthermore, the discrepency between $\gamma_n$ and $\gamma$ becomes exponentially small as $n > N$ increases, as
\begin{center}
$| \gamma_n - \gamma | \leq \frac{|P_n(\gamma)|}{\inf_{\xi \in [\gamma,\gamma_n]} |P_n'(\xi)|} 
  \leq  \frac{\gamma^{n+1}}{(1 - \gamma)\delta}$  .
  \end{center}
 \end{proof}
 
\begin{remark} In $\cite{solom96}$ it is shown that $\delta$-transversality holds on $[0, .63]$ with $\delta = \delta_{.63} = .07$.  This interval corresponds to $\epsilon = .6491 - .63 = .0191$ in Theorem $\ref{mytheorem}$.   If $\gamma$ is known a priori to be less than $.63$, then $\frac{\gamma}{1 - \gamma} \leq \frac{.63}{1-.63} \approx 1.7027$, and $N$ of Theorem $\eqref{mytheorem}$ satisfies
\begin{center}
 $N \leq \frac{\log{\epsilon\delta} - \log{1.7027}}{\log{.63}} \leq 16$.
 \end{center}
 \end{remark}
 
 Of course, even if we know $P_n(t)$,  $n$ will be too large to solve for $\gamma_n$ analytically.  However, if we can approximate $\gamma_n$ by $\tilde{\gamma}_n$ with a precision of $O(\gamma^n)$, then $|\gamma - \tilde{\gamma}_n| \leq |\gamma - \gamma_n| + |\gamma_n - \tilde{\gamma}_n|$ will also be of order $\gamma^n$, so that the estimates $\tilde{\gamma}_n$ are still exponentially accurate approximations to the input signal.   
\\
\\
Indeed, any $\tilde{\gamma}_n \in [\phi^{-1},.6491 - \epsilon]$ satisfying $| P_n(\tilde{\gamma}_n) | \leq \phi^{-n}$ provides such an exponential approximation to $\gamma_n$.  Since $n \geq N$, it follows that $\phi^{-n} \leq \gamma^n \leq \gamma^N \leq \frac{(1 - \gamma) \epsilon \delta}{\gamma} \leq \delta$, and so $P_n'(t) \leq -\delta$ on the interval between $\tilde{\gamma}_n$ and $\gamma_n$.  Thus, 
\begin{eqnarray}
 | \gamma - \tilde{\gamma}_n | &\leq& | \gamma - \gamma_n | + | \gamma_n - \tilde{\gamma}_n | 
 \nonumber \\
 &\leq& C_1 \gamma^{n} +  \frac{|P_n(\tilde{\gamma}_n)|}{\inf_{\xi \in [\tilde{\gamma}_n,\gamma_n] or \xi \in [\gamma_n, \tilde{\gamma}_n]} |P_n'(\xi)|}  \nonumber \\
 &\leq& C_1 \gamma^{n} + \frac{\phi^{-n}}{\delta} \nonumber \\
 & \leq& C_1 \gamma^{n} + \frac{\gamma^n}{\delta} \nonumber \\
 & =  & C_2 {\gamma^n},
 \label{search}
\end{eqnarray}
where $C_2 = C_1 + \frac{1}{\delta} = \frac{1}{\delta(1 - \gamma)}$.

\subsubsection{Approximating $\gamma$ with beta-expansions of $(-x,x)$}

The method for approximating the value of $\gamma$ in the previous section requires a bitstream $b$ corresponding to running the GRE recursion $\eqref{(lambda recurs)}$ with specific input $x = 0$.  This assumes that the reference value $0$ can be measured with high precision, which is an impractical constraint.  We can try to adapt the argument using bitstreams of an encoded pair $(x,-x)$ as follows.   Let   $b$ and $c$ be bitstreams corresponding to $x$ and $-x$, respectively.  Define $d_j = b_j + c_j$. Put $k = min \{j \in N |  d_{j+1} \neq 0 \}$, and consider the sequence $\bar{d} = (\bar{d}_j)_{j=1}^{\infty}$ defined by $\bar{d}_j = \frac{1}{2} d_{j+k}$.  Since $d_j \in \{-2,0,2\}$, it follows that $\bar{d}_j \in \{-1,0,1\}$, and by Theorem $\ref{exp}$, we have that
\begin{eqnarray}
\sum_{n=1}^{\infty} \bar{d}_n \gamma^n &=&  \frac{1}{2} \gamma^{-k} \sum_{n=1}^{\infty} d_n \gamma^n \nonumber \\
          &=& \frac{1}{2} \gamma^{-k} \Big[ \sum_{n=1}^{\infty} b_n \gamma^n + \sum_{n=1}^{\infty} c_n \gamma^n \Big]  \nonumber \\
             &=& 0      
             \end{eqnarray}                  
so that 
\begin{equation}
| \sum_{n=1}^{N} \bar{d}_n \gamma^n | \leq C_{\gamma} \gamma^N,
\label{close}
\end{equation}
where the constant $C_{\gamma} = \frac{\gamma}{1 - \gamma}$.

Equation $\eqref{close}$, along with the fact that the polynomials $\bar{P}_N(t) :=  \bar{d}_1 + \sum_{j=1}^{N} \bar{d}_{j+1} t^j $ are of the form $\eqref{g+}$, allows us to apply Theorem $\ref{mytheorem}$ to conclude that for $N$ sufficiently large, the first positive root $\gamma_N$ of $\bar{P}_N$ becomes exponentially close to $\gamma$.  However, note that the encoding of $(x,-x)$ is {\it not} equivalent to the encoding of $0$.  The value of $k = min \{j \in N |  d_{j+1} \neq 0 \}$ used to define the sequence $\bar{d}$ can be {\it arbitrarily large}.  In fact, if an ideal quantizer $Q_{\alpha}^0$ is used, then the bitstreams $b$ and $c$ are uniquely defined by $b \equiv -c$, so that $d \equiv 0$.  Thus, this method for recovering $\gamma$ actually {\it requires} the use of a flaky quantizer $Q_{\alpha}^{\nu}$.  To this end, one could intentially implement GRE with a quantizer which toggles close to, but not exactly at zero.  One could alternatively send not only a single pair of bitstreams $(b,c)$, but {\it multiple} pairs of bitstreams $(b^l, c^l)$ corresponding to several pairs $(x_l,-x_l)$, to increase the chance of having a pair that has $b^{L^j} + c^{L^j} \neq 0$ for relatively small $j$. 
\\
\\
Figure ~\ref{fig:fig3} plots several instances of $\bar{P}_{8}, \bar{P}_{16}$, and $\bar{P}_{32}$, corresponding to $\gamma = .64375$.  The quantizer $Q_{\alpha}^{\nu}(u,v)$ is used, with $\alpha = 2$ and $\nu = .3$.  These values of $\alpha$ and $\nu$ generate bounded sequences $(u_n)$ for all $(\lambda_1,\lambda_2) \in [.9,1]^2$ by Theorem $\ref{range}$.  Figure ~\ref{fig:fig4} plots several instances of the same polynomials, but with root $\gamma = .75$.  
\\
\\
As shown in Figure ~\ref{fig:fig2}, numerical evidence suggests that 10 iterations of Newton's method starting from $x_0 = \phi^{-1} \approx .618$ will compute an approximation to $\gamma_N$, the first root of $\bar{P}_N$, with the desired exponential precision.  The figure plots $\gamma$ versus the error $| \gamma - \tilde{\gamma}_N|$, where $\tilde{\gamma}_N$ is the approximation to $\gamma_N$ obtained via a 10- step Newton Method, starting from $x_0 = .618$.  More precisely, for each $N$, we ran $100$ different trials, with $x$ and $\gamma$ picked randomly from the intervals $[-1,1]$ and $[.618,.7]$ respectively, and independently for each trial.   The worst case approximation error of the 100 trials is plotted in each case.  The quantizer used is $Q_{\alpha}^{\nu}(u,v)$ with $\nu = .3$, and $\alpha$ picked randomly in the interval $[1.7,2]$, independently for each trial.  Again, Theorem $\ref{range}$ shows that these values of $\alpha$ and $\nu$ generate bounded sequences $(u_n)$ for all $(\lambda_1,\lambda_2) \in [.9,1]^2$.

\subsection{The beta-encoder revisited}
Even though our analysis of the previous section was motivated by leaky GRE encoders, it can be applied to general beta-encoders to recover the value of $\beta$ at any time during quantization.  From the last section, we have: 
\begin{theorem}
Let $F(t) = \sum_{j=0}^{\infty}b_j t^j$ be a power series belonging to the class $B = \{ \pm 1 + \sum_{j=1}^{\infty} b_j x^j  \textrm{,     } b_j \in \{-1,0,1\} \}$.  Suppose that $F$ has a root at $\gamma \in [\gamma_{low},\gamma_{high}]$, where $\gamma_{high} = .6491 - \epsilon$ for some $\epsilon > 0$.   Let $\delta > 0$ be such that $\delta$-transversality holds on the interval $[0, \gamma_{high}]$.  Let $N$ be the smallest integer such that $\gamma^{N+1} \leq \epsilon \delta (1 - \gamma)$. Then for $n \geq N$, \begin{enumerate}
\renewcommand{\labelenumi}{(\alph{enumi})}
\item The polynomials $P_n(t) = \sum_{j=0}^{n}b_j t^j$ have a unique root $\gamma_n$ in $[0,.6491]$
\item Any $\tilde{\gamma} \in [\gamma_{low}, \gamma_{high}]$ which satisfies $| P_n(\tilde{\gamma}) | \leq (\gamma_{low})^n$ also satisfies $|\tilde{\gamma} - \gamma| \leq \tilde{C}|\gamma|^n$, where $\tilde{C} = \frac{1}{\delta(1 - \gamma)}$.
\end{enumerate}  
\label{mybigtheorem}
\end{theorem}
This theorem applies to beta encoders, corresponding to implementing the recursion $\eqref{iter}$ with flaky quantizer $Q^{\nu}$ defined by $\eqref{(fqs)}$, and with $\gamma = \beta^{-1}$ known a priori to be contained in an interval $[\gamma_{low}, \gamma_{high}]$.   If $\beta \geq 1.54059$, or $\gamma_{high} \leq .6491$, then we can recover $\gamma = \beta^{-1}$ from either a bitstream corresponding to 0, or a pair of bitstreams $(x,-x)$, using Theorem $\ref{mybigtheorem}$.  Of course we should not consider only the scheme $\eqref{iter}$, but rather a revised scheme which accounts for integrator leak on the (single) integrator used in the beta-encoder implementation.
  The revised beta-encoding scheme, with slightly different initial conditions, becomes
\begin{eqnarray}
u_1 &=& x \nonumber \\
b_1 &=& Q(u_1) \nonumber \\
\textrm{for $j \geq 1$ }: u_{j+1} &=& \lambda \beta(u_j - b_j) \nonumber \\
b_{j+1} &=& Q(u_{j+1}) 
\label{iterlambda}
\end{eqnarray}
where $\lambda$ is an unknown parameter in $[.9,1]$.  As long as $\tilde{\beta} = \lambda \beta > 1$, we still have that $|x - \sum_{i=0}^{N} b_{i+1} \gamma^{i}| \leq C_{\tilde{\beta}} \tilde{\beta}^{-N}$ where $C_{\tilde{\beta}} = \frac{1}{\tilde{\beta}-1}$; furthermore, we can use Theorem $\ref{mybigtheorem}$ to recover $\tilde{\beta} = \lambda \beta$ in $\eqref{iterlambda}$, although the specific values of $\lambda$ and $\beta$ cannot be distinguished, just as the specific values of $\lambda_1$ and $\lambda_2$ in the expression for $\gamma$ could not be distinguished in GRE.

\begin{figure*}[htbp]
    \mbox{
      \subfigure[]{ \includegraphics[width=2.5in]{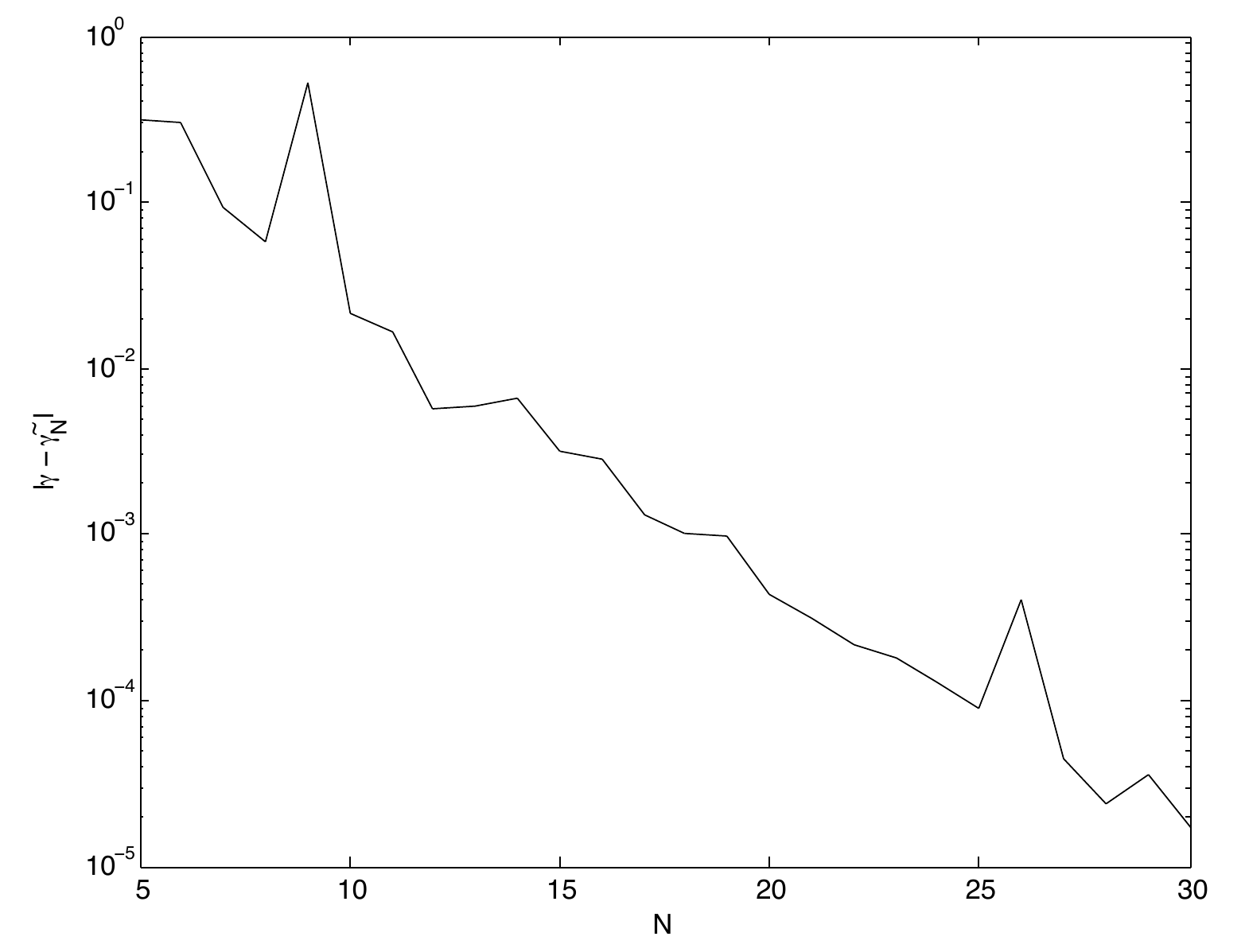}}
      \subfigure[]{\includegraphics[width=2.5in]{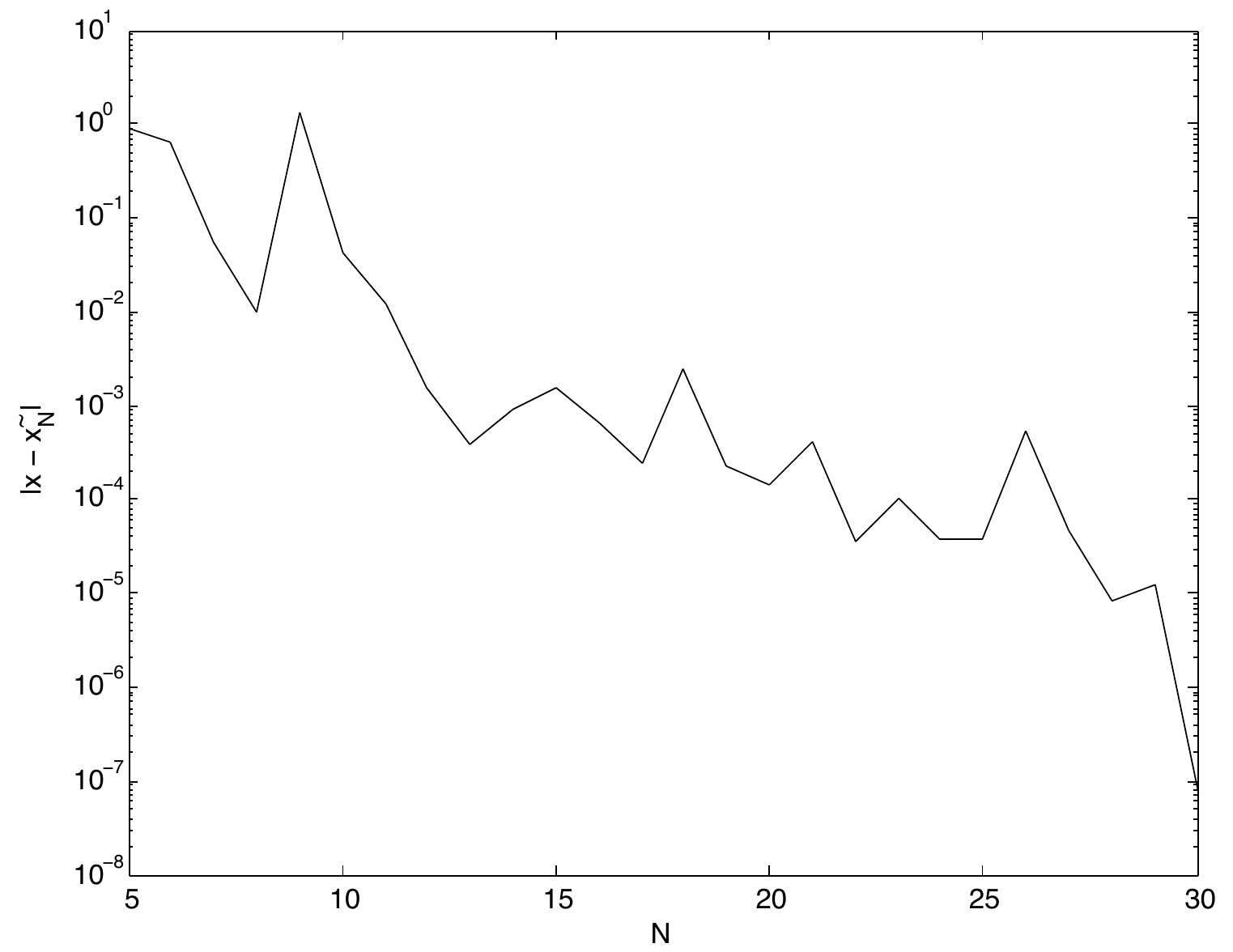}} 
      }
    \caption{(a) $N$ versus $ | \gamma - \tilde{\gamma}_N | $, where $\tilde{\gamma}_N$ is obtained via a 10 step Newton Method approximation to the first root of the polynomial $\bar{P}_N$ starting from $x_0 = .618$.  For each $N$, $| \gamma - \tilde{\gamma}_N |$ is the worst-case error among 100 experiments corresponding to different $(x_j,-x_j)$ pairs and different $\gamma_j$ chosen randomly from $[.618,.7]$.  (b) $N$ versus $ | x - \tilde{x}_N | $, where $\tilde{x}_N$ is reconstructed from $\tilde{\gamma}_N$ using the first $N$ bits.  The quantizer used in this experiment is $Q_{\alpha}^{\nu}$ with $\nu = .3$ and $\alpha$ chosen randomly in the interval $[1.7,2]$.}
 \label{fig:fig2}
\end{figure*}

\subsubsection{Remarks}
Figure ~\ref{fig:fig2} suggests that the first positive roots of the $P_n$ serve as exponentially precise approximations to $\gamma$ for values of $\gamma$ greater than $.6491$; 
Figure ~\ref{fig:fig4} suggests that the first positive root of $P_n$ will approximate values of $\gamma$ up to $\gamma = .75$.  Furthermore, these figures suggest that the constants $C_1$ and $C_2$ of $\eqref{search}$ in the exponential convergence of these roots to $\gamma$ can be made much sharper, even for larger values of $\gamma$. This should not be surprising, considering that nowhere in the analysis of the previous section did we exploit the specific structure of beta-expansions obtained via the particular recursions $\eqref{iterlambda}$ and $\eqref{(simple recurs)}$, such as the fact that such sequences $(b_n)$ cannot contain infinite strings of successive 1's or -1's.  It is precisely power series with such infinite strings that are the `extremal cases' which force the bound of $.6491$ in Theorem $\ref{trans}$.  It is difficult to provide more refined estimates for the constants $C_1$ and $C_2$ of $\eqref{search}$ in general, but in the idealized setting where beta-expansions of $0$ are available via the ideal GRE scheme $\eqref{(simple recurs)}$ without leakage,  or via the beta-encoding $\eqref{(simple recurs)}$ with $\beta = \phi$, the beta-expansions of $0$ have a very special structure:
\begin{proposition}
Consider the ideal GRE recursion $\eqref{(simple recurs)}$ with input $u_0 = u_1 = 0$ and $Q= Q_{\alpha}^{\nu}$, or the beta-encoder recursion $\eqref{iter}$ with $\beta = \phi$, $u_0 = 0$, and $Q = Q^{\nu}$.   As long as $\alpha >  \nu$ in $\eqref{(simple recurs)}$, or $\nu \leq 1$ in $\eqref{iter}$, then for each  $j \in \{0,1, ... \}$, $b_{3j}$ is equal to $-1$ or $+1$, and $b_{3j + 1} = b_{3j+2} = -b_{3j} $.    
\label{period}
\end{proposition}
The proof is straightforward, and we omit the details.
\\
\\
Proposition $\ref{period}$ can be used to prove directly that $\gamma$ must be the first positive root of the polynomials $P_N(t) = b_1 + \sum_{j=1}^{3N} b_{j+1} t^j$, when $\gamma = \phi^{-1}$ in either $\eqref{(simple recurs)}$ or $\eqref{iter}$.  Indeed, $P_N$ can be factored as follows:
\begin{eqnarray}
P_N(t)  &=& \sum_{j=0}^{N} (b_{3j+1} - b_{3j+1} t - b_{3j+1} t^2) t^{3j} \nonumber \\
        &=& (1 - t - t^2) \sum_{j=1}^{N} b_{3j+1} t^{3j} \nonumber \\
        &=& (1 - t - t^2) R_N(t)
\end{eqnarray}
where $R_N$ is a polynomial with random coefficients of the form $\sum_{j=0}^{N} \pm t^{3j}$.
\\
\\
$P_N(t) = (1 - t - t^2)R_N(t)$ clearly has a root at $t = \phi^{-1}$, and this root must be the only root of $P_N$ on $[0, \phi^{-1})$, since on this interval $R_N$ is bounded away from $0$ by $|R_N(t)| \geq 1 - \frac{t^3}{1 - t^3} \geq 1 - \frac{\phi^{-3}}{1 - \phi^{-3}} \approx .691$.  We can also obtain a lower bound on $|P_N'(\phi^{-1})| = -(1 + 2\phi^{-1})R_N(\phi^{-1})$ by $|P_N'(\phi^{-1})| \geq |1 + 2\phi^{-1}| |1 -  \frac{\phi^{-3}}{1 - \phi^{-3}}| \geq 1.545$.  Note that this bound holds also for the infinite sum $F(t) = b_1 + \sum_{j=1}^{\infty} b_{j+1} t^j$, and that the bound of $|F'(\gamma)| \geq 1.545$ is much sharper than the bound on $|F'(\gamma)|$ given by Theorem $\ref{trans}$; e.g., $\delta_{.63} = .07$, and $\delta_{.6491} = .00008$ (see $\cite{solom96}$).  Similar bounds on the derivatives $|P_N'(\gamma)|$ and $|F_N'(\gamma)|$ corresponding to beta-expansions of $0$ in a base $\gamma$ close to $\phi^{-1}$ should hold, leading to sharper estimates on the constants $C_1$ and $C_2$ of $\eqref{search}$ in the case where beta-expansions of $0$ are available.    

\subsection{Closing remarks}
We have shown that golden ratio encoders are robust with respect to leakage in the delay elements of their circuit implementation.  Although such leakage may change the base $\gamma$ in the reconstruction formula $|y - \sum_{j=1}^{N} b_j \gamma^j| \sim  O(\gamma^N)$ , we have shown that exponentially precise approximations $\tilde{\gamma}$ to $\gamma$ can be obtained from the bitstreams of a pair $(x,-x)$, and such approximations $\tilde{\gamma}$ are sufficient to reconstruct subsequent input $y$ by $|y - \sum_{j=1}^{N} b_j \tilde{\gamma}^j| \sim  O(\gamma^N)$.   
\\
\\
Our method can be extended to recover the base $\beta$ in general beta-encoders, as long as $\beta$ is known a priori to be sufficiently large; e.g. $\beta \geq 1.54$.   This method is similar to the method proposed in $\cite{DDGV06}$ for recovering $\beta$ in beta-encoders when $\{0,1\}$-quantizers are used, except that our method does not require a fixed reference level, which is difficult to measure with high precision in practice.

\begin{figure*}[htbp]
\mbox{
\subfigure[]{\includegraphics[width=2in]{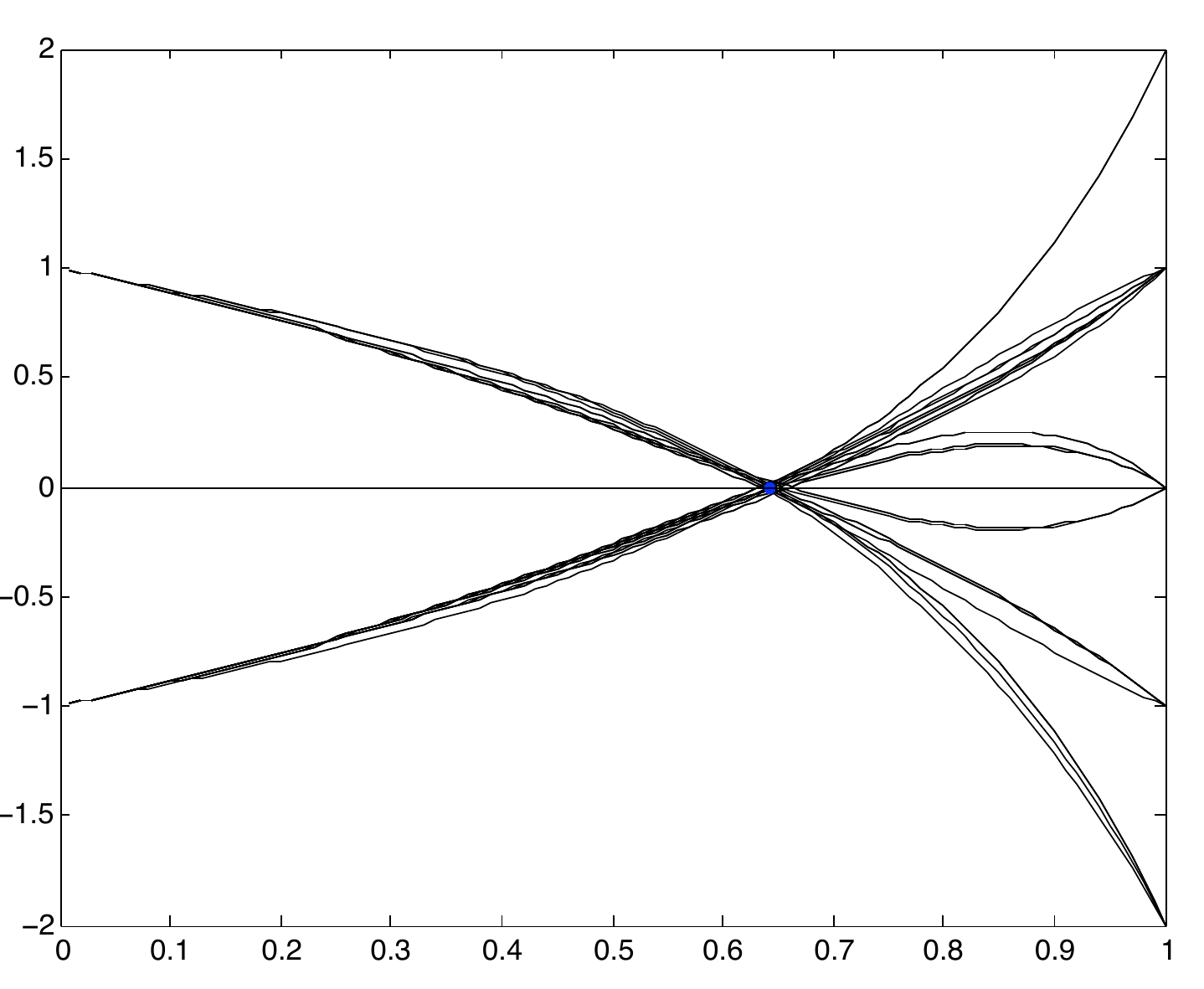}}
\subfigure[]{\includegraphics[width=2in]{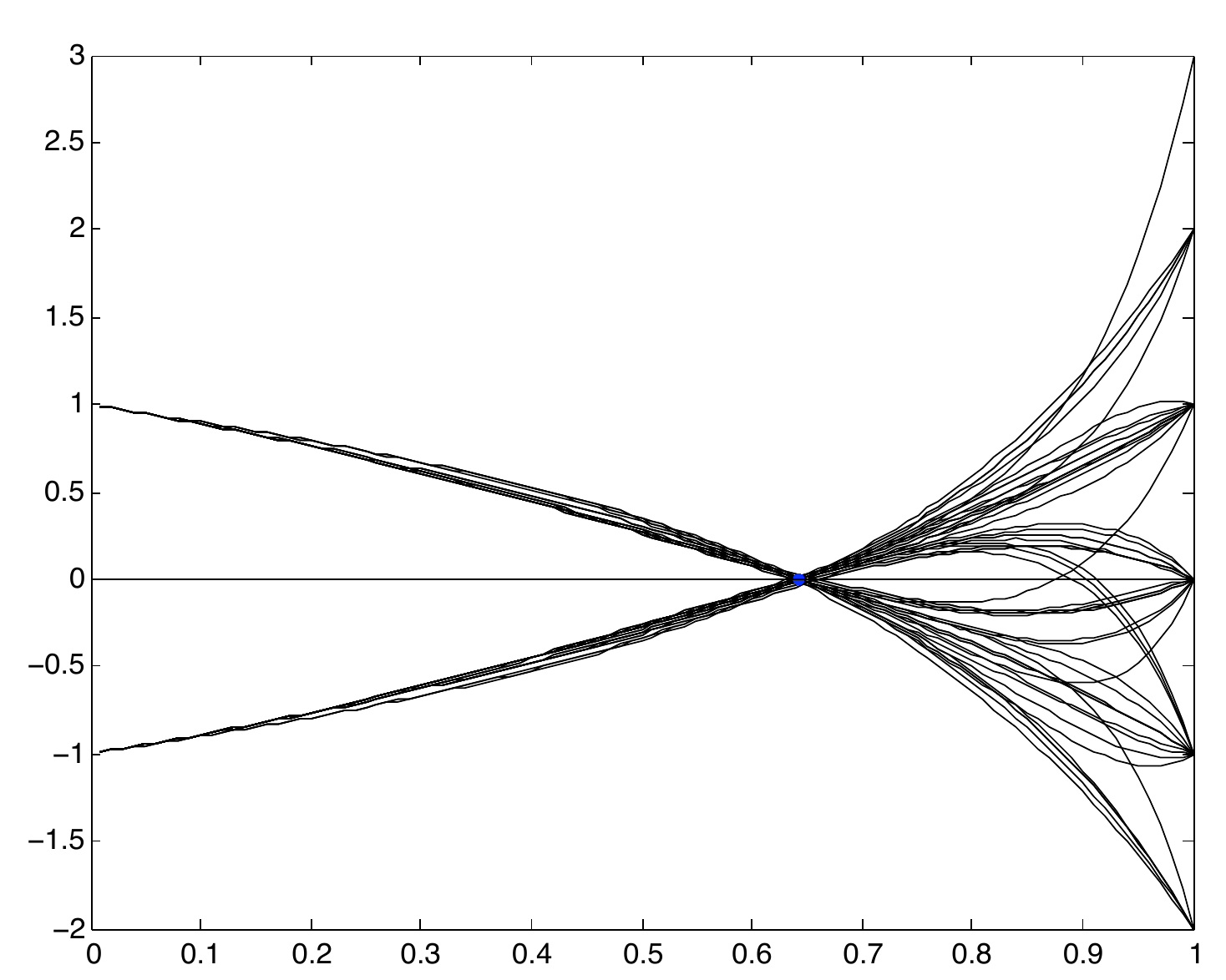}} 
\subfigure[]{\includegraphics[width=2in]{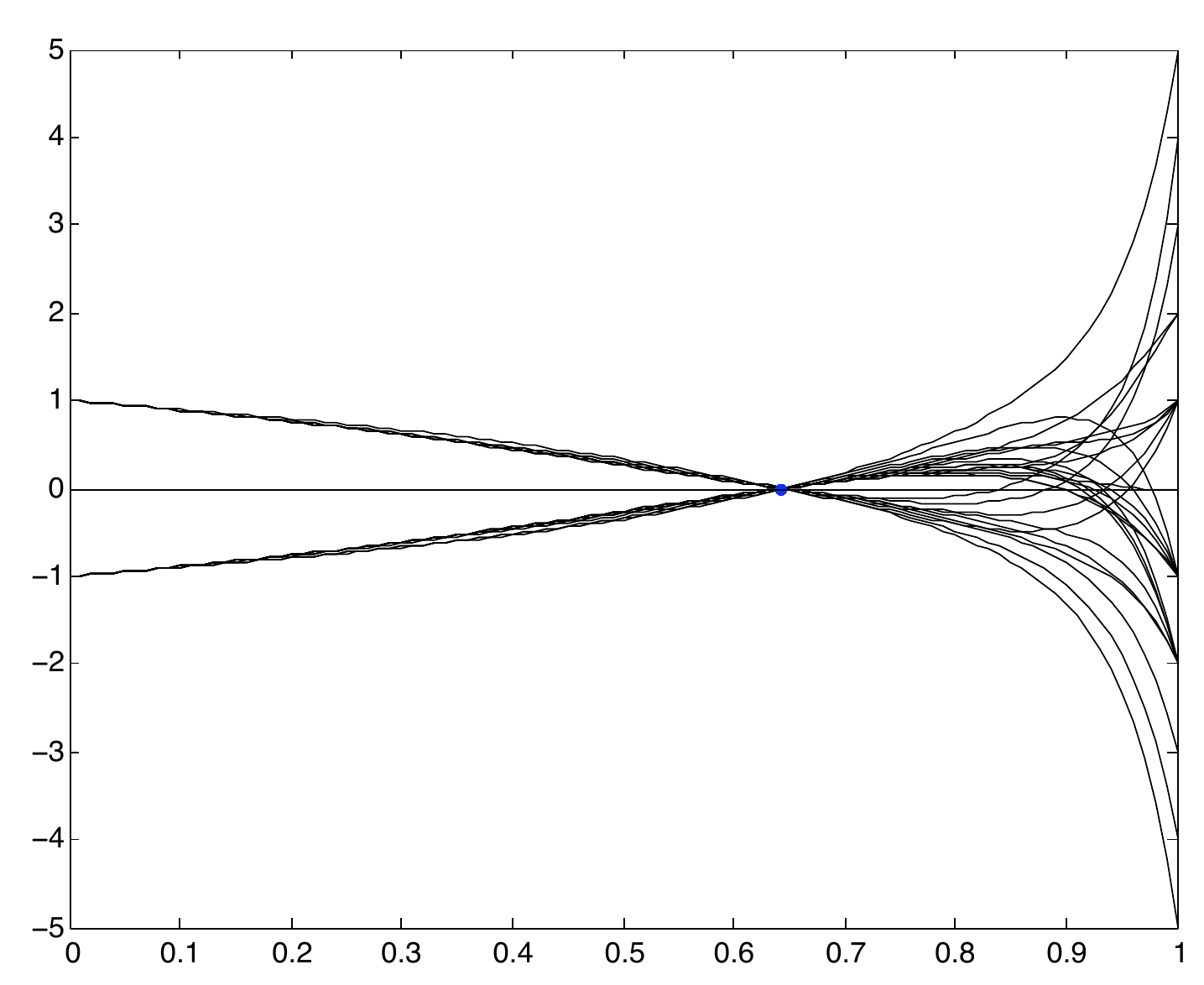}}
}
\caption{The graphs of (a) $\bar{P}_{8}$, (b) $\bar{P}_{16}$, and (c) $\bar{P}_{32}$ for several pairs $(x_j,-x_j)$ corresponding to $\gamma = .64575$.  These polynomials can have several roots on the unit interval, but the first of these roots must become exponentially close to $\gamma$ as $N$ increases.    The quantizer used is $Q_{\alpha}^{\nu}$ with parameter values $\nu = .3$, and $\alpha = 2$.}
\label{fig:fig3}
\end{figure*}

\begin{figure*}[htbp]
\mbox{
\subfigure[]{\includegraphics[width=2in]{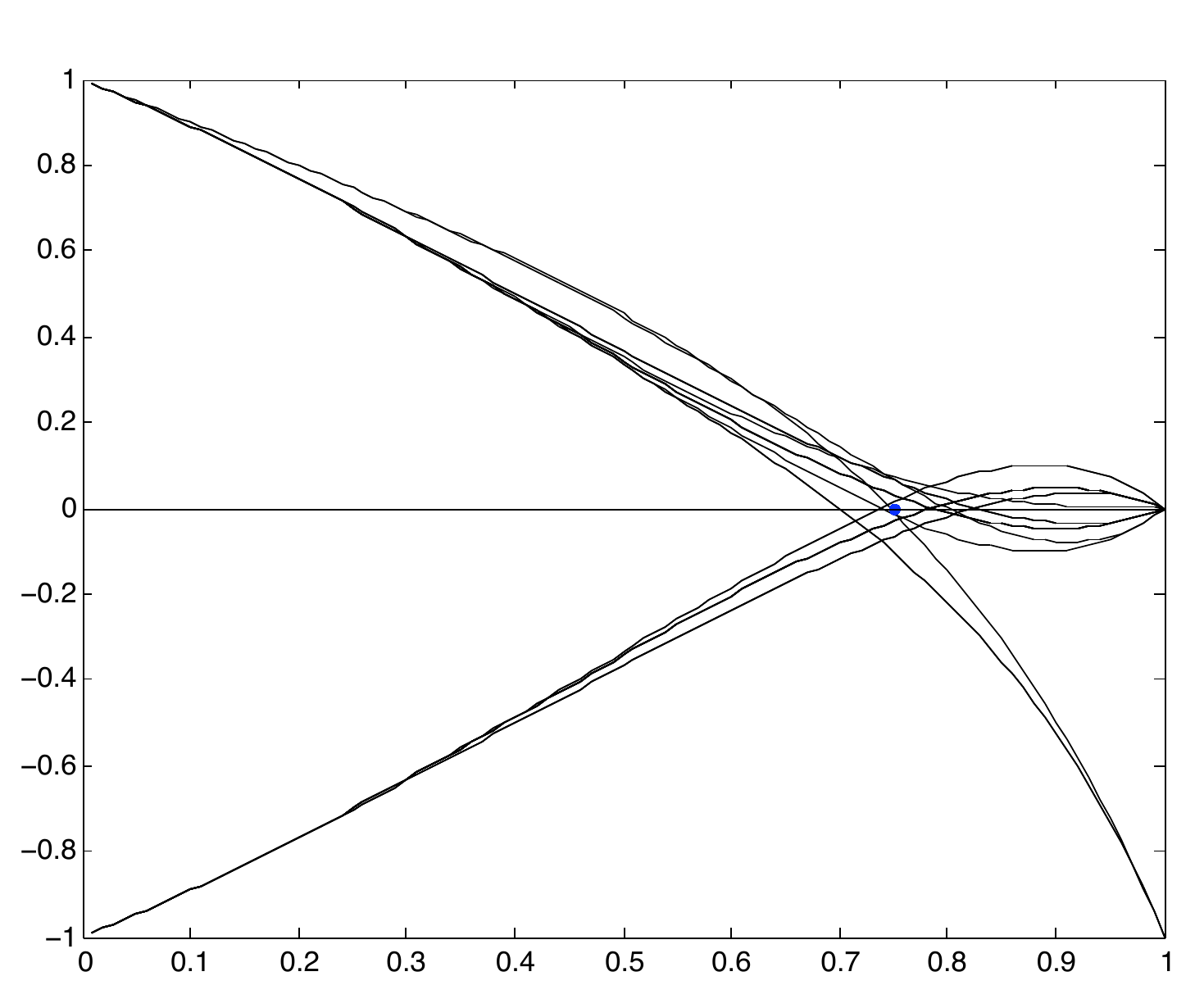}}
\subfigure[]{\includegraphics[width=2in]{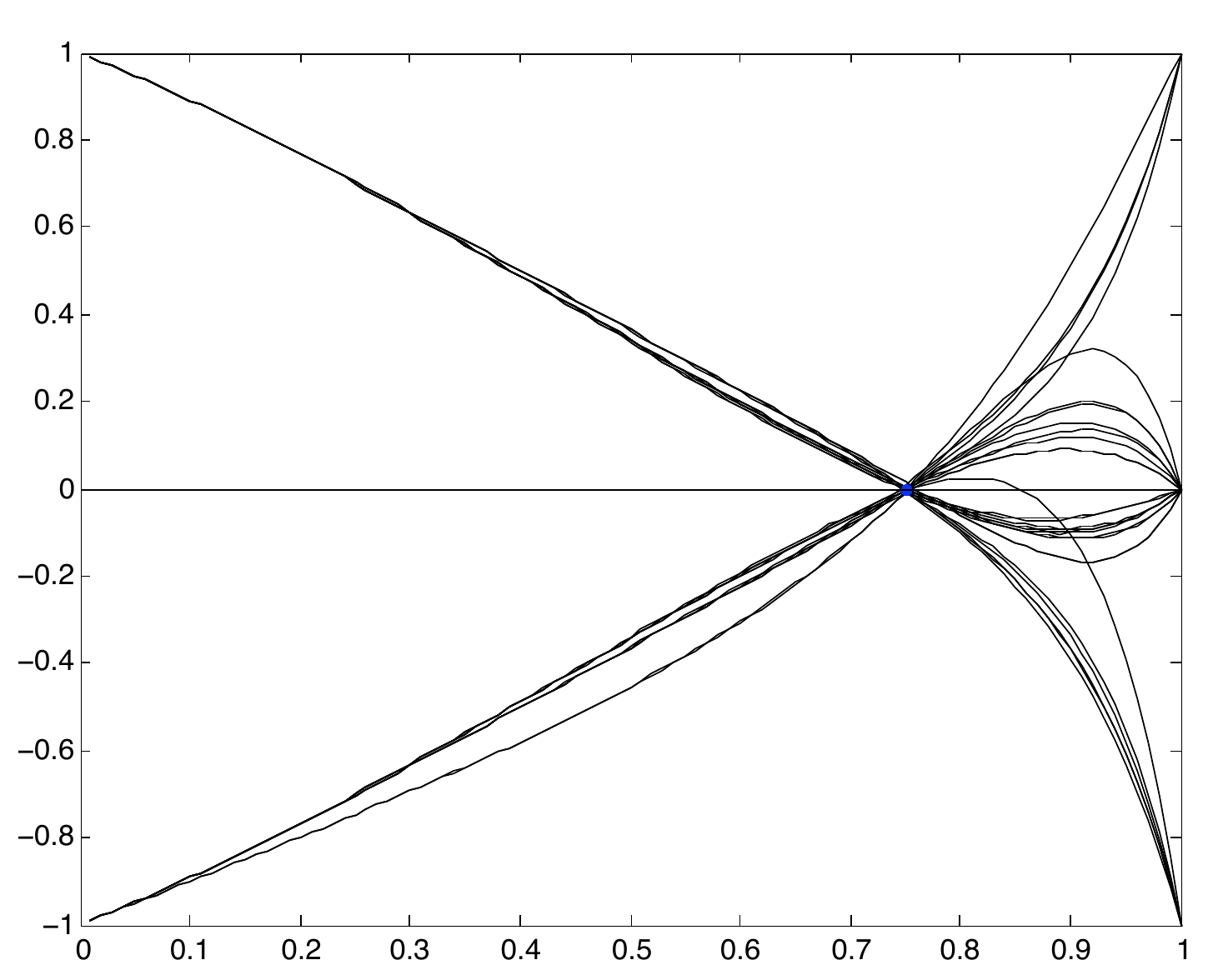}} 
\subfigure[]{\includegraphics[width=2in]{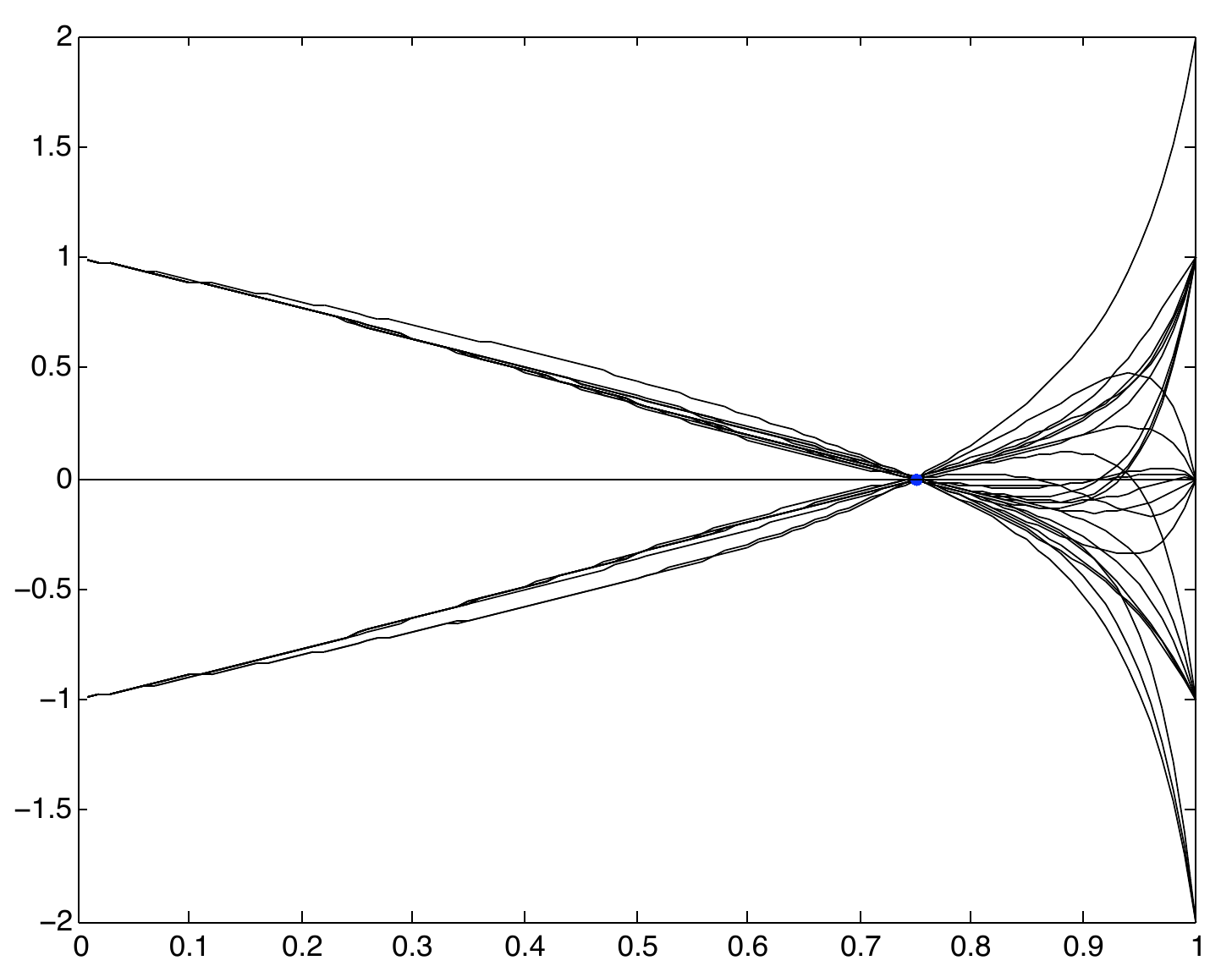}}
}
\caption{The graphs of (a) $\bar{P}_{8}$, (b) $\bar{P}_{16}$, and (c) $\bar{P}_{32}$ for several pairs $(x_j,-x_j)$ corresponding to $\gamma = .75$.  These graphs suggest that the first positive root of these polynomials becomes exponentially close to $\gamma$ as $N$ increases, although we do not have a proof of this result for values of $\gamma$ greater than approximately $.65$.   The quantizer used is $Q_{\alpha}^{\nu}$ with parameter values $\nu = .3$, and $\alpha = 2$.}
\label{fig:fig4}
\end{figure*}

\subsection{Appendix: proof of Theorem $\ref{range}$}

\begin{figure*}[htbp]
\begin{center}
\includegraphics[width=4in]{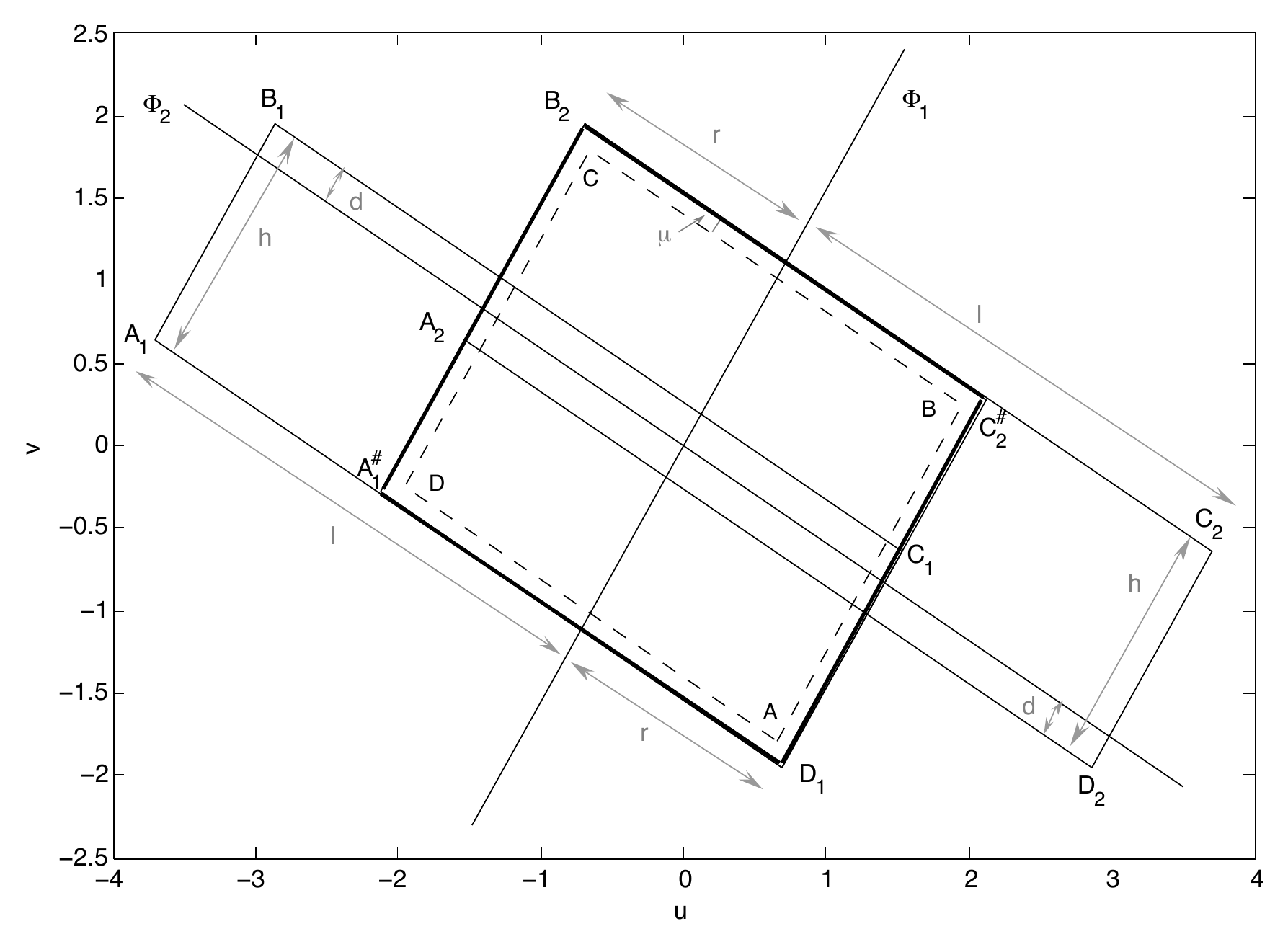}
\caption{The rectangle $R = A_1^{\#} B_2 C_2^{\#} D_1$ (bold outline) is a positively invariant set for the map $T_{\alpha}^{\nu}$.  In fact, $T_{\alpha}^{\nu} (R) $ is contained in the smaller rectangle $ABCD$ (dashed outline), illustrating that the revised GRE scheme $\eqref{(lambda recurs)}$ is robust with respect to small additive errors.  Here, $\Phi_1$ and $\Phi_2$ are the eigenvectors of the matrix in A in $\eqref{matrix}$ below.  In this figure, $\lambda_2 = .6$ and $\mu$ is taken to be $.0625$. The length parameters $l, r, h,$ and $d$ will be discussed later.}
\label{fig:fig5}
\end{center}
\end{figure*}

\noindent In this section we prove Theorem $\ref{range}$, which provides a range within which the "flakiness" parameter $\nu$ and the "amplifier" parameter $\alpha$ in the quantizer $Q_{\alpha}^{\nu}(u,v)$ can vary from iteration to iteration, without changing the fact that the sequences $(u_n)$ produced by the scheme $\eqref{(lambda recurs)}$ will be bounded.  The derived range for $\nu$ and $\alpha$ is independent of the specific values of the parameters $(\lambda_1, \lambda_2) \in [.9,1]^2$ in $\eqref{(lambda recurs)}$.  
\\
\\
The techniques of this section are borrowed in large part from those used in $\cite{GREsubmit}$ to prove a similar result for the ideal GRE scheme $\eqref{(simple recurs)}$; i.e., taking $\lambda_1 = \lambda_2 = 1$ in $\eqref{(T_Q)}$.  As is done in $\cite{GREsubmit}$, we first observe that the recursion formula $\eqref{(lambda recurs)}$ corresponding to the leaky GRE  is equivalent to the following piecewise affine discrete dynamical system on $\mathbb{R}^2$ :

\begin{equation}
		\left[\begin{array}{cl}
	u_{n+1} \\
	u_{n+2}
	   \end{array}\right]
           =T_{\alpha}^{\nu}\left[\begin{array}{cl}
	u_n \\
	u_{n+1}
	   \end{array}\right]
           \label{(DS)}
\end{equation}

where

\begin{equation}
	T_{\alpha}^{\nu}:	\left[\begin{array}{cl}
	u \\
	v
	   \end{array}\right]
           \rightarrow       \left[\begin{array}{cc}
	0 & 1 \\
	\lambda_1 \lambda_2 & \lambda_1
	   \end{array}\right] 
           \left[\begin{array}{cl}
	u \\
	v
	   \end{array}\right] 
           - Q_{\tilde{\alpha}}^{\tilde{\nu}}(u,v) \left[\begin{array}{cl}
	0 \\
	1
	   \end{array}\right]
             \label{(T_Q)}
 \end{equation}
where $\tilde{\alpha} = \frac{\alpha}{\lambda_2}$ and $\tilde{\nu} = \frac{\nu}{\lambda_1 \lambda_2}$.  We will construct a class of subsets $R = R(\mu) = R(\lambda_1,\lambda_2,\mu)$ of $\mathbb{R}^2$ for which $T_{\alpha}^{\nu}(R) + B_{\mu}(0) \subset R$, where $B_{\mu}(0)$ is the disk of radius $\mu$ centered at the origin; i.e. $B_{\mu}(0) = \{ (u,v): u^2 + v^2 \leq \mu^2 \}$.  Note that these sets are not only { \it positively invariant sets } of the map $T_{\alpha}^{\nu}$, but have the additional property that if $(u_0, v_0) \in R(\mu)$, the image $(u_{n+1}, v_{n+1}) = T_{\alpha}^{\nu} (u_n, v_n)$ may be perturbed at any time within a radius of $\mu$ (for example, by additive noise), and the the resulting sequence $(u_n)_{n=0}$ will still remain bounded within $R(\mu)$ for all time $n$.  
\\
\\
We refer the reader to Figure ~\ref{fig:fig5} as we detail the construction of the sets $R(\mu)$.   Rectangles $A_1 B_1 C_1 D_1$ and $A_2 B_2 C_2 D_2$ in Figure ~\ref{fig:fig5} are designed so that their respective images under the affine maps $T_1$ and $T_2$ (defined below) are both equal to the dashed rectangle $ABCD$.  
\begin{eqnarray}
	T_1: \left[\begin{array}{cl}
	u \\
	v
	   \end{array}\right]   \rightarrow   \left[\begin{array}{cc}
	0 & 1 \\
	\lambda_1 \lambda_2 & \lambda_1
	   \end{array}\right] 
           \left[\begin{array}{cl}
	u \\
	v
	   \end{array}\right] 
           - \left[\begin{array}{cl}
	0 \\
	1
	   \end{array}\right],  \nonumber \\
          T_2 : \left[\begin{array}{cl}
	u \\
	v
	   \end{array}\right]   \rightarrow   \left[\begin{array}{cc}
	0 & 1 \\
	\lambda_1 \lambda_2 & \lambda_1
	   \end{array}\right] 
           \left[\begin{array}{cl}
	u \\
	v
	   \end{array}\right] 
           + \left[\begin{array}{cl}
	0 \\
	1
	   \end{array}\right].
	   \label{T1}
	   \end{eqnarray}     
That is, rectangles $A_1 B_1 C_1 D_1$ and $A_2 B_2 C_2 D_2$ will satsify $T_1(A_1 B_1 C_1 D_1) = ABCD = T_2(A_2 B_2 C_2 D_2)$. More specifically, $T_1(A_1) = T_2(A_2) = A$, $T_1(B_1) = T_2(B_2) = B$, and so on.  Since the rectangle $R = A_1^{\#} B_2 C_2^{\#} D_1$ is contained within the union of $A_1 B_1 C_1 D_1$ and $A_2 B_2 C_2 D_2$, $R$ is a positively invariant set for any map $T(u,v)$ satisfying
\begin{equation}
T(u,v) =  \left\{\begin{array}{cl}
	T_1(u,v),  & (u,v)  \in R \setminus A_2 B_2 C_2 D_2   \\
	T_2(u,v),   &(u,v) \in R  \setminus A_1 B_1 C_1 D_1 \\
	T_1(u,v) \textrm{  } or \textrm{  } T_2(u,v), &(u,v) \in H \\
	   \end{array}\right.
	   \label{admin}
\end{equation}  
where $H =  A_1 B_1 C_1 D_1 \cap A_2 B_2 C_2 D_2$.  In particular, if the parameters $\alpha$ and $\nu$ in the map $T_{\alpha}^{\nu}$ are chosen such that the intersection of $R$ and the strip $F = \{(u,v): -\tilde{\nu} < u +\tilde{\alpha}v < \tilde{\nu} \}$ is a subset of $H$ (recall $\tilde{\alpha} = \alpha/\lambda_2$ and $\tilde{\nu} = \nu/(\lambda_1 \lambda_2)$, then $T_{\alpha}^{\nu}$ is of the form $\eqref{admin}$.  Indeed, $F$ is the region of the plane in which the quantizer $Q_{\tilde{\alpha}}^{\tilde{\nu}}(u,v)$ operates in flaky mode. This geometric setup is clarified with a figure, provided in Figure ~\ref{fig:fig6}.

\begin{figure}[htbp]
\begin{center}
\includegraphics[width=3in]{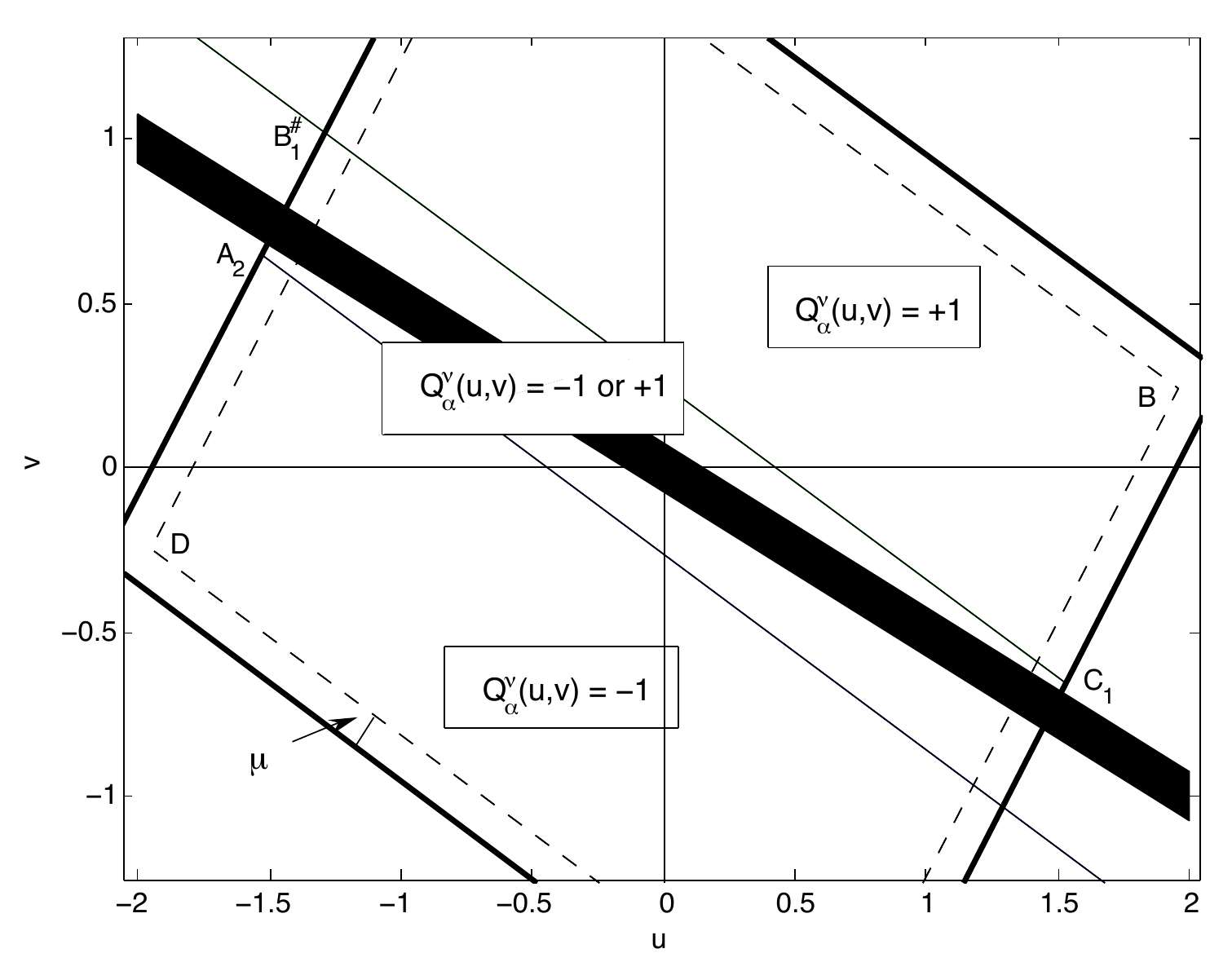}
\end{center}
\label{fig:fig6}
\caption{ {\it The flaky quantizer $Q_{\tilde{\alpha}}^{\tilde{\nu}}(u,v)$ outputs either $-1$ or $+1$ when the input $(u,v)$ belongs to the strip $F = \{(u,v): \tilde{\nu} < u + \tilde{\alpha} v < \tilde{\nu} \}$ (shaded above).  Here $\mu = .0625$, $\tilde{\alpha} = 2$, and $\tilde{\nu} = .15$.  For these parameter values, the intersection of $F$ and $R$ is a subset of  $A_1 B_1 C_1 D_1 \cap A_2 B_2 C_2 D_2$, meaning that the scheme $\eqref{(lambda recurs)}$ produces bounded sequences $(u_n)$, when implemented with this quantizer, and with input $(u_0,u_1) \in R$.}}
\end{figure}
\vspace{5mm}
\noindent It remains to verify the existence of at least one solution to the setup in Figure ~\ref{fig:fig5} for each $(\lambda_1,\lambda_2) \in [.9,1]^2$.  This can be done, following the lead of $\cite{GREsubmit}$, in terms of the parameters defined in Figure ~\ref{fig:fig5}.
\\
\\
Note that the matrix 
\begin{equation}
A =  \left[\begin{array}{cc}
	0 & 1 \\
	\lambda_1 \lambda_2 & \lambda_1
	   \end{array}\right]
	   \label{matrix}
\end{equation}
in $\eqref{T1}$ has as eigenvalues $\epsilon_1 = \frac{\lambda_1 + \sqrt{\lambda_1^2 + 4\lambda_1\lambda_2}}{2}$ and $- \epsilon_2  = -\big( \frac{- \lambda_1 + \sqrt{\lambda_1^2 - 4\lambda_1\lambda_2}}{2} \big)$.   In particular, when $\lambda_1 = \lambda_2 = 1$, we have that $\epsilon_1 = \phi \approx 1.618$ and $\epsilon_2 = \phi^{-1} \approx .618$.  The eigenvalues $\epsilon_1$ and $-\epsilon_2$ have respective normalized eigenvectors 
\begin{center}
$\Phi_1 = s_1^{-1}\left[\begin{array}{cl}
	1 \\
	\epsilon_1
	   \end{array}\right]$ 
	   and $\Phi_2 = - s_2^{-1} \left[\begin{array}{cl}
	1 \\
	\epsilon_2
	   \end{array}\right]$,
\end{center}
where $s_1 = \sqrt{1 + \epsilon_1^2}$ and $s_2 = \sqrt{1 + \epsilon_2^2}$.   It follows that the affine map $T_1$ acts as an expansion by a factor of $\epsilon_1$ along $\Phi_1$ and a reflection followed by contraction by a factor of $\epsilon_2$ along $\Phi_2$, followed by a vertical translation of +1.  $T_2$ is the same as $T_1$ except with a vertical translation of -1 instead of +1.  After some straightforward algebraic calculations, the mapping relations described above imply that the parameters in Figure ~\ref{fig:fig5} are given by
\begin{eqnarray}
h  &=& \frac{2\mu}{1-\epsilon_1} + \frac{2s_1}{\epsilon_1(\epsilon_1-1)(\epsilon_1 + \epsilon_2)} \nonumber \\
d  &=& \frac{\mu}{1-\epsilon_1} + \frac{s_1(2 - \epsilon_1)}{\epsilon_1(\epsilon_1-1)(\epsilon_1 + \epsilon_2)} \nonumber \\
l  &=& \frac{\mu}{1 - \epsilon_2} + \frac{s_2(2 - \epsilon_2)}{\epsilon_2(1 - \epsilon_2)(\epsilon_1 + \epsilon_2)} \nonumber \\
r &=& \frac{\mu}{1 - \epsilon_2} + \frac{s_2}{(1-\epsilon_2)(\epsilon_1 + \epsilon_2)}
\label{points}
\end{eqnarray}
The existence of the overlapping region $H = A_1 B_1 C_1 D_1 \cap A_2 B_2 C_2 D_2$ is equivalent to the condition $d > 0$ which in turn is equivalent to the condition $\mu < \frac{s_1(2-\epsilon_1)}{\epsilon_1(\epsilon_1 + \epsilon_2)}$.  This expression is minimized over the range $(\lambda_1,\lambda_2) \in V = [.9,1]^2 $ when $\lambda_1 = \lambda_2 = 1$, in which case this constraint simplifies to $\mu < \frac{2 - \phi}{\sqrt{\phi^2+1}} \approx .2008$. 
\\
\\
We are interested primarily in quantizing numbers in $[-1,1]$, so that this is the range of interest for the input $u_1 = x$; for $u_0$ we simply take $0$.   The set $\{0\} \times [-1,1] \subset  \bigcap_{(\lambda_1,\lambda_2) \in V} R(\lambda_1,\lambda_2,\mu)$ is equivalent to the condition that the line  $v = mu + b$ which passes through the points $B$ and $C$ in Figure ~\ref{fig:fig5} have y-intercept $b \geq 1$ for all $(\lambda_1,\lambda_2) \in V$.   This line has slope $m = -\epsilon_2$, and passes through the point $C$, so that its y-intercept is given by $b = s_1^{-1}(\epsilon_1 + \epsilon_2)(h - d)$.  Rearranging terms, this implies that $b \geq 1$ if and only if 
\begin{equation}
\mu \leq \frac{(\sqrt{1 + \epsilon_1^2})(2 - \epsilon_1)}{\epsilon_1 + \epsilon_2}.
\label{line}
\end{equation}
The right hand side of the above inequality is bounded below over the range $(\lambda_1,\lambda_2) \in V$ by $\frac{\sqrt{1 + (.9\phi)^2}(2 - \phi)}{\sqrt{5}} \approx .301$, so that indeed $\{0\} \times [-1,1] \subset R(\lambda_1,\lambda_2,\mu)$ is satisfied independent of $(\lambda_1,\lambda_2) \in V$, for all admissable $\mu \in [0,.2008] $.   
\\
\\
In practice, the "flaky" parameter $\nu$ of the quantizer $Q_{\alpha}^{\nu}$ is not known exactly; we will however be given a tolerance $\delta$ for which it is known that $|\nu| \leq \delta$.  It follows that for each $\delta$, we would like a range $(\alpha_{min}, \alpha_{max})$ for the "amplifier" $\alpha$ such that the GRE scheme $\eqref{(lambda recurs)}$, implemented with quantizer $Q_{\alpha}^{\nu}$, produces bounded sequences $(u_n)$ for all values of $(\lambda_1,\lambda_2) \in V$.  Now, it will be easier to derive all of the following results for $\tilde{\alpha}$ and $\tilde{\nu}$, and return to the original parameters $(\alpha, \nu)$ afterwards.  For a particular choice of $(\lambda_1, \lambda_2)$, it is not hard to derive an admissable range for $\tilde{\alpha}$ in terms of the eigenvalues $\epsilon_1$ and $\epsilon_2$ for fixed $\delta$.  We will take $\mu = 0$ in the following analysis for the sake of simplicity.  First of all, the tolerance $\delta$ must be admissable, i.e. the coordinate $(\delta, 0)$ should lie within the region $H$.  Indeed, for $\tilde{\nu} \in [0,\delta]$, $\tilde{\alpha}$ can vary within a small neighborhood of $\frac{1}{\epsilon_2}$, as long as the line with slope $-\frac{1}{\tilde{\alpha}}$ which passes through $(0,\delta)$ remains bounded above in the region $H$ by the line passing through $B_1$ and $C_1$.  In other words, the admissable range for $\tilde{\alpha}$ is obtained from the constraint $m_1 \leq -\frac{1}{\tilde{\alpha}} \leq m_2$, where $m_1$ is the slope  of the line through the points $\{ ( \delta, 0),  C_1 \}$ and $m_2$ is the slope of the line through the points $\{ ( \delta, 0 ), B_1^{\#} \}$ in Figure ~\ref{fig:fig6}.   Rewriting $C_1$ and $B_1^{\#}$ in terms of the eigenvalues $\epsilon_1$ and $\epsilon_2$ via the relations $\eqref{points}$, we have that

\begin{equation}
L(\epsilon_1,\epsilon_2) \leq \tilde{\alpha} \leq U(\epsilon_1,\epsilon_1 + \epsilon_2)
\label{bounds}
\end{equation}
where
\begin{eqnarray}
&&L(x,y) = N(x,y)/D(x,y) \nonumber \\
&=&  \frac{x(x - 1) - (2-x)(1 - y) + \delta x(x - 1)(x + y)(1 - y)}{x((2 - x)(1-y) + y(x-1))} \nonumber
\end{eqnarray}
and 
\begin{equation}
U(x,y) = \frac{2 + xy - 2y - \delta xy(x - 1)(1 - y + x)}{x(y - 2)}.
\end{equation}
 It is not hard to show that for any fixed $y \in [.9 \phi + .9 \phi^{-1}, \phi + \phi^{-1}] \approx [2.0124, 2.236]$, the function $U(x,y)$ increases as a function of $x$ over the range $x \in [.9\phi, \phi] \approx [1.456, 1.618]$, as long as $\tilde{\nu} \leq  .75$.  It follows that the minimum of $U(x,y)$ over the rectangle $S = [1.456, 1.618] \times [2.0124, 2.236]$ occurs along the edge corresponding to $x = 1.456$.  But $U(1.456,y)$ decreases as a function of $y$ over the interval $[2.0124, 2.236]$, so that the minimum of $U(x,y)$ over the entire rectangle $S$ occurs at $(x,y) = (1.456, 2.236)$, giving the following uniform lower bound on $\tilde{\alpha}_{max}$ for $(\lambda_1,\lambda_2) \in V$:
\begin{eqnarray}
\tilde{\alpha}_{max} &\geq& U(1.456, 2.236) \nonumber \\
                         &\approx& 2.281 - .952 \delta.
\label{max}
\end{eqnarray} 
We proceed in the same fashion to derive a uniform upper bound on $\tilde{\alpha}_{min}$ over $(\lambda_1,\lambda_2) \in V$, except that we analyze the numerator and the denominator in the expression for $L(x,y)$ separately.  The numerator $N(x,y)$ increases as a function of $x$ over the range $x \in [1.456, 1.618]$ for fixed $y \in [.9\phi^{-1}, \phi^{-1}] \approx [.5562,.618]$, so that $N(x,y)$ achieves its maximum over the rectangle $T = [1.456, 1.618]\times [.5562,.618]$ along the edge corresponding to $x = 1.618$.   $N(1.618,y)$ increases as a function of $y$ over the range $y \in [.5562,.618]$, so that $N(x,y)$ achieves its maximum over $T$ at $(x,y) = (1.618, .618)$.   A similar analysis shows that the denominator $D(x,y)$ is minimized over $T$ at $(x,y) = (1.456, .618)$.  It follows that for $(\lambda_1,\lambda_2) \in [.9,1]^2$, $\tilde{\alpha}_{min}$ is bounded above by
\begin{eqnarray}
\tilde{\alpha}_{min} &\leq& N(1.618, .618)/D(1.456, .618) \nonumber \\
                        &\approx& 1.198(1+\delta)
\label{m}
\end{eqnarray}
Finally, we note that $\delta$ is admissable for all $(\lambda_1, \lambda_2)$ if and only if $L(\epsilon_1, \epsilon_2) \leq U(\epsilon_1, \epsilon_1 + \epsilon_2)$ for this value of $\delta$ and for all coresponding $\epsilon_1$ and $\epsilon_2$.  Thus an  admissable range for $\delta$ is obtained by equating the lower bound of $2.281 - .952 \delta$ in $\eqref{max}$ and the upper bound of $1.198(1+\delta)$ in $\eqref{m}$; namely, $\delta \in [0, .5037]$.
\\
\\
In terms of the original parameters $\alpha = \lambda_2 \tilde{\alpha}$ and $\nu = \lambda_1 \lambda_2 \tilde{\nu})$,
\begin{eqnarray}
1.198(1 + \delta) &\leq& \alpha \leq 2.053 - .8568 \delta  \textrm{, and so} \nonumber \\
0 &\leq& \delta \leq .4161, \nonumber \\
\nu &=& \lambda_1 \lambda_2 \tilde{\nu} \leq \lambda_1\lambda_2 \delta \leq .337.
\end{eqnarray}
represents a uniformly admissible range for $(\alpha, \nu)$ over the interval $(\lambda_1, \lambda_2) \in [.9,1]^2$.

\chapter{Quiet quantization in analog to digital conversion}
%
\vspace{10mm}

In the last chapter, we studied robustness properties of the $\beta$-encoder, a quantization method in analog to digital conversion that is both robust with respect to quantizer imperfections like $\Sigma \Delta$ schemes, and also achieves optimal reconstruction guarantees like pulse code modulation (PCM).   While we were able to show that $\beta$-encoders are also robust with respect to other component imperfections, the $\beta$-encoder is \emph{not} robust with respect to additive noise; that is, if the discrete input $x_n$ to the $\beta$-encoder \eqref{iter} represent noisy signal samples $x_n = f_n+ \epsilon_n$, then unavoidably
\begin{equation}
\sum_{j=0}^{N-1} b_j \beta^{-j} = f_n + \epsilon_n + O(\beta^{-N}),
\label{noise_beta}
\end{equation}
and the situation is similar for the golden ratio encoder.  That is, no matter how many bits are spent per Nyquist interval, the reconstruction error of the beta-encoder and golden ratio encoder is limited by the level of noise on the samples $f_n$.   
\\
\\
\noindent On the other hand, $\Sigma \Delta$ schemes \emph{are} robust with respect to additive noise.  That is, the $O(\lambda^{-K})$ reconstruction accuracy of the $K$th order $\Sigma \Delta$ scheme is preserved under additive noise, as facilitated by the incorporation of all previous samples $f_k^{\lambda}$,  $k \leq n$, in determining the quantization output $b_n$. However, $\Sigma \Delta$ schemes suffer from an entirely different kind of instability in the context of audio signal processing, corresponding to a mechanical rather than mathematical reconstruction inaccuracy: at the onset of periodicities in the bit output $b_n$, the filters used in the reconstruction of the analog signal by $\Sigma \Delta$ quantizers are such that periodic oscillatory patterns in the quantization output $b_n$ generate low-level \emph{idle tones} that are not present in the signal.   These tones are particularly intolerable to the listener when no signal is present, such as between successive audio tracks.  As idle tone components are most prominent in low-order, 1-bit  $\Sigma \Delta$ schemes \cite{sigmadelta}, it is desirable to adapt these schemes in such a way as to eliminate the possibility of periodic behavior of the output $b_n$ for vanishing input $x_n$; we shall refer to such as \emph{quiet $\Sigma \Delta$ quantizers}, in line with the following definition:

\pagebreak

  \begin{definition}[{\bf Quiet quantization}]
A quantization scheme is said to be \emph{quiet} if the onset of identically zero input $x_n \equiv 0$ after some finite time $k$ induces constant bit output $b_n \equiv b_m$ after some later time $m \geq k$.
\label{def1}
\end{definition}

\noindent In the first order $\Sigma \Delta$ model $\eqref{1storderagain}$, which, starting from initial state $u_0$, follows the recursion
\begin{eqnarray}
b_n &=& Q(u_{n-1} + f^{\lambda}_n) \nonumber \\
u_n &=& u_{n-1} + f^{\lambda}_n - b_n,
\label{1storderagain}
\end{eqnarray} 
one can eliminate periodicities in the output $b_n$ at instantiation of small input $f_n$ by simply replacing the standard 2-level quantizer $Q(u) = \sgn(u)$ by a \emph{tri-level} quantizer $Q^{\tau}_{tri}$:
\begin{equation}
Q^{\tau}_{tri}(u) = \left\{\begin{array}{cl}
1 & u >  \tau, \\
0 & -\tau \leq u \leq \tau, \\
-1 & u < -\tau. \end{array}\right.
\label{tri-level}
\end{equation}
It is straightforward that the first-order $\Sigma \Delta$ scheme retains its first-order approximation accuracy when implemented with the tri-level quantizer $Q^{.5}_{tri}$, and is also quiet, in the sense of Definition \ref{def1}; however, first order schemes are rarely used in practice, and we would thus like a similar result for higher order schemes.  Unfortunately, the second-order $\Sigma \Delta$ scheme,
 \begin{eqnarray}
 b_n &=& Q(F(u_{n-1},v_{n-1})) \nonumber \\
u_n &=& u_{n-1} + f^{\lambda}_n - b_n \nonumber \\
v_n &=& v_{n-1} + u_n,
\label{2ndorder2}
 \end{eqnarray}
is not quiet, as shown in \cite{OYthesis}, even when implemented with tri-level quantizer $Q^{\tau}_{tri}$ and with linear rule $F_{\gamma}(u,v) = u + \gamma v$.  In fact, for most initial states $(u_0, v_0)$ at the onset of zero input $f^{\lambda}_1 = 0$, the sequence $b_n$ will become oscillatory.  An exception to this rule occurs for the initial condition $(u_0, v_0) = (0,0)$, in which case $(u_n, v_n) = (0,0)$ and $b_n = 0$ persist as long as $f_n^{\lambda} = 0$ remains zero.  That is, the origin is a fixed point of the \emph{zero-input} discrete map $M: (u_n, v_n) \rightarrow (u_{n+1}, v_{n+1})$ produced by \eqref{2ndorder2} with tri-level quantizer and zero input $f_n^{\lambda} = 0$,  but is not an \emph{attractive} fixed point.  Recall that an attractive fixed point for a discrete map $M$ is a fixed point $x_0 = M(x_0)$ having the property that for any value of $x$ in the domain that is close enough to $x_0$, the \emph{orbit} of $x$, 
\begin{center}
$(x, M(x), M \big(M(x) \big) = M^2 x, ... )$
\end{center}
converges to $x_0$. Such a fixed point is said to be a \emph{global} attracting fixed point if $M^n x \rightarrow x_0$ for any value of $x$ in the domain. In order to modify the second-order scheme \eqref{2ndorder2} so that the origin is an attractive fixed point for the corresponding zero-input map,  the state sequence $(u_n, v_n)$ must be made to \emph{decay}.   One way to induce decay is by applying a contraction $(u,v) \rightarrow (\rho u, \rho v)$ before each iteration of \eqref{2ndorder2}, resulting in the modification
 \begin{eqnarray}
 b_n &=& Q(F(\rho u_{n-1}, \rho v_{n-1})) \nonumber \\
u_n &=& \rho u_{n-1} + x_n - b_n, \nonumber \\
v_n &=& \rho v_{n-1} + u_n.
\label{2ndorder-finite}
\end{eqnarray}
This so-called \emph{finite-memory} scheme was first studied in \cite{OYthesis}.  As the accuracy of $\Sigma \Delta$ methods rely on `keeping track' of the previous input $f^{\lambda}_k, k \leq n$ through the sequence $(u_n, v_n)$, it is not clear that the finite-memory scheme, which loses a fraction of its `memory' at each step, should still be second-order.  However, as long as $1 - \rho$ is sufficiently small, second-order accuracy \emph{is} maintained, as proven in \cite{OYthesis}.  For completeness, we include this result below, starting with the analogous result for the first-order finite-memory scheme,
\begin{eqnarray}
 b_n &=& Q(\rho u_{n-1} + f_n^{\lambda}), \nonumber \\
u_n &=& \rho u_{n-1} + f^{\lambda}_n - b_n. 
\label{1storder-finite}
 \end{eqnarray}
 We note that the following result does not depend on the choice of quantizer, and requires only that the sequence $(u_n)$ be uniformly bounded.

\begin{lemma}[Yilmaz]
Suppose $\lambda > 0$, let $f \in L^2(\mathbb{R})$ be bandlimited with supp $f \subset [-\pi, \pi]$ and $\| f \|_{\infty} \leq 1$, and let $g$ be a function satisfying $\hat{g} = 1$ for $|\xi| \leq \pi$, $\hat{g} = 0$ for $|\xi| \geq \lambda \pi$, and $g \in {\cal C}^{\infty}$.   Suppose that the leakage factor $\rho \geq \rho^{\lambda} = 1 - \frac{1}{\lambda}$, and assume that the sequence $(u_n)$ generated by the finite-memory scheme \eqref{1storder-finite} is bounded.  If $(b_n^{\lambda})$ is the output of the first-order finite $\Sigma \Delta$-quantizer \eqref{1storder-finite}, then
\begin{equation}
\| f(t) - \tilde{f}(t) \|_{\infty} \leq \frac{ \| v \|_{\ell_{\infty}}}{\lambda} ( \| g' \|_{L^1} + C_g),
\end{equation}
where $C_g \leq \| g \|_{L^1} + \frac{1}{\lambda} \| g' \|_{L^1}$, and $\tilde{f}(t) = \frac{1}{\lambda} \sum b_n^{\lambda} g(t - \frac{n}{\lambda}).$
\label{leaky1}
\end{lemma}

\begin{proof}
We have $u_n - \rho u_{n-1} = f_n^{\lambda} - b_n^{\lambda}$.  Therefore, 
\begin{eqnarray}
f(t) - \tilde{f}(t) &=& \frac{1}{\lambda} \Big( \sum(u_n - u_{n-1}) g(t - \frac{n}{\lambda}) + (1 - \rho)\sum u_{n-1} g(t - \frac{n}{\lambda})  \Big) \nonumber \\
&=& \frac{1}{\lambda}  \Big( \sum u_n \big(g(t - \frac{n}{\lambda}) - g(t - \frac{n+1}{\lambda}) \big) + (1 - \rho) \sum u_n g(t - \frac{n}{\lambda}) \Big). \nonumber
\end{eqnarray}
Using the bound $1 - \rho \leq \frac{1}{\lambda}$, 
\begin{eqnarray}
| f(t) - \tilde{f}(t) | &\leq& \frac{\| u \|_{\ell^{\infty}}}{\lambda} \Big( \sum |g(t - \frac{n}{\lambda}) - g(t - \frac{n+1}{\lambda})| + \frac{1}{\lambda}\sum |g(t - \frac{n}{\lambda})| \Big) \nonumber \\
&\leq&  \frac{\| u \|_{\ell^{\infty}}}{\lambda} ( \| g' \|_{L^1} + C_g ). \nonumber 
\end{eqnarray}
\end{proof}
\noindent The analogous result for the second-order scheme follows the same argument, and we state the result without proof. 
\begin{lemma}[Yilmaz]
Let $f, g$, and $\rho$ be as in Lemma \ref{leaky1}.  Assume that the sequence $(u_n, v_n)$ generated by \eqref{2ndorder-finite} is bounded.  If $(b_n^\lambda)$ is the output of the second-order leaky $\Sigma \Delta$-quantizer given in \eqref{2ndorder-finite}, then
\begin{equation}
| f(t) - \tilde{f}(t) | \leq \frac{ \| v \|_{\ell_{\infty}}}{\lambda^2} (\| g'' \|_{L^1} + 2 \| g' \|_{L^1} + 2 C_g),
\end{equation}
where $C_g$ and $\tilde{f}$ are as before.
\label{leaky2}
\end{lemma}

\noindent It is important to note that Lemma \ref{leaky2} is useful only in the finite regime, even though it is asymptotic statement.  Indeed, the oversampling ratio $\lambda$ is not taken to be arbitrarily large in practice as limited by the resolution of analog-to-digital technologies; in most A/D converters, $\lambda \leq 64$ is set.  For this range of $\lambda$, the assumption in Lemma \ref{leaky2} that the damping factor $\rho^{\lambda}$ can be prescribed up to resolution 
$\rho \in [.97,1]$ is certainly reasonable.
\\
\\
\noindent Lemma \ref{leaky2} rests on the assumption that the initial input $(u_0, v_0)$ can be prescribed so that the sequence $(u_n, v_n)$ remains uniformly bounded.  We can ensure such stability by constructing positively invariant sets for the discrete map $(u_n, v_n) \rightarrow (u_{n+1}, v_{n+1})$ corresponding to the finite-memory scheme \eqref{2ndorder-finite}.  Recall that a set $S$ is positively invariant for the discrete map $M$ if $x \in S$ implies that  $M(x) \in S$.  The following theorem, borrowed again from \cite{OYthesis}, gives explicit constructions of such positively invariant sets for the finite-memory map \eqref{2ndorder-finite} with linear rule $F_{\gamma}(u,v) = u + \gamma v$:   

\begin{theorem}[Yilmaz]
Fix $\alpha < 1$, and constant $C$ satisfying
\begin{equation}
C \geq \frac{1}{2- 2\alpha^2} + \frac{12 + 9 (1 + \alpha) }{8(1-\alpha)}. \nonumber
\end{equation}
Consider the following functions, 
\begin{equation}
B_1(u) = \left\{\begin{array}{cl}
	-\frac{u^2}{2(1-\alpha)} + \frac{u}{2} + C,  & u \geq 0 \\
	-\frac{u^2}{2(1+ \alpha)} + \frac{u}{2} + C,   &u < 0 
	   \end{array}\right.
           \label{(b1)} 
\end{equation}
\begin{equation}
B_2(u) = \left\{\begin{array}{cl}
	\frac{u^2}{2(1+ \alpha)} + \frac{u}{2} - C,  & u \geq 0 \\
	\frac{u^2}{2(1-\alpha)} + \frac{u}{2} - C,   &u < 0 
	   \end{array}\right.
           \label{(b2)} 
\end{equation}
and the subset of points in $\mathbb{R}^2$ lying between the graphs of $B_1$ and $B_2$:
\begin{eqnarray}
S &=& \{ (u,v):  v \leq B_1(u), v \geq B_2(u)\}.
\end{eqnarray}
Finally, suppose that $\gamma$ in the linear rule $F_{\gamma}(u,v) = u + \gamma v$ is set to a value in the range,
\begin{equation}
\frac{2 \big( 2C(1-\alpha^2) \big)^{1/2} - (1+\alpha)\}}{ \big( 2C(1 - \alpha^2) \big)^{1/2} + 2\alpha C} \leq  \gamma \leq \frac{4(1+\alpha)((1+\alpha) + 2)}{8C(1 + \alpha) - (1 + \alpha) - 4}.
\nonumber
\end{equation}
If $(u_{n+1}, v_{n+1})$ is obtained from $(u_n, v_n) \in S$ via the second-order recursion \eqref{2ndorder-finite} with input $f_n$ satisfying $|f_n| \leq \alpha$, quantizer $Q^{\tau}_{tri}$, $\tau \leq 1/2$, and linear rule $F_{\gamma}(u_n, v_n)$, then $(u_{n+1}, v_{n+1}) \in S$.  
\label{stability}
\end{theorem}

\noindent A qualitative depiction of the invariant region $S$ is shown in Figure $\ref{fig:stable}$ below, as generated by admissible parameter values $\alpha = .9$, $C = 40$, and $\gamma = .1$.   

\begin{figure*}[htbp]
\begin{center}
\includegraphics[width=3in]{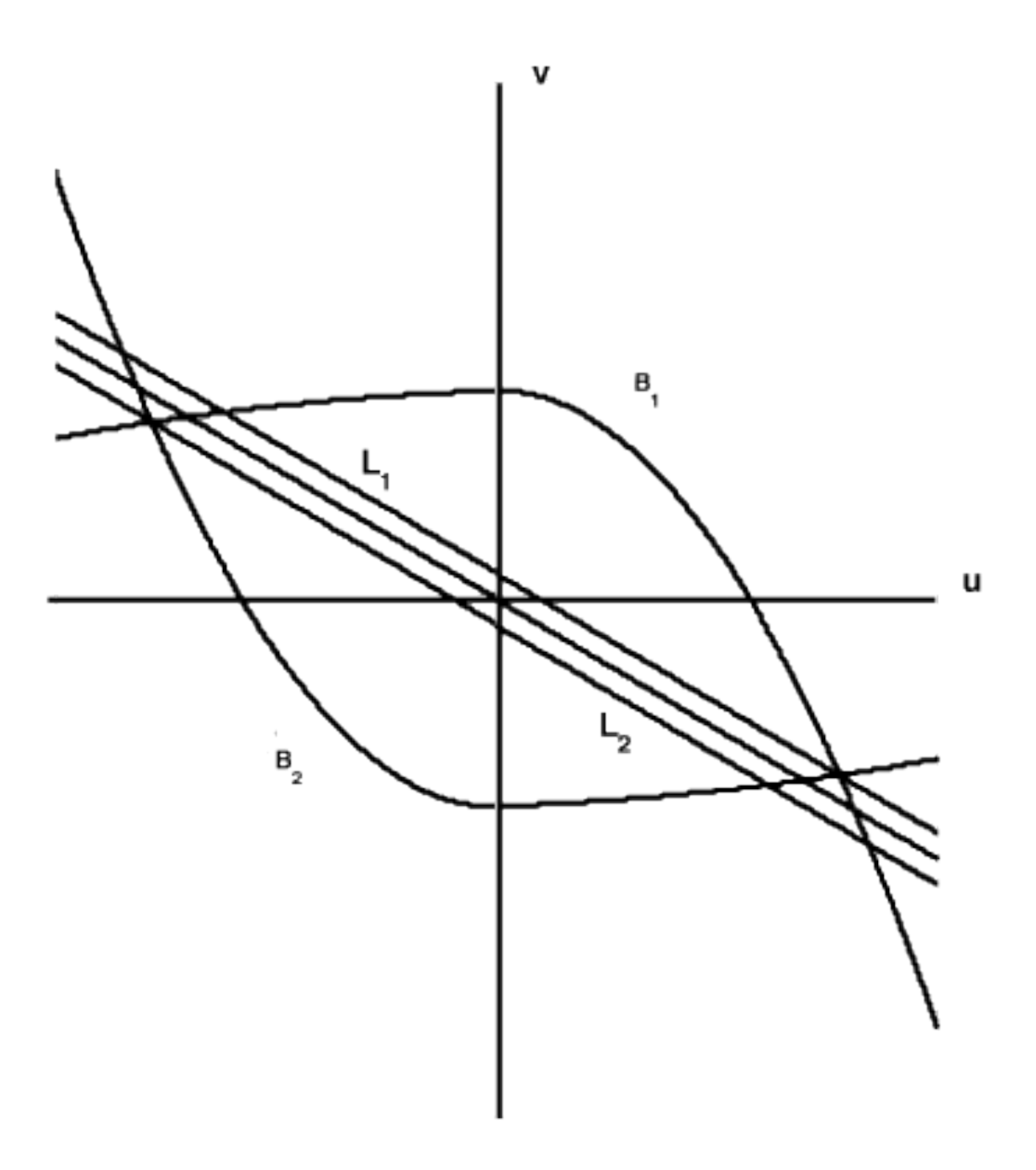}
\caption{The functions $B_1$ and $B_2$ from Theorem \ref{stability} are pictured along with the lines $L_1$ and $L_2$ consisting of the points $(u,v)$ such that $F_{\gamma}(u,v) = .5$ and $F_{\gamma}(u,v) = -.5$, respectively.  This figure was generated using admissible values $(\alpha, C, \gamma) = (.9, 40, .1)$}.
\label{fig:stable}
\end{center}
\end{figure*}

\subsection{The revised scheme: Quiet $\Sigma \Delta$ quantization}

\noindent Henceforth, we will assume that the oversampling factor $\lambda$ is fixed, and that the damping factor $\rho < 1$ in the finite-memory scheme \eqref{2ndorder-finite} satisfies the requirements of Lemma \ref{leaky2} for the finite-memory scheme to be second-order.  We further assume that the $\ell_{\infty}$ bound $|f_n| \leq \alpha < 1$ is fixed, and that $\gamma$ is in the admissible range, as specified by Theorem \ref{stability}, for the finite-memory $\Sigma \Delta$ scheme \eqref{2ndorder-finite} to be stable.  
\\
\\
\noindent  Unfortunately, the finite-memory scheme  \eqref{2ndorder-finite} implemented with tri-level quantizer $Q_{tri}^{.5}$ is not quiet for $\rho$ sufficiently close to $1$, as a range of initial conditions $(u_0, v_0)$ at the onset of zero input still lead to oscillatory behavior of the output.  In Figure \ref{fig:quietmap}, we indicate whether or not the iterates $(u_n, v_n)$ generated by the finite-memory scheme with zero input $x_n = 0$ remain oscillatory, for a set of perturbations $\rho \in [.96,1]$ and initial conditions $(u_0, 0) \in [-2,0] \times 0$.  

\begin{figure}[htbp]
\begin{center}
\includegraphics[width=4in]{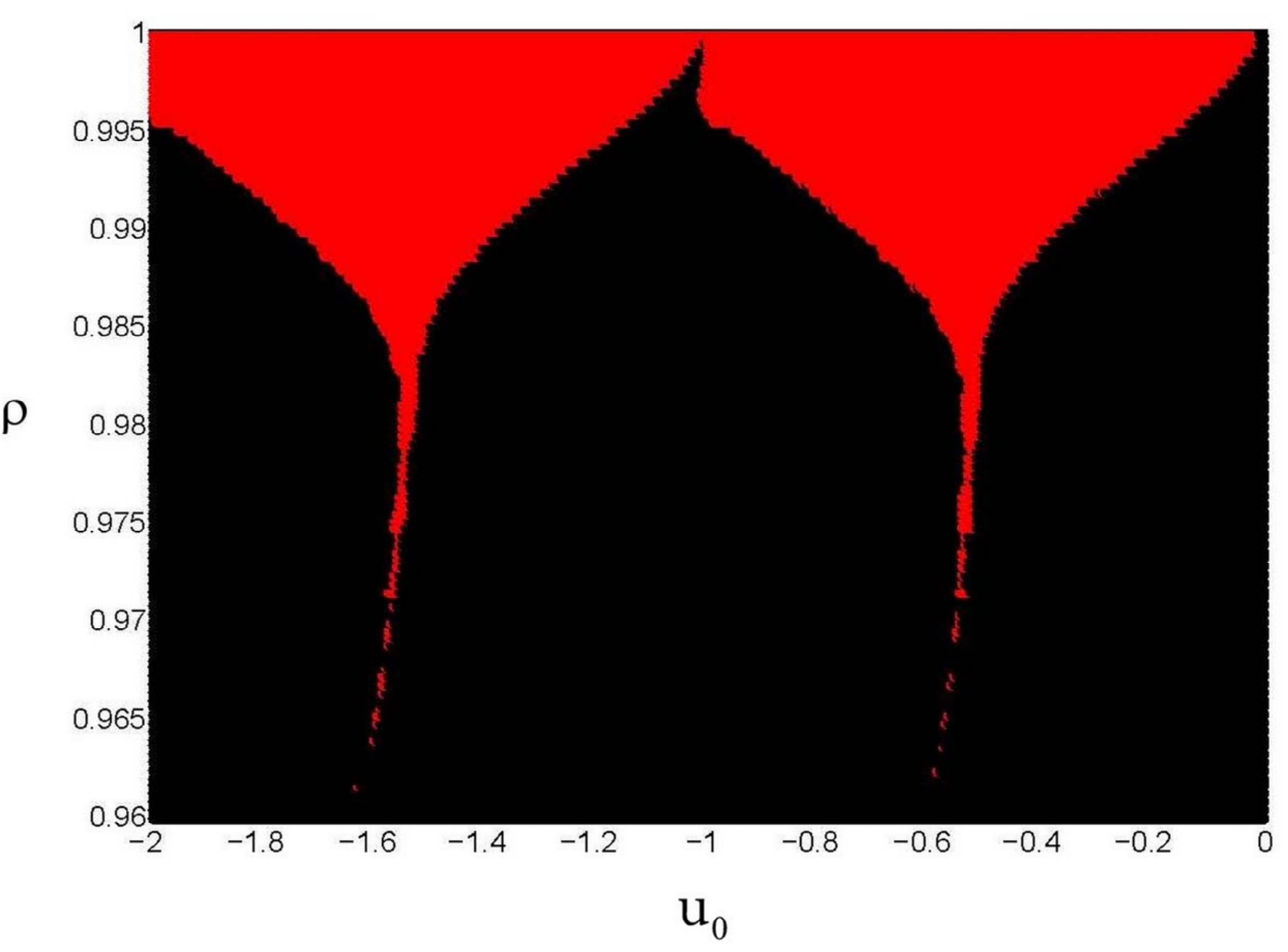}
\caption{We indicate whether or not the iterates $(u_n, v_n)$ converge to zero, as generated by the second order finite-memory scheme \eqref{2ndorder-finite} with zero input $x_n = 0$, for several values of damping factor $\rho \in [.96,1]$ and initial conditions $(u_0, 0) \in [-2,0] \times 0$.  In particular, $(u_n, ,v_n) \rightarrow 0$ can only occur over the black region, as a gray mark  indicates that after one million iterations $n \geq 1000000$, the quantization output $q_n$ is nonzero $|q_n| = 1$ at regular time interval not exceeding $100$ iterations, for one million additional steps $n \in [1000000, 2000000]$.  The finite-memory scheme \eqref{2ndorder-finite} was implemented with parameter $\rho = .2$ and tri-level quantizer $Q^{1/3}_{tri}$. }.

\label{fig:quietmap}
\end{center}
\end{figure}
\vspace{5mm}
\noindent Nevertheless, we show in the following sections that a simple modification to the finite-memory scheme \eqref{2ndorder-finite} will ensure quietness, in the sense of Definition \ref{def1}.  The key idea is to apply a damping factor $\rho$ as in the finite second order scheme $\eqref{2ndorder-finite}$, but not at  \emph{every} iteration;  only at iterations $n$ for which $u_n + \gamma v_n \geq 0$ (or the symmetric, when $u_n + \gamma v_n \leq 0$).  The situation $u_n + \gamma v_n \geq 0$ must keep reccurring if the system is to remain stable, and the application of an asymmetric damping \emph{forces} the state variables $u_n$, and in turn the variables $v_n$, to zero once $f_n = 0$.

\vspace{5mm}

\noindent More precisely, we will consider the following \emph{asymmetrically damped} modification to \eqref{2ndorder-finite}, which, starting from initial $(u_0, v_0)$, can be implemented according to the recursion
\begin{eqnarray}
(b_n, q_n) &=& Q_4^\tau {(u_{n-1}/\gamma + v_{n-1})}, \nonumber \\
u_n &=&  u_{n-1} + q_n(\rho-1)u_{n-1} - b_n + f_n^{\lambda}, \nonumber \\
v_n &=&  v_{n-1} + q_n(\rho-1)v_{n-1} + u_n,
\label{2ndorder-aleak}
 \end{eqnarray} 
with 4-level quantizer,
 \begin{eqnarray}
Q_{4}^{\tau}(u) &=& \left\{\begin{array}{cl}
	(-1,0),  &  u \leq - 1/2, \\
	(0,0),   & -1/2 < u \leq 0, \\
	(0,1), & 0 < u \leq 1/(2\tau), \\
	 (1,1), & u > 1/(2\tau).  \end{array}\right.
           \label{(b2)}
\label{q2}
 \end{eqnarray} 
Although we will take $\tau = \rho$ for the theoretical analysis of the next section, simulation results suggest that quietness persists for a much more general class of 4-level quantizers that includes the symmetric quantizer $Q_4^{1}$.
\\
\\
That the asymmetric scheme \eqref{2ndorder-aleak} maintains second order accuracy for $\rho \geq \rho^{\lambda} = 1 - \frac{1}{\lambda}$ is an immediate consequence of Proposition \ref{leaky2}.  Furthermore, the asymmetric scheme is stable and inherits the positively invariant sets $S_{\gamma}$ constructed in Proposition \ref{stability}.  It remains to prove that  \eqref{2ndorder-aleak} is quiet. Our accomplishments are summarized in the following theorem.

 \begin{theorem}[{\bf Main theorem}]
The asymmetrically-damped second-order $\Sigma \Delta$ scheme \eqref{2ndorder-aleak} is quiet when implemented with 4-level quantizer $Q_4^{\rho}$. That is, if the input $(f_n)$ to \eqref{2ndorder-aleak} becomes identically equal to zero after some time $n \geq N$, then the quantization output $(b_n, q_n) \equiv (0,1)$ becomes constant after a finite number of additional iterations $n \geq N + k$.   
 \label{mainthm}
 \end{theorem}

\noindent We prove Theorem \ref{mainthm} by showing that $(0,0)$ is a global attracting fixed point of the zero-input map governing $(u_n, v_n) \rightarrow (u_{n+1}, v_{n+1})$ in the asymmetric scheme  \eqref{2ndorder-aleak}.  As such, we need to develop a better understanding for this dynamical system. 
 
 \begin{observation}
 The discrete map $(u_n, v_n) \rightarrow (u_{n+1}, v_{n+1})$ produced by the asymmetrically-damped $\Sigma \Delta$ scheme \eqref{2ndorder-aleak} with zero input $f_n^{\lambda} \equiv 0$ and 4-level quantizer $Q_4^{\rho}$ corresponds to the orbit of a piecewise-affine dynamical system, 
\begin{equation}
(u_n, v_n) = M_{(\gamma, \rho)}(u_{n-1}, v_{n-1}) = A_{\gamma} \circ D_{(\gamma,\rho)}(u_{n-1}, v_{n-1}), 
\end{equation}
where   
\begin{equation}
 D_{(\gamma,\rho)}(u,v) =   \left\{\begin{array}{ll}
	(\rho u, \rho v), & \textrm{    } u/\gamma + v  > 0, \\
	(u,v), & \textrm{ else,}  \\
   \end{array}\right.
           \label{D}
 \end{equation}
 and
 \begin{equation}
A_{\gamma}(u,v) =   \left\{\begin{array}{ll}
	(u-1, u + v - 1), & \textrm{    } u/\gamma + v > 1/2, \\
	(u, u + v) , & {    }  | u/\gamma + v | \leq 1/2,  \\
	(u + 1, u + v + 1), & {   } u/\gamma + v <  -1/2. \\
   \end{array}\right.  \nonumber
 \end{equation}
\label{Mmaps} 
\end{observation}
\noindent The origin $(0,0)$ is clearly a fixed point of the map $M_{(\gamma,\rho)}$.  We will show that $(0,0)$ is a global attracting fixed point, or that $(u_n, v_n) = M^n_{(\gamma, \rho)}(u_0, v_0) \rightarrow (0,0)$ for any initial condition $(u_0, v_0)$, in two steps:

\begin{enumerate}
\item In Section $2.0.9$, we construct a positively invariant set $T$ for the map $M_{(\gamma, \rho)}$ having the property that all orbits initialized in $T$ converge to the origin.  
\item In Section $2.0.10$, we show that $T$ is also a \emph{trapping set} for $M_{(\gamma, \rho)},$ so that all orbits eventually become trapped in $T$.
\end{enumerate} 

\noindent It is not clear that the proof need necessarily be split into two steps; however, numerical results such as those in Figure \ref{fig:magnifty_orbit} suggest a marked change in the behavior of sequences $({\bf x}, M_{(\gamma, \rho)}{\bf x}, M^2_{(\gamma, \rho)}{\bf x}, ... )$ upon entering the invariant set $T$.  

\subsection{An invariant set $T$ and convergence to the origin} 
We begin by constructing a positively invariant set for the map $M_{(\gamma, \rho)}$.  We refer the reader to Figure \ref{fig:stable2} for reference.
\begin{proposition}
The union $T = T^+ \cup T^-$ of the affine regions
\begin{eqnarray}
T^+ &=& \{ (u,v) \textrm{:   } 0 < u < 1 \textrm{,  } -1/2 \leq u/\gamma + v \leq (1/2 + 1/\gamma) \}, \nonumber \\
T^- &=& \{ (u,v) \textrm{:  } -1 < u \leq 0 \textrm{,  } -(1/2 + 1/\gamma) \leq u/\gamma + v \leq 1/2  \} \nonumber
\end{eqnarray}
is a positively invariant set for the map $M_{(\gamma, \rho)}$.  Moreover, the bit sequence $b_n$ as in \eqref{2ndorder-aleak} associated to a sequence $M^n_{(\gamma, \rho)}(u_0, v_0) = (u_n, v_n)$ contained in $T$ has the following alternating structure:
\begin{itemize}
\item If $b_n= 1$ then $b_{n+1} \in \{-1,0\}$,
\item If $b_n= -1$ then $b_{n+1} \in \{0,1\}$.   
\end{itemize}
Finally, $b_n = 1$ if and only if $(u_n, v_n) \in T^+$ and $(u_{n+1}, v_{n+1}) \in T^-$, while $b_n = -1$ if and only if $(u_n, v_n) \in T^-$ and $(u_{n+1}, v_{n+1}) \in T^+$.
\label{invariantregion}
\end{proposition} 

 \begin{figure}
 \begin{center}
 \subfigure[The positively invariant set $T$]{\includegraphics[width=4in]{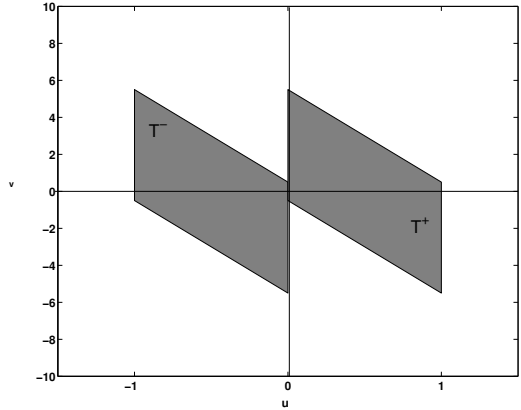}
\label{1}}
 \caption{The positively invariant set $T = T^+ \cup T^-$ for the map $M_{(\gamma, \rho)}$ with parameter $\gamma = .2$.}
 \label{fig:stable2} 
 \end{center}
 \end{figure}

\begin{proof}
First, as $T = T^+ \cap T^-$ is the union of two convex sets, each containing the origin, ${\bf x} \in T$ implies $\lambda {\bf x} \in T$ for $\lambda \in [0,1]$.  As such, if $T$ is positively invariant for the map $A_{\gamma}$, then $T$ must also be positively invariant for the full map $M_{(\gamma, \rho)} = A_{\gamma} \circ D_{(\gamma,\rho)}$.  In order to show that $T$ is positively invariant for $A_{\gamma}$, or that $(u,v) \in T$ implies $A_{\gamma} (u,v) \in T$, we can assume without loss that that $(u,v) \in T^+$ by symmetry of $T$ and $A_{\gamma}$.  Note that the bit $b$ in \eqref{2ndorder-aleak} associated to a point $(u,v) \in T^+$ in the asymmetric scheme \eqref{2ndorder-aleak} is restricted to $b \in \{0,1\}$. We split the proof into two cases,  according to whether $b=0$ or $b=1$.  
\begin{enumerate}
\item {\bf Case 1}: $b = 1$, or $\gamma/2 <  u + \gamma v \leq \gamma/2 + 1$:  In this case, $(u', v') = A_{\gamma}(u,v) = ( u - 1,  v + u -1)$.  Since $-1 < u' < 0$, $(u',v') \in T$ if $(u',v') \in T^-$, or if $(u',v') \in T^-$ if $-(\gamma/2 + 1) \leq u' + \gamma v' \leq \gamma/2$.  We can expand $u' + \gamma v'$ as 
$$ u' + \gamma v' = u + \gamma v - (\gamma + 1) + \gamma u,$$
and so indeed
$$ u' + \gamma v'  \leq (\gamma/2 + 1) - (\gamma + 1) + \gamma  \leq \gamma/2, $$
while also
$$ u' + \gamma v' \geq \gamma/2 - (\gamma + 1) + \gamma u \geq -(\gamma/2 + 1). $$
Thus, $(u',v') \in T^-$ and has associated bit $b' \in \{-1,0\}$.

\item {\bf Case 2:} $b = 0$, or $-\gamma/2 <  u + \gamma v \leq \gamma/2$: Within this region, $(u', v') = A_{\gamma}(u,v) = (u,  v + u)$.  Since $0 < u' \leq 1$, $(u',v') \in T$ if $(u',v') \in T^+$, or if $-\gamma/2 \leq u' + \gamma v' \leq \gamma/2 + 1$.  Proceeding as before,
$$ u' + \gamma v' = u + \gamma v + \gamma u \leq \gamma/2 + \gamma \leq \gamma/2 + 1, $$
while at the same time
$$ u' + \gamma v'  > -\gamma/2 + \gamma u > -\gamma/2. $$
Thus, $(u',v') \in T^+$ and has associated bit $b' \in \{0,1\}$.  
\end{enumerate}
The proposition follows by incorporating the symmetric result for initial conditions $(u,v) \in T^-$.
\end{proof}

\noindent The alternating pattern of the bit sequence $b_n$ associated to iterates $(u_n, v_n) \in T$ will be key in the following lemma.  In particular, this constraint on the bit sequence, in combination with the imbalance induced by damping $(u_{n+1}, v_{n+1}) = (\rho u_n, \rho v_n)$ only in the positive half-space $S^+ = \{ (u,v): u + \gamma v \geq 0 \}$, forces the $u_n$, and in turn the $v_n$, to zero.
 
 \begin{lemma} \label{utozero}
Suppose that $(u_0, v_0) \in T$.  Then the subsequence $(u_{n_j})$ from $(u_n, v_n) = M_{(\gamma,\rho)}^n(u_0,v_0)$ consisting of indices $n_j$ for which $u_{n_j} \in T^+$ satisfies $u_{n_j} \rightarrow 0$ as $n \rightarrow \infty$.    
 \end{lemma}
 
 \begin{proof}
First note that the event $(u_n, v_n) \in T^+$ must keep recurring, for if not, or if $(u_n, v_n) \notin T^+$ for all $n \geq 0$, then $b_n = 0$ for all $n$ according to Proposition \ref{invariantregion}, and 
$$v_n = v_0 + \sum^n_{j=0} u_j = v_0 + (n+1) u_0 $$
 diverges.  The index set ${\cal{I}^+} = \{ n_j : u_{n_j} \in T^+ \}$ is then an infinite subset of the natural numbers, so that $(u_{n_j})$ represents an infinite subsequence of $(u_n)$.  As $u_{n+1} = \rho u_n - 1$ in passing from $T^+$ to $T^-$, $u_{n+1} = u_n$ while both $(u_n,v_n)$ and $(u_{n+1}, v_{n+1})$ remain in $T^-$, and $u_{n+1} = u_n + 1$ in passing back from $T^-$ to $T^+$, it follows that $u_{n_{j+1}} = \rho u_{n_j}$ contracts along the index set $\cal I^+$. The subsequence $u_{n_j} = \rho^j u_{n_0}$ thus converges to zero as $n_j \rightarrow \infty$.
 \end{proof}
 
\noindent  Note that convergence of the subsequence $(u_{n_j})$ in Lemma \ref{utozero} does not guarantee convergence of the \emph{full} sequence $(u_n)$, as the residual sequence $(u_{n_k})$ over the index set ${\cal \cal{I}^-} = \mathbb{N} \setminus {\cal{I}^+}$ forms a subsequence satisfying $u_{n_k} \rightarrow -1$ as $k \rightarrow \infty$, if ${\cal \cal{I}^-}$ is an infinite set.   However, the following proposition ensures that this situation cannot occur.
 
 \begin{proposition}
 If $(u_0, v_0) \in T$, then the full state sequence $(u_n, v_n) = M^n_{(\gamma, \rho)}(u_0, v_0)$ satisfies $(u_n, v_n) \rightarrow (0^+,0^+)$ as $n \rightarrow \infty$. 
 \label{tozero!}
 \end{proposition}
 
 \begin{proof}
 Using the aforementioned results, we assume without loss that $(u_0, v_0) \in T^+$, and that $u_0 \leq \epsilon$ for arbitrarily small $\epsilon > 0$.  
 \begin{enumerate}
 \item Consider the situation where $(u_n, v_n) \in T^+$ for all $n \geq 0$.  In this case, $b_n = 0$ for all $n \geq 0$, $u_n \rightarrow 0^+$, and $v_n$ is defined recursively according to 
 \begin{eqnarray}
 v_{n+1} &=& \rho v_{n} + \rho u_{n-1} \nonumber \\
 &=& \rho^2 v_{n-1} + \rho^2 u_{n-2} + \rho u_{n-1} \nonumber \\
 &=& ... \nonumber \\
 &=& \rho^n v_0 + \sum^{n-1}_{j=0} \rho^{n-j} u_j \nonumber \\
 &=& \rho^n v_0 + \rho^n \epsilon,
 \end{eqnarray}
 where we use in the last equality that $u_j = \rho^j u_0$.  In this case, then, $v_n \rightarrow 0$.

\item Suppose that  $\rho  (u_0 + \gamma v_0) \leq \gamma/2$, in which case  $b_1 = 0, u_1 = \rho u_0$, and $v_1 = \rho  (v_0 + u_0)$, yielding
\begin{eqnarray}
\rho  (u_1 + \gamma v_1) &=& \rho ^2 (u_0 + \gamma v_0) + \rho ^2 \gamma u_0 \nonumber \\
&\leq& \rho (\gamma/2) + \rho ^2 \gamma \epsilon, \nonumber \\
&\leq& \gamma/2,
\end{eqnarray}
the last inequality requiring that $u_0 \leq \epsilon \leq (1 - \rho)/(2 \rho^2)$, which is satisfied as $\epsilon$ was arbitrary.  This case, then, reduces to the first case.

\item It remains only to consider sequences $(u_n, v_n)$ for which $(u_n, v_n) \in T^+$ implies $b_n = 1$, or $\rho \big( u_n + \gamma v_n \big) > \gamma/2$.  In terms of Figure \ref{fig:stable2}, such sequences avoid the lower parallel section of the set $T^+$.  The trajectory of such $(u_{n},v_{n})$ under the map $M_{(\gamma, \rho)}$ can then be described as follows:

\begin{enumerate}
\item If $(u_{n}, v_{n}) \in T^+$, then  \\
\vspace{1mm}
\hspace{10mm} $b_{n} = 1$, \\
\vspace{1mm}
\hspace{10mm} $u_{n+1} = \rho  u_{n}  - 1$,  \\
\vspace{1mm}
\hspace{10mm} $v_{n+1} = \rho v_{n} + \rho  u_{n} - 1 \leq \rho  v_{n} + \rho  \epsilon - 1$ \\ 

\noindent Else if $(u_{n}, v_{n}) \in T^-$,

\vspace{5mm}
\item  If $-1/2  \leq u_{n}/\gamma + v_{n} \leq 1/2$, then \\
\vspace{1mm}
\hspace{10mm} $b_{n} = 0$, \\
\vspace{1mm}
\hspace{10mm} $u_{n+1} =  u_{n} \leq \epsilon - 1$,  \\
\vspace{1mm}
\hspace{10mm} $v_{n+1} = v_{n} + u_{n} \leq v_{n} + \epsilon - 1$ \\ 

\item  If $u_{n}/\gamma + v_{n} \leq -1/2,$ then  \\
\vspace{1mm}
\hspace{10mm} $b_{n} = -1$, \\
\vspace{1mm}
\hspace{10mm} $u_{n+1} = u_{n}  + 1 \leq \epsilon$,  \\
\vspace{1mm}
\hspace{10mm} $v_{n+1} =  v_{n} +  u_{n} + 1 \leq v_{n} + \epsilon$ 
\end{enumerate}

Since $(c)$ cannot occur in successive iterations according to the alternating nature of $b_n$, we arrive at the period-2 inequality
\begin{equation}
v_{n+2} \leq v_n + 2\epsilon - 1,
\end{equation}
indicating that the iterates $v_n$ diverge. This case, in conclusion, cannot occur.
\end{enumerate}  
\end{proof}

\subsection{The invariant set $T$ as a global trapping set for $M_{(\gamma, \rho)}$}
We prove in this section that the positively invariant set $T$ is also a global \emph{trapping set} for $M_{(\gamma, \rho)}$; this result in combination with the results of the last section guarantee that the origin is a global fixed point for $M_{(\gamma, \rho)}$.
\\
\\
To be precise,
\begin{definition}
A global trapping set of a discrete mapping $M: \mathbb{R}^2 \rightarrow \mathbb{R}^2$ is any set $S \subset \mathbb{R}^2$ such that
\begin{enumerate}
\item $M(S) \subset S$ (that is, $S$ is a positively invariant set for $M$),
\item for any ${\bf x} \in \mathbb{R}^2$, there exists $n \geq 0$ such that $M^n {\bf x} \in S$.
\end{enumerate}
\end{definition}
\noindent The construction of global trapping sets for the \emph{full} second-order $\Sigma \Delta$ scheme was recently developed by Sidong \cite{SidongThesis}.   We will follow the approach of that work and show that $T$ is a global trapping set for the map $M_{(\gamma, \rho)}$ using a \emph{Lyapunov function} argument.   In the context of discrete dynamical systems, a Lyapunov function refers loosely to a nonnegative convex energy functional $h: \mathbb{R}^2 \rightarrow \mathbb{R}$ that contracts along the action of the discrete map, $h(M{\bf x}) \leq (1 - \delta)h({\bf x})$.    If $h$ contracts as such for all ${\bf x} \in \mathbb{R}^2$, and $h({\bf x}) = 0$ only if ${\bf x} = {\bf 0}$, then all orbits $(u_n, v_n) = M^n(u_0, v_0)$ must converge to the global fixed point $(0,0)$.  More generally, if $h$ decreases along orbits $h(M^n{\bf x}) \leq h({\bf x}) - \delta$ while the iterates $M^n{\bf x}$ remain outside an invariant set $S \in \mathbb{R}^2$, then $S$ can be shown to be a globally attracting set, as follows:

\begin{lemma}
Let $S$ be a positively invariant set for the discrete map $M$.  Suppose there exists a nonnegative function $h: \mathbb{R}^2 \rightarrow \mathbb{R}^+$ and a parameter $\delta > 0$ for which, if ${\bf x}$ is not in $S$, then either $M^k {\bf x} \in S$ or $h(M^k{\bf x}) - h({\bf x}) \leq - \delta$ after some finite time $k$. It follows that $S$ is a global trapping set for $M$.
\label{globaltrap}
\end{lemma}

\begin{proof}
Suppose ${\bf x} \notin S$ is such that $M^k{\bf x} \notin S$ for all $k \geq 0$, and denote $c = h({\bf x})$. From the stated hypotheses, $h(M^{k_1}{\bf x}) \leq c - \delta$ after some finite time $k_1$, and, by induction,  $h(M^{k_{n}}{\bf x}) \leq c - n \delta$ after a finite time $k_n$ for any positive integer $n$.   But then eventually $h(M^{k}{\bf x}) \leq 0$, which is impossible since $h \geq 0$ by assumption.
\end{proof}

\noindent Lemma \eqref{globaltrap} has the following implication for the map $M_{(\gamma, \rho)} = A_{\gamma} \circ D_{(\gamma, \rho)}$ and invariant set $T$ under consideration.

\begin{proposition}
Consider the positively invariant set $T$ as constructed in Proposition \eqref{invariantregion}.  $T$ is a global trapping set for $M_{(\gamma, \rho)}$ if  there exists a nonnegative and \emph{convex} function $h: \mathbb{R}^2 \rightarrow \mathbb{R}^+$ and a parameter $\epsilon > 0$ for which the following sets are contained in $T$:
\begin{eqnarray}
\Delta_h &:&  \{(u,v) \in \mathbb{R}^2 : h(u,v) \leq h(A_{\gamma}(u,v)) \}, \nonumber \\\Omega^{\epsilon}_{h} &:& \{(u,v) \in \mathbb{R}^2 : h(u,v) \leq \epsilon \}
\end{eqnarray}
\label{globaltrap2}
\end{proposition}

\begin{proof}
We verify that the conditions for Lemma \eqref{globaltrap} hold for invariant set $T$ and parameter $\delta = \epsilon (1 - \rho) > 0$.  The proof is split into two cases, according to whether or not ${\bf x} \in S^+ = \{ (u,v) \in \mathbb{R}^2: u + \gamma v \geq 0 \}$ or $S^- = \mathbb{R}^2 \setminus S^+$.

\begin{enumerate}
\item {\bf Case 1:} Suppose first that ${\bf x} \in S^+$. On this half-space, $D_{(\gamma,\rho)}: {\bf x} \rightarrow \rho {\bf x}$ acts as a contraction. If ${\bf x} \notin T$, but $D_{(\gamma,\rho)} {\bf x} \in T$, then $M_{(\gamma, \rho)}{\bf x} = A_{\gamma} \circ D_{(\gamma,\rho)} {\bf x} \in T$ by positive invariance of $T$ for $A_{\gamma}$.  If alternatively  $D_{(\gamma,\rho)} {\bf x} \notin T$, then
\begin{eqnarray}
h(M_{(\gamma,\rho)} {\bf x}) - h({\bf x}) &=&  h(A_{\gamma} (D_{(\gamma, \rho)} {\bf x})) - h({\bf x}) \nonumber \\
&\leq& h(D_{(\gamma, \rho)} {\bf x}) - h({\bf x})  \textrm{, as } \Delta_h \in T  \nonumber \\
&=& h(\rho {\bf x}) - h({\bf x}) \nonumber \\
&\leq& \rho h({\bf x}) - h({\bf x}) \textrm{, using convexity of $h$ and $h({\bf 0}) = 0$} \nonumber \\ 
&\leq& -\epsilon(1 - \rho), \textrm{ since $\Omega_h^{\epsilon} \in T$} \nonumber \\
&=& -\delta.  \nonumber
\end{eqnarray}

\item {\bf Case 2}: Suppose now that $M_{(\gamma,\rho)}^k {\bf x} \in (S^+)^c \cap T^c$ for all $k \geq 0$.  Under this assumption, the quantization output is restricted to $b_j \in \{-1, 0\}$ for all $j$, so that $u_j = u_{j-1} - b_j$ forms a monotonically nondecreasing sequence.  If $u_j < 0$ for all $j$, then $| v_j | = | v_0 + \sum_{k = 0}^j u_k| $ diverges, which is impossible since $(u_j,v_j)$ belongs to a bounded set. If alternatively $u_j > 0$ at some index $j$, then $u_n > 0$ for all $n \geq j$ by monotonicity and, still, $|v_j| \rightarrow \infty$.  The case $u_j = 0$ for all $j \geq n$ cannot happen since $(u_n,v_n) \notin T$ by assumption.  It follows that this case is impossible.
\end{enumerate}

We can conclude that after a finite number of iterations $k$, either $M_{(\gamma, \rho)}^k \in T$ or $M^k_{(\gamma, \rho)} \in S^+$.  The result follows by application of Lemma \ref{globaltrap}.
\end{proof}

\noindent Following the approach in  \cite{SidongThesis}, we consider the following Lyapunov function for the map $A_{\gamma}$:
\begin{equation}
h(u,v) = u^2 + |2v - u|.
\end{equation}
Note that $h$ is a convex function on $\mathbb{R}^2$. Letting $h^+(u,v) = u^2 + 2v - u$ and $h^-(u,v) = u^2 -2v + u$, it is easily verified that $h(u,v) = \max{ \{h^+(u,v), h^-(u,v) \} }$.  The motivation for this choice of Lyapunov function is that $h^+$ and $h^-$ are the unique functional solutions satisfying the relations 
\begin{eqnarray}
\left. h^+(A_{\gamma}(u,v)) \right \vert_{ u + \gamma v > \gamma/2} = h^+(u - 1, u + v -1) &=& h^+(u,v),  \nonumber \\ 
\left. h^-(A_{\gamma}(u,v)) \right \vert_{ u + \gamma v < -\gamma/2} = h^-(u + 1, u + v +1) &=& h^-(u,v). 
\end{eqnarray}
It is straightforward that the positively invariant set $T$ contains the set $\Omega_{h}^{\epsilon} = \{(u,v): h(u,v) \leq \epsilon \}$ for sufficiently small $\epsilon$.   In order to apply Proposition \eqref{globaltrap2}, it remains only to verify that $T$ also contains the set $\Delta_h = \{ (u,v) \in \mathbb{R}^2: h(u,v) \leq h \big( A_{\gamma}(u,v) \big) \}$.


 \begin{proposition}
Consider the map $A_{\gamma}$ as defined in \eqref{Mmaps}, the convex function $h(u,v) = u^2 + | 2v - u|$, and the set $R = R_1 \cup R_2$, with
 \begin{eqnarray}
 R_1 &=& \{ (u,v) \in \mathbb{R}^2: u + \gamma v \geq 0,  2v + u \leq 1,  u \leq 1/2 \} \textrm{, and} \nonumber \\
 R_2 &=& \{ (u,v) \in \mathbb{R}^2 : u + \gamma v < 0, 2v + u \geq -1, u \geq -1/2 \}.
 \label{R}
 \end{eqnarray}
The set $R$ is contained in the positively invariant set $T$ (as shown in Figure \ref{fig:stable3}).  Moreover, $R$ contains the set $\Delta_h = \{ (u,v) \in \mathbb{R}^2: h(u,v) \leq h \big( A_{\gamma}(u,v) \big) \}$.
 \label{trap_eps0}
 \end{proposition}
\noindent  We defer the proof of Proposition \ref{trap_eps0}, which amounts to a straightforward case by case analysis, to Appendix 1.  
\\
\\
\noindent The main result of this chapter, Theorem \ref{mainthm}, follows from Proposition \ref{globaltrap2} and Lemma \ref{tozero!}.

\begin{proof}[Proof of Theorem \ref{mainthm}]
By Proposition \eqref{globaltrap2}, $T$ is a global trapping set for $M_{(\gamma, \rho)}$, the discrete map governing $(u_n, v_n) \rightarrow (u_{n+1}, v_{n+1})$ in the asymmetrically-damped scheme \eqref{2ndorder-aleak} with 4-level quantizer $Q_4^{\rho}$ and zero input $f_n \equiv 0$.  In other words, for any initial condition $(u_0, v_0)$, the iterates $(u_n, v_n) = M^n_{(\gamma, \rho)} (u_0, v_0)$ become trapped in the set $T$ after a finite number of iterations. From Lemma \ref{tozero!}, it follows that $(u_n, v_n) \rightarrow (0^+, 0^+)$; once the iterates $(u_n, v_n) \in T^+ \cap \{0 \leq u + \gamma v \leq \gamma/2 \}$ become trapped in the positive quadrant, the bit output $(b_n, q_n) = (0,1)$ is constant.  
\end{proof}
 
 \begin{figure}
 \begin{center}
 \includegraphics[width=12cm]{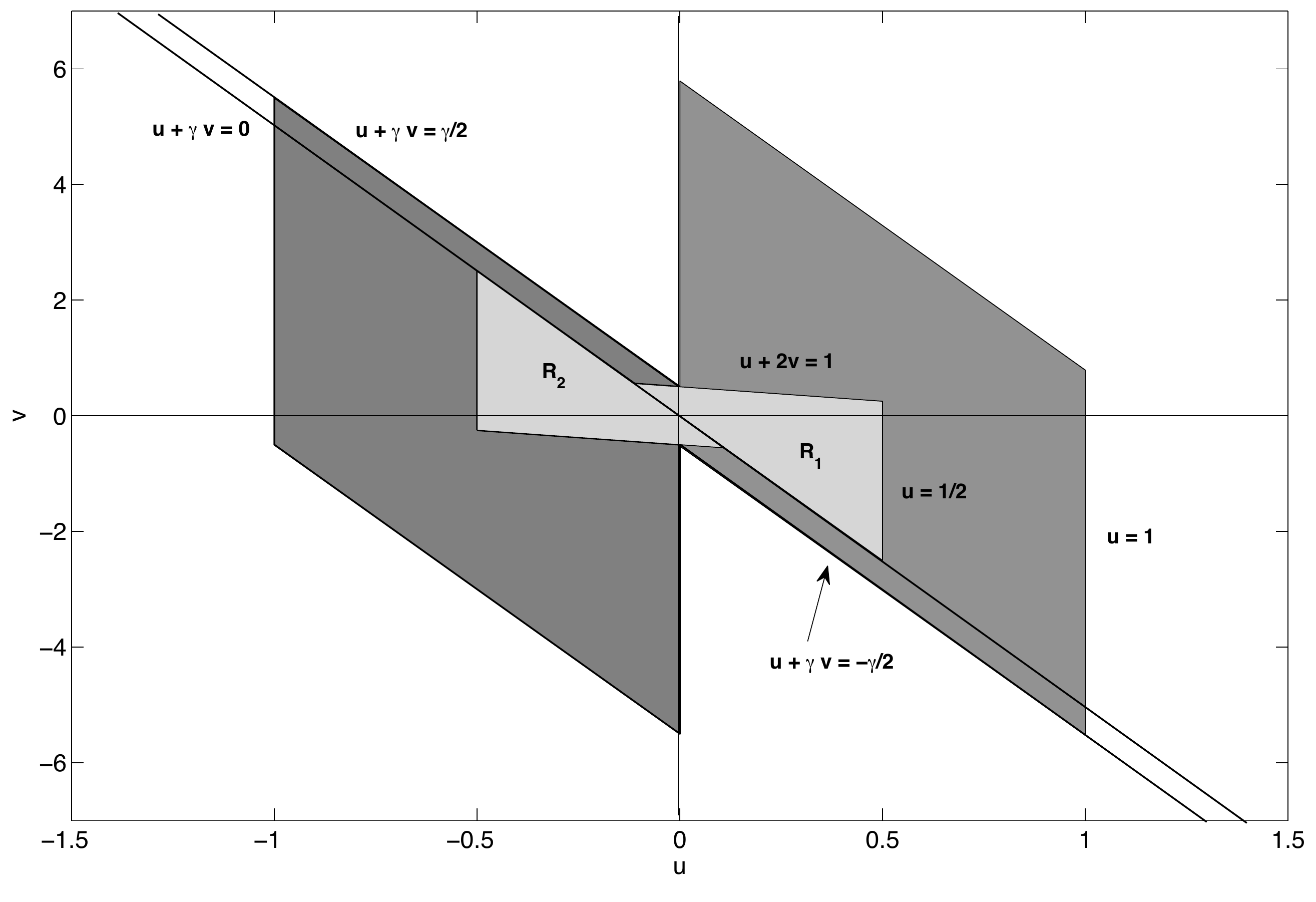}
 \caption{The positively invariant set $T$ is the union of two parallelograms, as shown in dark gray for parameter $\gamma = .2$. As the region $\Delta_h$, on which the map $A_{\gamma}$ may satisfy $A_{\gamma}(h(u,v)) - A_{\gamma}(u,v) > 0$, is contained in $R$ (light gray triangles) which is in turn contained in $T$, the set $T$ is a globally attracting set.  }

 \label{fig:stable3} 
 \end{center}
 \end{figure}

\begin{figure*}[htbp]
\begin{center}
\includegraphics[width=3in]{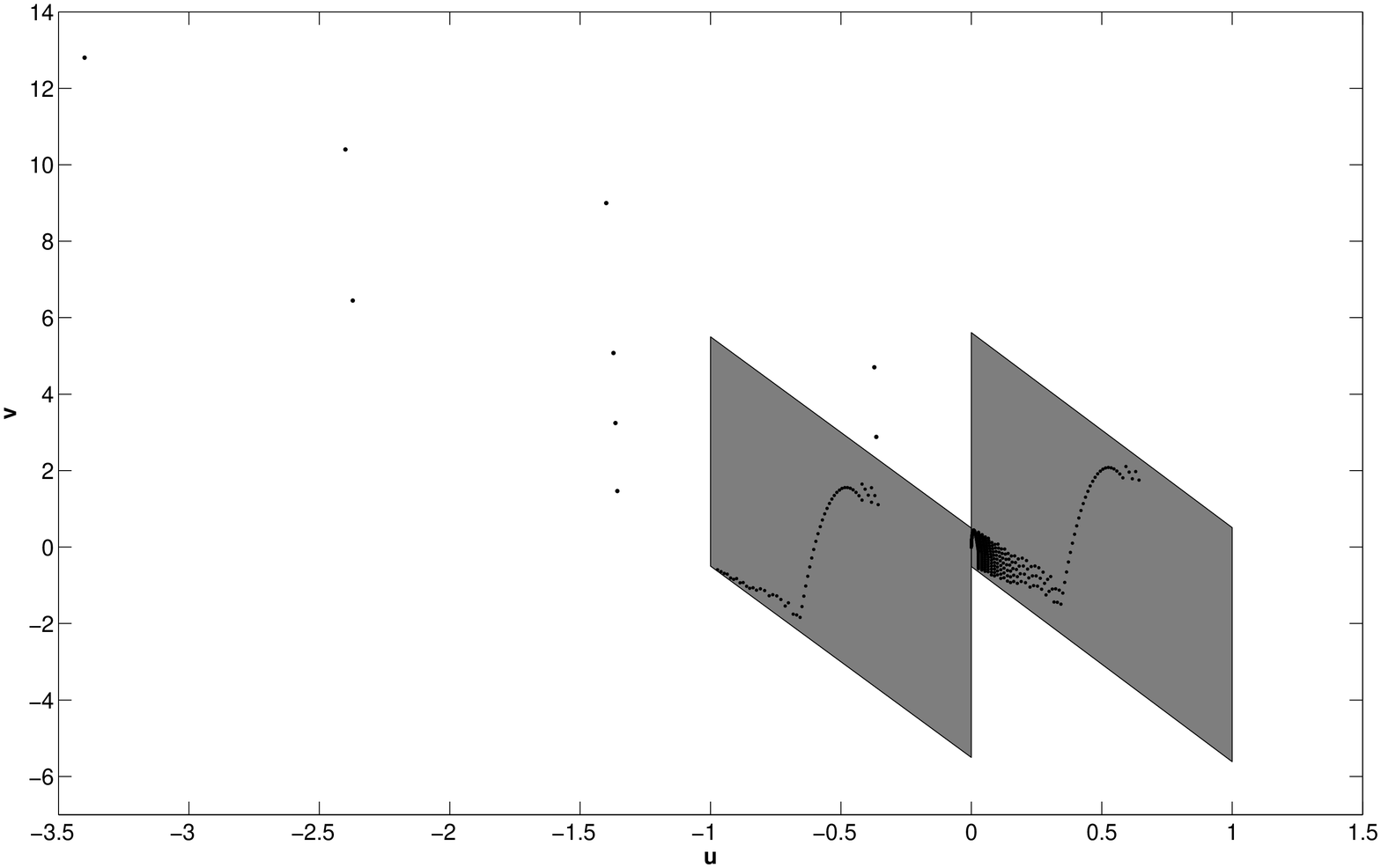}
\includegraphics[width=3in]{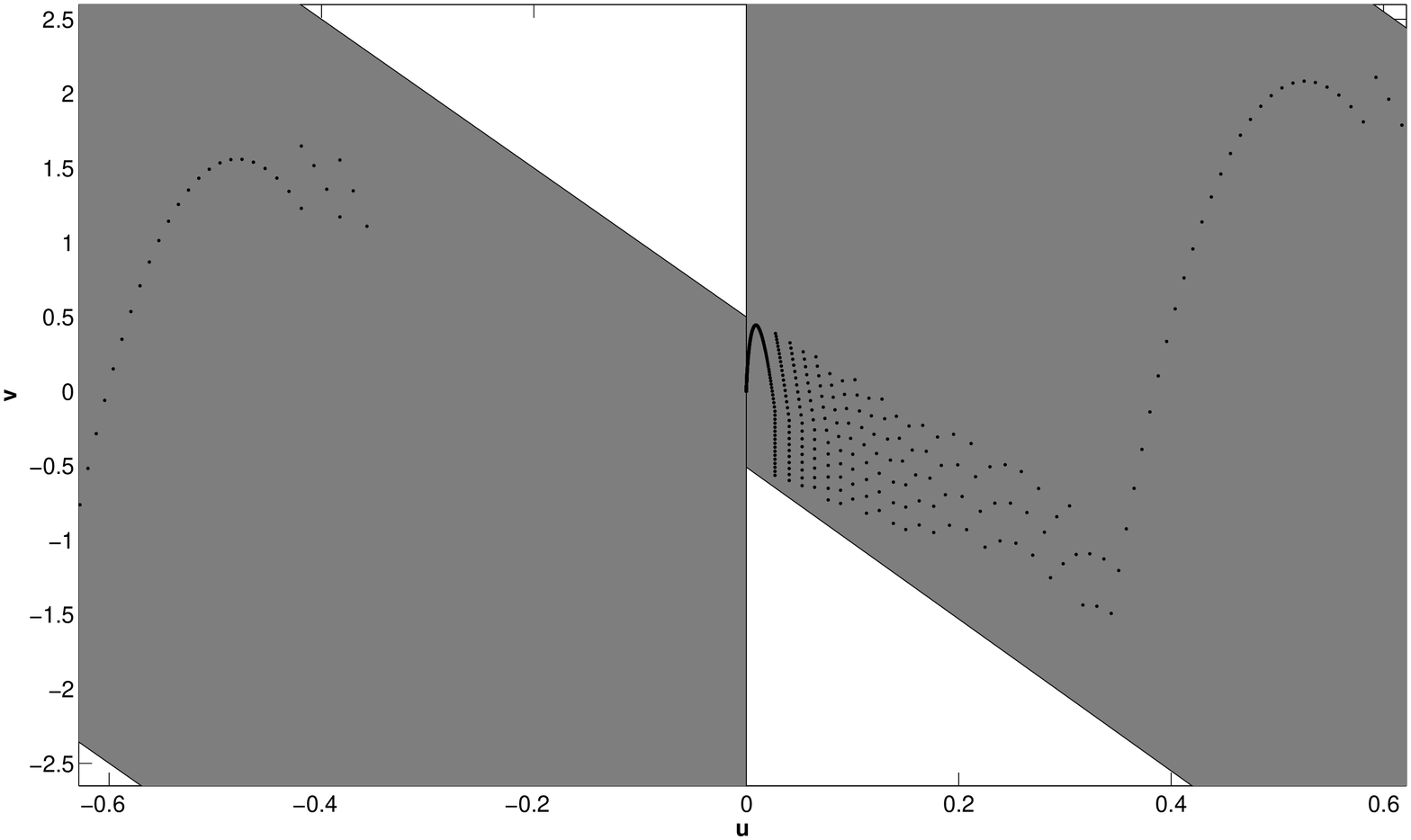}
\includegraphics[width=3in]{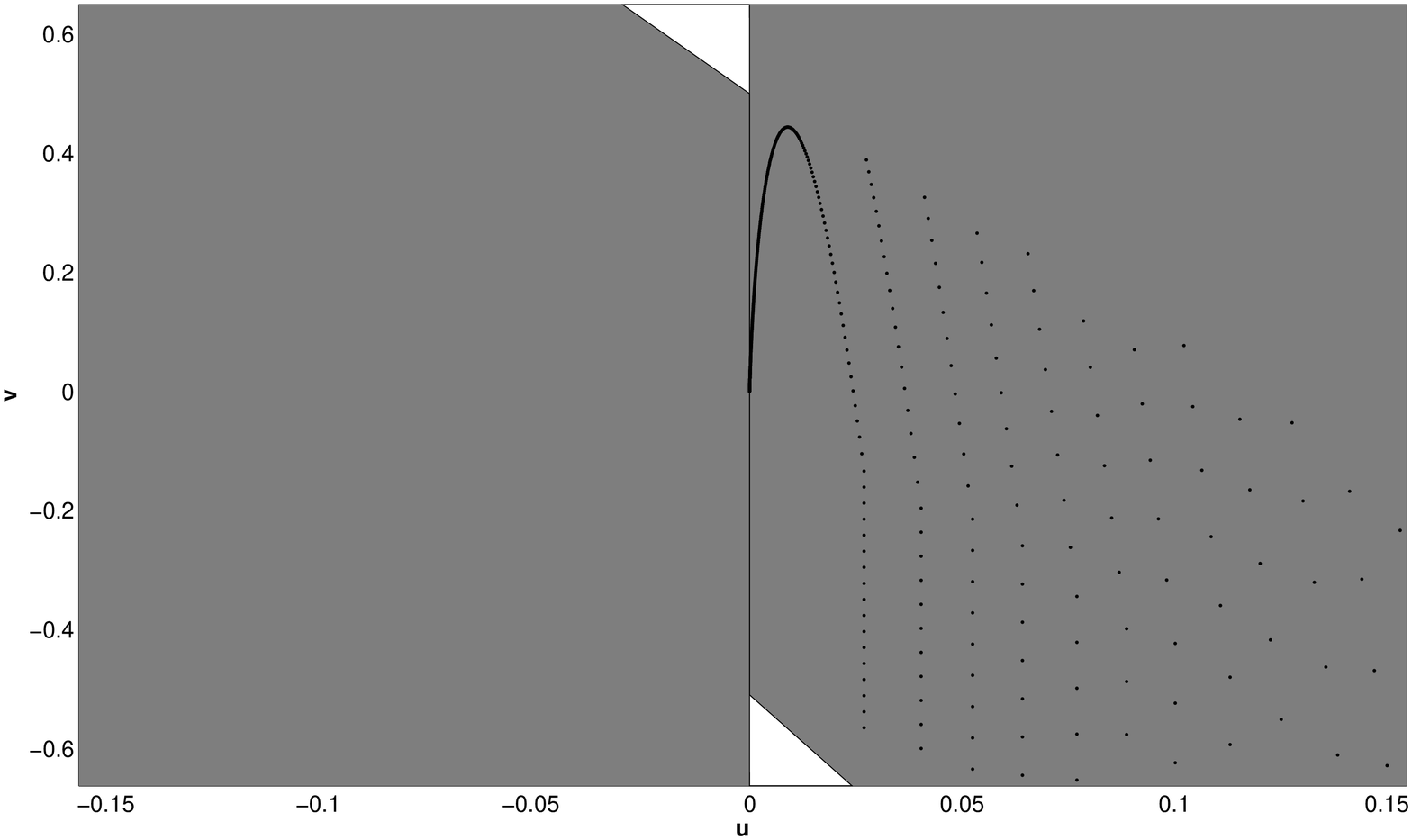}
\caption{ \small Different magnifications of an orbit of the map $M_{(\gamma, \rho)}$ for $\rho = .98$ and $\gamma = .2$.  The initial point $(u_0, v_0) = (-3.4, 12.7)$ can be seen in the top image; this point and the first few iterations are outside the trapping set $T$ (dark gray parallelograms).   Once trapped in $T$, the iterates $(u_n, v_n)$ converge to the globally attracting fixed point $(0,0)$.}
\label{fig:magnifty_orbit}
\end{center}
\end{figure*}
 
 \pagebreak
 
 \section{Future Directions}
We have introduced a $2$-bit second order $\Sigma \Delta$ scheme that is guaranteed to be \emph{quiet}:  periodicities in the bit output are eliminated when the input vanishes. It remains to analyze the stability of such quietness in the face of unavoidable component imprecisions.  Although a full analysis is beyond the scope of the current presentation, we outline a few key issues below.
\begin{enumerate}

\item {\bf Randomness helps.}  Any introduction of randomness to a quantization scheme will only \emph{increase} the aperiodicity of the bit output; this includes zero-mean additive noise on the samples $x_n = f_n + \epsilon_n$ and unbiased `flakiness' in the quantizer element.  Quietness  for the asymmetrically-perturbed scheme \eqref{2ndorder-aleak} is threatened more from \emph{biased} imprecisions, such as those resulting from an offset in the 4-level quantizer, $\tilde{Q}^{\rho}_4 (u) = Q_4^{\rho}(u + \delta)$.   Numerical simulations suggest that quietness of the asymmetric scheme nevertheless persists in the face of such offsets: the state sequence $(u_n, v_n)$ still converges at the onset of zero input, but the iterates $v_n$ approach a nonzero limit $L(\delta)$ that grows with the magnitude of the shift; if $\delta > 0$ is small enough, then quietness is still attained.  More generally, quietness of the asymetrically-damped scheme appears to persist for a much more general class of quantizers than the  the particular 4-level quantizer $Q^{\rho}_4$ that was needed for the theoretical analysis of the previous section; indeed, indistinguishable results are observed by implementing \eqref{2ndorder-aleak} with the far simpler symmetric 4-level quantizer,

\begin{equation}
Q_4(u) = \left\{\begin{array}{cl}
	(-1,0),  &  u \leq -1/2, \\
	(0,0),   & -1/2 < u \leq 0, \\
	(0,1), & 0 < u \leq1/2, \\
	 (1,1), & u > 1/2.  \end{array}\right.
\label{q4}
\end{equation}

\item {\bf Other robustness issues.}

\begin{enumerate}
\item  {\bf Imprecisions in $\rho$}: The damping factor $\rho$ does not have to be constant from iteration to iteration nor known precisely; all of the analysis of the previous section holds for 
$\rho_n > \rho^{\lambda}$ varying at each iteration, as long as the sequence remains bounded from below as to maintain second-order accuracy of the scheme \eqref{2ndorder-aleak}, and is not chosen adversarialy to converge $\rho_n \rightarrow 1$.  

\item {\bf Imprecisions in $\gamma$:} We have until now assumed that the parameter $\gamma$ is fixed throughout the iterations.  However, numerical evidence suggests that quietness does not require that $\gamma$ be fixed in the expression $u_n + \gamma v_n$, as long as $\gamma$ is within the range of stability, as stated in Theorem \eqref{stability}.   

\item {\bf Leaky integrators.}  More realistically, one should incorporate integrator leakage into the proposed model \eqref{2ndorder-aleak}, analogous to the modification studied for the golden ratio encoder in the last chapter. That is, the asymmetric scheme, being second order, requires two delay elements in its implementation to hold each of the states $u_n$ and $v_n$ over one iteration.  After one clock time, the stored input in the first delay is reduced to $\lambda_1$ times its original value, while the stored input in the second delay is replaced by $\lambda_2$ times its original value; we can incorporate such leakage into the asymmetric model \eqref{2ndorder-aleak} as follows:
\begin{eqnarray}
(b_n, q_n) &=& Q^4{(\lambda_1 u_{n-1}/\gamma +  \lambda_2 v_{n-1})}, \nonumber \\
u_n &=&  \lambda_1 \big( u_{n-1} + q_n(\rho_1-1)u_{n-1} \big) - b_n + f_n^{\lambda}, \nonumber \\
v_n &=& \lambda_2 \big( v_{n-1} + q_n(\rho_2-1)v_{n-1} \big) + u_n.
\label{aleak_intleak}
 \end{eqnarray} 
 Above, $(\lambda_1, \lambda_2) \in [.95,1]^2$ in most circuits on interest; the specific leakage factors within this window are generally unknown and may vary slowly in time.  When $\lambda_1 = \lambda_2$, the positively invariant set $T$ is still a trapping set for the revised scheme; however, once in this set, it is not immediately clear how to modify the analysis of section $2.0.9$ to prove that $(u_n, v_n) \rightarrow (0,0)$.  Numerical studies indicate that quietness persists in the face of such leakage, requiring only that the perturbations remain asymmetric, or that each of $\lambda_1 \rho_n$ and $\lambda_2 \rho_n$ is strictly smaller than either of $\lambda_1$ and $\lambda_2$. 

\end{enumerate}

\item{\bf Hybrid Chaotic-Quiet $\Sigma \Delta$ Quantization}

While the asymmetry of the perturbation $1 - \rho_n$ in \eqref{2ndorder-aleak} is key for inducing quietness, it is not clear that this perturbation need be nonnegative.  So-called \emph{chaotic $\Sigma \Delta$ quantization} \cite{SDChaos} has been proposed as a means to eliminate idle tones in $\Sigma \Delta$ quantization, not just at the onset of zero input, but at the onset of any constant, or DC input sequence.  In the context of 1-bit, second-order $\Sigma \Delta$ quantization, chaotic $\Sigma \Delta$ quantization corresponds to \emph{expanding}, rather than contracting, by a small factor $(1 + \epsilon) > 0$ at each iteration:
 \begin{eqnarray}
 b_n &=& Q(u_{n-1} + \gamma v_{n-1}),  \nonumber \\
u_n &=& (1 + \epsilon) u_{n-1} + f^{\lambda}_n - b_n, \nonumber \\
v_n &=& (1 + \epsilon) v_{n-1} + u_n.
\label{2ndorder_chaos}
 \end{eqnarray}
This modification is termed chaotic because it has the effect of generating aperiodic output at the onset of DC input; however, a complete stability analysis of the scheme \eqref{2ndorder_chaos} remains an open problem.  Numerical studies, such as those in Figure \ref{chaos}, indicate that the chaotic scheme outperforms the asymmetrically-damped scheme \eqref{2ndorder-aleak} in generating aperiodic output for general DC input, although the asymmetric scheme shows a marked improvement over the fully damped tri-level scheme.  While the chaotic scheme \eqref{2ndorder_chaos} is not quiet in the sense of Definition \eqref{def1}, we expect that quietness can be achieved by introducing asymmetry into the model \eqref{2ndorder_chaos}.  Specifically, the hybrid scheme,
\begin{eqnarray}
(b_n, q_n) &=& Q_4({u_{n-1}/\gamma + v_{n-1})}, \nonumber \\
u_n &=&  u_{n-1} + q_n \epsilon u_{n-1} - b_n + f_n^{\lambda}, \nonumber \\
v_n &=& v_{n-1} + q_n\epsilon v_{n-1} + u_n,
\label{aleak_intleak}
 \end{eqnarray} 
 with $q_n$ taking values in $\{-1,1\}$ so that either damping $(\rho_n  = 1 - \epsilon)$ or expansion $(\rho_n = 1 + \epsilon)$ is applied at each iteration, appears to inherit the aperiodic orbit structure of its parent chaotic map \eqref{2ndorder_chaos}, and also quietness induced by its asymmetry.  We hope to study the stability and aperiodicity for the chaotic maps \eqref{2ndorder_chaos}, \eqref{aleak_intleak} in the future.
 
 \end{enumerate}

\begin{figure*}[htbp]
\begin{center}
\subfigure[Original and damped second order $\Sigma \Delta$ scheme \eqref{2ndorder-finite} with $\rho = 1$ and $\rho = .995$, respectively]{\label{1} \includegraphics[width=2.5in]{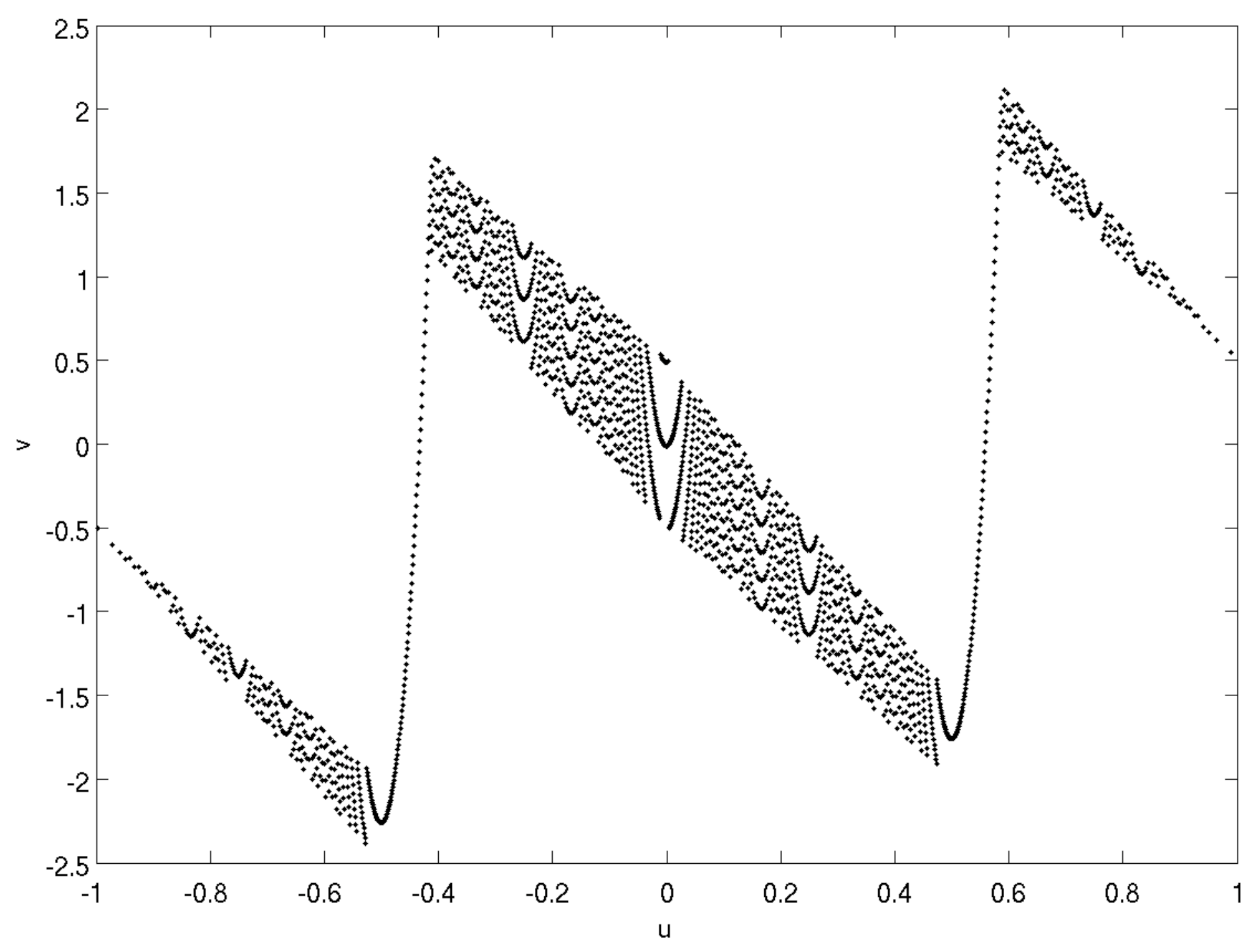}\includegraphics[width=2.5in]{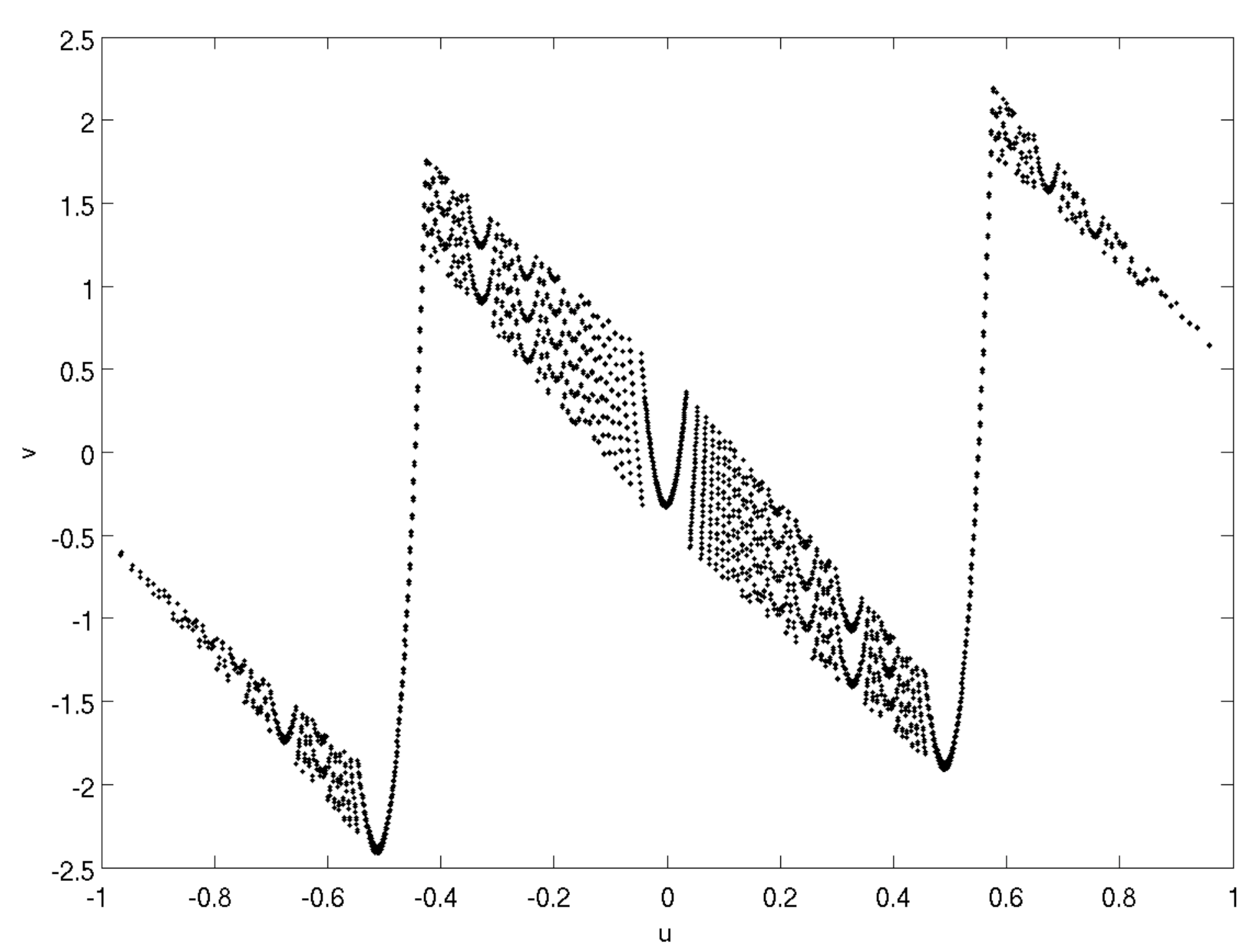}}

\subfigure[Asymmetrically damped $\Sigma \Delta$ scheme with $\rho = .995$.]{\label{2} \includegraphics[width=2.5in]{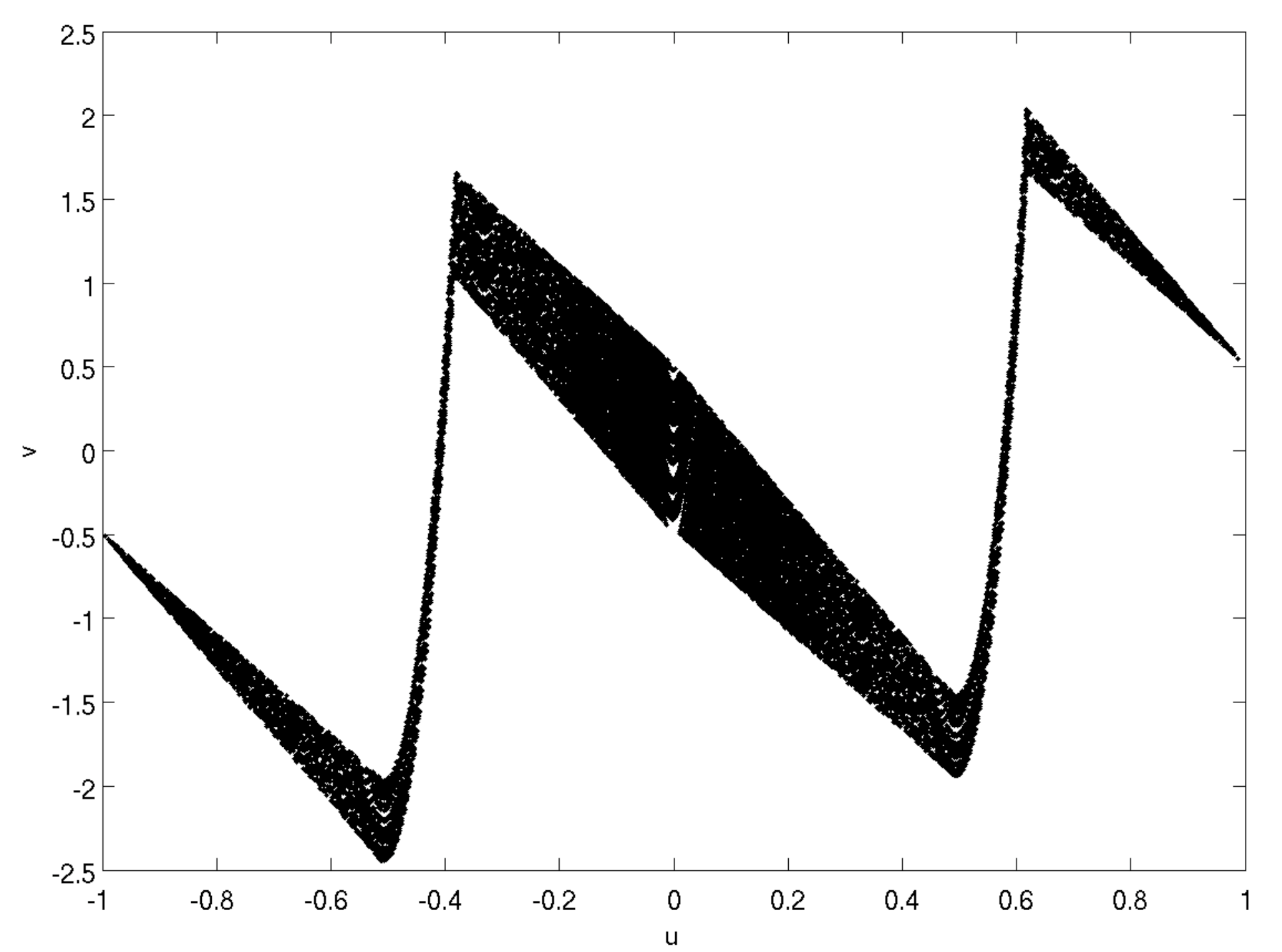}}

\subfigure[Chaotic $\Sigma \Delta$ scheme with $\rho = 1.01$.]
{\label{3} \includegraphics[width=2.5in]{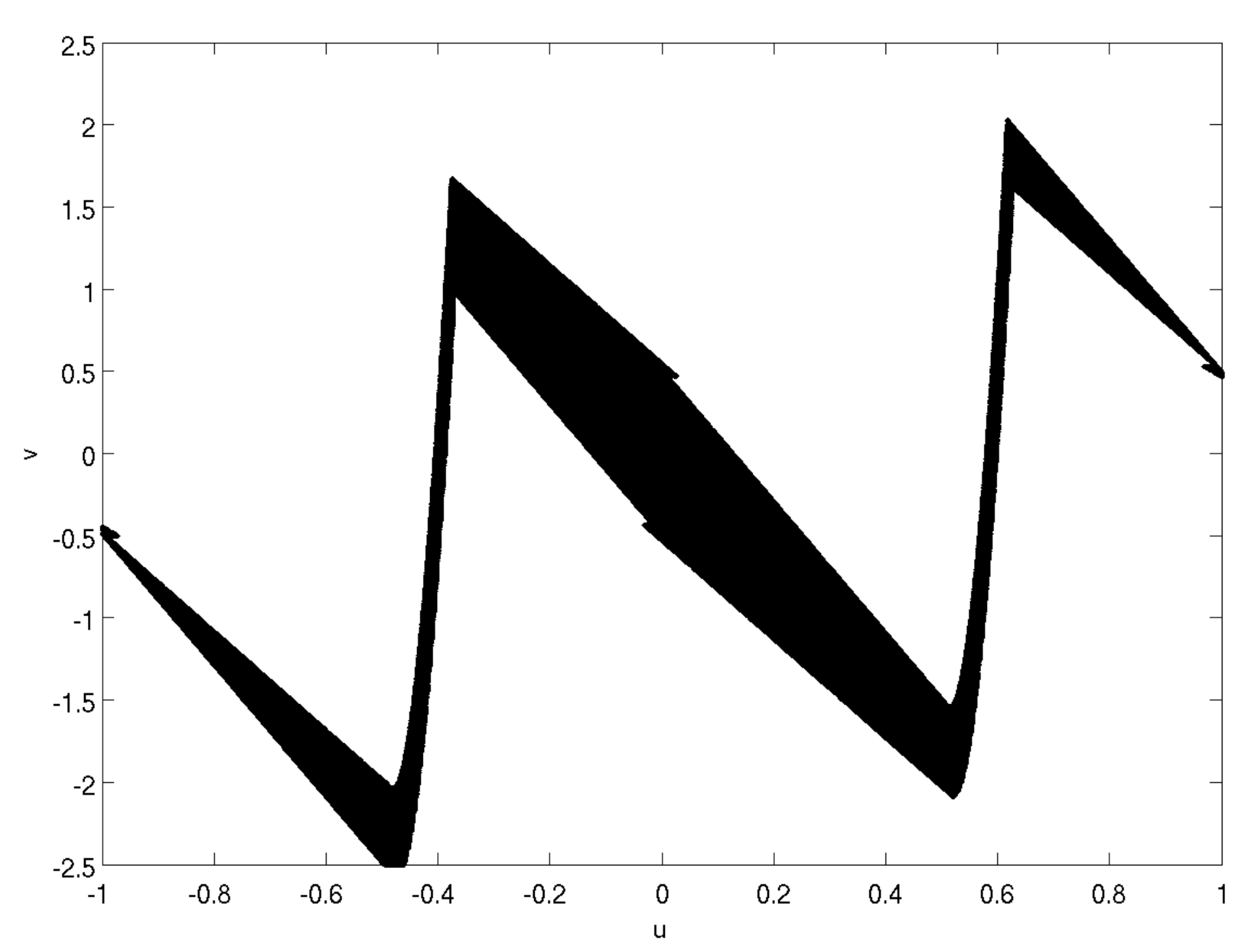}}

\caption{One million iterates $(u_n, v_n)$ from a single orbit under (a) the original damped second order scheme \eqref{2ndorder-finite} with damping factor $\rho = 1$ and $\rho = .995$, respectively, (b) the asymmetrically-damped scheme \eqref{2ndorder-aleak}, with asymmetric damping factor $\rho = .995$, and (c) the chaotic map \eqref{2ndorder_chaos}, with expansion factor $\rho = 1.01$.  Constant input $f_n = -.001$ was used for all cases, along with parameter $\gamma = .2$ and either tri-level quantizer $(a)$ or 4-level quantizer $(b), (c)$.  The asymmetrically-damped orbit $(b)$ is seen to fill up more space in the plane than the fully damped orbits $(a)$, but not as densely as the chaotic orbit $(c)$, which appears to be aperiodic.}
\label{fig:chaos}
\end{center}
\label{chaos}
\end{figure*}

\pagebreak

 \section{Appendix : proof of Proposition \ref{trap_eps0}}
Let $S^+ = \{ (u,v) \in \mathbb{R}^2: u + \gamma v \geq 0 \}$ and let $S^- =  \{ (u,v) \in \mathbb{R}^2: u + \gamma v < 0 \} = \mathbb{R}^2 \setminus S^+$.  Suppose that $(u,v) \in S^+ \setminus R_1$ is such that $(u',v') = A_{\gamma}(u,v) = (u - 1, u + v - 1)$.  Our first aim is to show that $h(M_1(u,v)) \leq h(u,v)$ in this situation.  
 \begin{enumerate}
\item {\bf Case 1:} If $2v - u \geq 0$, then $h(u,v) = u^2 + 2v - u$, while $h(u',v') = u^2 - 2u + 1 + | u + 2v -1|$,so 
$$
h(u',v') \leq h(u,v) \Longleftrightarrow | u + 2v - 1| \leq u + 2v  - 1 \Longleftrightarrow u + 2v \geq 1.
$$  
But since $(u,v) \in S^+ \setminus R_1$, we know that $u > 1/2$, so $u + 2v \geq 2u > 1$, and $h(u',v') \leq h(u,v)$ holds in this case.
\item {\bf Case 2:} If $2v - u \leq 0$, then $h(u,v) = u^2 + u - 2v$, while the expression for $h(u',v')$ remains unchanged, so
\begin{equation}
h(u',v') \leq h(u,v) \Longleftrightarrow | u + 2v - 1 | \leq 3u - 2v - 1.
\label{cond}
\end{equation}
We split this case into two subcases:
\begin{enumerate}
\item {\bf Case 2(a)}: If, on the other hand, $u + 2v > 1$, then $|u + 2v - 1| = u + 2v - 1$, and \eqref{cond} simplifies to 
$$
h(u',v') \leq h(u,v) \Longleftrightarrow  u + 2v - 1 \leq 3u - 2v - 1  \Longleftrightarrow 2v \leq u.
$$
 But since $2v - u \leq 0$ by assumption, the result holds in this subcase.

\item {\bf Case 2(b)}: If $u + 2v \leq 1$, then 
\begin{eqnarray}
h(u',v') \leq h(u,v) &\Longleftrightarrow&  -u - 2v + 1 \leq 3u - 2v - 1 \nonumber \\
&\Longleftrightarrow& u \geq 1/2.
\end{eqnarray}
But of course, the condition $u \geq 1/2$ holds throughout $S^+ \setminus R_1$.
\end{enumerate}
\end{enumerate}

\noindent We have shown thus far that $h(u',v') \leq h(u,v)$ if $(u,v) \in S^+ \setminus R_1$ and $M_1(u,v) = (u-1, u + v -1)$.  It remains to show that $h(u, u + v) \leq h(u,v)$ if $(u,v) \in S^+ \setminus R_1$ and $F_{\gamma}(u,v) \leq \gamma/2$.  By inspection of Figure 1, the intersection of the regions $S^+ \setminus R_1$ and $\{ (u,v): F_{\gamma}(u,v) \leq \gamma/2\} $ consists of two disjoint sets: $P_1:  \{ (u,v): 0 \leq u + \gamma v \leq \gamma/2, u + 2v \geq 1, u \leq 0 \}$, and $P_2 :  \{ (u,v): 0 \leq u + \gamma v \leq \gamma/2, u + 2v \leq -1,  u \geq 0 \}.$

\begin{enumerate}
\item {\bf Case 1: $(u,v) \in P_1$}: As $P_1 \subset \{ (u,v): u \leq 0, v \geq 0 \}$, the restriction $2v - u \geq 0$ trivially holds, and so $h(u,v) = u^2 + 2v - u$, and
\begin{eqnarray}
h(u,u+v) \leq h(u,v) &\Longleftrightarrow& u^2 + | 2(u + v) - u | \leq u^2 + 2v - u  \nonumber \\        &\Longleftrightarrow& |2v + u| \leq 2v - u \nonumber \\
&\Longleftrightarrow& 2v + u \leq 2v - u \nonumber \\
&\Longleftrightarrow& u \leq 0 \nonumber
\end{eqnarray}
which is satisfied by assumption.  
\item {\bf Case 2: $(u,v) \in P_2$ }: $P_2$ is contained in the quadrant $\{ (u,v): u \geq  0, v \leq 0 \}$, and so $2v - u < 0$, $h(u,v) = u^2 - 2v + u$, and
\begin{equation}
h(u,u+v) \leq h(u,v) \Longleftrightarrow | u + 2v | \leq u - 2v \Longleftrightarrow -(u + 2v) \leq u - 2v \Longleftrightarrow u \geq 0, \nonumber
\end{equation}
which again is satisfied by assumption.
\end{enumerate}
By symmetry of the set $R$ and the map $A_{\gamma}$, we have also that $h(A_{\gamma}(u,v)) \leq h(u,v)$ if $(u,v) \in S^- \cap \mathbb{R} \setminus R_2$.

\chapter{Cross validation in compressed sensing}
\section{Compressed sensing: Redefining traditional notions of sampling}
Compressed Sensing (CS) is a fast developing area in applied mathematics, motivated by the reality that most data we store and transmit contains far less information than its dimension suggests.  For example, a one-dimensional slice through the pixels in a typical grayscale image will contain segments of smoothly varying intensity, with sharp changes between adjacent pixels appearing only at edges in the image.  If a large data vector contains only $k << N$ nonzero entries, or is { \it k-sparse}, it is common practice to temporarily store the entire vector, possibly with the intent to go back and replace this vector with a smaller dimensional vector encoding the location and magnitude of its $k$ significant coefficients.  In compressed sensing, one instead collects fewer fixed linear measurements of the data to start with, sufficient in number  to recover the location and numerical value of the $k$ nonzero coordinates at a later time.   Finding "good" linear measurements, as well as fast, accurate, and simple algorithms for recovering the original data from these measurements, are the twofold goals of Compressed Sensing research today. \\[\baselineskip]
\noindent In the sequel, we restrict our focus to signals that can be represented as real-valued vectors $x = (x_1, x_2, ... , x_N) \in \mathbb{R}^N$.  This signal model works well for digital images, which are naturally sparse: for example, a one-dimensional slice through the pixels in a typical grayscale image will contain segments of smoothly varying intensity, with sharp changes between adjacent pixels appearing only at edges in the image. Often this sparsity in information translates into a sparse (or approximately sparse) representation of the data with respect to some standard basis; for the image example, the basis would be a wavelet of curvelet basis.   For such $N$ dimensional data vectors that are well approximated by a $k$-sparse vector (or a vector that contains at most $k << N$ nonzero entries), it is common practice to temporarily store the entire vector, possibly with the intent to go back and replace this vector with a smaller dimensional vector encoding the location and magnitude of its $k$ significant coefficients.  In compressed sensing, one instead collects fewer fixed linear measurements of the data to start with, sufficient in number  to recover the location and numerical value of the $k$ nonzero coordinates at a later time.   Finding `good' linear measurements, as well as fast, accurate, and simple algorithms for recovering the original data from these measurements, are the twofold goals of compressed sensing research today.\\ \\
\noindent { \bf Review of basic CS setup}.  The data of interest is taken to be a real-valued vector $x \in  \mathbb{R}^N$ that is {\it unknown}, but from which we are allowed up to $m < N$ linear measurements, in the form of inner products of $x$ with $m$ vectors $v_j \in \mathbb{R}^N$ of our choosing.  Letting $\Phi$ denote the $m \times N$ matrix whose $j$th row is the vector $v_j$, this is equivalent to saying that we have the freedom to choose and store an $m \times N$ matrix $\Phi$, along with the $m$-dimensional measurement vector $y = \Phi x$.  Of course, since $\Phi$ maps vectors in $\mathbb{R}^N$ to vectors in a smaller dimensional space $\mathbb{R}^m$, the matrix $\Phi$ is not invertible, and we thus have no hope of being able to reconstruct an arbitrary $N$ dimensional vector $x$ from such measurements.\\ \\
\noindent However, if the otherwise unknown vector $x$ is specified to be $k$-sparse, and $k$ is fairly small compared with $N$, then there do exist matrices $\Phi$ for which $y = \Phi x$ uniquely determines $x$, and allows recovery of $x$ using fast and simple algorithms.  It was the interpretation of this phenomenon given by Candes and Tao $\cite{CandesTao}$, $\cite{CRT}$, and Donoho $\cite{DonohoElad}$, that gave rise to compressed sensing.  In particular, these authors define classes of matrices that possess this property.   One particularly elegant characterization of this class is via the {\it Restricted Isometry Property} (RIP) $\cite{CRT}$.  A matrix $\Phi$ with unit normed columns is said to be $k$-RIP if all singular values of any $k$ column submatrix of $\Phi$ lie in the interval $[1 - \delta, 1 + \delta]$ for a given constant $\delta < 1$.  
For a fixed value of $\delta > 0$,  an $m \times N$ matrix $\Phi$ whose entries $\Phi_{i,j}$ are independent realizations of a Gaussian or Bernoulli random variable satisfies $2k$-RIP of level
\begin{equation}
k = K(\delta, m,N) := c_1(\delta)m/ \log(N/m)), \textrm{where } m \leq \frac{1}{2}N,
\label{k}
\end{equation}
with probability $\geq 1 - 2\exp{-c_2(\delta) m}$ $\cite{bddw}$.   Also with high probability, an $m \times N$ matrix obtained by selecting $m$ rows at random from the $N \times N$ discrete Fourier matrix satisfies $2k$-RIP of the same order as $\eqref{k}$ up to an additional $\log{(\log N)}$ factor \cite{RV06}.  In fact, the order of $k$ given by $\eqref{k}$ is optimal given $m$ and $N$, as shown in $\cite{cdd09}$ using classical results on Gelfand widths of $l_1^N$ unit balls in $l_2^N$.   To date, there exist no deterministic constructions of RIP matrices of this order.\\ \\
\noindent { \bf Recovering or approximating x. } As shown in $ \cite{FL08}$, the following approximation results hold for matrices $\Phi$ that satisfy $2k$-RIP with constant $\delta_{2k} \leq 2(3-\sqrt{2})/7 \approx .4531$:
\begin{enumerate}
\item If $x \in \mathbb{R}^N$ is $k$-sparse,  then $x$ can be reconstructed from $\Phi$ and the measurement vector $y = \Phi x$ as the solution to the following $\ell_1$ minimization problem:
\begin{equation}
x = {\cal L}_1(\Phi, y) := \arg \min_{\Phi z = y} \| z \|_1.
\label{l1}
\end{equation} 
\item If $x$ is not $k$-sparse, the error between $x$ and the approximation $\hat{x} =  {\cal L}_1(\Phi, y)$ is still bounded by
\begin{equation}
\| x - \hat{x} \|_2 \leq \frac{c}{\sqrt{k}} \sigma_k(x)_{l_1^N},
\label{approxerr1}
\end{equation}
where $c = 2(1 + \delta_{2k})/(1 - \delta_{2k})$, and $\sigma_k(x)_{l_p^N} := \inf_{|z| \leq k} \|x - z\|_{l_p^N}$ denotes the best possible approximation error in the metric of $l^N_p$ between $x$ and the set of $k$-sparse signals in $\mathbb{R}^N$.  The approximation error $\sigma_k(x)_{l_p^N}$ is realized by the $k$-sparse vector $x_k \in \mathbb{R}^N$ that corresponds to the vector $x$ with all but the $k$ largest entries set to zero, independent of the $l^N_p$ norm in the approximation $\sigma_k(x)_{l_p^N}$. 
\end{enumerate} 

\noindent This immediately suggests to use the $\ell_1$-minimizer ${ \cal L}_{1}$ as a means to recover or approximate an unknown $x$ with sparsity constraint.  Several other decoding algorithms are used as alternatives to $\ell_1$ minimization for recovering a sparse vector $x$ from its image $y = \Phi x$, not because they offer better accuracy ( $\ell_1$ minimization gives optimal approximation bounds when $\Phi$ satisfies RIP), but because they can be faster and easier to implement.  For a comprehensive survey on compressed sensing decoding algorithms, we refer the reader to \cite{cosamp}. 

\section{Estimating the accuracy of compressed sensing estimates}

According to the bound $\eqref{approxerr1}$, the quality of a compressed sensing estimate $\hat{x} = \triangle(\Phi, y)$ depends on how well $x$ can be approximated by a $k$-sparse vector, where the value of $k$ is determined by the number of rows $m$ composing $\Phi$.   While $k$ is assumed to be known and fixed in the compressed sensing literature, no such bound is guaranteed for real-world signal models such as vectors $x \in \mathbb{R}^N$ corresponding to wavelet coefficient sequences of discrete photograph-like images.  Thus,  the quality of a compressed sensing estimate $\hat{x}$ in general is not guaranteed. \\ \\
If the error $\|x - \hat{x} \|_2/\|x\|_2$ incurred by a particular approximation $\hat{x}$ were \emph{observed} to be large, then decoding could be repeated using more measurements,  perhaps at increasing measurement levels $\{m_1, m_2, ... , m_p \}$, until the error $\|x - \hat{x}_j \|_2/\|x\|_2$ corresponding to $m_j$ measurements were observed to be sufficiently small.  Of course, the errors $\|x - \hat{x}_j \|_2$ and $\| x - \hat{x}_j \|_2 / \| x \|_2$ are typically not known, as $x$ is unknown.  Our main observation in this chapter is that one can apply the Johnson-Lindenstrauss lemma $\cite{JL}$ to the set of $p$ points,
\begin{equation}
\{ (x - \hat{x}_1), (x - \hat{x}_2), ... , (x - \hat{x}_p) \}.
\end{equation}
In particular, $r = O(\log p)$ measurements of $x$, provided by $y_{\Psi} = \Psi x$, when $\Psi$ is, e.g. a Gaussian or Bernoulli random matrix, are sufficient to guarantee that with high probability,
\begin{equation}
4/5 \| y_{\Psi} - \Psi \hat{x}_j \|_2 \leq  \|x - \hat{x}_j \|_2  \leq 4/3 \| y_{\Psi} - \Psi \hat{x}_j \|_2 \label{same}
\end{equation}
and
\begin{equation}
1/3 \frac{\| y_{\Psi} - \Psi \hat{x}_j \|_2}{ \|y_{\Psi}\|_2} \leq \frac{ \|x - \hat{x}_j \|_2}{ \|x\|_2}  \leq 2 \frac{\| y_{\Psi} - \Psi \hat{x}_j \|_2}{\|y_{\Psi}\|_2}
\label{same2}
\end{equation}
for any $p$ compressed sensing estimates.   The equivalences $\eqref{same}$ and  
$\eqref{same2}$ allow the {\it measurable} quantities $\|y_{\Psi} - \Psi \hat{x}_j \|_2$ and $\| y_{\Psi} - \Psi \hat{x}_j \|_2/ \|y_{\Psi}\|_2$ to function as proxies for the {\it unknown} quantities $\| x - \hat{x}_j \|_2$ and $\|x - \hat{x}_j \|_2 / \|x\|_2$; these proxies can be used to
\begin{enumerate}
\renewcommand{\labelenumi}{(\alph{enumi})}
\item provide tight numerical upper and lower bounds on the errors $ \| x - \hat{x}_j \|_2$ and $\|x - \hat{x}_j \|_2/\|x\|_2$ at up to $p$ compressed sensing estimates $\hat{x}_j$, 

\item provide estimates of the underlying $k$-term approximations $\| x - x_k \|_2$ of $x$ for up to $p$ different values of $k$, and

\item return from among a sequence of estimates $( \hat{x}_1, ... , \hat{x}_p )$ with different initialization parameters, an estimate $\hat{x}_{cv} = \arg \min_j \| y_{\Psi}- \Psi \hat{x}_j \|_2$ having error $\| x - \hat{x}_{cv}\|_2$ that does not exceed a small multiplicative factor of the best possible error in the metric of $\ell_2^N$ between $x$ and an element from the sequence at hand.  
\end{enumerate}
More precisely, all CS decoding algorithms require as input a parameter $m$ corresponding to the number of rows in $\Phi$; some compressed sensing decoding algorithms (such as greedy algorithms) require also a parameter $k$ indicating the sparsity level of $x$, and other algorithms require as input a bound $\gamma$ on the expected amount of energy in $x$ outside of its significant coefficients.  All CS decoding algorithms can be symbolically represented by functions of the form $\triangle(\Phi, y, k, \gamma)$,  and we will give examples where each of the parameters $m$, $k$, and $\gamma$ can be optimized over a sequence of estimates $(\hat{x}_1, \hat{x}_2, ..., \hat{x}_p)$ indexed by increasing hypotheses on each of the parameters $m$, $k$, and $\gamma$.  \\ \\
The method we propose for parameter selection and error estimation in compressed sensing is reminiscent of {\it cross validation}, which is a technique used in statistics and learning theory whereby a data set is separated into a training/estimation set and a test/cross validation set, and the test set is used to prevent overfitting on the training set by estimating underlying noise parameters. Indeed, we  take a set of $m$ measurements of the unknown $x$, and use $m - r$ of these measurements, $\Phi x$, in a compressed sensing decoding algorithm to return a sequence $(\hat{x}_1, \hat{x}_2, ... )$ of candidate approximations to $x$.  The remaining $r$ measurements, $\Psi x$, are then used to identify from among this sequence a `best' approximation $\hat{x} = \hat{x}_j$, along with an estimate of the sparsity level of $x$.   Although the application of cross validation in compressed sensing has been previously proposed by Boufounos, Duarte, and Baraniuk in $\cite{cv_bdb}$, the context in which it is studied there is different from that of the present paper (we will discuss this difference  further in the last section), and in their application one cannot immediately apply the mathematical justification of the Johnson Lindenstrauss lemma that we present below.

 \section{Preliminary notation}
 Throughout the paper, we will be dealing with large dimensional vectors that have few nonzero coefficients.  We use the notation $| x | = n$ to indicate that a vector $x \in \mathbb{R}$$^N$ has exactly $n$ nonzero coordinates.\\ \\
\noindent We will sometimes use the notation $a \sim_{\epsilon} b$ as shorthand for the multiplicative relation
\begin{equation} 
(1 - \epsilon)a \leq b \leq (1 + \epsilon)a,
\end{equation}
that can be worded as ``the quantity $a$ approximates the quantity $b$ to within a multiplicative factor of $(1 \pm \epsilon)$".  Note that the relation $\sim_{\epsilon}$ is not symmetric.  Properties of the relation $a \sim_{\epsilon} b$ are listed below; we omit the proofs, which amount to a string of simple inequalities.

\begin{lemma} Fix $\epsilon \in (0,1)$.
\begin{enumerate}
\item  If $a, b \in \mathbb{R}^+$ satisfy $a \sim_{\epsilon} b$, then $b/\big[(1+\epsilon)(1-\epsilon)\big] \sim_{\epsilon} a$. 
\item If $a,b,c,d \in \mathbb{R}^+$ satisfy  $a \sim_{\epsilon} b$ and  $c \sim_{\epsilon} d$, then $a/c \sim_{\delta} b/d$ for parameter $\delta = 2\epsilon/1-\epsilon$.
\item If $( a_1, a_2, ... , a_p )$ and $( b_1, b_2, ... , b_p )$ are sequences in $\mathbb{R}^{+}$, and $a_j \sim_{\epsilon} b_j$ for each $1 \leq j \leq p$,  then $\min_j a_j \sim_{\epsilon} \min_j b_j$.
\end{enumerate}
\label{babylemma}
\end{lemma}

\section{Mathematical foundations}
 
The Johnson Lindenstrauss (JL) lemma, in its original form, states that any set of $p$ points in high dimensional Euclidean space can be embedded into $\epsilon^{-2} \log(p)$ dimensions, without distorting the distance between any two points by more than a factor of $(1 \pm \epsilon)$ $\cite{JL}$.   In the same paper, it was shown that a random orthogonal projection would provide such an embedding with positive probability.  Following several simplifications to the original proof $\cite{FM}$, $\cite{Indyk98}$, $\cite{Dasgupta}$, it is now understood that Gaussian random matrices, among other purely random matrix constructions, can substitute for the random projection in the original proof of Johnson and Lindenstrauss.  Of the several versions of the lemma now appearing in the literature, the following variant presented in Matousek $\cite{mat08}$ is most applicable to the current presentation.

\begin{lemma}[Johnson-Lindenstrauss Lemma]
Fix an accuracy parameter $\epsilon \in (0, 1/2]$, a confidence parameter $\delta \in (0,1)$, and an integer $r \geq r_0 = C\epsilon^{-2}\log{\frac{1}{2\delta}}$. \\ \\
Let ${ \cal M}$ be a random $r \times N$ matrix whose entries ${\cal M}_{i,j}$ are independent realizations of a random variable R that satisfies:
\begin{enumerate}
\item { \it Var} $\big[ R \big] = 1/r$  (so that the columns of ${\cal M}$ have expected $\ell_2$ norm 1)
\item $\mathbb{E}\big[ R \big] = 0$,  
\item For some fixed $a > 0$ and for all $\lambda$,
\begin{equation}
\textrm{Prob} \big[ |R| > \lambda \big] \leq 2e^{-a \lambda^2}
\label{a} 
\end{equation}
\end{enumerate}
Then for a predetermined $x \in \mathbb{R}^N$, 
\begin{equation}
(1 - \epsilon)\| x\|_{l_2^N} \leq \| {\cal M} x \|_{l_2^r} \leq (1 + \epsilon)\| x\|_{l_2^N}
\label{jlbound} 
\end{equation}
is satisfied with probability exceeding $1 - \delta$.
\label{psi}
\end{lemma}
\noindent The constant $C$ bounding $r_0$ in Lemma $\ref{psi}$ grows with the parameter $a$ specific to the construction of ${\cal M}$ $\eqref{a}$.   Gaussian and Bernoulli random variables $R$ will satisfy the concentration inequality $\eqref{a}$ for a relatively small parameter $a$ (as can be verified directly), and for these matrices one can take $C = 8$ in Lemma $\ref{psi}$. \\ \\
\noindent The Johnson Lindenstrauss lemma can be made intuitive with a few observations.  Since $\mathbb{E}\big[R\big] =0$ and $\textrm{Var}\big[R\big] = \frac{1}{r}$, the random variable $\| {\cal M} x \|_2^2$ equals $\|x\|_2^2$ in expected value; that is,
\begin{equation}
\mathbb{E} \big[ \textrm{ } \| {\cal M} x \|^2_2\textrm{ }  \big] = \| x \|_2^2.
\end{equation} 
Additionally, $\| {\cal M} x \|^2_2$ inherits from the random variable $R$ a nice concentration inequality:
\begin{eqnarray}
\textrm{Prob} \big[ \| {\cal M} x \|_2^2 - \|x\|_2^2 > \epsilon \|x\|_2^2 \big] &\leq& e^{-a(2\epsilon\sqrt{r})^2} 
\leq \delta/2.
\label{bound}
\end{eqnarray} 
The first inequality above is at the heart of the JL lemma; its proof can be found in $\cite{mat08}$.  The second inequality follows using that $r \geq  (2a\epsilon^2)^{-1}\log(\frac{\delta}{2})$ and $\epsilon \leq 1/2$ by construction.  A bound similar to $\eqref{bound}$ holds for $\textrm{Prob} \big[ \| {\cal M} x \|_2^2 - \|x\|_2^2 < - \epsilon \|x\|_2^2 \big]$ as well, and combining these two bounds gives desired result $\eqref{jlbound}$.\\ \\
\noindent For fixed $x \in \mathbb{R}^N$, a random matrix ${\cal M}$ constructed according to Lemma $\ref{psi}$ fails to satisfy the concentration bound $\eqref{jlbound}$ with probability at most $\delta$.  Applying Boole's inequality, ${\cal M}$ then fails to satisfy the stated concentration on any of $p$ predetermined points $\{x_j\}_{j=1}^p$, $x_j \in \mathbb{R}^N$, with probability at most $\xi = p\delta$.   In fact, a specific value of $\xi \in (0,1)$ may be imposed for fixed $p$ by setting $\delta = \xi/p$.  These observations are summarized in the following corollary to Lemma $\ref{psi}$.  

\begin{corollary}
Fix an accuracy parameter $\epsilon \in (0,1/2]$, a confidence parameter $\xi \in (0,1)$, and fix a set of $p$ points $\{x_j\}_{j=1}^p \subset \mathbb{R}^N$.  Set $\delta = \xi/p$, and fix an integer $r \geq r_0 = C\epsilon^{-2}\log{\frac{1}{2\delta}} = C\epsilon^{-2}\log{\frac{p}{2\xi}}$.  If ${\cal M}$ is a $r \times N$ matrix constructed according to Lemma $\ref{psi}$, then with probability $\geq 1 - \xi$, the bound
\begin{equation}
(1 - \epsilon) \| x_j \|_{l_2^N} \leq \| {\cal M} x_j \|_{l_2^r} \leq (1 + \epsilon) \| x_j \|_{l_2^N}
\label{jlplus}
 \end{equation}
obtains for each $j=1,2, ... ,p$.
\label{myjl}
\end{corollary}

\section{Cross validation in compressed sensing}
We return to the situation where we would like to approximate a vector $x \in \mathbb{R}^N$ with an assumed sparsity constraint using $m < N$ linear measurements $y = {\cal A} x$ and  $m \times N$ matrix ${\cal A}$ of our choosing.  Continuing the discussion in Section 1, we will not reconstruct $x$ in the standard way by $\hat{x} = \Delta({\cal A}, y, k, \gamma)$ for fixed values of the input parameters, but instead separate the $m \times N$ matrix ${\cal A}$ into an $n \times N$ {\it implementation} matrix $\Phi$ and an $r \times N$ {\it cross validation} matrix $\Psi$, and separate the measurements $y$ accordingly into $y_{\Phi}$ and $y_{\Psi}$.  We use the implementation matrix $\Phi$ and corresponding measurements $y_{\Phi}$ as input into the decoding algorithm to obtain a sequence of possible estimates $(\hat{x}_1, ..., \hat{x}_p)$ corresponding to increasing one of the input parameters $m$, $k$, or $\gamma$.  We reserve the cross validation matrix $ \Psi$ and measurements $y_{\Psi}$ to estimate each of the error terms $\| x - \hat{x}_j \|_2$ in terms of the computable $\|y_{\Psi} - \Psi \hat{x}_j\|_2$. 
We will also estimate the unknown \emph{oracle error},
\begin{equation}
\eta_{or} = \min_{1\leq j \leq p}  \|x - \hat{x}_j \|_{l_2^N},
\end{equation}
corresponding to the best possible approximation to $x$ in the metric of $l_2^N$ from the sequence $(\hat{x}_1, ..., \hat{x}_p)$, using the computable \emph{cross validation error},
\begin{equation}
\widehat{\eta_{cv}} = \min_{1 \leq j \leq p} \| \Psi (x - \hat{x}_j) \|_{l_2^r}.
\end{equation}
 Our main result, which follows from Corollary $\eqref{myjl}$, details how the number of cross validation measurements $r$ should be chosen in terms of the desired accuracy $\epsilon$ of estimation, confidence level $\xi$ in the prediction, and number $p$ of estimates $\hat{x}_j$ to be measured.    
\begin{theorem}
For a given accuracy $\epsilon \in (0,1/2]$, confidence $\xi \in (0,1)$, and number $p$ of estimates $\hat{x}_j \in \mathbb{R}^N$, it suffices to allocate $r  = \lceil C\epsilon^{-2}\log{\frac{p}{2\xi}} \rceil$ rows to a cross validation matrix $\Psi$ of Gaussian or Bernoulli type, normalized according to Lemma \ref{myjl} and independent of the estimates $\hat{x}_j$, to obtain with probability greater than or equal to $1 - \xi$, and for each $j = 1,2, ..., p$, the bounds
\\
\\
\begin{equation}
 \frac{1}{1+\epsilon} \leq \frac{\|x - \hat{x}_j \|_{l_2^N}}{\| \Psi (x - \hat{x}_j) \|_{l_2^{\ell}}}\leq \frac{1}{1-\epsilon}
 \label{1a}
\end{equation}
and 
\begin{eqnarray}
 \frac{1-3\epsilon}{(1+\epsilon)(1-\epsilon)^2} &\leq& \frac{\|x - \hat{x}_j \|_{l_2^N}/\|x\|_{l_2^N}}{\| \Psi (x - \hat{x}_j) \|_{l_2^{\ell}}/\|\Psi x \|_{l_2^{\ell}}} \leq \frac{1}{(1-\epsilon)^2} \nonumber \\ \nonumber \\
\label{2a}
\end{eqnarray}
and also
\begin{equation}
 \frac{1}{1+\epsilon} \leq \frac{\eta_{or}}{\widehat{\eta_{cv}}} \leq \frac{1}{1-\epsilon}.
\label{1b}
\end{equation}
\label{mainthm}
\end{theorem} 

\begin{proof}.
\begin{itemize}
\item The bounds in \eqref{1a} are obtained by application of Lemma $\ref{myjl}$ to the $p$ points $u_j = x - \hat{x}_j$, and rearranging the resulting bounds according to Lemma $\ref{babylemma}$ part $(1)$.   The bound \eqref{1b} follows from the bounds \eqref{1a} and part (3) of Lemma $\ref{babylemma}$.
\item  The bounds in \eqref{2a} are obtained by application of Lemma $\ref{myjl}$ to the $p+1$ points $u_0 = x, u_j = x - \hat{x}_j$, and regrouping the resulting bounds according to part (2) of Lemma $\ref{babylemma}$.
\end{itemize}
\end{proof}
\begin{remark}
\emph{The measurements making up the cross validation matrix $\Psi$ must be independent of the measurements comprising the rows of the implementation matrix $\Phi$.  This comes from the requirement in Lemma $\ref{psi}$ that the matrix $\Psi$ be independent of the points $u_j = x - \hat{x}_j$.  This requirement is crucial, as observed  when $\hat{x}$ solves the $\ell_1$ minimization problem 
\begin{equation}
\hat{x} = \arg \min_{z \in \mathbb{R}^N} \| z \|_1 \textrm{ subject to } \Phi z = \Phi x,
\end{equation}
in which case the constraint $\Phi(\hat{x} - x) = 0$ clearly precludes the rows of $\Phi$ from giving any information about the error $\| \hat{x} - x \|_2$. }
\end{remark}

\begin{remark}
\emph{ If the full compressed sensing matrix $\cal{A} =$[$\Phi$ ; $\Psi $] is itself of Gaussian or Bernoulli type, then one can obtain a more accurate approximation to the unknown quantities $ \| x - \hat{x}_j \| / \| x \| $ by using all of the measurements ${\cal A}x$ in the estimates $\| \Psi (x - \hat{x}_j) \|_{l_2^{\ell}}/\|{\cal A} x \|_{l_2^{m}}$.}
\end{remark}

\begin{remark}
\emph{Proposition \ref{mainthm} should be applied with a different level of care depending on what information about the sequence $(x - x_1, x - x_2, ..., x - x_p)$ is sought.  If the minimizer $\hat{x} = \arg \min_{1 \leq j \leq p} \| \Psi (x - \hat{x}_j) \|_{l_2^r}$ is sufficient for one's purposes, then the precise normalization of $\Psi$ in Proposition \ref{mainthm} is not important.  The normalization doesn't matter either for estimating the normalized quantities $\|x - \hat{x}_j\|_2/\|x\|_2$.  On the other hand, if one is using cross validation to obtain estimates for the quantities $\|x - \hat{x}_j\|_2$, then normalization is absolutely crucial, and one must observe the normalization factor given by Lemma $\ref{myjl}$ that depends on the number of rows $r$ allocated to the cross validation matrix $\Psi$.}
\end{remark}

\section{Applications}

\subsection{Estimation of the best $k$-term approximation error}
We have already seen that if the $m \times N$ matrix $\Phi$ satisfies $2k$-RIP with parameter $\delta \leq \approx .45$, and $\hat{x} = {\cal L}_1(\Phi, \Phi x)$ is returned as the solution to the $\ell_1$ minimization problem $\eqref{l1}$, then the error between $x$ and the approximation $\hat{x}$ is bounded by
\begin{equation}
\| x - \hat{x} \|_2 \leq \frac{c}{\sqrt{k}} \sigma_k(x)_{l_1^N}.
\label{approxerr}
\end{equation}
Several other decoding algorithms in addition to $\ell_1$ minimization enjoy the reconstruction guarantee $\eqref{approxerr}$ under similar bounds on $\Phi$, such as the Iteratively Reweighted Least Squares algorithm (IRLS) $\cite{irls}$, and the greedy algorithms CoSAMP $\cite{cosamp}$ and Subspace Pursuit $\cite{subspace}$.  It has recently been shown \cite{InstOpt} \cite{InstOpt2} that if the bound \eqref{approxerr} is obtained, and if $x - \hat{x}$ lies in the null space of $\Phi$ (as is the case for the decoding algorithms just mentioned), then if $\Phi$ is a Gaussian or a Bernoulli random matrix, the error $\|x - \hat{x} \|_2$ also satisfies a bound, with high probability on $\Phi$, with respect to the $\ell_2^N$ residual, namely
\begin{equation}
\| x - \hat{x} \|_2 \leq c\sigma_k(x)_{l_2^N},
\label{w2}
\end{equation}
for a reasonable constant $c$ depending on the RIP constant $\delta_{2k}$ of $\Phi$.  In the event that $\eqref{w2}$ is obtained, a cross validation estimate $\| \Psi (x - \hat{x}) \|_{l_2^{r}}$ can be used to lower bound the residual $\sigma_k(x)_{l_2^N}$, with high probability, according to
\begin{eqnarray}
(1 - \epsilon) \| \Psi (x - \hat{x}) \|_{l_2^{\ell}} \leq \|x - \hat{x} \|_{l_2^N}  \leq c \sigma_k(x)_{l_2^N},
\end{eqnarray}
with $O(\frac{1}{\epsilon^2})$ rows reserved for the matrix $\Psi$ $\eqref{mainthm}$.  At this point, we will use Corollary 3.2 of $\cite{Romp07}$, where it is proved that if the bound $\eqref{approxerr}$ holds for $\hat{x}$ with constant $c$, then the same bound will hold for
\begin{equation}
\hat{x}_k = \arg \min_{z: |z| \leq k} \| \hat{x} - z \|_{l_2^N},
\end{equation}
the best $k$-sparse approximation to $\hat{x}$, with constant $\tilde{c} = 3c$.  Thus, we may assume without loss of generality that $\hat{x}$ is $k$-sparse,  in which case $\| \Psi (x - \hat{x}) \|_{l_2^{r}}$ also provides an upper bound on the residual $\sigma_k(x)_{l_2^N}$ by
\begin{equation}
(1 + \epsilon) \| \Psi (x - \hat{x}) \|_{l_2^{r}}  \geq  \|x - \hat{x} \|_{l_2^N} \geq \sigma_k(x)_{l_2^N}.
\end{equation} 
With almost no effort then, cross validation can be incorporated into many decoding algorithms to obtain tight upper and lower bounds on the unknown $k$-sparse approximation error $\sigma_k(x)_{l_2^N}$ of $x$.  More generally, the allocation of $10\log{p}$ measurements to the cross validation matrix $\Psi$ is sufficient to estimate the errors $(\|x - x_{k_j}\|_2)_{j=1}^p$ or the normalized approximation errors $(\|x - x_{k_j}\|_2/\|x\|_2)^p_{j=1}$ at $p$ sparsity levels $k_j$ by decoding $p$ times, adding $m_j$ measurements to the implementation matrix $\Phi_j$ at each repetition.  Recall that the quantities $k_j$ and $m_j$ are related according to $\eqref{k}$. 

\subsection{Choice of the number of measurements $m$}
Photograph-like images have wavelet or curvelet coefficient sequences $x \in \mathbb{R}^N$ that are \emph{compressible} $\cite{DeVore92}$ $\cite{curvelet}$, having entries that obey a power law decay
\begin{equation}
|x|_{(k)} \leq c_s k^{-s},
\label{power}
\end{equation}
where $x_{(k)}$ denotes the $k$th largest coefficient of $x$ in absolute value, the parameter $s > 1$ indicates the level of compressibility of the underlying image, and $c_s$ is a constant that depends only on $s$ and the normalization of $x$.  From the definition $\eqref{power}$, compressible signals are immediately seen to satisfy
\begin{equation}
\|x - x_k \|_1/\sqrt{k} \leq c_s' k^{-s + 1/2}, 
\label{compress}
\end{equation}
so that the solution $\hat{x}_m = {\cal L}_1(\Phi, \Phi x)$ to the $\ell_1$ minimization problem $\eqref{l1}$ using an $m \times N$ matrix $\Phi$ of optimal RIP order $k = cm/ \log(N/m))$ satisfies 
\begin{equation}
\| x - \hat{x}_{m} \|_2 \leq c_{s,\delta} k^{-s+1/2}.
\label{levelcomp}
\end{equation}
The number of measurements $m$ needed to obtain an estimate $\hat{x}_{m}$ satisfying $ \|x - \hat{x}_m \|_2 \leq \tau$ for a predetermined threshold $\tau$ will vary according to the compressibility of the image at hand.  Armed with a total of $m$ measurements, the following decoding method that \emph{adaptively} chooses the number of measurements for a given signal $x$ presents a more democratic alternative to standard compressed sensing decoding structure:  

\begin{table}[h]
\label{BDC}
{ \small
\begin{center}
\caption{ {\it \small CS decoding structure with adaptive number of measurements} } 
\end{center}
\begin{enumerate}
\item { \it Input}: The $m$-dimensional vector $y = \Phi x$, the $m \times N$ matrix $\Phi$, (in some algorithms) the sparsity level $k$, and  (again, in some algorithms) a bound $\gamma$ on the noise level of $x$, the number $p$ of row subsets of $\Phi$,  $(\Phi_1, \Phi_2, ..., \Phi_p)$, corresponding to increasing number of rows $m_1 < m_2 < ... < m_p < m$, and threshold $\tau > 0$.
\item { \it Initialize} the decoding algorithm at $j = 1$.
\item {\it Estimate} $\hat{x}_j = \triangle(\Phi_{m_j}, y_{m_j}, k, \gamma)$ with the decoder $\triangle$ at hand, using only the first $m_j$ measurement rows of $\Phi$.  The previous estimate $\hat{x}_{j-1}$ can be used for ``warm initialization" of the algorithm, if applicable.  The remaining $r_j = m - m_j$ measurement rows are allocated to a cross validation matrix $\Psi_j$ that is used to estimate the resulting error $\| x - \hat{x}_j \|_2/\|x\|_2$. 
\item { \it Increment}  $j$ by 1, and iterate from step 3 if stopping rule 
is not satisfied.
\item { \it Stop}: at index $j = j^* < p $ if $\|x - x_{m_j}\|_2/\| x \|_2 \leq \tau$ holds with near certainty, as indicated by
\begin{equation}
\frac{\sqrt{r_j}\|\Psi (x - x_{m_j}) \|_2/ \| \Phi x\|_2}{\sqrt{r_j} - 3\log{p}} \leq \tau
\label{r_j}
\end{equation}
according to Proposition \ref{mainthm}.  If the maximal number of decoding measurements $m_p < m$ have been used at iteration $p$, and $\eqref{r_j}$ indicates that $\| x - \hat{x}_{m_p} \|_2/\|x\|_2 > \tau$ still, return $\hat{x}_m = \triangle(\Phi, y, k, \gamma)$ using all $m$ measurements, but with a warning that the underlying image $x$ is probably too dense, and its reconstruction is not trustworthy.     
\label{CS decode}
\end{enumerate}  
}
\end{table}

\subsection{Choice of regularization parameter in homotopy-type  algorithms}
Certain compressed sensing decoding algorithms iterate through a sequence of intermediate estimates $\hat{x}_j$ that could be potential optimal solutions to $x$ under certain reconstruction parameter choices.  This is the case for greedy and homotopy-continuation based algorithms.  In this section, we study the application of cross validation to the intermediate estimates of decoding algorithms of homotopy-continuation type.  
\\
\\
LASSO is the name coined in $\cite{lasso}$ for the problem of minimizing of the following convex program:
\begin{equation}
\hat{x}^{[\tau]} = \arg \min_{\| z \|_1 \leq \tau} \| \Phi x - \Phi z \|_{\ell_{2}} 
\label{lasso}
\end{equation}
The two terms in the LASSO optimization problem $\eqref{lasso}$ enforce data fidelity and sparsity, respectively, as balanced by the regularization parameter $\tau$.  In general, choosing an appropriate value for $\tau$ in $\eqref{lasso}$ is a hard problem; when $\Phi$ is an underdetermined matrix, as is the case in compressed sensing, the function $f(\tau) = \| x - \hat{x}^{[\tau]} \|_2$ is unknown to the user but is seen empirically to have a minimum at a value of $\tau$ in the interval $[0, \| \Phi x \|_{\infty}]$ that depends on the unknown noise level and/or and compressibility level of $x$.      
\\
\\
The \emph{homotopy continuation} algorithm \cite{homotopy}, which can be viewed as the appropriate variant of LARS \cite{homotopy}, is one of many algorithms for solving the LASSO problem $\eqref{lasso}$ at a predetermined value of $\tau$; it proceeds by first initializing $\tau' $ to a value sufficiently large to ensure that the $\ell_1$ penalization term in $\eqref{lasso}$ completely dominates the minimization problem and $x^{[\tau']} = 0$ trivially.  The homotopy continuation algorithm goes on to generate $x^{[\tau']}$ for decreasing $\tau'$ until the desired level for $\tau$ is reached.   If $\tau = 0$, then the homotopy method traces through the entire solution path $x^{[\tau']} \in \mathbb{R}^N$ for $\tau' \geq 0$ before reaching the final algorithm output $x^{[0]} = {\cal L}_1(\Phi, y)$ corresponding to the $\ell_1$ minimizer $\eqref{l1}$.   
\\
\\
From the non-smooth optimality conditions for the convex functional $\eqref{lasso}$, it can be shown that the solution path $\hat{x}^{[\tau]} \in \mathbb{R}^N$ is a piecewise-affine function of $\tau$ \cite{homotopy}, with ``kinks" possible only at a finite number of points $\tau \in \{\tau_1, \tau_2, ... \}$.  Proposition \ref{mainthm} suggests a method whereby an appropriate value of $\tau*$ can be chosen from among a subsequence of the kinks $(\tau_1, \tau_2, ... , \tau_p)$ by solving the minimization problem $\hat{x}^{[\tau*]} = \arg \min_{j \leq p} \| \Psi (x - \hat{x}^{[\tau_j]}) \|_2$ for appropriate cross validation matrix $\Psi$.  Moreover, since the solution $x - \hat{x}^{[\tau]}$ for $\tau_j \leq \tau \leq \tau_{j+1}$ is restricted to lie in the two-dimensional subspace spanned by $x - \hat{x}^{[\tau_j]}$ and $x - \hat{x}^{[\tau_{j+1}]}$,  one can combine the Johnson Lindenstrauss Lemma with a covering argument analogous to that used to derive the RIP property for Gaussian and Bernoulli random matrices in $\cite{bddw}$, to cross validate the entire \emph{continuum} of solutions $\hat{x}^{[\tau]}$ between $\tau_1 \leq \tau \leq \tau_p$.  More precisely, the following bound holds under the conditions outlined in Theorem $\eqref{mainthm}$, with the exception that $2r$ (as opposed to $r$) measurements are reserved to $\Psi$:

\begin{equation}
 \frac{1}{1+\epsilon} \leq \frac{\min_{\tau_1 \leq \tau \leq \tau_p} \| x - \hat{x}^{[\tau]} \|_2}{\min_{\tau_1 \leq \tau \leq \tau_p} \|\Psi ( x - \hat{x}^{[\tau]})\|_2} \leq \frac{1}{1-\epsilon}
 \label{lassobound}
\end{equation}

Unfortunately, the bound $\eqref{lassobound}$ is not strong enough to \emph{provably} evaluate the entire solution path $\hat{x}^{[\tau]}$ for $\tau \geq 0$, because the best upper bound on the number of kinks on a generic  LASSO solution path can be very large.  One can prove that this number is bounded by $3^N$, by observing that if $\hat{x}^{[\tau_1]}$ and $\hat{x}^{[\tau_2]}$ have the same sign pattern, then $\hat{x}^{[\tau]}$ also has the same sign pattern for $\tau_1 \leq \tau \leq \tau_2$.  Applying Proposition \ref{mainthm} to $p = 3^N$ points $x - \hat{x}_j$, this suggests that $O(N)$ rows would need to be allocated to a cross validation matrix $\Psi$ in order for Proposition \ref{mainthm} and the corollary $\eqref{lassobound}$ to apply to the entire solution path, which clearly defeats the compressed sensing purpose.  However, whenever the matrix $\Phi$ is an $m \times N$ compressed sensing matrix of random Gaussian, Bernoulli, or partial Fourier construction, it is observed \emph{empirically} that the number of kinks along a homotopy solution path behaves like $O(m)$ for generic vectors $x \in \mathbb{R}^N$ used to generate the path, see Figure \ref{fig:homotopy}.  This suggests, at least heuristically, that the allocation of $O(\log{m})$ out of $m$ compressed sensing measurements of this type suffices to ensure that the error $\| x - \hat{x}^{[\tau]} \|_2$ for the solution $\hat{x}^{[\tau]} = \arg \min_{\tau \geq 0} \| \Psi (x - \hat{x}^{[\tau]}) \|_2$ will be within a small multiplicative factor of the best possible error in the metric of $\ell_2^N$ obtainable by any approximant $\hat{x}^{[\tau]}$ along the solution curve $\tau \geq 0$.   At the value of $\tau$ corresponding to $\hat{x}^{[\tau]}$, the LASSO solution $\eqref{lasso}$ can be computed using all $m$ measurements $\Phi$ as a final approximation to $x$.
\\
\\

\begin{figure*}
\begin{center}
\includegraphics[width=8cm]{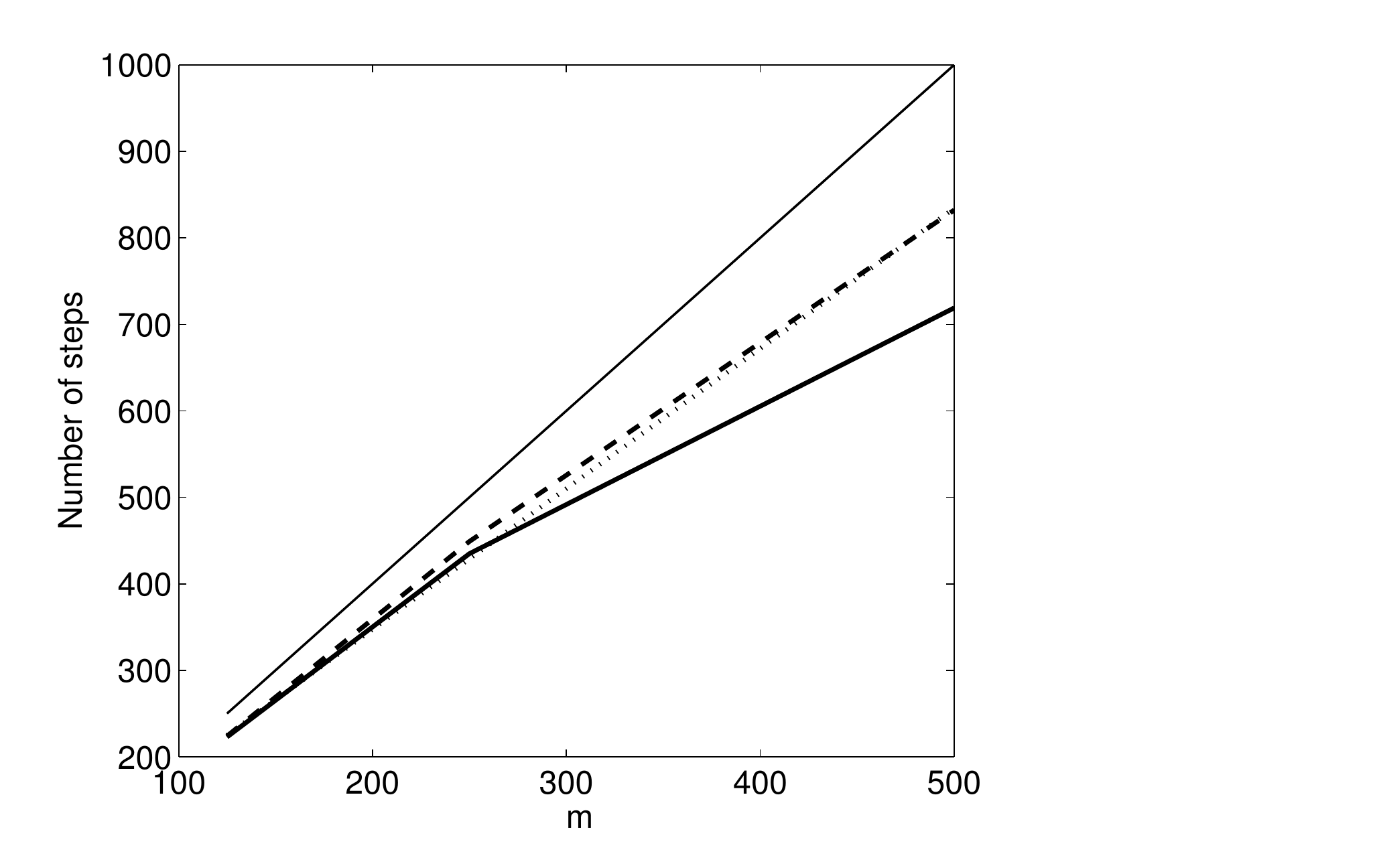}
\caption{Number of steps in the homotopy continuation algorithm, as described in Section $3.6.3$, as a function of the number of rows $m$ comprising the encoding matrix $\Phi$.  The thick solid, dashed, and dotted lines represent a bound satisfied by $95$ out of $100$ trials corresponding to independent realizations of an $m \times N$ Gaussian random matrix $\Phi$ and $N$-dimensional Gaussian random vector $x$, when $N = 500, 1000,$ and $2000$, respectively. The lines are generated from data at $m = 125, 250,$ and $500$. The number of steps appears to depend sublinearly on the number of rows $m$ and does not seem to depend on the ambient dimension $N$.  The thin solid line corresponds to the observed upper bound of $2m$.}
\label{fig:homotopy}
\end{center}
\end{figure*}

The \emph{Dantzig selector} (DS) \cite{dantzig1} refers to a minimization problem that is similar in form to the LASSO problem: 
\begin{equation}
\hat{x}_{\tau} = \arg \min_{z \in \mathbb{R}^N} \| \Phi x - \Phi z \|_{\ell_{\infty}} + \tau \| z \|_1
\label{DS} 
\end{equation}
The difference between the DS $\eqref{DS}$ and LASSO $\eqref{lasso}$ is the choice of norm ($\ell_{\infty}$ versus $\ell_2$) on the fidelity-promoting term.  Homotopy-continuation based algorithms have also been developed to solve the minimization problem $\eqref{DS}$ by tracing through the solution path $\hat{x}_{\tau'}$ for $\tau' \geq \tau$.  As the minimization problem $\eqref{DS}$ can be reformulated as a linear program, its solution path $\hat{x}_{\tau} \in \mathbb{R}^N$ is seen to be a piecewise \emph{constant} function of $\tau$, in contrast to the LASSO solution path.  In practice, the total number of breakpoints $(\tau_1, \tau_2, ... )$ in the domain $0 \leq \tau$ is observed to be on the same order of magnitude as $m$ when the $m \times N$ matrix $\Phi$ satisfies RIP \cite{DASSO09}; thus, the procedure just described to cross validate the LASSO solution path can be adapted to cross validate the solution path of $\eqref{DS}$ as well. 
\\
\\
Thus far we have not discussed the possibility of using cross validation as a stopping criterion for homotopy-type decoding algorithms.  Along the LARS homotopy curve $\eqref{lasso}$, most of the breakpoints $(\tau_1, \tau_2, ... )$ appear only near the end of the curve in a very small neighborhood of $\tau = 0$.  These breakpoints incur only miniscule changes in the error $\| x - \hat{x}_{\tau_j} \|_2$ even though they account for most of the computational expense of the LARS decoding algorithm.  Therefore, it would be interesting to adapt such algorithms, perhaps using cross validation, to stop once $\tau^*$ is reached for which the error $\|x - \hat{x}_{\tau^*} \|_2$ is sensed to be sufficiently small.

\subsection{Choice of sparsity parameter in greedy-type algorithms}
Greedy compressed sensing decoding algorithms also iterate through a sequence of intermediate estimates $\hat{x}_j$ that could be potential optimal solutions to $x$ under certain reconstruction parameter choices.  \emph{Orthogonal Matching Pursuit} (OMP), which can be viewed as the prototypical greedy algorithm in compressed sensing, picks columns from the implementation matrix $\Phi$ one at a time in a greedy fashion until, after $k$ iterations, the $k$-sparse vector $\hat{x}_k$, a linear combination of the $k$ columns of $\Phi$ chosen in the successive iteration steps, is returned as an approximation to $x$.  The OMP algorithm is listed in Table 2.  Although we will not describe the algorithm in full detail, a comprehensive study of OMP can be found in $\cite{TroGil}$.   \\ \\
\begin{table}[h]
{ \small 
\label{omp}
 \begin{center}
\caption{{\it \small Orthogonal Matching Pursuit Basic Structure}}
\end{center}
\begin{enumerate}
\item { \it Input}: The $m$-dimensional vector $y = {\cal B} x$, the $m \times N$ encoding matrix $\Phi$ whose $j^{th}$ column is labeled $\phi_j$, and the sparsity bound $k$.
\item { \it Initialize} the decoding algorithm at $j = 1$, the residual $r_0 = y$, and the index set $\Lambda_0 = \emptyset$.
\item {\it Estimate} 
\begin{enumerate}
\item Find an index $\lambda_j$ that realizes the bound $ (\Phi^T r_{j-1} )_{\lambda_j} = \| \Phi^T r_{j-1} \|_{\infty} $.
\item Update the index set $\Lambda_j = \Lambda_{j-1} \cup { \lambda_j}$ and the submatrix of contributing columns: $\Phi_j = [ \Phi_{j-1}\textrm{,  }  \phi_{\lambda_j} ]$
\item Update the residual:
\begin{eqnarray}
s_j &=& \arg \min_x \| \Phi_j x - y \|_2 = (\Phi_j^T\Phi_j)^{-1}\Phi_j^Ty, \nonumber \\
a_j &=& \Phi_j x_j \nonumber \\
r_j &=& r_{j-1} - a_j. \nonumber
\end{eqnarray}
\item The estimate $\hat{x}_j$ for the signal has nonzero indices at the components listed in $\Lambda_j$, and the value of the estimate $\hat{x}_j$ in component $\lambda_i$ equals the $ith$ component of $s_j$.
\end{enumerate} 
\item { \it Increment}  $j$ by 1 and iterate from step 3, if $j < k$.
\item { \it Stop}: at $j = k$.  Output $\hat{x}_{omp} = \hat{x}_k$ as approximation to $x$.
\end{enumerate}
}
\end{table}

\pagebreak

Note in particular that OMP requires as input a parameter $k$ corresponding to the expected sparsity level for $x \in \mathbb{R}^N$.  Such input is typical among greedy algorithms in compressive sensing (in particular, we refer the reader to $\cite{cosamp}$,  $\cite{irls}$, and  $\cite{subspace}$).  As shown in $\cite{TroGil}$, OMP will recover with high probability a vector $x$ having at most $k \leq m / \log{N}$ nonzero coordinates from its image $\Phi x$ if $\Phi$ is a (known) $m \times N$ Gaussian or Bernoulli matrix with high probability.   Over the more general class of vectors $x = x_d + {\cal N}$ that can be decomposed into a $d$-sparse vector $x_d$ (with $d$ presumably less than or equal to $k$) and additive noise vector ${\cal N}$, we might expect an intermediate estimate $\hat{x}_s$ to be a better estimate to $x$ than the final OMP output $\hat{x}_k$, at least when  $d << k$.  Assuming that the signal $x$ admits a decomposition of the form $x = x_d + {\cal N}$, the sequence of intermediate estimates $(\hat{x}_1, ... , \hat{x}_k)$ of an OMP algorithm can be cross validated in order to estimate the noise level and recover a better approximation to $x$.  We will study this particular application of cross validation in more detail below.

\section{Orthogonal matching pursuit: a case study}
As detailed in Table 2, a single index $\lambda_j$ is added to a set $\Lambda_j$ estimated as the $j$ most significant coefficients of $x$ at each iteration $j$ of OMP; following the selection of $\Lambda_j$, an estimate $\hat{x}_j$ to $x$ is determined by the least squares solution,
\begin{equation}
\hat{x}_j = \arg \min_{\textrm{supp}(z) \in \Lambda_j} \| \Phi z - y \|_2,
\end{equation}
among the subspace of vectors $z \in \mathbb{R}$$^N$ having nonzero coordinates in the index set $\Lambda_j$.  OMP continues as such, adding a single index $\lambda_j$ to the set $\Lambda_j$ at iteration $j$, until $j = k$ at which point the algorithm terminates and returns the $k$-sparse vector $\hat{x}_{omp} = \hat{x}_k$ as approximation to $x$.  
\\
\\
Suppose $x$ has only $d$ significant coordinates.  If $d$ could be specified beforehand, then the estimate $\hat{x}_d$ at iteration $j = d$ of OMP would be returned as an approximation to $x$.   However, the sparsity $d$ is not known in advance, and $k$ will instead be an upper bound on $d$.  As the estimate $\hat{x}_j$ in OMP can be then identified with the hypothesis that $x$ has $j$ significant coordinates, the application of cross-validation as described in the previous section applies in a very natural way to OMP.  In particular, we expect $\hat{x}_{or}$ and $\hat{x}_{cv}$ of Proposition \ref{mainthm} to be close to the estimate $\hat{x}_j$ at index $j = |x|$ corresponding to the true sparsity of $x$; furthermore, in the case that $|x|$ is significantly less than $k$, we expect the cross validation estimate $\hat{x}_{cv}$ to be a better approximation to $x$ than the OMP-returned estimate $\hat{x}_k$.  We will put this intuition to the test in the following numerical experiment. 

\subsection{Experimental setup}
We initialize a signal $x_0$ of length $N = 3600$ and sparsity level $d = 100$ as
\begin{equation}
 x_0(j) = 
 \left \{ \begin{array}{ll}
1,  &  \textrm{ for  } j = 1... 100 \\
0, & \textrm{ else}.
\end{array} \right.
\label{x0}
\end{equation}
\noindent Noise is then added to $x_a = x_0 + {\cal N}_a$ in the form of  a Gaussian random variable ${\cal N}_a$ distributed according to
\begin{equation}
{\cal N}_a \sim N(0, .05),
\end{equation}
and the resulting vector $x_a$ is renormalized to satisfy $\|x_a\|_{l_2^N} = 1$.  This yields an expected noise level of 
\begin{equation}
\mathbb{E} \big[ \sigma_d(x_a) \big] \approx .284.
\end{equation}
\\
We fix the input $k=200$ in Table 2, and assume we have a total number of compressed sensing measurements $m = 800$.  A number $r$ of these $m$ measurements are allotted to cross validation, while the remaining $n = m - r$ measurements are allocated as input to the OMP algorithm in Table 2. This experiment aims to numerically verify Proposition \ref{mainthm}; to this end, we specify a confidence $\xi = 1/100$, and solve for the accuracy $\epsilon$ according to the relation $r =  \epsilon^{-2} \log(\frac{k}{2\xi})$; that is, 

\begin{equation}
\epsilon(r) = \sqrt{\frac{\log(\frac{k}{2\xi})}{r}} \approx \frac{3}{\sqrt{r}}.
\label{rach}
\end{equation} 

\noindent Note that the specification $\eqref{rach}$ corresponds to setting the constant $C = 1$ in Proposition \ref{mainthm}.  Although $C \geq 8$ is needed for  the proof of the Johnson Lindenstrauss lemma at present, we find that in practice $C = 1$ already upper bounds the optimal constant needed for Proposition \ref{mainthm} for Gaussian and Bernoulli random ensembles. 
\\
\\
A single (properly normalized) Gaussian $n \times N$ measurement matrix $\Phi$ is generated (recall that $n$ = $m$ - $r$) , and this matrix and the measurements $y = \Phi x$ are provided as input to the OMP algorithm; the resulting sequence of estimates  $(\hat{x}_1, \hat{x}_2, ..., \hat{x}_k)$ is stored.  The final estimate $\hat{x}_k$ from this sequence is the returned OMP estimate $\hat{x}_{omp}$ to $x$.  The error $\eta_{omp} = \| \hat{x}_{omp} - x \|_2$  is greater than or equal to the oracle error of the sequence, $\eta_{or} = \min_{\hat{x}_j} \| x - \hat{x}_j \|_2$.
\\
\\
With the sequence $(\hat{x}_1, \hat{x}_2, ... , \hat{x}_k)$ at hand, we consider $1000$ realizations $\Psi_q$ of an $r \times N$ cross validation matrix having the same componentwise distribution as $\Phi$, but normalized to have variance $1/r$ according to Theorem $\ref{psi}$.  The cross validation error
 \begin{equation}
 \widehat{\eta_{cv}}(q) = \min_j \|\Psi_q (x - \hat{x}_j)\|_{l_2^r}
 \end{equation}
 is measured at each realization $\Psi_q$; we plot the average $\bar{\widehat{\eta_{cv}}}$ of these $1000$ values and intervals centered at $\bar{\widehat{\eta_{cv}}}$ having length equal to twice the empirical standard deviation.  Note that we are effectively testing $1000$ trials of OMP-CV, the algorithm which modifies OMP to incorporate cross validation so that $(\hat{x}_{cv}, \widehat{\eta_{cv}})$ are output instead of $\hat{x}_{omp} = \hat{x}_k$.
 \\
 \\
 At the specified value of $\xi$, Proposition \ref{mainthm} part $\eqref{2a}$ (with constant $C = 1$) implies that 
 \begin{equation}
 \big(1 -\epsilon)\eta_{or} \leq \widehat{\eta_{cv}}(q) \leq \big( 1 +\epsilon \big)\eta_{or}
 \label{theory}
 \end{equation}
 should obtain on at least $990$ of the $1000$ estimates $\widehat{\eta_{cv}}(q)$; in other words, at least $990$ of the $1000$ discrepancies $| \eta_{or} - \widehat{\eta_{cv}}(q)|$ should be bounded by
 \begin{equation}
 0 \leq | \widehat{\eta_{cv}}(q) - \eta_{or} | \leq \epsilon \eta_{or}.
 \label{theory+}
 \end{equation}
Using the relation $\eqref{rach}$ between $\epsilon$ and $r$, this bound becomes tighter as the number $r$ of CV measurements increases; however, at the same time, the oracle error $\eta_{or}$ increases with $r$ for fixed $m$ as fewer measurements $n = m - r$ are input to OMP.  An ideal number $r$ of CV measurements should not be too large or too small;  Figure 1 suggests that setting aside just enough measurements $r$ such that $\epsilon \leq.6$ is satisfied in $\eqref{rach}$ serves as a good heuristic to choose the number of cross validation measurements (in Figure 1, $\epsilon \leq .6$ is satisfied by taking only $r = 30$ measurements).  
 \\
 \\
We indicate the theoretical bound $\eqref{theory}$ with dark gray in Figure 1, which is compared to the interval in light gray of the $990$ values of $\eta_{cv}(q)$ that are closest to $\eta_{or}$ in actuality.   
 \\
 \\
This experiment is run for several values of $r$ within the interval $[5, 90]$; the results are plotted in Figure 1(a), with the particular range $r \in [5, 30]$ blown up in Figure 1(b). 
\\
\\
We have also carried out this experiment with a smaller noise variance; i.e. $x_b = x_0 + {\cal N}_b$ is subject to additive noise
\begin{equation}
{\cal N}_b \sim N(0, .02).
\end{equation}
The signal $x_b$ is again renormalized to satisfy $\|x_b\|_{l_2^N} = 1$; it now has an expected noise level of 
\begin{equation}
\mathbb{E} \big[ \sigma_d(x_b) \big] \approx .116.
\end{equation}
The results of this experiment are plotted in Figure 1(c).
\\
\\
\begin{figure}[htp]
{ \label{}
\includegraphics[width=6.5in]{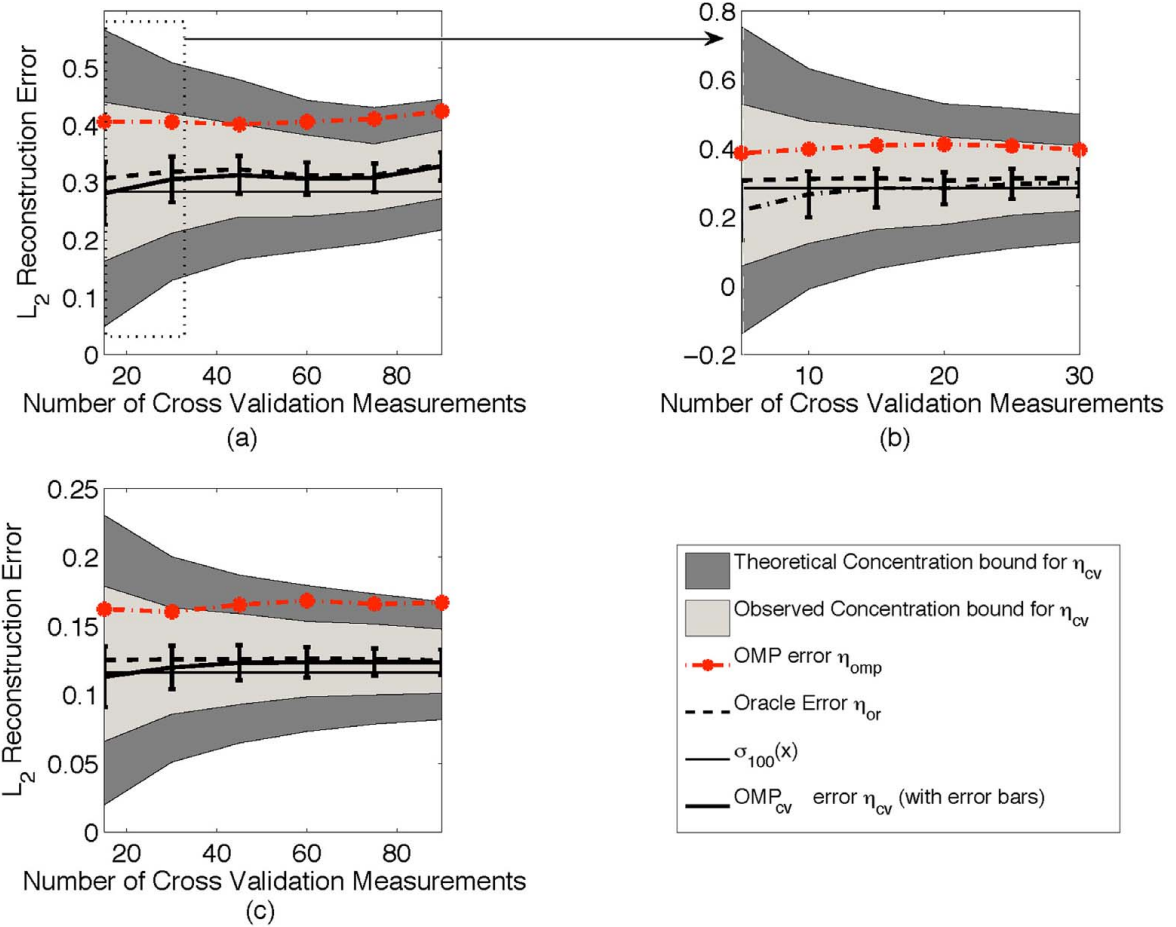}}

\caption{\small{Comparison of the reconstruction algorithms OMP and OMP-CV.  We fix the parameters $N = 3600, m = 800, k = 200$, and underlying sparsity $d = 100$, but vary the number $r$ of the total $m$ measurements reserved for cross validation from $5$ to $90$, using the remaining $n = m-r$ measurements for training.  The underlying signal has residual $\sigma_{100}(x) \approx .284$ in Figures 1(a) and 1(b), and $\sigma_{100}(x) \approx .116$ in Figure 1(c), as shown for reference by the thin horizontal line.  In both cases, the OMP-CV  error $\eta_{cv}$ (the solid black line with error bars; each point represents the average of 1000 trials) gives a better approximation to the residual error than does OMP (dot-dashed line) with very high probability, even when as few as $20$ of the total $800$ measurements are used for cross validation.  Even though $\eta_{cv}$ is guaranteed to provide a tighter bound for $\eta_{or}$ as the number $r$ of CV measurements increases, at the same time, the oracle error $\eta_{or}$ becomes a worse indicator of the residual $\sigma_{100}(x)$ because fewer measurements $n = m - r$ are input to OMP.  } }
\label{figure 1}
\end{figure}

\subsection{Experimental results}
\begin{enumerate}

\item We remind the reader that the cross-validation estimates $\widehat{\eta_{cv}}$ are observable to the user, while the values of $\eta_{omp}$, $\eta_{or}$, along with the noise level $\sigma_d(x)$, are not available to the user.  Nevertheless, $\widehat{\eta_{cv}}$ can serve as a proxy for $\eta_{or}$ according to $\eqref{theory}$, and this is verified by the plots in Figure 1.  $\widehat{\eta_{cv}}$ can also provide an upper bound on $\sigma_d(x)$, as is detailed in Section 5.1.

\item The theoretical bound $\eqref{theory}$ is seen to be tight, when compared with the observed concentration bounds in Figure 1.  

\item With high probability, the estimate $\hat{x}_{cv}(15)$ using $r = 15$ out of the alloted $m = 800$ measurements will be a better estimate of $x$ than the OMP estimate: $\| \hat{x}_{cv}(15) - x \|_2 \leq \| \hat{x}_{omp}(15) - x \|_2$.  With \emph{overwhelming probability}, the estimate $\hat{x}_{cv}(30)$ will result in error $\| \hat{x}_{cv}(30) - x \|_2 \leq \| \hat{x}_{omp}(30) - x\|_2$.  We note that the estimates $\hat{x}_{cv}(15)$ and $\hat{x}_{cv}(30)$ correspond to accuracy parameters $\epsilon(15) = .8405$ and $\epsilon(30) = .5943$ in \eqref{rach}, indicating that $\epsilon \leq.6$ is a good heuristic to determine when enough CV measurements have been reserved.

\item The OMP-CV estimate $\hat{x}_{cv}$ will have more pronounced improvement over the OMP estimate $\hat{x}_{omp}$ when there is larger discrepancy between the true sparsity $d$ of $x_0$ and the upper bound $k$ used by OMP (in Figure (1), $d = 100$ and $k = 200$).   In contrast, OMP-CV will not outperform OMP in approximation accuracy when $d$ is close to $k$; however, the multiplicative relation $\eqref{theory}$ guarantees that OMP-CV will not underperform OMP, either.   
\end{enumerate}

\section{Beyond compressed sensing}

The Compressed Sensing setup can be viewed within the more general class of {\it underdetermined linear inverse problems}, in which $x \in \mathbb{R}^N$ is to be reconstructed from a known $m \times N$ underdetermined matrix ${\cal A}$ and lower dimensional vector $y = {\cal A} x$ using a decoding algorithm $\Delta: \mathbb{R}^m \rightarrow \mathbb{R}^N$; in this broader context, ${\cal A}$ is given to the user, but not necessarily {\it specified by} the user as in compressed sensing.  In many cases, a prior assumption of sparsity is imposed on $x$, and an iterative decoding algorithm for solving the LASSO problem $\eqref{lasso}$ will be used to reconstruct $x$ from $y$ \cite{DDD}.   If it is possible to take on the order of $r = \log{p}$ additional measurements of $x$ by an $r \times N$ matrix $\Psi$ satisfying the conditions of Lemma $\ref{psi}$, then all of the analysis presented in this paper applies to this more general setting.  In particular, the error $\|x - \hat{x}_j \|_{l_2^N}$ at up to $j \leq p$ successive approximations $\hat{x}_j$ of the decoding algorithm $\Delta$ may be bounded from below and above using the quantities $\| \Psi (x - \hat{x}_j) \|_{\ell_2^r}$, and the final approximation $\hat{x}$ to $x$ can be chosen from among the entire sequence of estimates $\hat{x}_j$ as outlined in Proposition \ref{mainthm};  an earlier estimate $\hat{x}_j$ may approximate $x$ better than a final estimate $\hat{x}_p$ which contains the artifacts of parameter overfitting occurring at later stages of iteration. 

\section{Extensions}
We have presented an alternative approach to compressed sensing in which a certain number $r$ of the $m$ allowed measurements of a signal $x \in \mathbb{R}$$^N$ are reserved to track the error in decoding by the remaining $m - r$ measurements, allowing us to choose a best approximation to $x$ in the metric of $\ell^N_2$ out of a sequence of $p$ estimates $(\hat{x}_j)_{j=1}^p$, and estimate the error between $x$ and its best approximation by a $k$-sparse vector, again with respect to the metric of $\ell^N_2$.   We detailed how the number $r$ of such measurements should be chosen in terms of desired accuracy $\epsilon$ of estimation, confidence level $\xi$ in the prediction, and number $p$ of decoding iterations to be measured; in general, $r = O(\log(p))$ measurements suffice.   Several important issues remain unresolved; we mention only a few below.

\begin{enumerate}

\item {\bf Noisy measurements.}
 Consider the noisy measurement model,
\begin{equation}
y = {\cal A}x + {\cal N},
\label{noise}
\end{equation}
where $x \in \mathbb{R}^N$, $y \in \mathbb{R}^M$, and ${\cal N} \sim N(0,\sigma^2)$ is a Gaussian random variable that accounts for both noise and quantization error on the measurements ${\cal A} x$.  Because measurement noise and quantization error are unavoidable in any real-world sensing device, any proposed compressed sensing technique should extend to the model $\eqref{noise}$.   Cross validation is studied in $\cite{cv_bdb}$ in the context of noisy measurements, assuming that $x$ is truly sparse and that ${\cal N} \sim N(0,\sigma^2)$ is Gaussian noise.   The experimental results in $\cite{cv_bdb}$ indicate that cross validation works well in the presence of noise, and Proposition \ref{mainthm} provides intuition as to why.  Suppose that $x \in \mathbb{R}^N$ is $k$-sparse, so that the best possible reconstruction to $x$ using $n$ noisy measurements $y_1 = \Phi x + {\cal N}_1$ of the $m$ total noisy measurements $y = {\cal A}x + {\cal N}$ is bounded by the $\ell_2$ norm of ${\cal N}_1$: 
\begin{equation}
\| x  - \hat{x} \|_{\ell_2^N} \sim \mathbb{E}\big[ \| {\cal N}_1 \|_{\ell_2^n} \big]  = \sigma \sqrt{n}.
\end{equation}
Alternatively, the noise vector ${\cal N}_{cv}$ corresponding to the remaining $r = m - n$ cross validation measurements has \emph{exponentially} smaller expected $\ell_2$ norm $\mathbb{E}[ \| {\cal N}_{cv} \|_{\ell_2^r} ] = \sigma \sqrt{10\log{n}}$, assuming that $r = 10\log{n}$.  It follows by application of the triangle inequality,
\begin{eqnarray}
 (1 - \epsilon) \| x - \hat{x} \|_{\ell_2^r} - \sigma \sqrt{\log{n}} 
&\leq& \| \Psi(x  - \hat{x}) + {\cal N}_{cv} \|_{\ell_2^r}  \nonumber \\
&\leq& (1 + \epsilon) \| x - \hat{x} \|_{\ell_2^r} + \sigma \sqrt{\log{n}}, \nonumber 
\end{eqnarray}
whence the approximation error $\| x - \hat{x} \|_{\ell_2^r} \sim \sigma \sqrt{n}$ dominates both lower and upper bounds on the noisy cross validation error.  

\item {\bf Other cross validation techniques.} The cross validation technique promoted in this paper corresponds in particular to the technique of \emph{holdout cross validation} in statistics, where a data set is partitioned into a single training and cross validation set (as a rule of thumb, the cross validation set is usually taken to be less than or equal to a third of the size of the training set; in the the current paper, we have shown that the Johnson Lindenstrauss lemma provides a theoretical justification of how many, or, more precisely, how few, cross validation measurements are needed in the context of compressed sensing).  Other forms of cross validation, such as \emph{repeated random subsampling} cross validation or \emph{K-fold} cross validation, remain to be analyzed in the context of compressed sensing.  The former technique corresponds to repeated application of holdout cross validation, with  $r$ cross validation measurements out of the total $m$ measurements chosen by random selection at each application.  The results are then averaged (or otherwise combined) to produce a single estimation.   The latter technique, $K$-fold cross validation, also corresponds to repeated application of holdout cross validation.  In this case, the $m$ measurements are partitioned into $K$ subsets of equal size $r$, and cross-validation is repeated exactly $K$ times with each of the $K$ subsets of measurements used exactly once as the validation set. The $K$ results are again combined to produce a single estimation.   Although Proposition $\ref{mainthm}$ does not directly apply to these cross validation models, the experimental results of Section 6 suggest that, equiped with an $m \times N$ matrix satisfying the requirements of Lemma $\ref{psi}$, the application of $K$ fold cross validation to subsets of the measurements of size $r << m - r$ just large enough that $\epsilon > 0$ in Proposition \ref{mainthm} for fixed accuracy $\xi$ and constant $C = 1$ can be combined to accurately approximate the underlying signal with near certainty.  

\item {\bf Cross validation in different reconstruction metrics.} We have only considered cross validation over the metric of $\ell_2$.  However, the error $\|x - \hat{x}\|_{\ell_2^N}$, or \emph{root mean squared error}, is just one of several metrics used in image processing for analyzing the quality of a reconstruction $\hat{x} \in \mathbb{R}^N$ to a (known) image $x \in \mathbb{R}^N$.  In fact, the $\ell_1$ reconstruction error $\|x - \hat{x} \|_{\ell_1^N}$ has been argued to outperform the root mean squared error as an indicator of reconstruction quality $\cite{DeVore92}$.  Unfortunately, Proposition \ref{mainthm} cannot be extended to the metric of $\ell_1$, as there exists no $\ell_1$ analog of the Johnson Lindenstrauss Lemma \cite{Charikar}.  However, it remains to understand the extent to which cross validation in compressed sensing can be applied over a broader class of image reconstruction metrics, perhaps using more refined techniques than those considered in this paper.

\item {\bf Cross validation with more general CS matrix ensembles.} Many more compressed sensing matrices in addition to those satisfying the requirements of Lemma $\ref{psi}$ can be used for cross validation purposes, in the sense that a concentration bound similar to \eqref{psi} will hold for \emph{most} input $x$. Indeed, it has been recently shown \cite{fastJL} that such random partial Fourier matrices, if multiplied on the right by an $N \times N$ diagonal matrix of independent and identically distributed Bernoulli random variables, will satisfy the concentration bound $\eqref{psi}$ with slightly larger number of measurements $r \geq \epsilon^{-2}\log^3{\frac{1}{2\delta}}$.   

\end{enumerate}

\chapter{Free discontinuity problems meet iterative thresholding}
\section{Introduction}
Free-discontinuity problems describe situations where the solution of interest is defined by a function and a lower dimensional set consisting of the discontinuities of the function.  Hence, the derivative of the solution is assumed to be a `small' function almost everywhere except on sets where it concentrates as a singular measure. This is the case, for instance, in crack detection from fracture mechanics or in certain digital image segmentation problems. If we discretize such situations for numerical purposes, the free-discontinuity problem in the discrete setting can be re-formulated as that of finding a derivative vector with small components at all but a few entries that exceed a certain threshold.  This problem is similar to those encountered in Compressed sensing, where vectors with a small number of dominating components in absolute value are recovered from a few given linear measurements via the minimization of related energy functionals. Several iterative thresholding algorithms that intertwine gradient-type iterations with thresholding steps have been designed to recover sparse solutions in this setting.  It is natural to wonder if and/or how such algorithms can be used towards solving discrete free-discontinuity problems.  The current chapter explores this connection, and, by establishing an iterative thresholding algorithm for discrete free-discontinuity problems, provides new insights on properties of minimizing solutions thereof.

\subsection{Free-discontinuity problems: the Mumford-Shah functional}
The terminology `free-discontinuity problem' was introduced by De Giorgi \cite{dg91}
to indicate a class of variational problems that consist in the minimization of a functional, involving both volume and surface energies, depending on a closed set $K \subset \mathbb R^d$, and a function $u$ on $\mathbb R^d$ usually smooth outside of $K$. In particular, 
\begin{itemize}
\item $K$ is not fixed a priori and is an unknown of the problem;
\item $K$ is not a boundary in general, but a free-surface inside the domain of the problem.
\end{itemize}

\noindent The best-known example of a free-discontinuity problem is the one modelled by the so-called Mumford-Shah functional \cite{MS89}, which is defined by
$$
J(u,K):= \int_{\Omega \setminus K} \left [ | \nabla u |^2 + \alpha ( u - g )^2 \right ] dx +\beta \mathcal{H}^{d-1}(K \cap \Omega).
$$
The set $\Omega$ is a bounded open subset of $\mathbb{R}^d$, $\alpha, \beta >0$ are fixed constants, and $g \in L^\infty(\Omega)$. Here $\mathcal H^N$ denotes the $N$-dimensional Hausdorff measure. Throughout this paper, the dimension of the underlying Euclidean space $\mathbb R^d$ will always be $d=1$ or $d=2$.  In the context of visual analysis, $g$ is a given noisy image that we want to approximate by the minimizing function $u \in W^{1,2}(\Omega \setminus K)$; the set $K$ is simultaneously used in order to {\it segment} the image into connected components.  For a broad overview on free-discontinuity problems, their analysis, and applications, we refer the reader to \cite{AFP}.
\\

\noindent If the set $K$ were fixed, then the minimization of $J$ with respect to $u$ would be a relatively simple problem, equivalent to solving the following system of equations:
\begin{eqnarray*}
\Delta u&=& \alpha ( u - g ), \qquad \mbox{in } \Omega \setminus K,\\
\frac{\partial u}{\partial \nu} &=&0, \qquad  \qquad \quad \mbox{ on } \partial \Omega \cup K,   
\end{eqnarray*}
where $\nu$ is the outward-pointing normal vector at any $x \in \partial \Omega \cup K$. Therefore the relevant unknown in free-discontinuity problems is the set $K$.
Ensuring the existence of minimizers $(u,K)$ of $J$ is a challenging problem because there is no topology on the closed sets that ensures 
\begin{itemize} 
\item[(a)] compactness of minimizing sequences and 
\item[(b)] lower semicontinuity of the Hausdorff measure. 
\end{itemize}
Indeed, it is well-known, by the direct method of calculus of variations \cite[Chapter 1]{DM}, that the two previous conditions ensure the existence of minimizers.  However, the problem becomes more manageable if we restrict our domain to functions $u \in BV(\Omega) \cap  W^{1,2}(\Omega \setminus K)$, and make the identification $K \equiv \overline{S_u}$ where $S_u$ is the well-defined discontinuity set of $u$.  In this case, we need to work only with a topology on the space $BV (\Omega)$ of bounded variation, and no set topology is anymore required. \\ \\
Unfortunately the space $BV(\Omega)$ is `too large'; it contains Cantor-like functions whose approximate gradient vanishes, $\nabla u= 0$, almost everywhere, and whose discontinuity set has measure zero, $\mathcal{H}^{d-1}(S_u)=0$.  As these functions are dense in $L^2(\Omega)$, the problem is trivialized; see \cite{AFP} for details.
\\
\\
Nevertheless, it is possible to give a meaningful formulation of the functional $J$ if we exclude such functions and restrict $J$ to the space $SBV(\Omega)$ constituted of $BV$-functions with vanishing Cantor part.   If we assume again $K \equiv \overline{S_u}$,  the solution can be recast as the minimization of
\begin{equation}
\label{amb}
\mathcal J ( u ) = \int_{\Omega \setminus S_u} \left [ | \nabla u |^2 + \alpha ( u - g )^2 \right ] dx +\beta \mathcal{H}^{d-1}(S_u).
\end{equation}
The existence of minimizers in $SBV$ for the functional \eqref{amb} was established by Ambrosio on the basis of his fundamental compactness theorem in \cite{AM89}, see also \cite[Theorem 4.7 and Theorem 4.8]{AFP}.

\subsection{$\Gamma$-convergence approximation to free-discontinuity problems}

The discontinuity set $S_u$ of a $SBV$-function $u$ is not an object that can be easily handled, especially numerically. This difficulty gave rise to the development of approximation methods for the Mumford-Shah functional and its minimizers where \emph{sets} are no longer involved, and instead substituted by suitable \emph{indicator functions}.  In order to understand the theoretical basis for these approximations, we need to introduce the notion of $\Gamma$-convergence, which is today considered one of the most successful notions of `variational convergence'; we state only the definition of $\Gamma$-convergence below, but refer the reader to \cite{DM,br02} for a broad introduction.
\begin{definition}
Let $(X,d)$ be a metric space\footnote{ Observe that by \cite[Proposition 8.7]{DM} suitable bounded sets $X$ endowed with the weak topology induced by a larger Banach space are indeed metrizable, so this condition is not that restrictive.} and let $f,f_n : X \to [0,\infty]$ be functions for $n \in \mathbb N$. We say that $(f_n)_{n \in \mathbb N}$ $\Gamma$-converges to $f$ if the following two conditions are satisfied:
\begin{itemize}
\item[i)] for any sequence $(x_n)_n \subset X$ converging to $x$, 
$$
\liminf_n f_n(x_n) \geq f(x);
$$
\item[ii)] for any $x \in X$, there exists a sequence $(x_n)_n \subset X$ converging to $x$ such that
$$
\limsup_n f_n(x_n) \leq f(x).
$$

\end{itemize}
\label{GammaDef}
\end{definition}

One important consequence of Definition \ref{GammaDef} is that if a sequence of functionals $f_n$ $\Gamma$-converges to a target functional $f$, then the corresponding minimizers of $f_n$ also converge to minimizers of $f$, see \cite[Corollary 7.30]{DM}. 
\\
\\
We define now
\begin{equation}
F_{\varepsilon}(u,v) := \int_{\Omega} \left [ v^2 | \nabla u |^2 + \alpha ( u - g )^2 \right ]  +\frac{\beta}{2 }  \left ( \varepsilon |\nabla v|^2 + \frac{(1 -v)^2}{\varepsilon} \right ) dx
\end{equation}
over the domain $(L^2(\Omega))^2$, along with the related functional
\begin{equation}
\label{amto}
\mathcal J_\varepsilon(u,v) :=   \left\{\begin{array}{cl} F_\varepsilon(u,v)
	  & \textrm{, if } v \in W^{1,2}(\Omega) \textrm{, }u v \in W^{1,2}(\Omega) \textrm{, and } 0 \leq v \leq 1,\\
	\infty   & \textrm{, else.}
	   \end{array}\right.
\end{equation}
Note that at the minimizer $(u,v)$ of $\mathcal J_\varepsilon$, the function $0 \leq v \leq 1$ tends to indicate the discontinuity set $S_u$ of the functional \eqref{amb} as $\varepsilon \to 0$. In \cite{amto90} Ambrosio and Tortorelli proved the following $\Gamma$-approximation result:
\begin{theorem}[Ambrosio-Tortorelli '90] 
For any infinitesimal sequence $(\varepsilon_n)_n$, the functional $\mathcal J_{\varepsilon_n}(u,v)$ $\Gamma$-converges in $(L^2(\Omega))^2$ to the functional
\begin{equation}
\mathcal J ( u,v):=  \left\{\begin{array}{cl}  \mathcal J(u) & \textrm{, if } v \equiv 1, \\
 \infty & \textrm{, otherwise.}
 \end{array}\right.
 \end{equation}
\end{theorem}

\subsection{Discrete approximation}

In fact, the Mumford-Shah functional is the continuous version of a previous discrete formulation of the image segmentation problem proposed by Geman and Geman in \cite{GG}; see also the work of Blake and Zisserman in \cite{BZ}. Let us recall this discrete approach.
For simplicity let $d=2$ (as for image processing problems),  $\Omega = [0,1]^2$, and let $u_{i,j}=u(h i,h j)$, $(i,j) \in \mathbb{Z}^2$ be a discrete function defined on $\Omega_h:=\Omega \cap h \mathbb{Z}^2$, for $h>0$.
Define $W_h ( t ) = \min\{ t^2, \beta/h\}$ to be the truncated quadratic potential, and
\begin{eqnarray}
\mathcal J_{\sqrt{\beta/h}}(u ) &:=& h^2 \sum_{(h i, h j) \in \Omega_h} W_h \left ( \frac{u_{i+1,j} - u_{i,j}}{h} \right ) \nonumber \\
&+& h^2 \sum_{ ( h i, h j) \in \Omega_h} W_h \left ( \frac{u_{i,j+1} - u_{i,j}}{h} \right )\nonumber \\
&+& \alpha h^2 \sum_{(h i,h j) \in \Omega_h} ( u_{i,j} - g_{i,j})^2.
\label{discrete}
\end{eqnarray}
Chambolle \cite{ca92,ca95} gave formal clarification as to how the discrete functional $\mathcal J_{\sqrt{\beta/h}}$ approximates the continuous functional $\mathcal J$ of Ambrosio: discrete sequences can be interpolated by piecewise linear functions in such a way as to allow for discontinuities when the discrete finite differences of the sampling values are large enough. On the basis of this identification of discrete functions on $\Omega_h$ and functions defined on the `continuous domain' $\Omega$, we have the following result:
\begin{theorem}[Chambolle '95]
\label{chthm} 
The functional $\mathcal J_{\sqrt{\beta/h}}$ $\Gamma$-converges in $\mathcal B(\Omega) $ (the space of Borel-measurable functions, which is metrizable, see \cite{ca95} for details) to 
$$
\mathcal J^{cab} ( u ) = \int_{\Omega \setminus S_u} \left [ | \nabla u |^2 + \alpha ( u - g )^2 \right ] dx +\beta \mathcal{C}(S_u),
$$
as $h \to 0$, where $\mathcal{C}$ is the so-called `cab-driver' measure defined below.
\end{theorem}
Basically $\mathcal{C}$ measures the length of a curve only through its projections along horizontal and vertical axes; for a regular $C^1$ curve $c=\gamma([0,1])$, with $\gamma(t)=(\gamma_1(t), \gamma_2(t)) \in \Omega$, we have
$$
\mathcal{C}(c) = \int_0^1 \left ( |\gamma_1'(t)| + |\gamma_2'(t)| \right) dt.
$$
The reason this anisotropic (or, direction dependent) measure appears, in place of the Hausdorff measure in the Mumford-Shah functional, is due to the approximation of derivatives by finite differences defined on a `rigid' squared geometry. A discretization of derivatives based on meshes {\it adapted} to the morphology of the discontinuity indeed leads to precise approximations of the Mumford-Shah functional  \cite{chdm99,boch00}.  

\subsection{Free-discontinuity problems and discrete derivatives}
 In the literature, several methods have been proposed to numerically approximate minimizers of the Mumford-Shah functional \cite{beco94,boch00,ca92,ca95, ma92}.
In particular, a relaxation algorithm, based essentially on alternated minimization of a finite element approximation of the Ambrosio and Tortorelli functional \eqref{amto}, leads to iterated solutions of suitable elliptic PDEs, where the differential part includes the auxiliary variable $v$ which encodes and indicates information about the discontinuity set. These implementations are basically finite dimensional approximations to the following algorithm: \emph{Starting with $v^{(0)} \equiv 1$, iterate}
$$
\left \{ \begin{array}{l}
u^{(n+1)} := \arg \min_{u \in W^{1,2}(\Omega)} \mathcal J_\varepsilon(u ,v^{(n)})\\
v^{(n+1)} := \arg \min_{v \in W^{1,2}(\Omega)} \mathcal J_\varepsilon(u^{(n+1)} ,v).
\end{array}
\right .
$$
{  However, neither has a proof of convergence of this iterative process to its stationary points been explicitly provided in the literature, nor have the properties of such stationary points been investigated, especially in case of genuine inverse problems (see the discussion in Subsection \ref{MS4invprob}).

In this paper, we take a different approach and investigate how minimization of the $\Gamma$-approximating discrete functionals \eqref{discrete} can be implemented efficiently by iterative thresholding on the discrete derivatives. 
Unlike the aforementioned approach, we will be able to provide a rigorous proof of convergence to stationary points, which coincide with local minimizers of the discrete Mumford-Shah functional.  Moreover, we are able to characterize stability properties of such stationary points, and demonstrate the stability of global minimizers of the discrete Mumford Shah functional.}

Let us recall: the solutions $u$ of a free-discontinuity problem are supposed to be smooth out of a minimal ipersurface $K$. This means that the distributional derivative of $u$ is a `small function' everywhere except on $K$ where it coincides with a singular measure. In the discrete approximation $\eqref{discrete}$, the vector of finite differences $(w_j) = (\frac{u_{i,j+1} - u_{i,j}}{h}, \frac{u_{i+1,j} - u_{i,j}}{h} )$ corresponds to a piecewise constant function that is small everywhere except for a few locations, corresponding to $|w_j| \geq \sqrt{\beta/h}$, that approximate the discontinuity set $K$.   So, in terms of derivatives, solutions of $\eqref{discrete}$ are vectors having only few large entries.  In the next section, we clarify how we can indeed work with just derivatives and forget the primal problem.
\subsubsection{The 1-D case}
Let us assume for simplicity that the dimension $d=1$, the domain $\Omega = [0,1]$, and the parameters $\alpha=\beta=1$.  Denote by $u_{i}=u(h i)$ a discrete function defined on $h i \in \Omega_h:=\Omega \cap h \mathbb{Z}$, for $h>0$; note that the vector $(u_i) \in \mathbb{R}^n$ for $n = \lfloor1/h \rfloor$. In this setting, the discrete functional $\eqref{discrete}$ reduces to
\begin{eqnarray*}
\mathcal J_{\sqrt{1/h}}(u ) &=& h \sum_{(h i) \in \Omega_h} W_h \left ( \frac{u_{i+1} - u_{i}}{h} \right ) \\
&+& h \sum_{(h i) \in \Omega_h} ( u_{i} - g_{i})^2,
\end{eqnarray*}
where we recall that $W_h(t) = \min\{t^2, 1/h \}$.  Since no geometrical anisotropy is now involved ($d=1$), it is possible to show that this discrete functional $\Gamma$-converges precisely  to the corresponding Mumford-Shah functional on intervals \cite{ca92}.
\\
\\
For $(u_i)_{ hi \in \Omega_h}$ we define the discrete derivative as the matrix $D_h: \mathbb{R}^n \rightarrow \mathbb{R}^{n-1}$ that maps $(u_i)_{ hi \in \Omega_h}$ into $\left (\frac{u_{i+1}-u_i}{h} \right )_{i}$, given by
\begin{equation}
\label{dermtrx}
D_h = \frac{1}{h} \left ( \begin{array}{cccccc} -1 & 1 & 0& \dots & \dots &0\\
0&-1&1&0 &\dots&0\\
\dots\\
0&0&\dots&\dots&-1&1
\end{array} \right ).
\end{equation}
It is not too difficult to show that
$$
 u =D^\dagger_h D_h u + c,
$$
where $D_h^\dagger$ is the pseudo-inverse matrix of $D_h$ (in the Moore-Penrose sense; note that $D^\dagger_h$ maps $\mathbb R^{n-1}$ into $\mathbb R^n$ and is an injective operator) and $c$ is a constant vector which depends on $u$, and the values of its entries coincide with the mean value $h \sum_{ hi \in \Omega_h} u_i$ of $u$.
Therefore, any vector $u$ is uniquely identified by the pair $(D_h u , c)$.
\\

 \noindent Since constant vectors comprise the null space of $D_h$, the orthogonality relation $\langle D^\dagger_h D_h u,c\rangle_{\ell_2^n} =0$ holds for any vector $u$ and any constant  vector $c$. Here the scalar product $\langle \cdot, \cdot \rangle_{\ell_2^n} = \sum_i u_i v_i$ is the standard Euclidean scalar product, which induces the Euclidean norm $\| u \|_{\ell_2^n} := \left ( \sum_ i u_i^2 \right )^{1/2}$. 
Using this orthogonality property, we have that
\begin{eqnarray*}
\| u - g \|_{\ell_2^n}^2 &=& \| D^\dagger_h D_h u -  D^\dagger_h D_h g + ( c - c_g )\|_{\ell_2^n}^2 \\
&=&    \| D^\dagger_h D_h u -  D^\dagger_h D_h g\|_{\ell_2^n}^2 + \| c - c_g \|_{\ell_2^n}^2 
\end{eqnarray*}
Hence, with a slight abuse of notation, we can reformulate the original problem in terms of derivatives, and mean values, by
\begin{eqnarray*}
\mathcal J_{1/\sqrt h}(z,c) &=& h \| D_h^\dagger z  - f \|_{\ell_2^{n}}^2 + h \| c- c_{g} \|_{\ell_2^n}^2 + h \sum_{i} \min \left \{|z_i|^2, \frac{1}{h}  \right \} \\
\end{eqnarray*}
where $z = D_h u$ and $f =  D^\dagger_h D_h g$. Of course at the minimizer $u$ we have $c = c_g$, since this term in $\mathcal J_{1/\sqrt h}$ does not depend on $z$. Therefore, $\| c - c_g \|_2^2$ does not play any role in the minimization and can be neglected. Once the minimal derivative vector $z$ is computed, we can assemble the minimal $u$ by incorporating the mean value of $g$ as follows:
$$
 u = D_h^\dagger z + c_g.
$$

\subsubsection{The 2-D case, discrete Schwartz conditions, and constrained optimization}
\label{2dcase}
Let us assume now $d=2, \Omega = [0,1]^2$, and again $\alpha=\beta=1$.  Denote $u_{i,j}=u(h i,h j)$, $(i,j) \in \mathbb{Z}^2$, a discrete function defined on $\Omega_h:=\Omega \cap h \mathbb{Z}^2$, $n = \lfloor 1/h \rfloor$,  and
\begin{eqnarray*}
\mathcal J_{1/\sqrt{h}}(u ) &:=& h^2 \sum_{(h i, h j) \in \Omega_h} W_h \left ( \frac{u_{i+1,j} - u_{i,j}}{h} \right ) \\
&+& h^2 \sum_{ ( h i, h j) \in \Omega_h} W_h \left ( \frac{u_{i,j+1} - u_{i,j}}{h} \right )\\
&+& h^2 \sum_{(h i,h j) \in \Omega_h} ( u_{i,j} - g_{i,j})^2.
\end{eqnarray*} 
In two dimensions, we have to  consider the derivative matrix $D_h : \mathbb R^{n^2} \to \mathbb R^{2n(n-1)}$ that maps the vector $(u_{j + (i-1) n}) := (u_{i,j})$ to the vector composed of the finite differences in the horizontal and vertical directions $u_x$ and $u_y$ respectively, given by
$$
D_h u := \left [ 
\begin{array}{l}
u_x\\
u_y
\end{array}
\right ], \quad \left \{ \begin{array}{ll} (u_x)_{j + n(i-1) }:= (u_x)_{i,j}:= \frac{u_{i+1,j} - u_{i,j}}{h}, i=1,\dots,n-1, j=1,\dots,n\\
 (u_y)_{j + (n-1)(i-1) }:= (u_y)_{i,j}:= \frac{u_{i,j+1} - u_{i,j}}{h}, i=1,\dots,n, j=1,\dots,n-1
\end{array} \right. .
$$
{ Note that its range $R(D_h) \subset \mathbb{R}^{2n(n-1)}$ is a $(n^2-1)$-dimensional subspace} because $D_h c = 0$ for constant vectors $c \in \mathbb{R}^{n^2}$.  Again, we have the differentiation-integration formula, given by
$$
 u =D^\dagger_h D_h u + c,
$$
where $D_h^\dagger$ is the pseudo-inverse matrix of $D_h$ (in the Moore-Penrose sense); note that $D^\dagger_h$ maps $R(D_h)$ injectively into $\mathbb R^{n^2}$.  Also, $c$ is a constant vector that depends on $u$, and the values of its entries coincide with the mean value $h^2\sum_{ (hi,hj) \in \Omega_h} u_{i,j}$ of $u$.
\\
\\
Proceeding as before and again with a slight abuse of notation, we can reformulate the original discrete functional $\eqref{discrete}$ in terms of derivatives, and mean values, by
\begin{eqnarray*}
\mathcal J_{1/\sqrt h}(z,c) &=& h^2 \big[ \| D_h^\dagger z  - f \|^2_{\ell_2^{n^2}} + \| c - c_{g} \|_{\ell_2^{n^2}}^2 + \sum_{i,j} \min \left \{|z_{i,j}|^2, \frac{1}{h}  \right \} \big].
\label{2Dunconstrained}
\end{eqnarray*}
where $z = D_h u \in \mathbb{R}^{2n(n-1)}$, and $f =  D^{\dagger}_h D_h g \in \mathbb{R}^{n^2}$. Of course $c = c_g$ is again assumed at the minimizer $u$, since this latter term in $\mathcal J_{1/\sqrt h}$ does not depend on $z$.  However, in order to minimize only over vectors in $\mathbb R^{2n(n-1)}$ that are derivatives of vectors in $\mathbb R^{n^2}$, we must minimize $\mathcal J_{1/\sqrt{h}}(z,c)$ subject to { the constraint $D_h D^{\dagger}_h z = z$.}
\begin{figure}[htp]
\begin{center}
\includegraphics[width=2.5in]{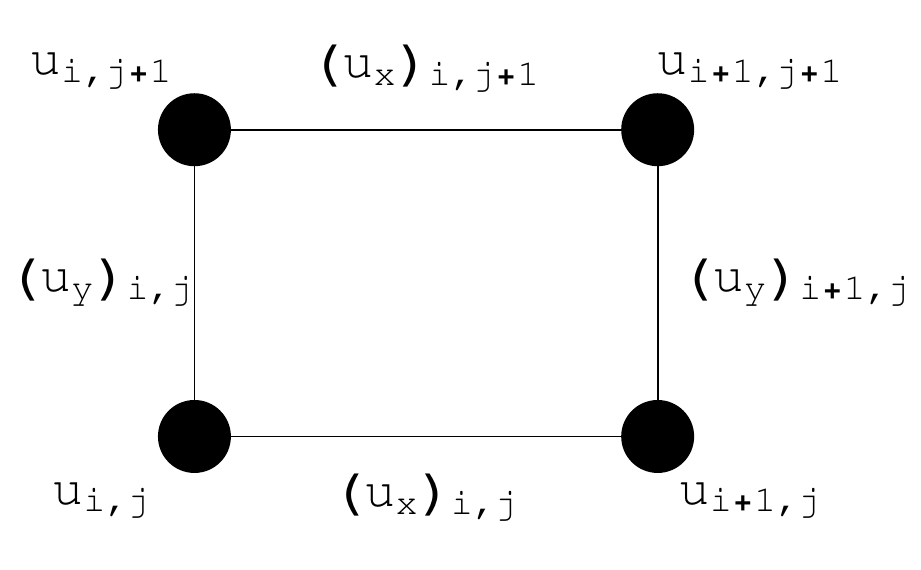}
\end{center}

\caption{Compatibility conditions of derivatives in 2D.}\label{compatib}
\end{figure}
\\
\\
The { $2n(n-1)$ linearly independent constraints $D_h D^{\dagger}_h z = z$} are equivalent to the \emph{discrete Schwartz constraints}\footnote{These discrete conditions correspond to the well-known  Schwartz mixed derivative theorem for which $\partial_{xy} u = \partial_{yx} u$ for any $u \in C^2(\Omega)$.}, 
\begin{equation}
\label{comp}
(u_y)_{i,j} + (u_x)_{i,j+1} =  (u_y)_{i+1,j} + (u_x)_{i,j},
\end{equation}
that establish the equivalence of the length of the paths from $u_{i,j}$ to $u_{i+1,j+1}$, whether one moves in vertical first and then in horizontal direction or in horizontal first and then in vertical direction (see Figure \ref{compatib}).
\\
\\
In short, we arrive at the following constrained optimization problem:

\begin{eqnarray}
&&
 \left\{ \begin{array}{llll}
\textrm{Minimize} &  \mathcal J_{1/\sqrt h}(z) = h^2 \big[ \| T z  - f \|^2_{\ell_2^{n^2}} +  \sum_{i,j} \min \left \{|z_{i,j}|^2, \frac{1}{h}  \right \} \big].
\\
\\
 \textrm{subject to} & {\cal Q} z = 0, \\\end{array}
\right.
\label{MS2d}
\end{eqnarray}
for $T = D_h^\dagger$ and { ${\cal Q} = {\cal I}- D_h D^{\dagger}_h $.}  Once the minimal derivative vector $z$ is computed, we can assemble the minimal $u$ by incorporating the mean value of $g$ as follows:
$$
 u = D_h^\dagger z + c_g.
$$

\subsubsection{Regularization of inverse problems by means of the Mumford-Shah constraint}
\label{MS4invprob}

The Mumford-Shah regularization term 
\begin{equation}
\label{MSreg}
MS ( u ) = \int_{\Omega \setminus S_u} | \nabla u |^2  +\beta \mathcal{H}^{d-1}(S_u),
\end{equation}
has been used frequently in inverse problems for image processing \cite{essh02,rari07}, such as inpainting and tomographic inversion. Despite the successful numerical results observed  in the aforementioned papers for the minimization of functionals of the type 
\begin{equation}
\label{regMS}
\mathcal J (u) = \alpha \| K u - g \|_{L^2(\Omega)} + MS(u),
\end{equation}
where $K:L^2 (\Omega) \to L^2 (\Omega)$ is a bounded operator which is not boundedly invertible, no rigorous results on existence of minimizers are currently available in the literature. Indeed, the Ambrosio compactness theorem \cite{AM89} used for the proof of the case $K = I$ does not apply in general. A few attempts towards using the regularization $MS$ for inverse problems in fracture detection appear in the work of Rondi \cite{ro07,ro08-2,ro08-1}, although restrictive technical assumptions on the admissible discontinuities of the solutions are required.
\\

\noindent As one of the contributions to this paper, we show that discretizations of regularized functionals of the type \eqref{regMS} \emph{always} have minimizers (see Theorem \ref{existmin}).  More precisely, these discretizations correspond to functionals of the form,
\begin{equation}
\mathcal J_{\sqrt{\beta/h}}(u ):=  \alpha h^2 \| K u - g\|_{\ell_2}^2 + h^2 \sum_{(h i, h j) \in \Omega_h} \left [ W_h \left ( \frac{u_{i+1,j} - u_{i,j}}{h} \right ) + W_h \left ( \frac{u_{i,j+1} - u_{i,j}}{h} \right ) \right ].
\label{regdiscrete}
\end{equation} 
and we prove that such functionals admit minimizers. Note that the discrete Mumford-Shah approximation $\eqref{discrete}$ can be written in this form.   We go on to show that such minimizers can be characterized by certain fixed point conditions, see Theorem 
 \ref{mintheorem} and Theorem \ref{globalmin}.  
As a consequence of these achievements we can prove that global minimizers are always isolated, although not necessarily unique, whereas local minimizers may constitute a continuum of unstable equilibria. Hence, our analysis will shed light on fundamental properties, virtues, and limitations, of regularization by means of the Mumford-Shah functional $MS$, and provide a rigorous justification of the numerical results appearing in the literature.
\\
\\
\noindent It is useful to show how the discrete functional \eqref{regdiscrete} can be still expressed in terms of the sole derivatives for general $K$. As done before in the case $K = {\cal I}$, and with the now usual identification $u= (D_h u, c)$, we can rewrite the functional in terms of derivatives and mean value as follows:
\begin{eqnarray}
\mathcal J_{\sqrt{\beta/h}}(z,c) &=&  h^2 \alpha \| K D_h^\dagger z  - (g - K c) \|_2^2+ h^2  \sum_{i,j} \min \left \{|z_{i,j}|^2, \frac{\beta}{h}  \right \},
\label{withc}
\end{eqnarray}
Note that in general we cannot anymore split orthogonally the discrepancy $\| K D_h^\dagger z  - (g - K c) \|_2^2$ into a sum of two terms which depend only on derivatives $z$ and mean value $c$ respectively. Nevertheless, for fixed $z$, it is straightforward to show that { $\bar{c} = \arg \min_c \mathcal J_{\sqrt{\beta/h}}(z,c)$  depends on $z$ via an affine map}. Indeed we can compute
$$
{\it \bar{c}} = \left ( \frac{ \langle K \mathbf 1, g - K D_h^\dagger z \rangle}{\| K \mathbf 1\|_{\ell^2}^2} \right ) \mathbf 1,
\label{c}
$$
where $\mathbf 1$ is the constant vector with entries identically $1$. Here we assume that $\mathbf 1 \notin \ker K$, that is a necessary condition in order to be able to identify the mean value of minimizers (a similar condition is required anytime we deal with regularization functionals which depend on the sole derivatives, see, e.g., \cite{chli97,ve01}).  By substituting this expression for $\bar{c}$ into \eqref{withc}, it is clear that the minimization of functionals $\eqref{regdiscrete}$ can be reformulated, in terms of the sole derivatives, as constrained minimization problems of the form $\eqref{MS2d}$.  

\section{Existence of minimizers for a class of discrete free-discontinuity problems}
In light of the observations above, we can transform the problem of the minimization of functionals of the type \eqref{regMS}, by means of discretization first and then reduction to sole derivatives, into the (possibly, but not necessarily) constrained minimization problem:
\begin{eqnarray}
&&
 \left\{ \begin{array}{llll}
\textrm{Minimize} &  \mathcal J_r (u) = \big[ \| T u  - g \|^2_{\ell_2^M} +  \sum^N_{i=1} \min \left \{|u_i|^2, r^2  \right \} \big].
\\
 \textrm{subject to} & {\cal Q} u = 0. \\\end{array}
\right.
\label{L2general}
\end{eqnarray}

Our first result ensures the existence of minimizers for the constrained optimization problem \eqref{L2general}:
\begin{proposition}
Assume $r>0$, and fix linear operators $T:\mathbb R^N \to \mathbb R^M$ and ${\cal Q}: \mathbb R^N \to \mathbb R^{M'}$, which are identified in the following with their matrices with respect to the canonical bases. We also fix $g \in \mathbb R^M$. The constrained minimization problem
\begin{eqnarray}
&&
 \left\{ \begin{array}{llll}
\textrm{Minimize} &  \mathcal J_r (u) = \big[ \| T u  - g \|^2_{\ell_2^M} +  \sum^N_{i=1} \min \left \{|u_i|^2, r^2  \right \} \big]
\\
 \textrm{subject to} & {\cal Q} u = 0. \\\end{array}
\right.
\label{L2}
\end{eqnarray}
has minimizers $u^*$. 
\label{existmin2}
\end{proposition}
\begin{proof}
We begin by noting that $\inf_{{\cal Q}u = 0} {\cal{J}}_r({u})$ is well-defined and finite, since ${\cal{J}}_r \geq 0$ is bounded from below.  It remains to show that there exists a vector ${u^*}$ that satisfies ${\cal{J}}_r({u^*}) = \inf_{u \in \mathbb R^N} {\cal{J}}_r({u})$.  Towards this goal, consider the following partition $\mathcal P = \{ \mathcal U_{\mathcal I_j } \}_{j=1}^{2^N}$ of $\mathbb{R}^N$ indexed by the subsets $\mathcal I_j$ of the index set $\mathcal I = \{1,2,...,N\}$, as follows:
\begin{equation}
\mathcal U_{{\cal I}_j} := \{u \in \mathbb{R}^{N}: |u_i| \leq r, i \in {\cal I}_j, |u_i| > r, i \in {\cal I} / {\cal I}_j \}.
\end{equation}
The minimization  of ${\cal{J}}_r$ subject to ${\cal Q}u = 0$ and constrained to the closure of the subset $\mathcal U_{\mathcal I_j}$ can be reformulated as a quadratic optimization problem, for which the classical Frank-Wolfe theorem  $\cite{bashsh93}$ guarantees the existence of a minimizer $u(\mathcal I_j)$.  Now, since $\mathbb{R}^N = \cup_j {\cal I}_j$, the minimal value of $\mathcal J_r$ subject to $Qu = 0$ and over all of $\mathbb{R}^N$ is just the minimal value from the finite set $\{ \mathcal J_r(u({\cal I}_j)): \quad j=1, \dots, 2^N\}$; that is,
\begin{eqnarray}
\min_{{\cal Q}u = 0} {\cal{J}}_r({u}) &=& \min_{{\cal I}_j \subset {\cal I}} {\cal{J}}_r({u}({\cal I}_j))
\nonumber
\end{eqnarray}
and $u^*=\arg \min_{{\cal Q}u = 0} {\cal{J}}_r({u}) = u\big(\arg \min_{{\cal I}_j \subset {\cal I}} {\cal{J}}_r({u}({\cal I}_j))\big).$
\end{proof}
In fact, Proposition \ref{existmin2} extends to a much larger class of free-discontinuity type minimization problems; by the same reasoning as before, we arrive at the more general result:
\begin{theorem}
The constrained minimization problem
\begin{eqnarray}
&&
 \left\{ \begin{array}{llll}
\textrm{Minimize} &  \mathcal J^p_r (u) = \big[ \| T u  - g \|^2_{\ell_2^M} +  \sum^N_{i=1} \min \left \{|u_i|^p, r^p  \right \} \big]
\\
 \textrm{subject to} & {\cal Q} u = 0. \\\end{array}
\right.
\label{Lpgeneral}
\end{eqnarray}
has minimizers $u^*$ for any real-valued parameter $p \geq 1$. 
\label{existmin}
\end{theorem}
The Frank-Wolfe theorem, which guarantees the existence of minimizers for quadratic programs with bounded objective function, does not apply to the general case $p \geq 1$ where the objective function $\mathcal J^p_r$ is not necessarily quadratic.  Nevertheless, with the following generalization for the Frank-Wolfe theorem, Theorem \ref{existmin} follows directly from a similar argument as for Proposition \ref{existmin2}.

\begin{proposition}
Suppose $A$ is an $N \times N$ positive semidefinite matrix, and suppose $b$ and $c$ are $N \times 1$ vectors.   Suppose also that $X$ is a nonempty convex polyhedral subset of $\mathbb{R}^N$.  The convex optimization problem
\begin{eqnarray}
 \left\{ \begin{array}{llll}
\textrm{minimize} & u^t A u + b^t u + \sum_{1 \leq j \leq N} c_j |u_j|^p
\\
 \textrm{subject to} & u \in X.  \\\end{array} 
\right.
\label{FWgen}
\end{eqnarray}
admits minimizers for any real parameter $p \geq 1$, as long as the objective function is bounded from below.  
\label{firstlemma}
\end{proposition}
For ease of presentation, we reserve the proof of Proposition \ref{firstlemma} to the Appendix.

\noindent From the proof of Theorem \ref{existmin}, one could in principle obtain a minimizer for ${\cal J}^p_r$ by computing a minimizer ${u}({\cal I}_j)$ for each subset ${\cal I}_j \subset {\cal I}$ using a quadratic program solver $\cite{bashsh93}$, and then minimizing ${\cal J}^p_r$ over the finite set of points $\{ u({\cal I}_j) \}$.  Unfortunately, this algorithm is computationally infeasible as the number of subsets of the index set $\{1,2, ..., N \}$ grows exponentially with the dimension $N$ of the underlying space.   

Indeed, the minimization problem \eqref{Lpgeneral} is  NP hard, as the known NP-complete problem SUBSET-SUM can be reduced to this problem.  A complete discussion about the NP-hardness of \eqref{Lpgeneral} can be found in \cite{RachelPrep} . 

\section{An iterative thresholding algorithm for 1-D free-discontinuity inverse problems}

\subsection{Overview of the algorithm}

In this section, we introduce an algorithm that is guaranteed to converge to a local minimizer of the real-valued functional ${\cal{J}}_r^p: \ell_2(\mathcal I) \rightarrow \mathbb{R}$ having the form 
\begin{equation}
{\cal{J}}^p_r ({ u}) =  \| T {u} - g \|^2_{\ell_2(\mathcal K)} +  \sum_{i \in \mathcal I} \min \{ |u_i|^p, r^p \},
\label{generalform}
\end{equation}
subject to the { conditions:
\begin{itemize}
\item  $\mathcal I$ and $\mathcal K$ are countable sets of indices, and $T:\ell_2(\mathcal I) \to  \ell_2(\mathcal K)$ is a bounded linear operator, which is in the following identified with its matrix associated to the canonical basis;
\item the operator $T$ has spectral norm $\|T\|  < 1$. Note that this requirement is easily met by an appropriate scaling for the functional, i.e., we may have to consider instead
$$
{\cal{J}}^p_r ({ u}) =  \gamma \| T{ u} - g \|^2_{\ell_2(\mathcal K)} +  \gamma \sum_{i \in \mathcal I} \min \{ |u_i|^p, r^p \}, \quad \gamma \leq 1.
$$ 
This modification leads to minor changes in the analysis that follows (see also Subsection \ref{denoising}), and throughout this paper we assume, without loss of generality, that $\gamma =1$;

\item  the parameter $p$ is in the range $1 \leq p \leq 2$.  In case the index set $\mathcal I$ is \emph{finite}, only the restriction $p \geq 1$ is necessary.  

\end{itemize}
We note that the scaled 1D discrete Mumford-Shah functional $\frac{1}{h}{\cal{J}}_{1/\sqrt h}$ is clearly a functional of the form \eqref{generalform} having $r=1/\sqrt h$, index set $\mathcal I  = \{1, \dots, \lfloor r^2 \rfloor\}$, parameter $p = 2$, and operator $T = D_{1/r^2}^{\dagger}: \mathbb{R}^{\lfloor r^2 \rfloor - 1} \rightarrow \mathbb{R}^{\lfloor r^2 \rfloor}$ .  As shown in the Appendix, the operators $D_{1/r^2}^{\dagger}$ satisfy the uniform bound $\| D_{1/r^2}^{\dagger} \| \leq 1/2$, independent of dimension, so a scaling factor is not needed in this case.  
\\

\noindent In the following, we will not minimize ${\cal{J}}^p_r$ directly.  Instead, we propose a \emph{majorization-minimization} algorithm for finding solutions to ${\cal{J}}^p_r$, motivated by the recent application of such algorithms for minimizing energy functionals arising in sparse signal recovery and image denoising \cite{blda??,DDD}.    More precisely, consider the following \emph{surrogate} objective function,
\begin{equation}
 {\cal{J}}^{p, surr}_r (u,a) := {\cal{J}}^p_r(u) - \| T{u} - T {a} ||_{\ell_2(\mathcal K)}^2 + \| {u} - {a} \|_{\ell_2(\mathcal I)}^2. \quad u,a \in \ell_2(\mathcal I).
\label{surr1}
 \end{equation}
The surrogate functional ${\cal{J}}^{p,surr}_r$ satisfies ${{\cal{J}}}^{p,surr}_r (u, {a}) \geq {\cal{J}}^p_r({ u})$ everywhere, with equality if and only if ${u = a}$, and is such that the sequence
\begin{equation}
{ u}^{n+1} = \arg \min_{ u} {{\cal{J}}}^{p,surr}_r ({ u}, { u}^{n})
\label{map}
\end{equation}
obtained by successive minimizations of ${{\cal{J}}}^{p,surr}_r ({ u, a})$ in ${ u}$ for fixed ${ a}$ results in a nonincreasing sequence of the original functional ${\cal{J}}^p_r (u^n)$ (see Lemmas \ref{l1} and \ref{l2}).   We will study the implementation and the convergence properties of the iteration \eqref{map} as follows:
\begin{itemize}
\item in Section $3.2$, we review the standard properties of majorization-minimization iterations,
\item in Section $3.3$, we explicitly compute $u$-global minimizers of the surrogate functional ${{\cal{J}}}^{p,surr}_r ({ u, a})$, for $a$ fixed; 
\item in Section $4.1$ we discuss connections between the resulting thresholding functions and thresholding functions used in compressed sensing,
\item in Sections $5.1$, $5.2$, and $5.3$, we show that  the sequence $({ u}^n )_{n \in \mathbb N}$ defined by \eqref{map} will converge to a stationary value $\bar{u}= \arg \min_{ u} {{\cal{J}}}^{p,surr}_r ({ u}, \bar{ u})$, starting from any initial value ${ u}^0$ for which ${\cal{J}}^p_r({ u}^0) < \infty$,
\item in Section $5.4$, we show that such stationary values $\bar{u}$ are also local minimizers of the original functional ${\cal{J}}^p_r$ that satisfy a certain fixed point condition, and 
\item in Section $5.5$, it is shown that any global minimizer of ${\cal{J}}^p_r$ is among the set of possible fixed points $\bar{u}$ of the iteration $\eqref{map}$.
\end{itemize}By means of the thresholding algorithm, we also show that global minimizers of the functional ${\cal{J}}^p_r$ are isolated, and moreover possess a certain segmentation property that is also shared by fixed points of the algorithm.
}

\subsection{Preliminary lemmas}
The lemmas in this section are standard when using surrogate functionals (see $\cite{DDD}$ and $\cite{blda??}$), and concern general real-valued surrogate functionals of the form
\begin{equation}
{\cal{F}}^{surr}({ u, a}) = {\cal{F}}({ u}) - \|T{ u} - T{ a}\|^2_{\ell_2(\mathcal K)}+ \|{ u - a}\|^2_{\ell_2(\mathcal I)}.
\label{surrgen}
\end{equation}
The lemmas in this section hold independent of the specific form of the functional ${\cal{F}}:\ell_2(\mathcal I) \to \mathbb R^+$,  but do rely on the restriction that $\|T\| < 1$.  
\begin{lemma}
If the real-valued functionals ${\cal{F}}({ u})$ and ${\cal{F}}^{surr}({ u,a})$ satisfy the relation $\eqref{surrgen}$ and the sequence $({ u^n})_{n \in \mathbb N}$ defined by ${ u}^{n+1} = \arg \min_{ u \in \ell_2(\mathcal I)} {\cal{F}}^{surr} ({ u, u^n})$ is initialized in such a way that ${\cal{F}}({ u^0}) < \infty$, then the sequences ${\cal{F}} ({ u^n})$ and ${{\cal{F}}}^{surr} ({ u^{n+1}, u^{n}})$ are non-increasing as long as $\|T\| < 1$.
\label{l1}
\end{lemma}
\begin{proof}
Since $\|T\| < 1$, also $\|T^*T\| < 1$,  and so the operator $L = \sqrt{I - T^*T}$ is a well-defined positive operator whose spectrum is contained within a closed interval $[c, 1]$ that is bounded away from zero $c > 0$.   We can then rewrite ${{\cal{F}}}^{surr} ({ u^{n+1} u^{n}})$ as ${{\cal{F}}}^{surr} ({ u^{n+1}, u^{n}}) = {{\cal{F}}}({ u^{n+1}}) + \|L({ u^{n+1} - u^n)}\|_{\ell_2(\mathcal I)}^2$, from which it follows that
\begin{eqnarray}
{{\cal{F}}}({ u^{n+1}}) &\leq& {{\cal{F}}}({ u^{n+1}}) +  \|L({ u^{n+1} - u^n})\|_{\ell_2(\mathcal I)}^2 \nonumber \\
&=& {{\cal{F}}}^{surr} ({ u^{n+1}, u^{n}}) \nonumber \\
&\leq& {{\cal{F}}}^{surr} ({ u^{n}, u^{n}}) \nonumber \\
&=& {{\cal{F}}}({ u^{n}}) \nonumber \\
&\leq& {{\cal{F}}}({ u^{n}}) +  \|L({ u^{n} - u^{n-1}})\|_{\ell_2(\mathcal I)}^2 \nonumber \\
&=& {{\cal{F}}}^{surr} ({ u^{n}, u^{n-1}}),
\label{est}
\end{eqnarray}
where the second inequality follows from ${ u^{n+1}}$ being a minimizer of $ {{\cal{F}}}^{surr} ({ u, u^{n}})$.
\end{proof}
\noindent From Lemma \ref{l1} we obtain the following corollary:
\begin{lemma}
As long as the conditions of Lemma $\ref{l1}$ are satisfied, one can choose $N \in \mathbb N$ sufficiently large such that for all $n \geq N$, $\|{ u^{n+1} - u^n} \|_{\ell_2(\mathcal I)} \leq \epsilon$, i.e., $$\lim_{n \to \infty}  \|{ u^{n+1} - u^n} \|_{\ell_2(\mathcal I)}=0.$$
\label{l2}
\end{lemma}
\begin{proof}
From Lemma  \ref{l1}, it follows that ${\cal F} (u^{n}) \geq 0$ is a nonincreasing sequence, therefore it converges, and ${\cal F} (u^{n}) - {\cal F} (u^{n+1}) \to 0$ for $n \to \infty$. The lemma follows from \eqref{est}, and the estimates
$$
{\cal F} (u^{n}) - {\cal F} (u^{n+1}) \geq \|L({ u^{n+1} - u^n})\|_{\ell_2(\mathcal I)}^2 \geq (1 - \|T\|^2) \|  u^{n+1} - u^n\|_{\ell_2(\mathcal I)}^2.
$$
\end{proof}



\subsection{The surrogate functional ${{\cal{J}}}^{p, surr}_r$, its explicit minimization, and a new thresholding operator}
It is not immediately clear that the surrogate functional ${{\cal{J}}}^{p,surr}_r$ in $\eqref{surr1}$ is any easier to manage than its parent functional ${{\cal{J}}}^p_r$.  However, expanding the squared terms on the right hand side of $\eqref{surr1}$, ${{\cal{J}}}^{p,surr}_r(u,a)$ can be equivalently expressed as
\begin{eqnarray}
{{\cal{J}}}^{p,surr}_r ({ u, a}) &=& \| { u} - (I - T^*T){ a} + T^*{ g}  \|_{\ell_2(\mathcal I)}^2 + \sum_{i \in \mathcal I} \min \{ |u_i|^p, r^p \} + C\nonumber \\ 
&=& \sum_{i \in \mathcal I} \Big[ ( u_i - [a - T^*Ta + T^*g]_i )^2 + \min \{ |u_i|^p, r^p \} \Big] + C, \nonumber
\label{surr}
\end{eqnarray}
where the term $C = C(T,a,g)$ depends only on $T$, ${ a}$ and ${ g}$.  Indeed, unlike the original functional ${\cal{J}}^p_r$, the surrogate functional ${{\cal{J}}}^{p,surr}_r$ {\it decouples} in the variables $u_i$, due to the cancellation of terms involving $\|T{ u}\|_{\ell_2}^2$.  Because of this decoupling, global ${ u}$-minimizers of ${{\cal{J}}}^{p,surr}_r({ u,a})$, for $a$ fixed, can be computed {\it component-wise} according to
\begin{equation}
\bar{u}_i = \arg \min_{t \in \mathbb R}  \Big[ (t - [a - T^*Ta + T^*g]_i )^2 + \min \{ |t|^p, r^p \} \Big], \quad i \in \mathcal I.
\label{surmin}
\end{equation} 
One can solve $\eqref{surmin}$ explicitly when e.g. $p = 2$, $p = 3/2$, and $p = 1$; in the general case $p \geq 1$, we have the following result: 

\begin{proposition}[Minimizers of ${{\cal{J}}}^{p,surr}_r({ u,a})$ for $a$ fixed]\label{thmthrs}.
\begin{enumerate}
\item If {$p > 1$}, the minimization problem ${ \bar{u}} = \arg \min_{ u \in \ell_2(\mathcal I)} {{\cal{J}}}^{p,surr}_r ({ u}, { a})$ can be solved component-wise by
\begin{equation}
\bar{u}_i = H_{(p,r)}([{ a} - T^*T{ a} + T^*{ g}]_i), \quad i \in \mathcal I,
\label{three}
\end{equation}
where $H_{(p,r)}: \mathbb{R} \rightarrow \mathbb{R}$ is the \emph{`thresholding function'},
\begin{equation}
H_{(p,r)} (\lambda) =
\left\{
\begin{array}{ll}
F_p^{-1}(\lambda), & |\lambda| \leq \lambda'(r,p)  \\
\lambda, & |\lambda|  > \lambda'(r,p). \\  
\end{array}
\right.
\label{thresh}
\end{equation}
Here, $F_p^{-1}(\lambda)$ is the inverse of the function $F_p(t) = t + \frac{p}{2}\sgn{t}{|t|}^{p-1}$,
and $\lambda':=\lambda'(r,p) \in (r, r + \frac{p}{2}r^{p-1})$ is the unique positive value at which
\begin{equation}
(F_p^{-1}(\lambda') - \lambda')^2 + |F_p^{-1}(\lambda')|^p = r^p.
\end{equation}
\item When {  $p=1$}, the general form $\eqref{three}$ still holds, but we have to consider two cases:
\begin{enumerate}
\item If $r > 1/4$, the thresholding function $H_{(1,r)}: \mathbb{R} \rightarrow \mathbb{R}$ satisfies
\begin{equation}
H_{(1,r)} (\lambda) =
\left\{
\begin{array}{ll}
0, & |\lambda| \leq 1/2  \\
(|\lambda| - 1/2)\sgn{\lambda}, &1/2 < |\lambda|  \leq r + 1/4 =  \lambda'(r,1) \\
\lambda, & |\lambda| >  r + 1/4
\end{array}
\right.
\label{thresh}
\end{equation}
\item If, on the other hand, $r \leq1/4$, the function $H_{(1,r)}$ satisfies
\begin{equation}
H_{(1,r)} (\lambda) =
\left\{
\begin{array}{ll}
0, & |\lambda| \leq \sqrt{r}=  \lambda'(r,1)  \\
\lambda, & |\lambda| > \sqrt{r}
\end{array}
\right.
\label{thresh}
\end{equation}
\end{enumerate}
\end{enumerate}
In all cases, the function $H_{(p,r)}$ is continuous except at $\lambda'(r,p)$, where $H_{(p,r)}$ has a jump-discontinuity of size $\delta(r,p) = |\lambda' - H_{(p,r)}(\lambda')| > 0$ if $r > 0$.  In particular, it holds that $\lambda'(r,p) > r$ while $H_{(p,r)}(\lambda') < r$.
\end{proposition}
\noindent We leave the proof of Proposition \ref{thmthrs} to the Appendix.  

\begin{remark}
\emph{ In the particular case $p = 2$ corresponding to classical Mumford-Shah regularization \eqref{L2general}, the thresholding function $H_{(2,r)}: \mathbb{R} \rightarrow \mathbb{R}$ has a particularly simple explicit form:}
\begin{equation}
H_{(2,r)} (\lambda) =
\left\{
\begin{array}{ll}
\lambda/2, & |\lambda| \leq \sqrt{2}r  \\
\lambda, & |\lambda|  > \sqrt{2}r \\  
\end{array}
\right. 
\label{thresh}
\end{equation}
\emph{In addition to $H_{(2,r)}$ and $H_{(1,r)}$, the thresholding operator $H_{(3/2,r)}(\lambda)$ corresponding to $p = 3/2$ can also be computed explicitly, by solving for the positive root of a suitable polynomial of third degree.  In Figure $\ref{threshold}$ below, we plot $H_{(2,1)}, H_{(3/2,1)}$, and $H_{(1,1)}$ with parameter $r = 1$.  For general noninteger values of $p$, $H_{(p,r)}$ cannot be solved in closed form.   However, recall the following general properties of $H_{(p,r)}$:
\begin{itemize}
\item $H_{(p,r)}$ is an odd function, 
\item $H_{(p,r)}(0) = 0$, and 
\item $H_{(p,r)}(\lambda) = \lambda$ once $|\lambda| > r + \frac{p}{2}r^{p-1}$. 
\end{itemize}
In fact, we can effectively \emph{pre}compute $H_{(p,r)}$ by numerically solving for the value of $H_{(p,r)}(\lambda_j)$ on a discrete set $\{\lambda_j\}$ of points $\lambda_j  \in (0,\frac{p}{2}r^{p-1} + r]$.  At $\lambda_j$, one just needs to solve the real equation
\begin{equation}
h_j + \frac{p}{2}h_j^{p-1} - \lambda_j = 0
 \end{equation}
which can be computed effortlessly via a root-finding procedure such as Newton's method:  while $h_j$ satisfies $(h_j - \lambda_j)^2 + (h_j)^p \leq r^p$, set $H_{(p,r)}(\lambda_j) = h_j$; once this constraint is violated, set $H_{(p,r)}(\lambda_j) = \lambda_j$. }
\end{remark}

\begin{figure}
\begin{center}
\subfigure[]{ \label{1(a)}
\includegraphics[width=3in]{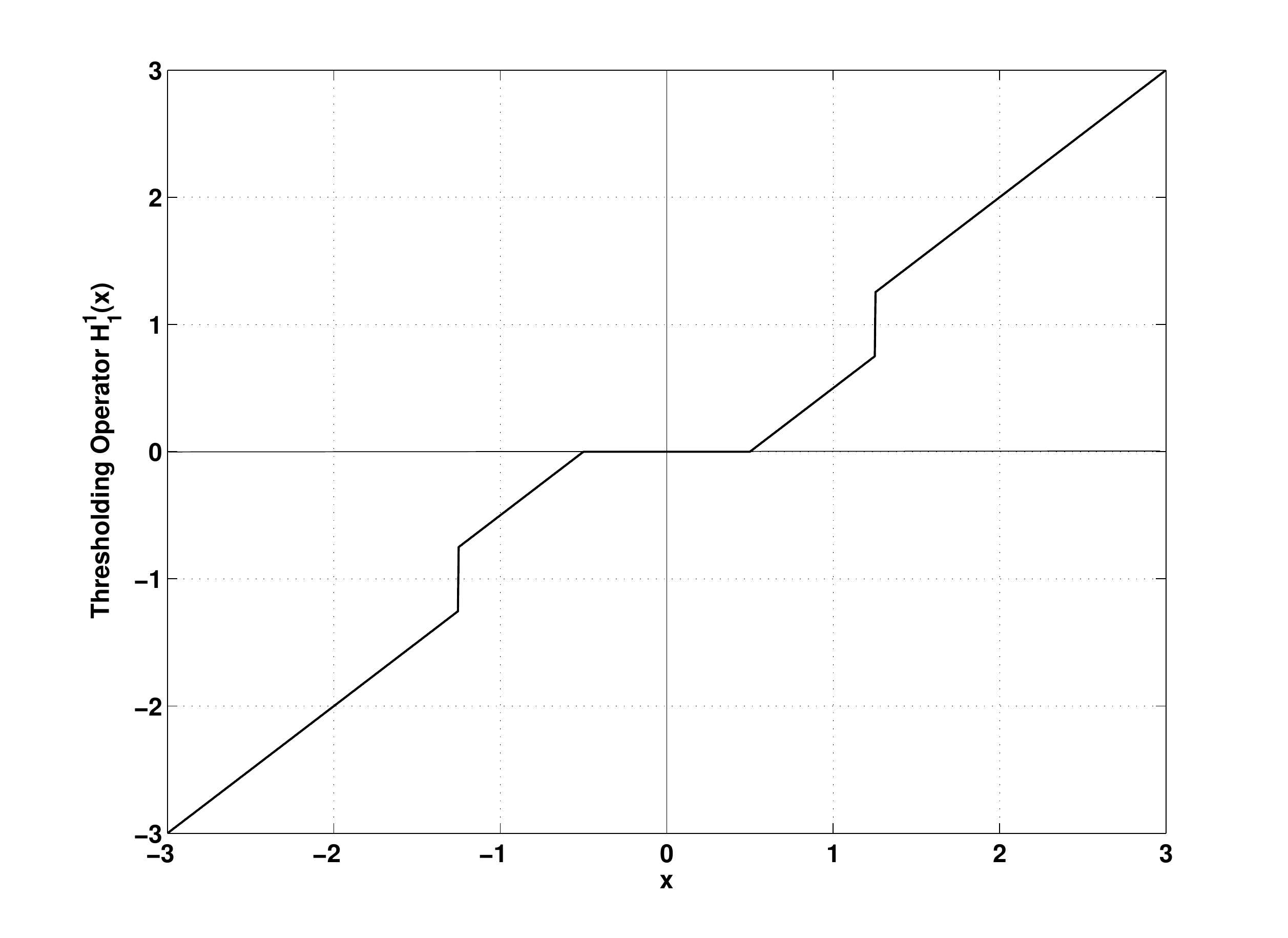}}
\subfigure[]{\label{1(b)}
\includegraphics[width=3in]{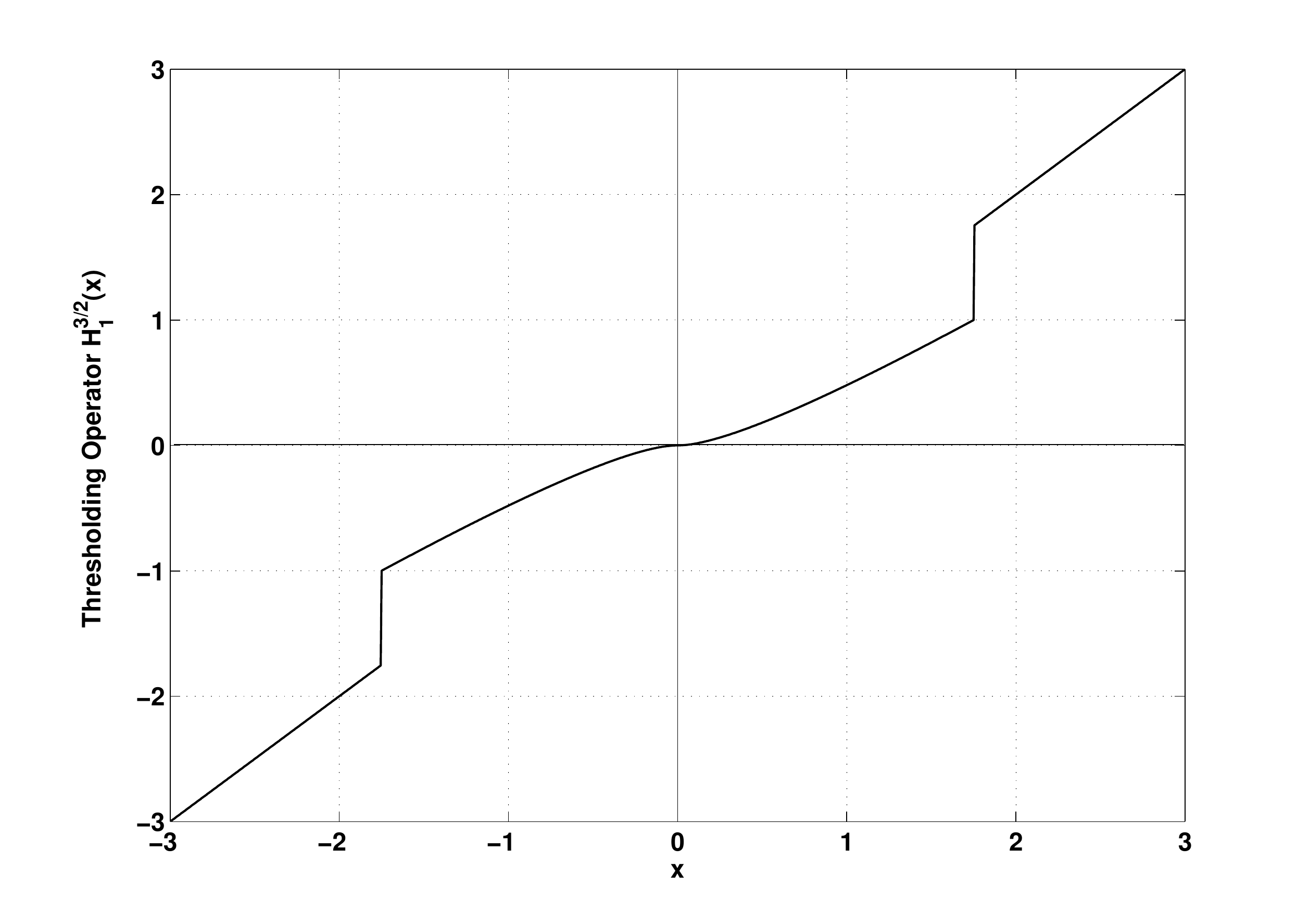}}
\subfigure[]{\label{1(c)}
\includegraphics[width=3in]{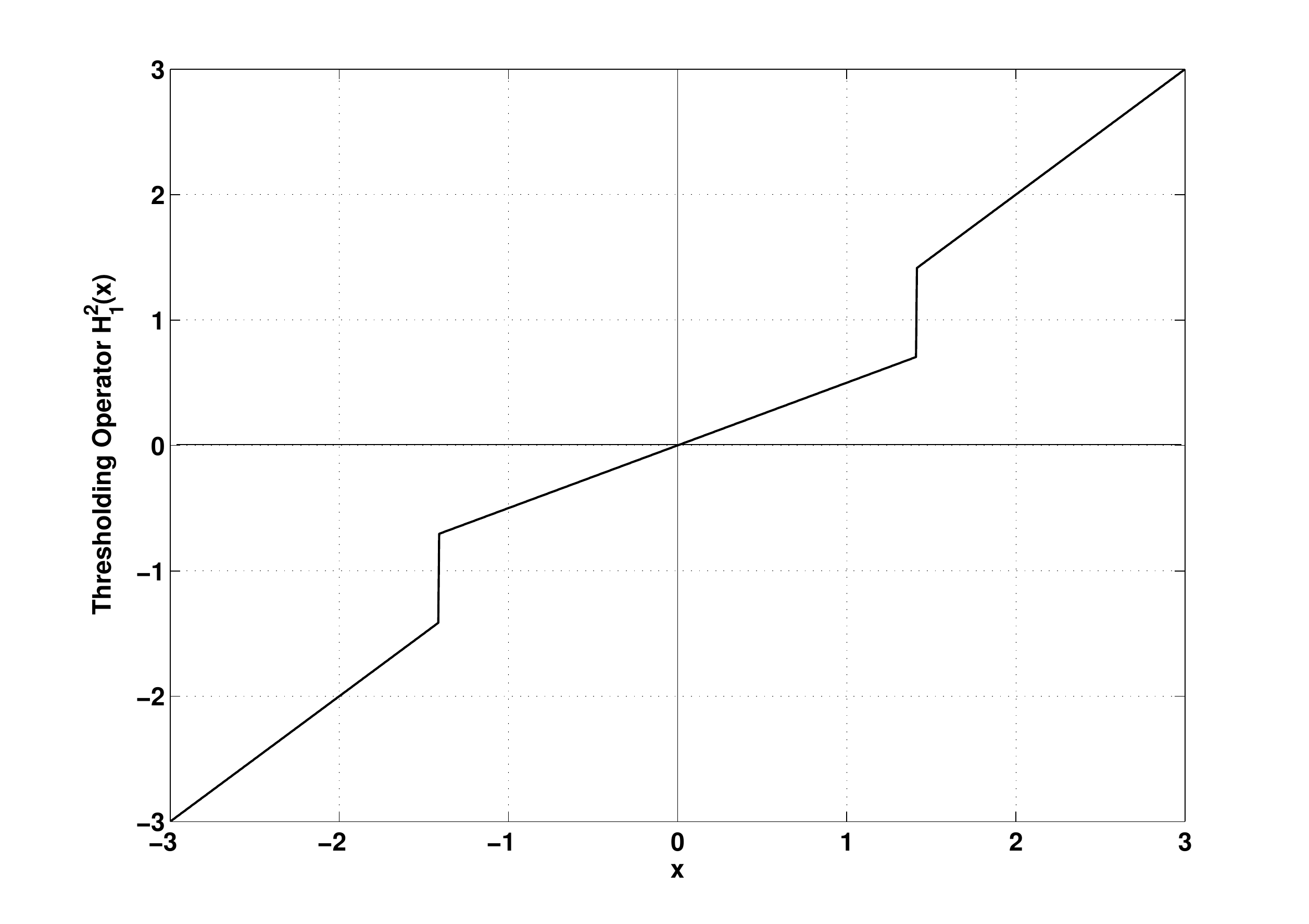}}
\end{center}
\caption{The discontinuous thresholding functions  $H_{(1,1)}$,  $H_{(3/2,1)}$, and $H_{(2,1)}$, with parameters $p=1,3/2$, and $2$, respectively, and $r = 1$.}\label{threshold}
\end{figure}

\pagebreak

\section{A connection to compressive sensing}

When $p=1$ and $r \leq 1/4$, we know from Theorem \ref{thmthrs} that the iterative algorithm
\begin{equation}
{ u}^{n+1} = \arg \min_{ u} {{\cal{J}}}^{p,surr}_r ({ u}, { u}^{n})
\end{equation}
reduces to the component-wise thresholding
\begin{eqnarray}\label{hard}
u^{n+1}_i &=& H_{\sqrt{r}}([u^n - T^*T{ u^n} + T^*g]_i),
\end{eqnarray}
where 
\begin{equation}
H_{\gamma}(\lambda) =
\left\{
\begin{array}{ll}
0, & |\lambda| \leq \gamma  \\
\lambda, & |\lambda|  > \gamma.
\end{array} \right.
\end{equation}
This thresholding function $H_{\gamma}: \mathbb{R} \rightarrow \mathbb{R}$ is referred to as {\it hard-thresholding} in the area of sparse recovery, and the iteration \eqref{hard} generated by successive applications of hard thresholding has been previously studied \cite{blda??}.  In particular, the iteration \eqref{hard} was shown in  \cite{blda??} to correspond to successive minimization in $u$ for fixed $a$ of the surrogate functional ${\cal F}_r^{0,surr} (u, a)$ corresponding to the $\ell_0$ regularized functional, 
\begin{equation}\label{ell_0}
 \mathcal F^0_{r}(u) = \|T{ u} - { g}\|^2_{\ell_2(\mathcal K)} + r \|{ u}\|_{\ell_0(\mathcal I)}.
\end{equation}
Here, the $\ell_0$ quasi-norm  $\| u \|_{\ell_0(\mathcal I)} := \sum_{i \in \mathcal I } | u_i|_0$  is defined component-wise by $$
|u_i|_0= \left \{ 
\begin{array}{ll}
0, & \textrm{ if } u_i= 0\\
1, & \textrm{ otherwise} 
\end{array} \right . 
$$  

It is not a coincidence that hard thresholding comes out from applying iterative thresholding to both the $\ell_0$ regularized problem \eqref{ell_0}, and free-discontinuity problem, 
$$
{\cal{J}}^1_r ({ u}) =  \| T{ u} - g \|^2_{\ell_2(\mathcal K)} + \sum_{i \in \mathcal I} \min \{ |u_i|, r \}$$ 
Indeed, consider the more general class of free-discontinuity-type functionals,
\begin{equation}\label{onegen}
{\cal{J}}^1_{(r,a)} ({ u}) =  \| T{ u} - g \|^2_{\ell_2(\mathcal K)} + a\sum_{i \in \mathcal I} \min \{ |u_i|, r \},
\end{equation}
which have an additional degree of freedom in the scaling parameter $a$ that is present in the discrete MS functional \eqref{discrete}, but which was omitted in the previous section to ease computations.  
\\
\\
\noindent The free-discontinuity functional ${\cal{J}}^1_{(r,a)}$ is related to the $\ell_0$ regularized functional $F^0_{\gamma}$ as follows:

\begin{proposition}
For any sequence $(r_n)$ satisfying $r_n \rightarrow 0$, the functional sequence ${\cal{J}}^1_{(\gamma/r_n, r_n)}$ $\Gamma$-converges to $\mathcal F^0_{\gamma}$ in the metric of $\ell_2^N$.
\end{proposition}

\begin{proof}
We have to verify the conditions for gamma convergence as established in Definition\eqref{GammaDef}.  Note that we can write
$$
{\cal{J}}^1_{(\gamma/r_n, r_n)} = \| T{ u} - g \|^2_{\ell_2(\mathcal K)} + \gamma \sum_{j \in \mathcal I} \min \{ |u_j|/r_n, 1 \},
$$ 
so that $\Gamma$ convergence of the full functional sequence ${\cal{J}}^1_{(\gamma/r_n, r_n)}$ is established if we can show $\Gamma$ convergence of the scalar function $\min \{ |x|/r_n, 1 \}$ to $| x |_0$.  
\begin{itemize}
\item Suppose that $x_n \rightarrow x$.  Our aim is to show that $\liminf_{r_n \rightarrow 0} \min \{ |x^n|/r_n, 1 \} \geq | x |_0$:
\begin{itemize} 
\item If $|x| \geq \epsilon > 0$, then $|x^n| \geq \epsilon/2$ for $n$ sufficiently large, and \\
$ \lim_{r_n \rightarrow 0} \min \{ |x^n|/r_n, 1\} = 1$.
\item If on the other hand $x = 0$, then $ \min \{ |x^n|/r_n, 1\} \geq  \min \{ |x|/r_n, 1\} = 0.$
\end{itemize}

\item On the other hand, it is easily seen that $\lim_{r_n \rightarrow 0} \min \{ |x|/r_n, 1 \} = | x |_0$ for any $x \in \mathbb{R}$, since $$\lim_{r_n \rightarrow 0} \min\{ |x|/r_n, 1\} = \left\{
\begin{array}{ll}
1, & |x| > 0, \\
0, & \textrm{else}.
\end{array}
\right. $$
\end{itemize}
The desired result is established, according to \eqref{GammaDef}.
\end{proof}

\begin{remark}
\emph{ We showed that ${\cal{J}}^1_{(\gamma/r_n, r_n)}$ $\Gamma$ converges to $\mathcal F_{\gamma}^0$ for sequences $r_n \rightarrow 0$.  More generally, ${\cal{J}}^p_{(\gamma/(r_n)^p, r_n)}$ $\Gamma$ converges to $\mathcal F_{\gamma}^0$ for any $p \in [1,\infty)$.  It is not hard to show that the thresholding functions $H_{(p, \gamma/(r_n)^p, r_n)}$ corresponding to ${\cal{J}}^p_{(\gamma/(r_n)^p, r_n)}$, which can be derived  following the proof of Proposition \ref{thmthrs}, converge to the hard thresholding function $H_{\sqrt{\gamma}}$ corresponding to $\mathcal F_{\gamma}^0$.    } 
\end{remark}
\vspace{5mm}
\noindent To ease presentation, we again take $a = 1$ in the sequel, although all of the following analysis extends to the case where $a$ is a free parameter.  Because a convergence analysis of the iteration $\eqref{hard}$ corresponding to hard thresholding has been studied already $\cite{blda??}$, we omit the case $p = 1$ and $r \leq 1/4$ in the sequel.


\section{Convergence of the iterative thresholding algorithm \eqref{map} }
\subsection{Fixation of the discontinuity set}
We prove now that the sequence $({ u}^n)_{n \in \mathbb N}$ defined by
\begin{eqnarray}
{ u}^{n+1} &=& \arg \min_{ u} {{\cal{J}}}^{p,surr}_r ({ u}, { u}^{n}) 
\end{eqnarray} 
or equivalently, according to Proposition $\ref{thmthrs}$, component-wise by
\begin{eqnarray}\label{map3}
u^{n+1}_i &=& H_{(p,r)}([u^n - T^*Tu^n + T^*g]_i), \quad i \in \mathcal I,
\end{eqnarray}
will converge, granted that $p \geq 1$ and $\|T\| < 1$.  To ease notation, we define the operator $\mathbb{H}: \ell_2(\mathcal I) \rightarrow \ell_2(\mathcal I)$ by its component-wise action,
\begin{equation}
[\mathbb{H}(u)]_i := H_{(p,r)}([u - T^*T{ u} + T^*g]_i);
 \label{u}
\end{equation}
so that the iteration $\eqref{map3}$ can be written more concisely in operator notation as
\begin{equation}
u^{n+1} = \mathbb{H}(u^n).
\end{equation}
We omit the dependence of $\mathbb{H}$ on the parameters $p, r$, and the function ${ g}$ for continuity of presentation.  At the core of the convergence proof is the fact that the `discontinuity set', indicated below by ${\cal I}_1^n$, of $u^n$ must eventually fix during the iteration \eqref{map3}, at which point the `free-discontinuity' problem is transformed into a simpler `fixed-discontinuity' problem. 

\begin{lemma}[Fixation of the index set ${\cal I}_1$]
Fix $p \geq 1, r \in \mathbb{R}^+$, and ${ g} \in \ell_2(\mathcal K)$.  Consider the iteration
\begin{equation}
u^{n+1} = \mathbb{H}(u^n)
\end{equation}
and the time-dependent partition of the index set $\mathcal I$ into `small' set
\begin{eqnarray}
\mathcal I_0^n &=&  \{ i \in \mathcal I: |u^n_i| \leq \lambda'(r,p) \} 
\label{lam0}
\end{eqnarray}
and 'large' set
\begin{eqnarray}
\mathcal I_1^n &=& \{ i  \in \mathcal I:  |u^n_i| > \lambda'(r,p) \} 
\label{lam1}
\end{eqnarray}
where $\lambda'(r,p)$ is the position of the jump discontinuity of the thresholding function, as defined in Proposition $\ref{thmthrs}$.  For $N \in \mathbb N$ sufficiently large, this partition fixes during the iteration $u^{n+1} = \mathbb{H} (u^n)$; that is, there exists a set $\mathcal I_0$ such that for all $n \geq N$, $\mathcal I_0^n = \mathcal I_0$ and $\mathcal I_1^n = \mathcal I_1 := \mathcal I \setminus \mathcal I_0$.  
\label{l3}
\end{lemma}

\begin{proof}
By {\it discontinuity} of the thresholding operator $H_{(p,r)}(\lambda)$, each sequence component
\begin{equation}
u^{n}_i = H_{(p,r)} ([u^{n-1} - T^*Tu^{n-1} + T^*g]_i)
\end{equation}
satisfies 
\begin{enumerate}
\item[(a)] $| u^n_i | \leq \lambda'(r,p) - \delta(r,p) < \lambda'(r,p) $, if $i \in \mathcal I_0^n$, or
\item[(b)] $| u^n_i | > \lambda'(r,p)$, if $i \in \mathcal I_1^n$.
\end{enumerate}
Thus, $|u^{n+1}_i - u^n_i| \geq \delta(r,p)$ if $i \in \mathcal I_0^{n+1} \bigcap \mathcal I_1^n$, or vice versa if $i \in \mathcal I_0^{n} \bigcap \mathcal I_1^{n+1}$.  At the same time, Lemma $\ref{l2}$ implies
\begin{equation}
| u^{n+1}_i - u^n_i | \leq \|u^{n+1} - u^n \|_{\ell_2(\mathcal I)} \leq \epsilon, 
\label{contra}
\end{equation}   
once $n \geq N(\epsilon)$, and $\epsilon > 0$ can be taken arbitrarily small.  In particular, $\eqref{contra}$ implies that $\mathcal I_0$ and $\mathcal I_1$ must be fixed once $n \geq N(\epsilon)$ and $\epsilon < \delta(r,p)$.
\end{proof}

\noindent After fixation of the index set $\mathcal I_0 = \{i \in \mathcal I: |u_i^n| \leq \lambda'(r,p)\}$, $\mathbb{H}(u^n) =\mathbb{U}_{\mathcal I_0}(u^n)$ and $\mathbb{U}_{\mathcal I_0}$ is an operator having component-wise action{, for $p>1$,}
\begin{eqnarray}
[\mathbb{U}_{\mathcal I_0} u]_i &=& \left\{
\begin{array}{ll} F_p^{-1}([(I - T^*T)u + T^*g]_i), & \textrm{if } i \in \mathcal I_0 \\  
\big( (I - T^*T)u + T^*g \big)_i, & \textrm{if } i \in \mathcal I_1
\end{array} \right.
\label{separate}
\end{eqnarray}

Here, as in Proposition $\ref{thmthrs}$, the function $F_p^{-1}$ is the inverse of the function $F_p(t) = t + \frac{p}{2}\sgn{t}|t|^{p-1}$.  Again, for ease of presentation, we omit the dependence of $\mathbb{U}_{\mathcal I_0}$ on the parameters $p,r$, and ${ g}$. { For $p=1$ the description is similar, and in general,} one easily verifies the equivalence
\begin{eqnarray}
\mathbb{U}_{\mathcal I_0} ({ v}) =  \arg \min_{ u \in \ell_2(\mathcal I)} {\cal{J}}_{\mathcal I_0}^{p,surr} ({ u, v})
\label{equiv}
\end{eqnarray}
where ${\cal{J}}_{\mathcal I_0}^{p,surr}$ is a surrogate for the {\it convex} functional,
\begin{equation}
 {\cal{J}}_{\mathcal I_0}^p({ u}) :=  \| T{ u} - { g} \|_{\ell_2(\mathcal K)}^2 +  \sum_{i \in \mathcal I_0} |u_i|^p. 
\end{equation}
That is, fixation of the index set $\mathcal I_0$ implies that the sequence $( u^n)_{n \in \mathbb N}$ has become constrained to a subset of $\ell_2(\mathcal I)$ on which the map $\mathbb{H}$ agrees with a map $\mathbb{U}_{\mathcal I_0}$, associated to the convex functional ${\cal J}_{\mathcal I_0}^p$.  As we will see, this implies that  the nonconvex functional ${\cal J}_r^p$ behaves \emph{locally} like a convex functional in neighborhoods of fixed points $u = \mathbb{H}(u)$, including the global minimizers of ${\cal J}_r^p$.

\subsection{On the nonexpansiveness and convergence for $T$ injective}
Given that $\mathbb{H} ({ u}^n) = \mathbb{U}_{\mathcal I_0} ({ u}^n)$ after a finite number of iterations, we can use well-known tools from convex analysis to prove that the sequence $({ u}^n)_{n \in \mathbb N}$ converges.
If the operator $T^*T: \ell_2(\mathcal I) \rightarrow  \ell_2(\mathcal K) $ is invertible, or, equivalently, if the operator $T$ maps onto its range and has a trivial null space -- as, for example, does the discrete pseudoinverse $D_{h}^{\dagger}$ in the 1D Mumford-Shah approximation -- then the mapping $\mathbb{U}_{\mathcal I_0}$ has the nice property of being a contraction mapping, so that a direct application of the Banach fixed point theorem ensures exponential convergence of the sequence $( u^n)_{n \in \mathbb N}$ after fixation of the index sets. \begin{theorem} \label{contract}
Suppose $T: \ell_2(\mathcal I) \rightarrow \ell_2(\mathcal K)$ maps onto $\ell_2(\mathcal K)$ and has a trivial null space.  Let $\delta > 0$ be a lower bound on the spectrum of $T^* T$.
 Then the sequence 
\begin{equation}
{ u^{n+1}} = \mathbb{H} ({ u^n}), 
\end{equation}
as defined in $\eqref{u}$, is guaranteed to converge in norm.  In particular, after a finite number of iterations $N \in \mathbb N$, this mapping takes the form
\begin{equation}
{ u}^{N+m} =  \mathbb{U}_{\mathcal I_0}^m ({{ u}}^N), \quad m \in \mathbb N \setminus \{0\},
\end{equation} 
and the sequence $( u^n)_{n \in \mathbb N}$ converges to the unique fixed point $\bar{u}$ of the map $\mathbb{U}_{\mathcal I_0}$.  
Moreover, after fixation of of the index set $\mathcal I_0$, the rate of convergence becomes exponential:
\begin{equation}
\| { u}^{N+m} - { \bar{u}} \|_{\ell_2(\mathcal I)} =\| \mathbb{U}_{\mathcal I_0}^m ({ u}^{N}) - \mathbb{U}_{\mathcal I_0}^m ({ \bar{u}}) \|_{\ell_2(\mathcal I)}  \leq (1-\delta)^m\| {{ u}}^N - { \bar{u}} \|_{\ell_2(\mathcal I)}, \quad m \in \mathbb N \setminus \{0\}.
\end{equation} 
\end{theorem}
The proof of Theorem  \ref{contract} is deferred to the Appendix.  

\subsection{Convergence for general operators $T$}
Unfortunately, if $T^*T$ is not invertible (that is, if $\delta = 0$ belongs to its nonnegative spectrum), then the map $\mathbb{U}_{\mathcal I_0}$ is not necessarily a contraction, and we can no longer apply the Banach fixed point theorem to prove convergence of the sequence $( u^n)_{n \in \mathbb N}$.  
However, as long as $\|T\| < 1$, we observe by following the proof of Theorem \eqref{contract} that $\mathbb{U}_{\mathcal I_0}$ is still \emph{non-expansive}, meaning that for all $v, v' \in \ell_2(\mathcal I)$, $\| \mathbb{U}_{\mathcal I_0} (v) - \mathbb{U}_{\mathcal I_0} (v') \|_{\ell_2(\mathcal I)} \leq \| v - v' \|_{\ell_2(\mathcal I)}$.  The following Opial's theorem $\cite{op67}$, here reported adjusted to our notations and context, gives sufficient conditions under which non-expansive maps admit convergent successive iterations:
\begin{theorem}[Opial's Theorem]
Let the mapping $\mathbb{A}$ from $\ell_2(\mathcal I)$ to $\ell_2(\mathcal I)$ satisfy the following conditions:
\begin{enumerate}
\item $\mathbb{A}$ is asymptotically regular: for all $v \in \ell_2(\mathcal I)$, $\| \mathbb{A}^{n+1} (v) - \mathbb{A}^n (v) \|_{\ell_2(\mathcal I)} \rightarrow  0$ for $n \to \infty$;
\item $\mathbb{A}$ is non-expansive: for all $v, v' \in \ell_2(\mathcal I)$, $\| \mathbb{A} (v) - \mathbb{A} (v') \|_{\ell_2(\mathcal I)} \leq \| v - v' \|_{\ell_2(\mathcal I)}$;
\item the set $\operatorname{Fix}(\mathbb A)$ of the fixed points of $\mathbb{A}$ in $\ell_2(\mathcal I)$ is not empty.
\end{enumerate}
Then, for all $v \in \ell_2(\mathcal I)$, the sequence $( \mathbb{A}^n (v))_{n \in \mathbb{N}}$ converges weakly to a fixed point in $\operatorname{Fix}(\mathbb A)$. 
\label{opial}
\end{theorem}

\noindent In fact, we already know that $\mathbb{U}_{\mathcal I_0}$ is asymptotically regular, in addition to being nonexpansive -  this follows by application of Lemma $\ref{l1}$ and Lemma $\ref{l2}$ to the functional ${\cal J}_{\mathcal I_0}^p$.  Thus, in order to apply Opial's theorem, it remains only to show that $\mathbb{U}_{\mathcal I_0}$ has a fixed point; that is, that there exists a point $\bar{u} \in \ell_2(\mathcal I)$ for which $${ \bar{u}} = \mathbb{U}_{\mathcal I_0}({ \bar{u}}). $$
In more detail, we must prove the existence of a vector ${ \bar{u}} \in \ell_2(\mathcal I)$ satisfying
\begin{eqnarray}
\bar{u}_i &=& \left\{
\begin{array}{ll} F_p^{-1}([(I - T^*T)\bar u + T^*g]_i), & \textrm{if } i \in \mathcal I_0 \\  
\big( (I - T^*T)\bar{u} + T^*g \big)_i, & \textrm{if } i \in \mathcal I_1
\end{array} \right.
\label{seperate2}
\end{eqnarray}
The following lemma gives a simple yet useful characterization of points satisfying the fixed point relation $\eqref{seperate2}$:
\begin{lemma}
Suppose $p > 1$.  A vector ${ \bar{u}} \in \ell_2(\mathcal I)$ satisfies the fixed point relation ${ \bar{u}} = \mathbb{U}_{\mathcal I_0}({ \bar{u}})$ if and only if
\begin{equation}
\label{wha}
 \big[ T^*(g - T\bar{u}) \big]_i  = \left\{
\begin{array}{ll}
0, &  i \in \mathcal I_1 \\  
F_p(\bar{u}_i) - \bar{u}_i, & i \in \mathcal I_0,
\end{array}
\right.
\end{equation}
Alternatively, if $p = 1$ and $r \geq 1/4$,  ${ \bar{u}} = \mathbb{U}_{\mathcal I_0}({ \bar{u}})$ is satisfied if and only if
\begin{equation}
   \left\{
\begin{array}{ll}\label{wah}
\big[ T^*(g - T\bar{u}) \big]_i \in [-1/2, 1/2] , & i \in \mathcal I^a_0, \\  
\big[ T^*(g - T\bar{u}) \big]_i = 1/2\sgn{\bar{u}_i}, & i \in \mathcal I^b_0, \\
\big[ T^*(g - T\bar{u}) \big]_i =0, &  i \in \mathcal I_1,
\end{array}
\right.
\end{equation}
where in $\eqref{wah}$, the index set $\mathcal I_0$ is split into
\begin{itemize}
\item $\mathcal I^a_0 =  \left \{ i \in \mathcal I_0: |\bar{u}_i| \leq 1/2 \right \}$, and
\item $\mathcal I^b_0 =  \left \{ i \in \mathcal I_0: 1/2 < |\bar{u}_i| \leq r + 1/4 \right \}$.
\end{itemize}
\label{l4}
\end{lemma}
Again, recall the notation $ F_p(t) = t + \frac{p}{2}\sgn{t}{|t|}^{p-1}$, and observe that the fixed point relation \eqref{wha} has a very simple expression when $p = 2$.   The proof of Lemma \ref{l4} is given in the Appendix.
\\
\\
The fixed point characterization of Lemma $\ref{l4}$ will be crucial in the following theorem that ensures the existence of a fixed point  $\bar{{ u}} = \mathbb{U}_{\mathcal I_0} (\bar{{ u}})$.  We remind the reader that until now, all of the results of Section $1.3$ remain valid in the infinite-dimensional setting $| {\cal I} | = \infty$.  From this point on, however, certain results will only hold in finite dimensions; for clarity, we will account each such situation explicitly. 

\begin{proposition}
In finite dimensions $| {\cal I} | < \infty$, then there exist (global) minimizers of the convex functional,
\begin{equation}
 {\cal{J}}_{\mathcal I_0}^p({ u}) =  \| T{ u} - { g} \|_{\ell_2(\mathcal K)}^2 +  \sum_{i \in \mathcal I_0} |u_i|^p,
 \label{lp} 
\end{equation}
for all $p \geq 1$, and any minimizer $\bar{{ u}}$ of ${\cal{J}}_{\mathcal I_0}^p$ satisfies the fixed point relation $\bar{{ u}} = \mathbb{U}_{\mathcal I_0}(\bar{{ u}})$.  Restricted to the range $1 \leq p \leq 2$, the statement is true also in the limit $| {\cal I} | = \infty$.  
\label{l5}
\end{proposition}
\begin{proof}
In the finite-dimensional setting, minimizers necessarily exist for all $p \geq 1$ according to Proposition \ref{firstlemma}.   We now consider the general case. Consider the unique decomposition $u = u_0 + u_1$ into a vector $u_0$ supported on $\mathcal I_0$ and another $u_1$ supported on $\mathcal I_1$, i.e., the vectors $u_0 \in \ell_2^{\mathcal I_0}(\mathcal I) := \{u \in \ell_2(\mathcal I): u_i=0, \quad i \in \mathcal I_1\}$ and $u_1 \in \ell_2^{\mathcal I_1}(\mathcal I) := \{u \in \ell_2(\mathcal I): u_i=0, \quad i \in \mathcal I_0\}$. 
Let $\mathcal P: u \rightarrow u_1$ and $\mathcal P^{\perp} = \mathcal I - \mathcal P: u \rightarrow u_0$ denote the orthogonal projections onto the subspaces  $\ell_2^{\mathcal I_1}(\mathcal I)$ and $\ell_2^{\mathcal I_0}(\mathcal I)$, respectively.   Consider the operators $T_0 = T \mathcal P^{\perp}$ and $T_1 = T \mathcal P$; note that clearly $T = T_0 + T_1$ is satisfied.
The functional \eqref{lp} can be re-written with this decomposition according to
\begin{equation}
 {\cal{J}}_{\mathcal I_0}^p({ u_0} + { u_1}) =  \| T_0 u_0 + T_1 u_1 - g \|_{\ell_2(\mathcal K)}^2 +  \|u_0\|_{\ell_p^{\mathcal I_0}(\mathcal I)}^p
\end{equation}  
where $\|z \|_{\ell_p^{\mathcal I_0}(\mathcal I)}:= \left (\sum_{i \in \mathcal I_0} |z_i|^p\right)^{1/p} $ is the $\ell_p$-norm on vectors supported on $\mathcal I_0$.
\\
\\
\noindent Let $\mathcal P_1$ be the orthogonal projection onto the range of $T_1$ in $\ell_2(\mathcal K)$ (not to be confused with $\mathcal P$, which operates on the space $\ell_2(\mathcal I)$) and let $\mathcal P_1^{\perp} = \mathcal I  - \mathcal P_1$ be the orthogonal projection in $\ell_2(\mathcal K)$ onto the orthogonal complement of the range of $T_1$.  Then, fixing $u_0 \in  \ell_2^{\mathcal I_1}(\mathcal I)$, the vector $\mathcal P_1(g - T_0 u_0) \in \operatorname{range}(T_1) \subset \ell_2(\mathcal K)$ is the solution to the minimization problem
\begin{equation}
\mathcal P_1(g - T_0 u_0)= \arg \min_{v \in \operatorname{range}(T_1)} \| v - (g - T_0 u_0) \|^2_{\ell_2(\mathcal K)}, 
\end{equation}
so that minimizers of the functional ${\cal{F}}: \ell_2^{\mathcal I_0}(\mathcal I) \rightarrow \mathbb{R}^+$ defined by
\begin{eqnarray}
 {\cal{F}}(v) &=&  \| T_0 v + \mathcal P_1(g - T_0 v) - g \|^2_{\ell_2(\mathcal K)} + \|v \|_{\ell_p^{\mathcal I_0}(\mathcal I)}^p \nonumber \\
 &=& \| K v - y \|^2_{\ell_2(\mathcal K)} + \|v \|_{\ell_p^{\mathcal I_0}(\mathcal I)}^p
 \label{DDD}
\end{eqnarray}
with $K := \mathcal P_1^{\perp} T_0$, and $y := \mathcal P_1^{\perp} g$, will yield minimizers of ${\cal{J}}_{\mathcal I_0}^p$.   Functionals of the form $\eqref{DDD}$ were studied in $\cite{DDD}$; there, it is shown that as long as $1 \leq p \leq 2$, ${\cal{F}}$ has minimizers, and any minimizer $\bar{v}$ can be characterized by the fixed point relation
\begin{equation}
\bar{v}_i = F_p^{-1}([(I - K^*K)\bar{v} + K^*y]_i), \quad i \in \mathcal I_0;
\label{usual}
\end{equation}
(recall that $F_p^{-1}$ is the inverse of the function $F_p(t) = t + \frac{p}{2}\sgn t |t|^{p-1}$).
\\
 In the finite-dimensional setting $| \cal I | < \infty$, the Euler-Lagrange equations corresponding to minimizers of the convex functional ${\cal{F}}$ as in \eqref{DDD} imply the same fixed point relation \eqref{usual} also, for all $p \geq 1$. 
\\
By Lemma $\ref{l4}$, the characterization \eqref{usual} is equivalent to the condition
\begin{itemize}
\item $p > 1$:
\begin{equation}
 \big[ K^*(y - K\bar{v}) \big]_i  =  \frac{p}{2}\sgn{\bar{v}_i} |\bar{v}_i|^{p-1}, 
 \label{ey}
\end{equation}
\item $p = 1$: 
\begin{equation}
 \left\{
\begin{array}{ll}
 \big[ K^*(y - K\bar{v}) \big]_i  \in [-1/2, 1/2], & \textrm{if } |\bar{v}_i| \leq 1/2,  \\  
 \big[ K^*(y - K\bar{v}) \big]_i  = 1/2\sgn{\bar{v}_j}, & \textrm{if } 1/2 < |\bar{v}_i| \leq r + 1/4.
\end{array}
\right., \quad i \in \mathcal I_0.
\label{ew}
\end{equation}
\end{itemize}
Making the identification $\bar{u}_0 = \bar{v}$ and $T_1\bar{u}_1 = \mathcal P_1(g - T_0 \bar{v})$, and rewriting $K = \mathcal P_1^{\perp} T_0$, and $y = \mathcal P_1^{\perp} g$, the relations $\eqref{ey}$ and $\eqref{ew}$ imply the full fixed point characterization in Lemma \ref{l4}. 
\end{proof}

\begin{remark}
\emph{The restriction $p \leq 2$ that is necessary for the results of this paper in the infinite dimensional setting $| \mathcal I | = \infty$ was only used in the proof of Theorem $\ref{l5}$, where it comes from $\cite{DDD}$ and is needed there to prove the existence of minimizers of functionals ${\cal{F}}$ of the form $\eqref{DDD}$.   If that proof can be extended to functionals of the form $\eqref{DDD}$ for general $p \geq 1$, then the restriction $p \leq 2$ can be dropped in the current paper.  
For instance, if we additionally require that $T$ is a bounded operator from $\ell_p(\mathcal I)$ to $\ell_2(\mathcal I)$ for $1 \leq p < \infty$ then the existence of minimizers would be guaranteed also for $1 \leq p < \infty$ and $| \mathcal I | = \infty$. In this case we could consider a minimizing sequence $(v^k)$ of $\mathcal F$, which is necessarily bounded in $\ell_p$. Therefore, there exists a subsequence $(v^{k_h})$ which weakly converges in $\ell_p$ to a point $v^*$. This also implies the weak convergence of the sequence $K v^{k_h}$ in $\ell_2$; note that $\langle K v^{k_h}, w \rangle_{\ell_2 \times \ell_2} = \langle v^{k_h}, K^* w \rangle_{\ell_p \times \ell_{p'}}$, for $1/p+1/p' =1$. By Fatou's lemma we obtain $\mathcal F (v^*) \leq \lim \inf_h \mathcal F (v^{k_h})$ and $v^*$ is a minimizer of $\mathcal F$.
However, we still require that $p \geq 1$ for the proof of Proposition $\ref{thmthrs}$ and for the results of the next section to hold.  }
\end{remark}

\noindent Combining the results from this section, we obtain:

\begin{theorem}\label{mainth}
Suppose $1 \leq p \leq 2$.  Starting from any ${ u}^0$ satisfying ${\cal{J}}^p_r({ u}^0) < \infty$, the sequence $({ u}^n)_{n \in \mathbb N}$ defined by ${ u}^{n+1} = \mathbb{H}^n ({ u}^0)$ as in \eqref{u} will converge weakly to a vector $\bar{{ u}} \in \ell_2(\mathcal I)$ that satisfies the fixed point condition,
\begin{enumerate}
\item $|\bar{u}_i| \geq \lambda'(r,p)$, if $i \in \mathcal I_1 = \{j \in \mathcal I: |\bar{u}_j| > r\}$
\item $|\bar{u}_i| \leq F_p^{-1}(\lambda'(r,p))$, { for $p>1$,} if $i \in \mathcal I_0 = \{j \in \mathcal I: |\bar{u}_j| \leq r\}$, and
\item 
\begin{enumerate}
\item If $p > 1$: 
\begin{equation}
\label{wha2}
 \big[ T^*(g - T\bar{u}) \big]_i  = \left\{
\begin{array}{ll}
0, & \textrm{ if }  |\bar{u}_i| \geq  \lambda'(r,p) \\
F_p(\bar{u}_i) - \bar{u}_i, & \textrm{ if } |\bar{u}_i| \leq  \lambda'(r,p) - \delta(r,p)
\end{array}
\right.
\end{equation}
\item If $p = 1$ and $r \geq 1/4$: 

\begin{equation}
  \left\{
\begin{array}{ll}
\big[ T^*(g - T\bar{u}) \big]_i  \in [-1/2, 1/2] , & |\bar{u}_i| \leq 1/2 \\  
\big[ T^*(g - T\bar{u}) \big]_i  = 1/2\sgn{\bar{u}_i}, & 1/2 < |\bar{u}_i| \leq r - 1/4. \\
\big[ T^*(g - T\bar{u}) \big]_i  = 0, &  |\bar{u}_i| > r + 1/4.
\label{wah2}
\end{array}
\right.
\end{equation}
\end{enumerate}
\end{enumerate}
If the index set $| \cal I | < \infty$ is finite dimensional, the theorem holds for all $p \geq 1$.  
\label{bigthm2}
\end{theorem}
\begin{proof}
By Lemma $\ref{l3}$, the map $u^{n+1} = \mathbb{H}(u^n)$ becomes equivalent to a map of the form
$u^{n+1}  = \mathbb{U}_{\mathcal I_0} (u^n)$ after a finite number of iterations $N \in \mathbb N$.  By Lemma $\ref{l3}$ and Proposition $\ref{thmthrs}$, the subset $\mathcal I_0 \subset \mathcal I$ separates $\mathcal I$ in the sense that, for all $n \geq N$, 
\begin{itemize}
\item $|u_i^n| < F_p^{-1}(\lambda'(r,p))$, if $i \in \mathcal I_0$,
\item $|u_i^n| > \lambda'(r,p)$, if $i \in \mathcal I_1 = \mathcal I \setminus \mathcal I_0$.
\end{itemize}
That the sequence $(u^n)_{n \in \mathbb N}$ converges to a fixed point of the map $\mathbb{U}_{\mathcal I_0} $ follows from Opial's theorem applied to the map $\mathbb{U}_{\mathcal I_0}$:   
\begin{enumerate}
\item the asymptotic regularity of $\mathbb{U}_{\mathcal I_0}$ is a consequence of Lemmas $\ref{l1}$ and $\ref{l2}$;
\item the nonexpansiveness of $\mathbb{U}_{\mathcal I_0}$ follows from the proof of Theorem \eqref{contract}, and
\item Theorem $\ref{l5}$ guarantees that the set of fixed points of $\mathbb{U}_{\mathcal I_0}$ in $\ell_2(\mathcal{I})$ is nonempty.
\end{enumerate}
The limit $\bar{u}$ of the sequence $(u^n)$ will satisfy the fixed point conditions of Lemma \ref{l4}.  Since weak convergence implies component-wise convergence, it follows for all $i \in \mathcal I_0$ that
\begin{eqnarray}
|\bar{u}_i| &=& \lim_{n \rightarrow \infty} |u^n_i| \nonumber \\
&\leq& \lambda'(r,p) - \delta(r,p)
\end{eqnarray}
and the respective lower bound $|u_i^n| \geq \lambda'(r,p)$ holds analogously for $i \in \mathcal I_1$.
\end{proof}

\section{On minimizers of ${\cal J}^p_r$}

We are now in a position to explore the relationship between limit vectors $\bar{u}$ of the iterative thresholding algorithm \eqref{u} and minimizers of the free-discontinuity functional  ${{\cal{J}}}_r^p$ \eqref{generalform}.   As a first but important result in this direction, 

\begin{theorem}
A point $\bar{u}$ satisfying the fixed point relation of Theorem $\ref{bigthm2}$ is a local minimizer of the functional ${\cal{J}}^p_r$ defined in $\eqref{generalform}$. 
\label{mintheorem} 
\end{theorem}
The proof of Theorem \ref{mintheorem} is omitted at present but can be found in the Appendix.   This result should not be surprising, however.   Due to the separation of the entries of any fixed point $\bar u$, such that $\bar u_i < r < \bar u_j$ for $i \in \mathcal I_0$ and $j \in \mathcal I_1$,  we have also $\mathcal I_0 \equiv \{i \in \mathcal I: |u_i| \leq r\}$ and $\mathcal I_1 \equiv \{j \in \mathcal I: |u_j| > r\}$ for all $u \in B(\bar u,\varepsilon(r))$, where $B(\bar u,\varepsilon(r))$ is a ball around an equilibrium point $\bar u$ of radius $\varepsilon(r)>0$ sufficiently small.  On this neighborhood $B(\bar u,\varepsilon(r))$ of $\bar{u}$, the functional $\mathcal J_r^p$ is convex.  Since $\bar{u}$ is obtained as the limit of a sequence $(u^n)$ in $B(\bar u,\varepsilon(r))$ for which the sequence ${\cal J}^p_r(u^n)$ is nonincreasing, one would expect that $\bar{u}$ minimizes ${\cal J}^p_r(u^n)$ within this neighborhood.  
\\
\\
More surprising is that global minimizers of ${\cal J}^p_r$ are also fixed points, as shown in the following theorem.   Even though the existence of such minimizers is only guaranteed in the finite-dimensional setting (see Proposition \ref{firstlemma}), the following result is not restricted as such.

\begin{theorem}[Global minimizers of ${\cal J}_r^p$ are fixed points $\bar{u} = \mathbb{H}(\bar{u})$]
{ Any global minimizer $u^*$ of ${\cal{J}}^p_r$ satisfies the fixed point condition of the map $\mathbb{H}$ that is given in Theorem \ref{bigthm2}. } 
\label{globalmin}
\end{theorem}

The proof of Theorem \ref{globalmin} is rather long and we defer it to the Appendix.    We reiterate once more that on a ball $B(\bar u,\varepsilon(r))$ around an equilibrium point $\bar u$ of radius $\varepsilon(r)>0$ sufficiently small, the functional $\mathcal J_r^p$ is convex; following the proof of Theorem \ref{globalmin}, we see that $\mathcal J_r^p$ is in fact \emph{strictly} convex whenever $\bar u=u^*$ is a global minimizer, since the restriction of $T$ to the subspace $\ell_2^{\mathcal I_1}(\mathcal I) \subset \ell_2(\mathcal I)$ of vectors with support in $\mathcal I_1$ must be an injective operator in this case. Hence a global minimizer is necessarily an isolated minimizer, whereas we cannot ensure the same property for local minimizers if $T$ has a nontrivial null-space; in this case, local minimizers may form continuous sets, as it is shown in the bottom-right box of Figure \ref{patterns}. We conclude the following remark.
\begin{corollary}\label{isolated}
Minimizers of $\mathcal J_r^p$ are isolated.  
\end{corollary}  

\section{Numerical Experiments}

\subsection{Dynamical systems, stability, and equilibria}
Iterative thresholding algorithms have a natural interpretation as discrete-time dynamical systems with nonsmooth right-hand-side, and can be associated to continuous dynamical systems of the type:
\begin{eqnarray*}
\dot u(t) &=& F(u(t),t)\\&=&\tau \left ( H_{(p,r)}( u(t) +  T^*(g - T u(t))) - u(t) \right), \quad t \geq t_0, \quad \tau >0.
\end{eqnarray*}
The study of the existence, uniqueness, stability, and long-time behavior  of these ODE's is of fundamental interest in order to clarify also the stability properties of iterative thresholding algorithms. Indeed, other than soft-thresholding iterations \cite{DDD}, the corresponding right-hand-side is not Lipschitz continuous and can even be  discontinuous, as is the case for free-discontinuity problems. In  \cite{br88,fi88} conditions are established for the existence, uniqueness, and continuous dependence on the initial data (at finite time) { of solutions of dynamical systems with discontinuous right-hand-side}. However, very little is known about long-time properties of such dynamical systems and about the nature of their equilibrium points. \\

\begin{figure}
\[\begin{array}{cc}
  \centering
  \includegraphics[width=7cm]{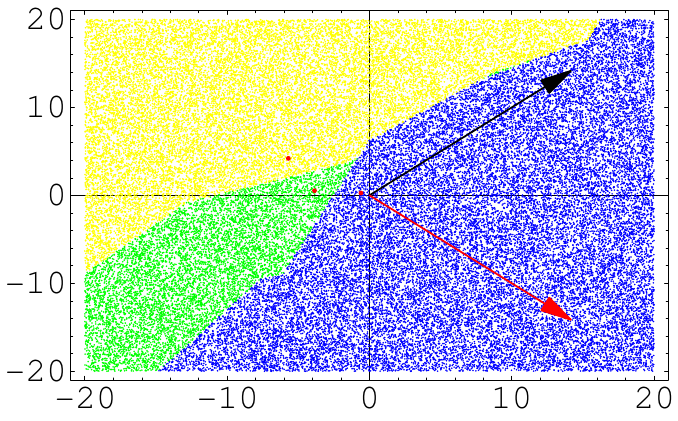} &  
  \includegraphics[width=7cm]{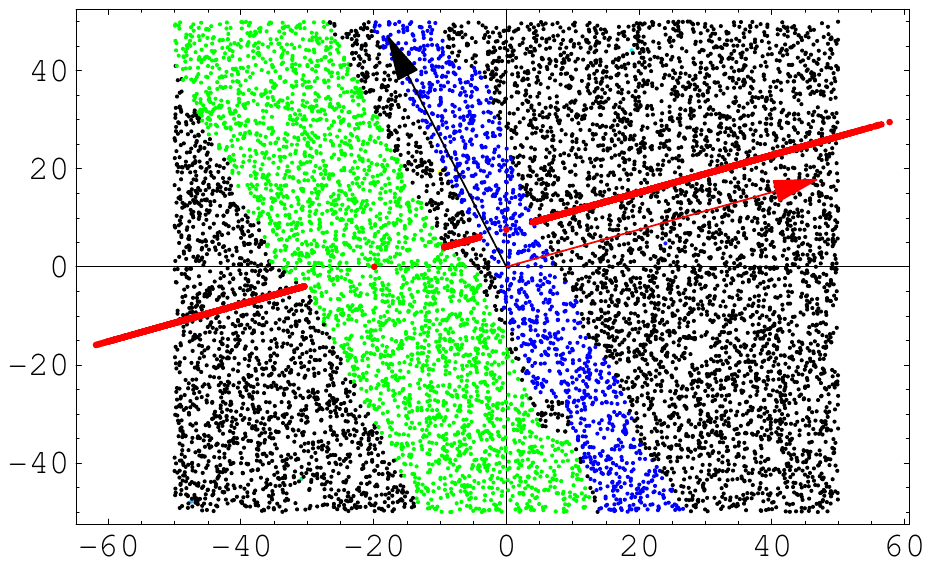} 
\end{array}\]
\caption{We show patterns in $\mathbb R^2$ formed by initial points $u^0$ colored according to the corresponding equilibria computed as limits of the iterative thresholding algorithm \eqref{map3}. For invertible $2\times 2$ squared matrices $T$, the equilibria are isolated and the region of initial points for which \eqref{map3} converges to a given equilibrium point do partition the space into sets which might be disconnected. Structures of the partition generated by a particular matrix $T$ is exemplified on the left, and in the right box we show a pattern related to iterations where the  $2\times 2$ squared matrix $T$ has nontrivial null-space. We can see again that global minimizers are isolated and correspond to the points on the axes, whereas local minimizers are continuously distributed along an affine space generated by the kernel of $T$. It is not difficult to show that this structure always occurs for such matrices.} 
\label{patterns}
\end{figure}

\noindent For several continuous thresholding functions, such as the ones introduced in \cite{DDD,fora06,fora07}, one can easily show, for instance by means of $\Gamma$-convergence arguments, that equilibrium points depend continuously on the parameters of the thresholding, see, e.g., \cite[Theorem 5.1]{fora07}. Nevertheless, for discontinuous thresholding functions $H_{(p,r)}$ such as those studied in this paper, sudden bifurcation phenomena and instabilities do appear in general. Figure \ref{patterns} shows that multiple equilibrium points can exist for these thresholding operators and their number may depend discontinuously on the thresholding shape parameters. Moreover, as established in Theorem \ref{globalmin}, global minimizers of $\mathcal J_r^p$ are always stable equilibria and isolated points, while local minimizers can be unstable equilibria and form a continuous set, as shown  in the bottom-right box of Figure \ref{patterns}.

\subsection{Denoising and segmentation of 1-D signals and digital images}
\label{denoising}

In this subsection, we are concerned with numerical experiments in the use of an iterative thresholding algorithm for the minimization of
\begin{equation}
{\cal{J}}_{r,\gamma}^2 ({u}) :=  \| D_h^\dagger u  - g \|^2_{\ell_2} +  \gamma \sum_{i=1}^N \min \{ u_i^2, r^2 \},
\end{equation}
modelling problems of denoising and segmentation.
\begin{figure}[htp]
\begin{center}
\includegraphics[width=6in]{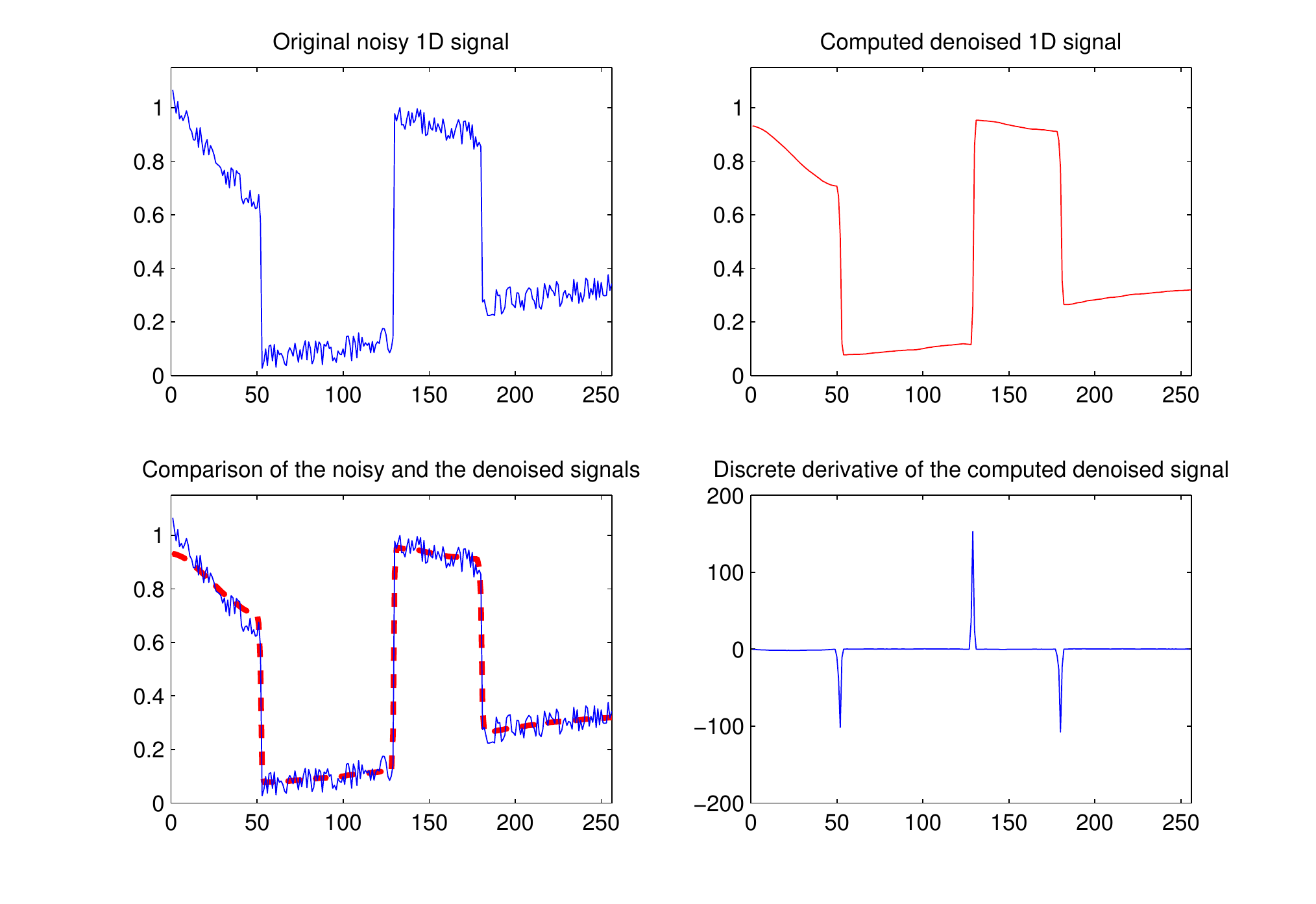}
\caption{We show the application of the  iterative thresholding algorithm \eqref{map3} for the classical denoising problem of 1-D signals where $K=I$ in \eqref{regMS}, and hence $T = D_h^\dagger$. The thresholding parameters used for the numerics are $r=2.2$ and $\gamma=0.002$.}\label{1Dsignal.pdf}
\end{center}
\end{figure}

\begin{figure}[htp]
\begin{center}
\includegraphics[width=6in]{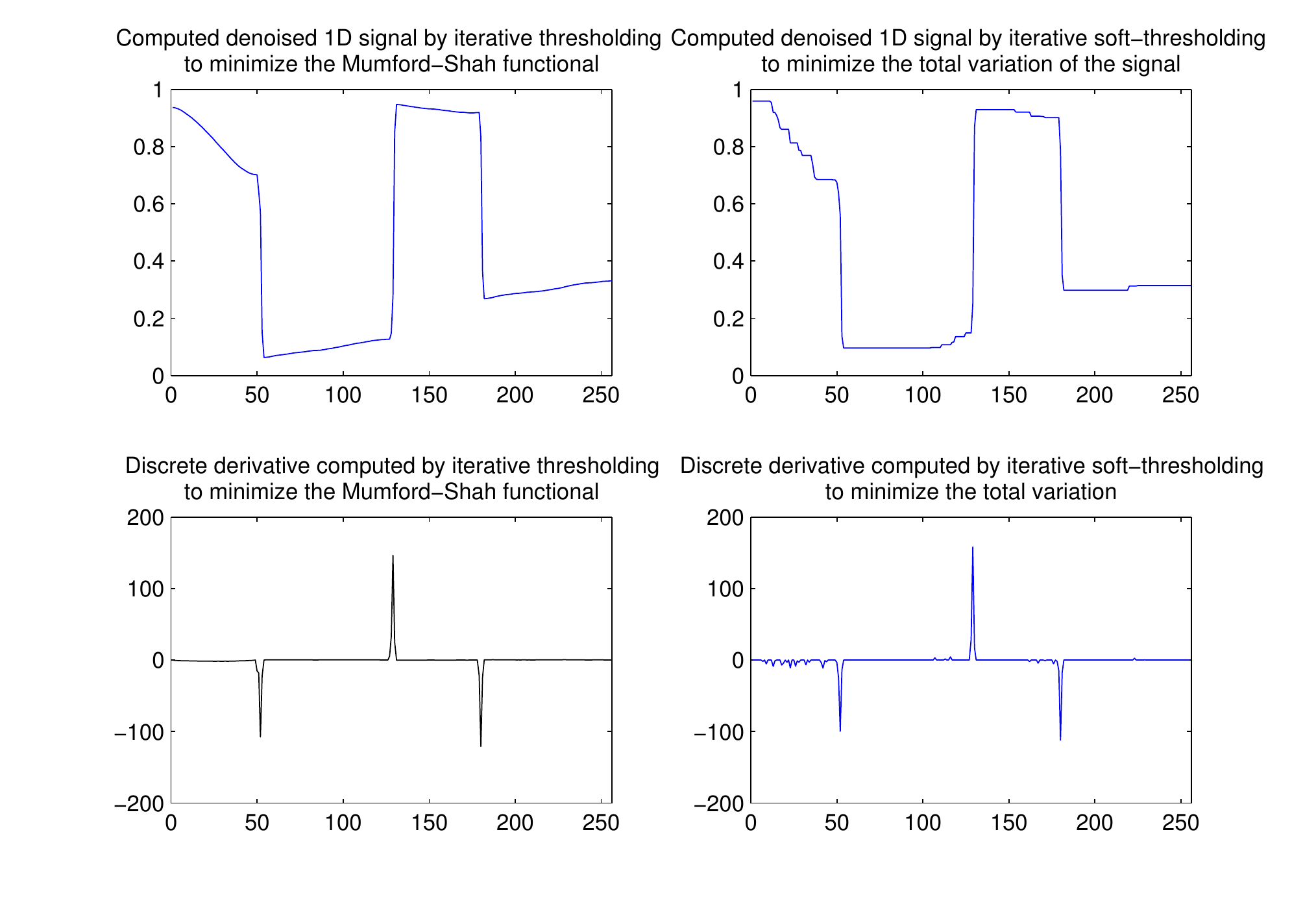}
\caption{A comparison of the denoising of the signal in Figure \ref{1Dsignal} by means of the algorithm \eqref{map3} and by iterative soft-thresholding \cite{DDD} applied to discrete derivatives. We can appreciate how the algorithm \eqref{map3} promotes piecewise smooth solutions, whereas the iterative soft-thresholding promotes the total variation minimization with the introduction of a `staircase effect'.  The thresholding parameters used for the numerics are $r=2.2$ and $\gamma=0.002$ for \eqref{modH}, and $\gamma=0.002$ for the soft-thresholding  \eqref{softthrs}.}\label{MS-TV}
\end{center}
\end{figure}

Note that we introduced an additional regularization parameter $\gamma>0$ which has the sole effect of modifying the thresholding function $H_{(2,r,\gamma)}$ as follows
\begin{equation}
\label{modH}
H_{(2,r,\gamma)}(z) = \left \{ 
\begin{array}{ll}
\frac{1}{1+\gamma} z, & |z| \leq \frac{\sqrt{1 + \gamma}r}{\sqrt{\gamma}}\\
z, & \mbox{ otherwise.}
\end{array}
\right.
\end{equation}
This thresholding function can be again easily computed by means of an argument similar to the proof of Proposition \ref{thmthrs}. In Figure \ref{1Dsignal} we show the results of applications of the iterative thresholding algorithm \eqref{map3}. In Figure \ref{MS-TV} we show a comparison of the use of the thresholding $H_{(2,r,\gamma)}$ and the soft-thresholding $S_\gamma$ (see its definition in \eqref{softthrs}); the former promotes the minimization of the Mumford-Shah constraint $MS$ and piecewise smooth solutions, whereas the latter promotes the minimization of a total variation constraint \cite{ROF}, which is also well-known to produce (almost) piecewise constant solutions with a perhaps unwanted `staircase effect'; see also \cite[Section 4]{chli97} for details.

\subsection{Inverse problems}

As already mentioned in Subsection \ref{MS4invprob} the Mumford-Shah term  $MS ( u ) = \int_{\Omega \setminus S_u} | \nabla u |^2  +\beta \mathcal{H}^{d-1}(S_u)$ is also used for regularizing inverse problems involving operators $T$ which are not boundedly invertible. In this section we present a numerical experiment on the use of the algorithm \eqref{map3} 
 for 1D interpolation (Figure \ref{1Dinterpolation}).
  In this case the operator $T$ is a multiplier by a characteristic function of a subdomain, i.e., $T u := \chi_D \cdot u$, for $D \subset \Omega$; see \cite{essh02} for other numerical examples previously obtained with the Mumford-Shah regularization.

\begin{figure}[htp]
\begin{center}
\includegraphics[width=6in]{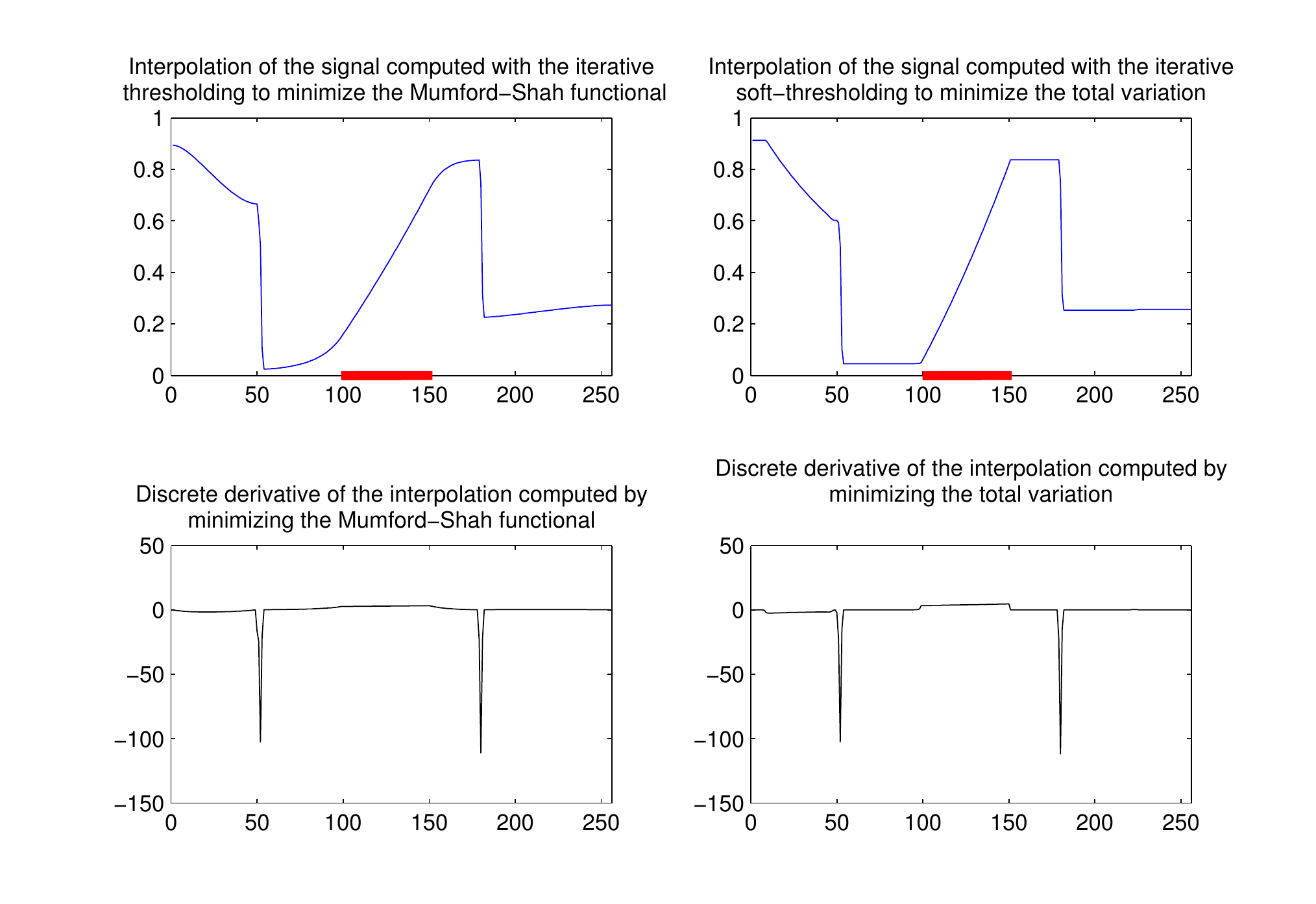}
\caption{Interpolation of an incomplete signal by means of the Mumford-Shah regularization and the total variation minimization provided by respective iterative thresholding algorithms. The red interval is the region where no information on the original signal is provided. The thresholding parameters used for the numerics are $r=2.2$ and $\gamma=0.002$  for \eqref{modH}, and $\gamma
=0.002$ for the soft-thresholding \eqref{softthrs}.}\label{1Dinterpolation}
\end{center}
\end{figure}

In Figure \ref{1Dinterpolation}  we show the reconstruction of the noiseless signal of Figure \ref{1Dsignal} provided information only out of the interval $[100,150]$ which has to be restored. On the left boxes we show the results due to algorithm \eqref{map3} and on the left ones the solution computed by iterative soft-thresholding. In the former the solution is again piecewise smooth and in the latter a (almost) piecewise constant solution is instead produced.



 \section{Appendix}

\subsection{Proof of Proposition \ref{firstlemma}}
 First, we recall Weierstrass' Theorem, which is used in the proof of Proposition $\ref{firstlemma}$ below.
\begin{theorem}[Weierstrass' Theorem]
The set of minima of a convex function $f$ over a subset $X \subset \mathbb{R}^N$ is nonempty and compact if $X$ is closed, $f$ is lower semicontinuous over $X$, and the function $\tilde{f}$, given by
\begin{equation}
\tilde{f} = \left\{ \begin{array}{ll}
f(x) & \textrm{, if } x \in X, \\
\infty & otherwise,
\end{array}
\right.
\end{equation}
is \emph{coercive}, i.e., for every sequence $( x_k) \subset X$ s.t. $\| x_k \| \rightarrow \infty$, we have $\lim_{k \rightarrow \infty} f(x_k) = \infty$.
\end{theorem}
 
The following two lemmas will be helpful in the proof of Proposition \ref{firstlemma}.  

\begin{lemma}
Let $F(u)$ be a convex function defined on $\mathbb{R}^N$ having the general form $F(u) = \Big[ u^t A u + b^t u + \sum_{1 \leq j \leq N} |u_j|^p \Big]$, for some $p \geq 1$.  Fix $x$ and $d$ in  
$\mathbb{R}^N$.  If $F$ is bounded above and below on the ray $\{x + td, t \geq 0\}$, then $F$ is constant on the line $x + td$.  
\label{appendlemma1}
\end{lemma}
\begin{proof}
Let $\mu(t) = F(x + td)$, and note that $\mu$ is convex because $F$ is convex. Moreover, $\mu$ has the general form $\mu(t) = P(t) + \sum_{1 \leq j \leq N} c_j \|x_j + td_j \|^p$ where $P(t)$ is a polynomial in $t$ of order at most $2$. Without loss of generality, suppose $0 \leq \mu(t) \leq 1$ for all values of $t \in \mathbb{R}^+$.  Then there exists a sequence of points $(t_n)_{n \in \mathbb N}$, $t_n \to \infty$ for $n \to \infty$, for which $\mu(t_n)$ is a convergent sequence; let us denote the limit of this sequence by $\gamma$.

\begin{enumerate} 
\item {\bf Case 1: $1 \leq p \leq 2$}.  To repeat, 
\begin{equation}
\lim_{n \rightarrow \infty} \mu(t_n) = \lim_{n \rightarrow \infty} P(t_n) + \sum_{1 \leq j \leq N} c_j \|x_j + t_n d_j \|^p = \gamma.
\end{equation}  
Since $0 = \lim_{n \rightarrow \infty} \mu(t_n) / t_n^2$, it follows that all coefficients in $\mu(t)$ of degree 2 must vanish.  In turn, then, $0 = \lim_{n \rightarrow \infty} \mu(t_n) / t_n^p $, has the implication that for each $j$, one of the coefficients $c_j$ or $d_j$ must vanish as well.   Following in the same manner, we conclude that all linear coefficients in $\mu(t)$ also vanish, leaving only the possibility that $\mu(t) \equiv \gamma$ is a constant function.  
\item {\bf Case 2: $p > 2$}: The proof in this case is identical to that of the previous case, and as such we leave the details to the reader.  
\end{enumerate}
\end{proof}

\begin{lemma}
Suppose $F$ is a convex function defined on $\mathbb{R}^N$ that is bounded from below, and has the property that if $F$ is bounded above on a ray $\{x + td, t \in \mathbb{R}^{+} \}$, then $F$ is constant on the line $x + td$.  Then if $F$ is constant on the line $x + td$, $F$ is also constant on any parallel line $y + td$.  
\label{appendlemma2}
\end{lemma}
\begin{proof}
Let $\mu(t) = F(x + td)$ which by assumption is a constant function $\mu(t) = \gamma$, and let $v(t) = F(y + td)$.  Fix $t \in \mathbb{R}^+$, and let $z$ be the point $z = x + 2(y-x)$, i.e. $y = \frac{1}{2}x + \frac{1}{2}z$.   By convexity of $F$, we have that 
\begin{equation}
F(y + td) = F \Big(\frac{1}{2}z + \frac{1}{2}(x + 2td) \Big) \leq \frac{1}{2}F(z) + \frac{1}{2}\mu(2t) = \alpha,
\end{equation}
for a constant $\alpha$.  It follows that $F$ is bounded above by $\alpha$ on the ray $\{y + td, t \in \mathbb{R}^{+} \}$, from which it follows, by assumption, that $F$ is constant on the line $y + td$.
\end{proof}
We now prove Proposition \ref{firstlemma}. 
Choosing $x_0 \in X$, we define the (nonempty) set
\begin{equation}
M := X \cap \{x \in \mathbb{R}^N, F(x) \leq F(x_0) \}.
\end{equation}
Obviously, the set $M$ is convex and closed.  By assumption, $F$ is bounded from below on $X$ and hence on $M$.  Therefore, if $M$ is bounded, then Weierstrass' Theorem yields the desired result.  
\\
\\
Thus, we may assume that $M$ is unbounded.  Then, the convexity of $M$ implies that $M$ contains a ray $r = \{z + td, t \geq 0\}$.   Denote by $r_1, r_2, ..., r_J$ a set of $J$ rays in $M$ corresponding to linearly independent vectors $d_1, ..., d_J$, so that any ray in $M$ can be expressed as a linear combination of the $r_1, ..., r_J$.   By definition of $M$ and by the assumption, $F$ is bounded on $M$, hence, $F$ is constant on each of the the lines $z_j + td_j$, according to Lemma \eqref{appendlemma1}.  From Lemma \eqref{appendlemma2}, it follows that $F$ is constant along each line $x + td_j$ for arbitrary $x \in \mathbb{R}^N$, from which we deduce that $F$ is constant along any line $x + td$ for arbitrary $d \in Y = \operatorname{span}\{d_1, ...., d_J \}$.  Thus, we project $X$ onto the subspace of $\mathbb{R}^N$ that is orthogonal to $Y$; call this subspace $\tilde{X}$.  
\\
\\
From the foregoing arguments, we have
\begin{equation}
\inf_{\tilde{X}} F(u) = \inf_X F(u)
\end{equation}
As $\tilde{X}$ is still a convex polyhedral set, and by construction $\tilde{M} = \tilde{X} \cap \{x \in \mathbb{R}^N \}$ contains no rays, Weierstrass' Theorem yields the desired result. 

\subsection{On uniform boundedness of $\| D^\dagger_h \|$}

\noindent The aim of the second part of the appendix is to prove the uniform bound $\| D^\dagger_h \| \leq 1/2$ eluded to in Section 3.1.  Again, $\| A \|$ denotes the spectral norm of the matrix $A$, and $D_h^\dagger: \mathbb{R}^{n-1} \rightarrow \mathbb R^n$ is the pseudo-inverse  of the discrete derivative matrix $D_h$ as given by \eqref{dermtrx}, with the identification $n = \lfloor 1/h \rfloor$.   From the expression for $D_h$, and the knowledge that $D_h D_h^\dagger = I$ is the identity operator and $D_h^\dagger D_h = (D_h^\dagger D_h)^*$ is self-adjoint, the $n \times (n-1)$ matrix $D_h^\dagger$ is identified as follows:
\begin{equation}
\label{dermtrxinverse}
D^{\dagger}_h = \frac{1}{n^2} \left ( \begin{array}{cccccc} -(n-1) & -(n-2) & -(n-3) & \dots & \dots & -1 \\
1 & -(n-2) & -(n-3) & \dots &\dots& -1 \\
1 & 2 & -(n-3) & \dots &\dots& -1 \\
\vdots & \vdots & \vdots & \vdots & \vdots & \vdots \\
1&2&3&\dots&\dots&n-1
\end{array} \right ).
\end{equation}
It is well-known that the spectral norm of an $m \times n$ matrix can be bounded by the more manageable entry-wise Frobenius norm, according to
\begin{equation}
\| A \| \leq \| A \|_F = \sqrt{\sum_{i=1}^m \sum_{j=1}^n | a_{i,j}|^2}.
\end{equation}
As such, we need only to bound the sum of the squares of the entries of $D^{\dagger}_h$.  The sum $S^1_n = \sum_{j=1}^{n-1} | d_{1,j}|^2$ over entries in the first row of $D^{\dagger}_h$ is given by $S^1_n = (n-1)(2n-1)/(6n^3)$, using the familiar formula $\sum_{j=1}^N j^2 = \frac{1}{6}N(N+1)(2N+1)$.  The analogous sum over entries in the $j^{th}$ row of $D^{\dagger}_h$ is seen inductively to satisfy $S^j_n = S^1_n - \frac{(j-1)}{n^2} +\frac{j(j-1)}{n^3}$.  The total sum $S_n = \sum_{j=1}^{n} S_n^j$ is then $S_n = \frac{1}{6} - \frac{1}{6n^2}$, and we arrive at the desired uniform bound:
\begin{equation}
\| D_h^\dagger \| \leq \sqrt{S_n} \leq \frac{1}{\sqrt{6}} < 1/2. 
\end{equation}  

\subsection{Proof of Proposition \ref{thmthrs}}

\noindent In order to help the reading of the current proof, as well as the proofs of Theorem \ref{mainth} and Theorem \ref{globalmin} in later appendices, we report in Table 1 the notation of the functions used in the proof of Proposition \ref{thmthrs} for the definition of $H_{(p,r)}$.
\begin{table}[h]
\begin{center}
\begin{tabular}{|l|l|}
\hline
$L_p(t,\lambda)$   & $=(t - \lambda)^2 + \min\{ |t|^p, r^p \}$ \\
\hline
$G_p(t, \lambda)  $ & $=(t - \lambda)^2 + |t|^p$ \\
\hline
$F_p(t) $  & $=t+\frac{p}{2}\sgn t |t|^{p-1}$, $p >1$ \\
\hline
$S_p(\lambda) $  & $=G_p(F^{-1}_p(\lambda),\lambda) = (F^{-1}_p(\lambda) - \lambda)^2 + | F^{-1}_p(\lambda)|^p$, $p>1$\\
\hline
$H_{(p,r)}(\lambda)  $ & $=\arg \min_{t \geq 0} L_p(t, \lambda)$ for general $\lambda \geq 0$, $p>1$ \\
\hline
& $=\arg \min_{0 \leq t \leq r} G_p(t, \lambda)= F_p^{-1}(\lambda)$ for $0\leq \lambda  \leq r$ \\
\hline
& $=\left\{ \begin{array}{ll}  F_p^{-1}(\lambda), & \textrm{if } G_p(F_p^{-1}(\lambda),\lambda) \leq r^p \nonumber \\
\lambda, & \textrm{else} \end{array} \right.$ for $\lambda > r$.\\
\hline

\end{tabular}
\caption{Notation of the functions involved in the definition of $H_{(p,r)}$ as  in the proof of Proposition \ref{thmthrs}.}
\end{center}
\end{table}

Consider the functions 
\begin{equation}
L_p(t,\lambda) = (t - \lambda)^2 + \min\{ |t|^p, r^p \},
\end{equation}
and
\begin{equation}
G_p(t, \lambda) = (t - \lambda)^2 + |t|^p.
\label{G}
\end{equation}
The proof reduces to solving for
\begin{equation}
H_{(p,r)}(\lambda) = \arg \min_{t \in \mathbb R} L_p(t, \lambda)
\label{hmin}
\end{equation}
as a function of $\lambda \in \mathbb{R}$.  Since $L_p(t, \lambda) = L_p(-t, -\lambda)$, the function $H_{(p,r)}(\lambda)$ will be odd, and since also $H_{(p,r)}(0) = 0$, we can, without loss of generality, restrict the domain of interest to $\lambda > 0$.  On this domain, $H_{(p,r)}(\lambda) = \arg \min_{t \in \mathbb R} L_p(t, \lambda)$ is nonnegative, since $L_p(t, \lambda) \leq L_p(-t, \lambda)$ when $t \geq 0$ and $\lambda \geq 0$. Hence, we can restrict the minimization of $L_p(t,\lambda)$ to $t \geq 0$.
\\

\noindent It will be convenient to split the proof into two cases: $1 < p$ and $p = 1$.
\begin{enumerate}
\item We first analyze the case $1 < p$.  
\\
Note that 
\begin{eqnarray}
\arg \min_{t \geq r} L_p(t, \lambda) &=&  \arg \min_{t \geq r}  (t - \lambda)^2 \nonumber \\
&=& \max \{\lambda, r \},
\end{eqnarray}
so that the minimization  $\eqref{hmin}$ naturally splits into the following two cases:
\begin{enumerate}
\item If $\lambda \leq r$, the minimizer has to be searched in $[0,r]$, hence
\begin{equation}
H_{(p,r)}(\lambda) = \arg \min_{0 \leq t \leq r} G_p(t, \lambda) = F_p^{-1}(\lambda) \leq \lambda 
\label{lessthan}
\end{equation}
where $F_p^{-1}(\lambda)$ is the functional inverse of the increasing, and continuous function
\begin{equation}
F_p(t) = t + \frac{p}{2}\sgn{t}{|t|}^{p-1}.
\end{equation}
\item On the other hand, if $\lambda > r$, the minimizer has to be searched in $[0, \lambda]$, hence
\begin{eqnarray}
H_{(p,r)}(\lambda) &=& \left\{ \begin{array}{ll}  F_p^{-1}(\lambda), & \textrm{if } G_p(F_p^{-1}(\lambda),\lambda) \leq r^p \nonumber \\
\lambda, & \textrm{else} \end{array} \right. .
\label{h2}
\end{eqnarray}
\end{enumerate}
By implicit differentiation of the functional relation $F_p(F_p^{-1}(\lambda)) = \lambda$, it is clear that the functions $F_p^{-1}(\lambda)$ and $S_p(\lambda) := G_p(F_p^{-1}(\lambda),\lambda)$ are strictly increasing functions in $\lambda$.  
Indeed, we have the bounds
$$
0 < \frac{d}{d \lambda}F_p^{-1}(\lambda) = \left (F'_p(F_p^{-1}(\lambda)) \right)^{-1} = \left (1 + \frac{p(p-1)}{2} (F_p^{-1}(\lambda))^{p-2} \right )^{-1} \leq 1,
$$
and
\begin{eqnarray*}
\frac{d}{d \lambda} S_p(\lambda) &=& \frac{\partial}{\partial t} G_p(F_p^{-1}(\lambda),\lambda)\frac{d}{d \lambda}F_p^{-1}(\lambda) + \frac{\partial}{\partial \lambda}   G_p(F_p^{-1}(\lambda),\lambda)\\
&=& (2 (F_p^{-1}(\lambda) - \lambda) + p (F_p^{-1}(\lambda))^{p-1})\frac{d}{d \lambda}F_p^{-1}(\lambda) - 2 (  F_p^{-1}(\lambda) - \lambda)\\
&=& 2 \left (1 - \frac{d}{d \lambda}F_p^{-1}(\lambda)\right ) (\lambda - F_p^{-1}(\lambda)) + p \frac{d}{d \lambda}F_p^{-1}(\lambda) (F_p^{-1}(\lambda))^{p-1} \geq 0,
\end{eqnarray*}
since $0\leq \frac{d}{d \lambda}F_p^{-1}(\lambda) \leq 1$, and 
\begin{equation}
\label{boundinvF}
0\leq F_p^{-1}(\lambda) \leq \lambda.
\end{equation}
Also observe that $F_p^{-1}(r + \frac{p}{2}r^{p-1}) = r$, and $S_p(r + \frac{p}{2}r^{p-1}) = r^p + \frac{p^2}{4}r^{2p-2} > r^p$.
This leads us to immediately conclude that
\begin{itemize}
\item[(i)] If $\lambda \leq r $, then $H_{(p,r)}(\lambda) = F_p^{-1}(\lambda)$ (from $\eqref{lessthan}$).
\item[(ii)] If $\lambda \geq r +\frac{p}{2}r^{p-1}$, then $S_p(\lambda) = G_p(F_p^{-1}(\lambda),\lambda) > r^p$, so that $H_{(p,r)}(\lambda) = \lambda$.  
\item[(iii)] Since $S_p(r) < r^p$ while $S_p(r+\frac{p}{2} r^{p-1})) > r^p$, the intermediate value theorem implies that there exists a unique value $\lambda'(r,p)$  lying {\it strictly within} the interval $\big(r , r^{p-1}(\frac{p}{2} + r^{2-p})\big)$ at which 
 \begin{equation}
S_p(\lambda')  = r^p,
\label{equal}
 \end{equation} 
 and  
 \begin{eqnarray}
 H_{(p,r)}(\lambda) &=& \left\{ \begin{array}{ll}  F_p^{-1}(\lambda) & \lambda < \lambda'(r,p)  \\ \lambda & \lambda > \lambda'(r,p) \end{array} \right. .
 \end{eqnarray}
At $\lambda'$, $H_{(p,r)}(\lambda') = \arg \min_{t \geq 0} L_p(t, \lambda')$ is not uniquely defined and is realized at $F_p^{-1}(\lambda')$ and at $\lambda'$. In this case, we identify $H_{(p,r)} (\lambda') = F_p^{-1}(\lambda)$ for the sequel; as will be made clear, this will not cause problems in the ensuing analysis.  Finally, note that
 \item[(iv)] 
 At $\lambda'$, the function $H_{(p,r)}$ has a discontinuity $\delta(r,p) = \lambda' - H_{(p,r)} (\lambda')$ that is strictly positive, as long as $r > 0$.  Indeed,
on the one hand, we know that $\lambda'(r,p) > r$, on the other hand, $H_{(p,r)} (\lambda') < r$. This follows because $H_{(p,r)}(\lambda') = F_p^{-1}(\lambda')$, and
$$
(F_p^{-1}(\lambda'))^p < (F_p^{-1}(\lambda') - \lambda')^2 + |F_p^{-1}(\lambda')|^p = S_p(\lambda')= r^p.
$$
 \end{itemize}

 \item The analysis of the case $p = 1$ is left to the reader since it follows a similar argument as for $p > 1$.
 \end{enumerate}

\subsection{Proof of Theorem \ref{contract} }

We assume that the operator $T^*T : \ell_2(\mathcal I) \rightarrow \ell_2(\mathcal I)$ is nonnegative, so that its spectrum lies within an interval $[\delta, 1]$ with $\delta \geq 0$, and the operator $I - T^*T$ has norm $\| I - T^*T\| \leq 1 - \delta$.  In particular, if $T^*T$ is invertible, then the inequality $\delta > 0$ is strict, and so $\|I - T^*T\| \leq 1 - \delta < 1$.
\\
We wish to show that the map $\mathbb{U}_{\mathcal I_0}$ with component-wise action
\begin{eqnarray}
[\mathbb{U}_{\mathcal I_0} u]_i &=& \left\{
\begin{array}{ll} F_p^{-1}([(I - T^*T)u + T^*g]_i), & \textrm{if } i \in \mathcal I_0 \\  
\big( (I - T^*T)u + T^*g \big)_i, & \textrm{if } i \in \mathcal I_1
\end{array} \right.
\label{separate}
\end{eqnarray}
is a contraction.  To this end, let ${ v, v'}$ be arbitrary vectors in $\ell_2(\mathcal I)$.  
\begin{enumerate}
\item If the index $i \in \mathcal I_1$, then $$|[\mathbb{U}_{\mathcal I_0}(v)]_i - [\mathbb{U}_{\mathcal I_0}(v')]_i|  = |\big[ (I - T^*T)(v - v') \big]_i |;$$
\item  If the index $i \in \mathcal I_0$, then we split the analysis in two cases $p>1$ and $p=1$:
\begin{enumerate}
\item for $p > 1$, we have
\begin{eqnarray}
|[\mathbb{U}_{\mathcal I_0}(v)]_i - [\mathbb{U}_{\mathcal I_0}(v')]_i| &=& \left | F_p^{-1}([(I - T^*T)v + T^*g]_i) - F_p^{-1}([(I - T^*T)v' + T^*g]_i) \right | \nonumber \\
&=& \left | \frac{d}{d\lambda}F_p^{-1}(\xi) \big[ (I - T^*T)(v - v') \big]_i \right | \nonumber \\
&<& \left |\big[ (I - T^*T)(v - v') \big]_i \right |
\end{eqnarray}
where the second equality is an application of the mean value theorem, which is valid since $F_p^{-1}(\lambda)$ is differentiable.  The final inequality above follows from implicit differentiation of the relation  $$F_p^{-1}(F_p(t)) = t $$
and the observation that $|\frac{d}{dt}F_p(t)| > 1$ (see the proof of Proposition \ref{thmthrs});
\item for $p = 1$, by analyzing all cases, we get also that
\begin{eqnarray}
|[\mathbb{U}_{\mathcal I_0}(v)]_i - [\mathbb{U}_{\mathcal I_0}(v')]_i| 
&\leq& |\big[ (I - T^*T)(v - v') \big]_i |
\end{eqnarray}
\end{enumerate}
\end{enumerate}
Together, we have 
\begin{eqnarray}
\| \mathbb{U}_{\mathcal I_0} ({ v}) - \mathbb{U}_{\mathcal I_0}({ v'}) \|^2_{\ell_2(\mathcal I)} &=& \sum_{i \in \mathcal I} |[\mathbb{U}_{\mathcal I_0}(v)]_i - [\mathbb{U}_{\mathcal I_0}(v')]_i|^2 \nonumber \\
&\leq&  \sum_{i \in \mathcal I} |\big( (I - T^*T) v - v' \big)_i |^2 \nonumber \\
&=& \| (I - T^*T) { v - v'} \|^2_{\ell_2(\mathcal I)} \nonumber \\
&\leq& \| I - T^*T \|^2 \|{ v - v'} \|^2_{\ell_2(\mathcal I)} \nonumber \\ 
&\leq& (1 - \delta) \|{ v - v'} \|^2_{\ell_2(\mathcal I)}.
\label{nonexpand}
\end{eqnarray}
As $\mathbb{U}_{\mathcal I_0}$ is a contraction, we arrive at the stated result by application of the Banach Fixed Point Theorem.

\subsection{Proof of Lemma \ref{l4}}
 If $i \in \mathcal I_1$, then $\bar{u}_i = \bar{u}_i + \big[ T^*(g - T\bar{u}) \big]_i$, which is satisfied if and only if $\big[ T^*(g - T\bar{u}) \big]_i = 0$ as stated.   It remains to analyze the case $i \in \mathcal I_0$, and, again, we split the argument in the cases $p>1$ and $p=1$.  
 \begin{enumerate}
 \item First suppose $p > 1$. Using the notation $\bar{\lambda} = \bar{u}_i + \big[ T^*(g - T\bar{u}) \big]_i$, the fixed point characterization $\eqref{seperate2}$ translates to $$ F_p^{-1}(\bar{\lambda}) = \bar{u}_i. $$
But of course $\lambda = F_p(\bar{u}_i)$ is the unique value at which $F_p^{-1}(\lambda) = \bar{u}_i$, and so this implies that
\begin{equation}
\big[ T^*(g - T\bar{u}) \big]_i = F_p(\bar{u}_i) - \bar{u}_i,
\label{invert}
\end{equation}
and, by reversing operations, the relation $\eqref{invert}$ in turn implies the fixed point condition $\eqref{seperate2}$.
\item The case $p = 1$, which is similar, is left to the reader.
\end{enumerate}

\subsection{Proof of Theorem \ref{mintheorem}}
The proof will be much simplified by the following lemma which characterizes vectors such as $\bar{u}$ that satisfy the fixed point relations \eqref{wha2} or \eqref{wah2}:
\begin{lemma}
If $u$ and $v$ are such that
\begin{equation}
{\cal{J}}_r^{p,surr} (u + v, u) - \| v \|_{\ell_2(\mathcal I)}^2\geq {\cal{J}}_r^{p,surr} (u, u) = {\cal{J}}_r^{p} (u), 
\label{eq}
\end{equation}
then ${\cal{J}}^p_r (u + v) \geq  {\cal{J}}^p_r ({ u})$.
\label{lemmata}
\end{lemma}
\begin{proof} 
For any $u$ and $v$, the following holds because $\|L\| \leq 1$:
\begin{equation}
{\cal{J}}^p_r(u + v) = {\cal{J}}_r^{p,surr} (u + v,u) - \|Lv \|_{\ell_2(\mathcal I)}^2 \geq  {\cal{J}}_r^{p,surr} (u + v, u) - \|v \|_{\ell_2(\mathcal I)}^2.
\end{equation}
If in addition $u$ and $v$ satisfy $\eqref{eq}$, then the desired result is achieved by virtue of the equality ${\cal{J}}_r^{p,surr} (u, u) =  {\cal{J}}^p_r ({ u})$.
\end{proof}

Let us show now the proof of Theorem \ref{mintheorem}.
By Lemma $\ref{lemmata}$, it suffices to show that at a fixed point $\bar{u}$ defined by \eqref{wha2} or \eqref{wah2}, any perturbation $\delta h \in \ell_2(\mathcal I)$ with norm $\| \delta h \|_{\ell_2(\mathcal I)} \leq \min\{ [\lambda'(r,p) - r], [r - H_{(p,r)}(\lambda')] \}$ will satisfy
\begin{equation}
{\cal{J}}_r^{p,surr} (\bar{u} + \delta h, \bar{u}) -{\cal{J}}_r^{p,surr} (\bar{u}, \bar{u}) \geq \| \delta h \|_{\ell_2(\mathcal I)}^2.
\label{min}
\end{equation}
After expanding the left-hand-side above, the inequality $\eqref{min}$ is seen to be equivalent to 
\begin{equation}
2 \sum_{i \in \mathcal I} \delta h_i [T^*(T\bar{u} - g)]_i + \sum_{i \in \mathcal I} \Big[ \min\{|\bar{u}_i + \delta h_i|^p, r^p\} - \min \{|\bar{u}_i|^p, r^p \} \Big] \geq 0.
\label{lhs}
\end{equation}
At this point, it is convenient to consider the summation over $i \in \mathcal I_0$ and $i \in \mathcal I_1$ separately. 
\\
\noindent By Lemma $\ref{l3}$, the first summand above vanishes over $\mathcal I_1$ and 
\begin{enumerate}
\item  if $1 < p $, then $\sum_{i \in \mathcal I} \delta h_i [T^*(T\bar{u} - g)]_i = - \sum_{i \in \mathcal I_0} \delta h_i \sgn{u_i} \frac{p}{2} |u_i|^{p-1}$;  
\item if $p = 1$, then $\sum_{i \in \mathcal I} \delta h_i [T^*(T\bar{u} - g)]_i = - 1/2 \sum_{i \in \mathcal I^b_0} \delta h_i \sgn{u_i} + \sum_{i \in \mathcal I^a_0} \delta h_i [T^*(T\bar{u} - g)]_i$.
\end{enumerate}
With respect to the second summation, observe from Proposition $\ref{thmthrs}$ that for all $1 \leq p$,  $|\bar{u}_i| \geq \lambda'(r,p) > r $ for $i \in \mathcal I_1$, so that this summation vanishes over $\mathcal I_1$ for any perturbation $\delta h$ satisfying the component-wise inequality $| \delta h_i | \leq \lambda'(r,p) - r$.  Similarly, $|\bar{u}_i| \leq H_{(p,r)} (\lambda') < r$ for $i \in \mathcal I_0$, so that for any perturbation $\delta h$ satisfying component-wise $| \delta h_i | \leq \min\{ [\lambda'(r,p) - r], [r - H_{(p,r)}(\lambda')] \}$, we have that
\begin{equation}
 \sum_{i \in \mathcal I} \Big[ \min\{|\bar{u}_i + \delta h_i|^p, r^p\} - \min \{|\bar{u}_i|^p, r^p \} \Big]  =  \sum_{i \in \mathcal I_0} |\bar{u}_i + \delta h_i|^p - |\bar{u}_i|^p.
\end{equation}
The desired result follows if we can show that
 \begin{enumerate}
 \item $1 < p \leq 2$:  $\big[ |\bar{u}_i + \delta h_i|^p - |\bar{u}_i|^p - \delta h_i p [\sgn{u_i}] |u_i|^{p-1} \big] \geq 0$, for all $i \in \mathcal I_0$
\item $p = 1$: 
\begin{enumerate}
\item $ | \delta h_i + \bar{u}_i| - |\bar{u}_i| - \delta h_i[\sgn{u_i}] \big] \geq 0$ for all $i \in \mathcal I_0^b$, and
\item  $\delta h_i [T^*(T\bar{u} - g)]_i + | \delta h_i | \geq 0$, for all $i \in \mathcal I_0^a$.
\end{enumerate}
\end{enumerate}
The inequality in $2(b)$ follows directly from Lemma $\ref{l3}$; by symmetry, $1$ and $2(a)$ follow if, for any $u \geq 0$,
\begin{equation}
\min_{v \in \mathbb R} \left [ f(v) := |u + v|^p - u^p - pu^{p-1}v \right ]  = \min_{v \geq -u} (u + v)^p - u^p - pu^{p-1}v \geq 0.
\end{equation}
When $p = 1$, the right-hand-side is identically zero and the result holds.  When $1 < p \leq 2$, differentiating the right-hand-side gives that $f(v)$ has a local minimum at $v = 0$, at which $f(0) = 0$, and, at the endpoint, $f(-u) = (p - 1)u^{p-1} \geq 0.$   

\subsection{Proof of Theorem \ref{globalmin}}
Suppose that $u^*$ is a minimizer of the functional ${\cal{J}}^p_r$.   Consider the partition of the index set $\mathcal I$ into $\mathcal I_0 = \{i \in \mathcal I: |u^*_i| \leq r \}$  and $\mathcal I_1 = \{i \in \mathcal I: |u^*_i| > r \}$, and note that $| \mathcal I_1 | < \infty$, or else $| \mathcal J_r^p(u^*)|$ would not be finite.  As in the proof of Theorem \eqref{l5}, consider the unique decomposition $u^* = u^*_0 + u^*_1$ into a vector $u^*_0$ supported on $\mathcal I_0$ and another $u^*_1$ supported on $\mathcal I_1$.  Again, let $\mathcal P: u \rightarrow u_1$ and $\mathcal P^{\perp} = \mathcal I - \mathcal P: u \rightarrow u_0$ denote the orthogonal projections onto the subspaces  $\ell_2^{\mathcal I_1}(\mathcal I)$ and $\ell_2^{\mathcal I_0}(\mathcal I)$, respectively, and consider the operators $T_0 = T \mathcal P^{\perp}$ and $T_1 = T \mathcal P$.
\\
\\
By minimality of ${ u^*}$, if we fix ${ u^*_0}$, the vector ${ u^*_1}$ satisfies ${ u_1^*} = \arg \min_{z \in \ell_2^{\mathcal I_1}(\mathcal I)} {\cal{J}}^p_{r,1}(z)$, where 
\begin{equation}
 {\cal{J}}^p_{r,1} ({ z}) :=  \| T_1 { z} - { (g - T_0 u^*_0)} \|^2_{\ell_2(\mathcal J)} +  \sum_{i \in \mathcal I_1} \min \{ |z_i|^p, r^p \}.
\label{restrict1}
\end{equation}
Since all coefficients in ${ u^*_1}$ have absolute value $|(u^*_1)_i| > r$, the vector ${ u^*_1}$ also minimizes the functional
\begin{eqnarray}
 \| T_1{ z} - { (g - T_0 u^*_0)} \|^2_{\ell_2(\mathcal J)}, 
\label{realrestrict}
\end{eqnarray}
or, else, the vector ${ z^*}$ minimizing $\eqref{realrestrict}$ would satisfy ${\cal{J}}^p_{r,1} ({ z^*}) <  {\cal{J}}^p_{r,1} ({ u^*_1})$, contradicting the minimality of ${ u^*_1}$.  In fact, ${ u^*_1}$ must be the {\it unique} vector minimizing $\eqref{realrestrict}$.  For, if another vector ${ u'}$ also minimized $\eqref{realrestrict}$, then the operator $T_1$ would have a nontrivial null space containing the span of some nonzero vector ${ v}$, so that all vectors in the affine space $\{ u^*_1 + t v : t \in \mathbb R\}$ would be minimal solutions for $\eqref{realrestrict}$.  In this case, we would have also the freedom of choosing from this affine subspace a vector ${ u'}$ having one coefficient $u'_i$ satisfying $|u_i'| < r$.  But such a vector ${ u'}$ satisfies ${\cal{J}}^p_{r,1} ({ u'}) <  {\cal{J}}^p_{r,1} ({ u^*_1})$, contradicting the minimality of $u^*_1$.  
\\
\\
\noindent It follows that the operator $T_1$ must have trivial null space, and ${ u^*_1}$ is the unique minimal least squares solution to $\eqref{realrestrict}$, well-known to be explicitly given by
\begin{equation}
{ u^*_1} = \big(T_1^*T_1\big)^{-1}T_1^*{ (g - T_0 u^*_0)},
\label{relu1u0}
\end{equation}
so that $T_1{ u^*_1}$ is the unique orthogonal projection of ${ (g - T_0 u^*_0)}$ onto the range of $T_1$. Actually $\mathcal P_1 = T_1 ( T_1^* T_1)^{-1} T_1^*$ is the orthogonal projection onto the range of $T_1$, due to the non-triviality of the null space of $T_1$.  Therefore we have $T_1{ u^*_1} = {\cal P}_1({ g - T_0 u^*_0})$. It easily follows that
\begin{equation}
T^*_1 \big(T_1{ u^*_1}  - { (g - T_0 u^*_0)} \big) = 0,
\end{equation}
or, in other words, 
\begin{equation}
 \big[ T^*(g - T{ u}^*) \big]_i  = 0 \textrm{,  for all }  i \in \mathcal I_1.
\end{equation}
Now, on the other hand, by observing that any optimal variable ${ u_1}$ for fixed $u_0$ depends on $u_0$ via the relationship ${ u_1} = \big(T_1^*T_1\big)^{-1}T_1^*{ (g - T_0 u_0)}$, we easily infer that the vector ${ u^*_0}$ minimizes
\begin{eqnarray}
{\cal{J}}^p_{r,0} ({ v})
&=& \| {\cal P}_1^{\perp} (T_0 { v} - { g}) \|^2_{\ell_2(\mathcal J)} +  \sum_{i \in \mathcal I_0} \min \{ |v_i|^p, r^p \}, 
\label{restrict1}
\end{eqnarray}
where ${\cal P}_1^{\perp}$ denotes the orthogonal projection operator onto the orthogonal complement of the range of $T_1$.
\\
\\
\noindent Consider the convex functional,
\begin{equation}
 {\cal{F}} ({ v}) :=  \|  {\cal P}_1^{\perp} (T_0 { v} - { g}) \|^2_{\ell_2(\mathcal J)} + \|v \|_{\ell_p^{\mathcal I_0}(\mathcal I)}^p,
\label{auxfunc}
\end{equation}
and note that ${\cal{J}}^p_{r,0} ({ u}) \leq  {\cal{F}} ({ u})$, while at the same time ${\cal{J}}_{r,0}^p({ u^*_0}) =  {\cal{F}} ({ u^*_0})$ by virtue of the fact that $|u^*_i| < r$. 
{ For $p>1$} it follows that ${ u^*_0}$ is also a minimizer of  ${\cal{F}} ({ u})$, and so satisfies the Euler-Lagrange equations $\cite{bashsh93}$,
\begin{equation}
\big(T_0^*{\cal P}_1^{\perp}(T_0{ u^*_0}  - { g}) \big) + \frac{p}{2}\sgn{{ u^*_0}}|{ u^*_0}|^{p-1} = 0,
\end{equation}   
which imply the fixed point conditions
\begin{equation}
 \big[ T^*(g - T{ u}^*) \big]_i = \frac{p}{2}\sgn{(u^*_0)_i}|(u_0^*)_i|^{p-1}  \textrm{,  for all } i \in \mathcal I_0.
\end{equation}
For $p=1$ one uses results from \cite{DDD} to conclude that
\begin{equation}
u_0^* = \mathbb S_{1/2} ( u_0^* + T_0^* \mathcal P_1^\perp( g - T_0 u_0^*)),
\label{fixptst}
\end{equation}
where $\mathbb S_\gamma$ is the so-called {\it soft-thresholding}, defined component-wise $\mathbb S_\gamma(v) = (S_\gamma (v_i))_{i \in \mathcal I}$, where
\begin{equation}
S_{\gamma}(\lambda) =
\left\{
\begin{array}{ll}
0, & |\lambda| \leq \gamma  \\
\lambda - \frac{\sgn \lambda}{2}, & |\lambda|  > \gamma.
\end{array} \right.
\label{softthrs}
\end{equation}
(Actually, \cite[Proposition 3.10]{DDD} only states that any fixed point of \eqref{fixptst} is a minimizer of \eqref{auxfunc}; nevertheless the converse also holds, see \cite[Remarks (1), pag. 2515]{fo07}.)
The fixed-point condition \eqref{fixptst} implies
\begin{equation}
  \left\{
\begin{array}{ll}
\big[ T^*(g - T {u}^*) \big]_i  \in [-1/2, 1/2] , & |{u}^*_i| \leq 1/2 \\  
\big[ T^*(g - T{u}^*) \big]_i  = 1/2\sgn{{u}^*_i}, & 1/2 < |{u}^*_i| \leq r. 
\end{array}
\right.
\end{equation}
It remains to verify that 
\begin{itemize}
\item $|u^*_i| \geq \lambda'(r,p)$, if $i \in \mathcal I_1$, and
\item $|u^*_i| \leq F_p^{-1}(\lambda'(r,p))$, { for $p>1$, and $|u^*_i| \leq r- 1/4$, for $p=1$,} if $i \in \mathcal I_0$.
\end{itemize}
{ We show these conditions for $p>1$ only, as the case $p=1$ is proved with an analogous argument.}
\begin{enumerate}
\item
We first show that $|u^*_i| \geq \lambda'(r,p)$ if $i \in \mathcal I_1$. From the first part of the proof, we know that at a minimizer $u^*$, the functional ${\cal J}^p_r(u^*)$ can be written as
\begin{equation}
 {\cal J}^p_r(u^*) = \| {\cal P}_1^{\perp} (T_0 u^*_0 - { g}) \|^2_{\ell_2(\mathcal K)} +  \|u_0^* \|_{\ell_p^{\mathcal I_0}(\mathcal I)}^p+ |\mathcal I_1|r^p
 \label{atmin}
\end{equation} 
Note that at this point we make  explicit use of the finite cardinality of $\mathcal I_1$.
Fix $i \in \mathcal I_1$ and any perturbation ${ h} = h_i e_i$, $h_i \in \mathbb R$, along the coordinate $i$ (here, $e_i$ is the $i^{th}$ vector of the canonical basis).  Consider the rank-one operator $t_i = T \mathcal P_i$, where we use $\mathcal P_i$ to denote the orthogonal projection onto the one-dimensional subspace spanned by $e_i$.  Observe that $|\ t_i u \| = | (u)_i | \| t_i \|$.    Since $t_i$ is orthogonal to the argument ${\cal P}_1^{\perp} (T_0 u^*_0 - { g})$ under the $\ell_2$ penalty in $\eqref{atmin}$, the minimality condition ${\cal J}^p(u^*) \leq {\cal J}^p(u^* + h)$ can be written as
\begin{eqnarray}
&& \| {\cal P}_1^{\perp} (T_0 u^*_0 - { g}) \|^2_{\ell_2(\mathcal K)}+ \|u_0^* \|_{\ell_p^{\mathcal I_0}(\mathcal I)}^p+ |\mathcal I_1|r^p \nonumber \\
&\leq&    \| {\cal P}_1^{\perp} (T_0 u^*_0 - { g}) \|^2_{\ell_2(\mathcal K)} + \|u_0^* \|_{\ell_p^{\mathcal I_0}(\mathcal I)}^p \nonumber \\
&& \phantom{XXXXX} + \|h_i t_i\|^2_{\ell_2(\mathcal I)} + \min\{r^p, |u^*_i + h_i|^p\} + r^p (|\mathcal I_1| -1) 
\end{eqnarray}
which is equivalent to the condition that
\begin{eqnarray}
r^p &\leq& \|h_i t_i \|^2_{\ell_2(\mathcal I)} + \min\{r^p, |u^*_i + h_i|^p \}
\label{up}
\end{eqnarray}
hold for all $h_i \in \mathbb{R}$.  Now, since $\|T\|<1$, it follows that $\|t_i\| \leq 1$, and $\eqref{up}$ implies that
\begin{eqnarray}
r^p &\leq& h_i^2 + \min\{r^p, |u^*_i + h_i|^p \}
\label{down}
\end{eqnarray}
holds for all $h_i \in \mathbb{R}$, or, after the change of variables $\alpha = u^*_i + h_i$, that
\begin{equation}
r^p \leq (\alpha - u^*_i)^2 + \min\{r^p, |\alpha|^p \} 
\label{alpha}
\end{equation}
holds for all $\alpha \in \mathbb{R}$.  In particular, the inequality $\eqref{alpha}$ must hold at the value $\alpha^*$ that minimizes the right-hand-side.  But we already know from Proposition $\ref{thmthrs}$ that such a minimizer $\alpha^*$ is of the form:
\begin{eqnarray}
\alpha^* &=& \left\{
\begin{array}{ll}
F_p^{-1}(u^*_i), &  |u^*_i| \leq \lambda'(r,p)  \\
u^*_i, & |u^*_i|  > \lambda'(r,p) \\  
\end{array}
\right.
\end{eqnarray}
Now, suppose $|u^*_i| < \lambda'(r,p).$ (We know that $|u^*_i| > r$, so then $r < |u^*_i| < \lambda'(r,p)$).    From the proof of Proposition $\ref{thmthrs}$ we know that the function $F_p^{-1}(\lambda)$ is increasing, so then $\alpha^* = F_p^{-1}(u^*_i) < F_p^{-1}(\lambda') < r$.  Since also $S_p$ is strictly increasing, it follows that $S_p(\alpha^*) < S_p(F_p^{-1}(\lambda')) \leq S_p(\lambda') = r^p$. In the last inequality we used \eqref{boundinvF}. (See also Table 1 for recalling the notations used here.) 
But this is a contradiction to the minimality condition, $\eqref{alpha}$, and so we must conclude that $|u^*_i| \geq \lambda'(r,p)$.  
\item We now show that $|u^*_i| \leq F_p^{-1}(\lambda'(r,p))$, if $|u^*_i| \leq r$.   Recall that for $i \in \mathcal I_0$, the coefficient $u^*_i$ satisfies the fixed point condition,
\begin{equation}
 \big[ T^*(g - T{ u}^*) \big]_i  = \frac{p}{2}\sgn{u^*_i}|u^*_i|^{p-1}.
 \label{fp0}
\end{equation}
Fix $i \in \mathcal I_0$, and consider as before any perturbation ${ h} = h_i e_i$ along the coordinate $i$, $h_i \in \mathbb R$.  Let $t_i$ be the rank-one operator as defined before.  Then, the minimality condition ${\cal J}^p_r(u^*) \leq {\cal J}^p_r(u^* + h)$ is easily seen to be equivalent to
\begin{eqnarray}
\|Tu^* - g\|^2_{\ell_2(\mathcal K)} + |u^*_i|^p &\leq&  \|Tu^* - g + h_i t_i\|^2_{\ell_2(\mathcal K)} \nonumber \\
 &&+ \min\{r^p, |u^*_i + h_i|^p\} \nonumber \\
 &=& \|Tu^* - g \|^2_{\ell_2(\mathcal K)} + \|h_i t_i \|^2_{\ell_2(\mathcal I)} + 2 h_i \langle t_i,Tu^* - g \rangle \nonumber \\
 &&+ \min\{r^p, |u^*_i + h_i|^p\} \nonumber \\
 &=& \|Tu^* - g \|^2_{\ell_2(\mathcal K)} + \|h_i t_i \|^2_{\ell_2(\mathcal I)} - 2 h_i\frac{p}{2}\sgn{u^*_i}|u^*_i|^{p-1}  \nonumber \\
 && + \min\{r^p, |u^*_i + h_i|^p\} 
 \label{string}
\end{eqnarray}
and the final equality follows directly from the fixed point condition $\eqref{fp0}$.  Now the chain of inequalities $\eqref{string}$ implies the minimality condition
\begin{eqnarray}
|u^*_i|^p &\leq&  \|h_i t_i \|^2_{\ell_2(\mathcal I)} - 2h_i\frac{p}{2}\sgn{u^*_i}|u^*_i|^{p-1} + \min\{r^p, |u^*_i + h_i|^p\}  \nonumber \\
 &\leq& h_i^2  - 2h_i\frac{p}{2}\sgn{u^*_i}|u^*_i|^{p-1} + \min\{r^p, |u^*_i + h_i|^p\},
\end{eqnarray}
or, again using the change of variables  $\alpha = u^*_i + h_i$, the inequality
\begin{equation}
|u_i^*|^p \leq (\alpha - u^*_i)^2 - 2(\alpha - u^*_i)\frac{p}{2}\sgn{u^*_i}|u^*_i|^{p-1} + \min\{r^p, |\alpha|^p\}.
\label{again}
\end{equation}
Again, the inequality $\eqref{again}$ should hold for all $\alpha$ by the minimality of $u^*$.   Minimizers $\alpha^*$ of the right-hand-side of $\eqref{again}$ also are minimizers of  
\begin{equation}
\big(\alpha - (u^*_i + \frac{p}{2}\sgn{u^*_i}|u^*_i|^{p-1})\big)^2 + \min\{r^p, |\alpha|^p\},
\end{equation}
which we know to have the form
\begin{eqnarray}
\alpha^* &=& \left\{
\begin{array}{ll}
F_p^{-1}(u^*_i + \frac{p}{2}\sgn{u^*_i}|u^*_i|^{p-1}), &  |u^*_i| + \frac{p}{2}|u^*_i|^{p-1} \leq \lambda'(r,p)  \\
u^*_i + \frac{p}{2}\sgn{u^*_i}|u^*_i|^{p-1}, & |u^*_i| + \frac{p}{2}|u^*_i|^{p-1}  > \lambda'(r,p) \\  
\end{array}
\right. .
\end{eqnarray}
But $u^*_i + \frac{p}{2}\sgn{u_i^*}|u_i^*|^{p-1} = F_p(u_i^*)$, so the above reduces to
\begin{eqnarray}
\alpha^* &=& \left\{
\begin{array}{ll}
u^*_i, &  F_p(u_i^*) \leq \lambda'(r,p)  \\
F_p(u_i^*), & F_p(u^*_i) > \lambda'(r,p) \\  
\end{array}
\right.
\end{eqnarray}
As before, the proof proceeds by contradiction.  Suppose that $F_p(u_i^*)  > \lambda'(r,p)$, so that $\alpha^* = F_p(u_i^*) > \lambda'(r,p)$ and $S_p(\alpha^*) > S_p(\lambda') = r^p$.   Note that, by recalling $F_p(u_i^*) = u_i^* - \frac{p}{2} \sgn(u_i^*) (u_i^*)^{p-1}$, we have
\begin{equation}
 S_p(\alpha^*) = (u_i^* - F_p(u_i^*))^2 + |u_i^*|^p = |u_i^*|^p + \frac{p^2}{4}|u_i^*|^{2p-2}.
\label{bad}
\end{equation}  
Plugging $\alpha^*$ into the right-hand-side of $\eqref{again}$, noting that $\lambda'(r,p) > r$ so that $|\alpha^*| \geq r$, and rearranging, yields the inequality
\begin{equation}
|u^*_i|^p \leq r^p -\frac{p^2}{4}|u_i^*|^{2p-2} \mbox{ or }  S_p(\alpha^*) = |u_i^*|^p + \frac{p^2}{4}|u_i^*|^{2p-2} \leq r^p.
\end{equation}
But this contradicts the assumption that the expression in $\eqref{bad}$ be larger than $r^p$.
\end{enumerate}


\bibliography{thesis.bib}


\end{document}